\documentclass[aihp,reqno]{imsart}
\RequirePackage{amsthm,amsmath,amsfonts,amssymb}
\RequirePackage[numbers]{natbib}
\RequirePackage[colorlinks,citecolor=blue,urlcolor=blue]{hyperref}
\RequirePackage{graphicx}
\RequirePackage[colorlinks,citecolor=blue,urlcolor=blue]{hyperref}
\usepackage[title]{appendix}
\usepackage{xkeyval}
\usepackage[utf8]{inputenc}
\usepackage[T1]{fontenc}
\usepackage[english]{babel}
\usepackage{csquotes}
\usepackage{enumerate}
\usepackage{bm}
\usepackage{verbatim}
\usepackage{indentfirst}
\usepackage{xcolor}
\usepackage{mathtools}
\usepackage{framed}
\usepackage{amsmath}
\usepackage{amsthm}	
\usepackage{amsfonts}
\usepackage{mathrsfs}	
\usepackage{amssymb}
\usepackage{dsfont}
\usepackage{bbm}
\allowdisplaybreaks

\newcommand{\sld}[1][]{  \frac{\delta^2{#1}}{\delta m^2}  }

\newcommand{\muinit}{\mu_{\textrm{\rm init}}}

\newcommand{\rvlaw}[1][]{  {\cL}{{(#1)}}  }
\newcommand{\nlaw}[1][]{  \mu^N_{{#1}}  }
\newcommand{\ld}[1][]{ \frac{\delta{#1}}{\delta m}  }
\DeclarePairedDelimiterX{\inp}[2]{\langle}{\rangle}{#1, #2}
\newcommand{\intrd}[1][]{ \int_{\bT^d}  }

\newcommand{\hregb}[2]{\textcolor{red}{(\text{Reg}-$b$-(${#1,#2}$))}}
\newcommand{\hintb}[2]{\textcolor{red}{(\text{Reg}-$b$-(${#1,#2}$))}}
\newcommand{\hlipb}[2]{\textcolor{red}{(\text{Lip}-$b$-(${#1,#2}$))}}

\newcommand{\hintphi}[2]{\textcolor{red}{(\text{Reg}-$\Phi$-(${#1,#2}$))}}

\newcommand{\hherg}[6]{\textcolor{red}{(\text{Erg}-({#1,#2})-[{#4,#5}])}}
\newcommand{\CLinear}[3]{\textcolor{red}{\text{Linear}-[#1,#2,#3]}}
\newcommand{\CSource}[3]{\textcolor{red}{\text{Source}-[#1,#2,#3]}}

\newcommand{\hergcoeff}[3]{\textcolor{red}{(\text{Erg}-({#1,#2}))}}
\newcommand{\ahergcoeff}[3]{\textcolor{red}{(\text{Erg}-({#1,#2}))}}
\newcommand{\hlocalcoeff}[3]{\textcolor{red}{(\text{Local}-({#1,#2,#3}))}}
\newcommand{\hergl}[4]{\textcolor{red}{(\text{Local}-({#1,#2,#3})-[{#4}])}}

\newcommand{\hergo}{\textcolor{red}{(\text{Erg})}}
\newcommand{\hergol}{\textcolor{red}{(\text{Local})}}
\newcommand{\ahergo}{\textcolor{red}{(\text{Erg})}}
\newcommand{\bE}{\mathbb{E}}

\newcommand{\bN}{\mathbb{N}}

\newcommand{\bR}{\mathbb{R}}

\newcommand{\bT}{\mathbb{T}}
\newcommand{\cL}{\mathcal{L}}

\newcommand{\cP}{\mathcal{P}}

\newcommand{\cU}{\mathcal{U}}
\newcommand{\cV}{\mathcal{V}}

\def\ud{\mathrm{d}}

\def \eps {\epsilon}

 \def\i{\bot}

\def\R{\mathbb{R}}
\def\one{{\mathbbm 1}}

 \renewcommand{\i}{{\mathrm{i}}}
\makeatletter
\DeclareFontFamily{OMX}{MnSymbolE}{}
\DeclareSymbolFont{MnLargeSymbols}{OMX}{MnSymbolE}{m}{n}
\SetSymbolFont{MnLargeSymbols}{bold}{OMX}{MnSymbolE}{b}{n}
\DeclareFontShape{OMX}{MnSymbolE}{m}{n}{
    <-6>  MnSymbolE5
   <6-7>  MnSymbolE6
   <7-8>  MnSymbolE7
   <8-9>  MnSymbolE8
   <9-10> MnSymbolE9
  <10-12> MnSymbolE10
  <12->   MnSymbolE12
}{}
\DeclareFontShape{OMX}{MnSymbolE}{b}{n}{
    <-6>  MnSymbolE-Bold5
   <6-7>  MnSymbolE-Bold6
   <7-8>  MnSymbolE-Bold7
   <8-9>  MnSymbolE-Bold8
   <9-10> MnSymbolE-Bold9
  <10-12> MnSymbolE-Bold10
  <12->   MnSymbolE-Bold12
}{}

\let\llangle\@undefined
\let\rrangle\@undefined
\DeclareMathDelimiter{\llangle}{\mathopen}%
                     {MnLargeSymbols}{'164}{MnLargeSymbols}{'164}
\DeclareMathDelimiter{\rrangle}{\mathclose}%
                     {MnLargeSymbols}{'171}{MnLargeSymbols}{'171}
\makeatother

\newcommand{\lev}{\left\langle}
\newcommand{\rev}{\right\rangle}

\newcounter{daggerfootnote}

\newcounter{starfootnote}

\theoremstyle{plain}
\newtheorem{theorem}{Theorem}[section]
\newtheorem{lemma}[theorem]{Lemma}
\newtheorem{proposition}[theorem]{Proposition}

\newtheorem{corollary}[theorem]{Corollary}
\theoremstyle{definition}
\newtheorem{definition}[theorem]{Definition}

\theoremstyle{definition}
\newtheorem{remark}[theorem]{Remark}
\newtheorem*{theorem*}{Theorem (Main result)}
\numberwithin{equation}{section}
\numberwithin{figure}{section}
\begin{document}
%\selectlanguage{english}
\begin{frontmatter}\title{Uniform in time weak propagation of chaos on the torus}
\runtitle{Uniform in time weak propagation of chaos on the torus}

\begin{aug}
%%%%%%%%%%%%%%%%%%%%%%%%%%%%%%%%%%%%%%%%%%%%%%
%%Only one address is permitted per author. %%
%%Only division, organization and e-mail is %%
%%included in the address.                  %%
%%Additional information can be included in %%
%%the Acknowledgments section if necessary. %%
%%%%%%%%%%%%%%%%%%%%%%%%%%%%%%%%%%%%%%%%%%%%%%
\author[A]{Fran\c{c}ois Delarue}
\and
\author[B]{Alvin Tse}
%%%%%%%%%%%%%%%%%%%%%%%%%%%%%%%%%%%%%%%%%%%%%%
%% Addresses                                %%
%%%%%%%%%%%%%%%%%%%%%%%%%%%%%%%%%%%%%%%%%%%%%%
\address[A]{Universit\'e C\^ote d'Azur, CNRS, Laboratoire J.A.Dieudonn\'{e},
Parc Valrose,
France-06108 NICE Cedex 2,
 francois.delarue@univ-cotedazur.fr}

\address[B]{Universit\'{e} Paris-Est, Cermics (ENPC), INRIA, F-77455 Marne-la-Vall\'{e}e, France, alvin.tse@enpc.fr \\
Hong Kong University of Science and Technology, Clear Water Bay, Hong Kong, alvintse@ust.hk}
\end{aug}

%\date{ \currenttime, \ddmmyyyydate\today}

%%%%%%%%%%%%%%%%%%%%%%%%%%%%%%%%%%%%%%%%%%%%%%%%%%%%%%%%%%%%%%%%%%%%%%%%%%%%%%%%%%%%%%%%%%%%%%%%%%%%%%%%%%%%%%%%%%%%%%%%
%%%%%%%%%%%%%%%%%%%%%%%%%%%%%%%%%%%%%%%%%%%%%%%%%%%%%%%%%%%%%%%%%%%%%%%%%%%%%%%%%%%%%%%%%%%%%%%%%%%%%%%%%%%%%%%%%%%%%%%%

\begin{abstract}
 We address the long time behaviour of weakly interacting diffusive particle systems on the $d$-dimensional torus. Our main result is to show that, under certain mild regularity conditions, the weak error between the empirical distribution of the particle system and the limiting theoretical law (governed by a 
Fokker-Planck equation) is of the order ${\mathcal O}(1/N)$, uniform in time on $[0,\infty)$, where $N $ is the number of particles in the interacting diffusion. This comprises 
Fokker-Planck equations with a globally attracting invariant measure for which 
the linearisation at the invariant measure 
enjoys appropriate ergodic properties. 
 Our approach relies on a systematic analysis of the long-time behaviour of the derivatives of the semigroup generated by the Fokker-Planck equation. This strategy is flexible enough to cover a wider broad of situations, including 
 the super-critical Kuramoto model, for which the corresponding Fokker-Planck equation has several invariant measures. 
 \end{abstract}



\end{frontmatter}

\section{Introduction} 

%\textit{Our model.}
In this paper, we are concerned with 
the large size and the large time behaviour of 
a weakly interacting particle system 
with toroidal data. Denoting by $N$ the number of particles, the system has the following generic form 
\begin{equation} 
\begin{cases}
       Y^{i,N}_t= \eta^i + \int_0^t  b \big( Y^{i,N}_s, \mu^{N}_s \big)  \, \ud s +  W^i_t, \quad i \in \{1,\cdots,N\}, \quad t \geq 0, 
       \\
               \mu^{N}_s := \frac{1}{N} \sum_{i=1}^N \delta_{Y^{i,N}_s},
               \end{cases} \label{eq particles}
\end{equation}
where $b$ is an $\bR^d$-valued function defined on $\bT^d \times \cP(\bT^d)$, 
$\bT^d:={\mathbb R}^d/ {\mathbb Z}^d$ denoting the $d$-dimensional torus and $\cP(\bT^d)$ the space of probability measures on $\bT^d$, which we equip (unless specified differently) with the ${\mathcal W}_{1}$-Wasserstein distance 
\begin{equation}
\label{eq:W1:Wasserstein}
{\mathcal W}_{1}(\mu,\nu) = \inf_{\pi} \biggl\{ \int_{\bT^d \times \bT^d} d_{\bT^d}(x,y) \ud \pi(x,y) \biggr\},
\end{equation}
the infimum being over all the probability measures $\pi$ on the product space 
$\bT^d \times \bT^d$ that have $\mu$ and $\nu$ as respective marginal measures. 
In \eqref{eq particles}, $W^i$, $i=1,\cdots,N$, are independent $d$-dimensional Brownian motions and  $\eta^i,$ $i=1,\cdots,N$, are $N$ $\bR^d$-valued random variables, with the two tuples 
$(\eta^{1},\cdots,\eta^{N})$ and $(W^1,\cdots,W^N)$ being independent. 
Most of the time, the random variables $\eta^i$, $i=1,\cdots,N$,
are also assumed to be independent and identically distributed (I.I.D.) with a common law ${\mu_{\text{init}}}$, but this might not be the case in some of our results (in those cases we emphasise it very clearly). In physical applications, this type of processes arises when we consider interacting particle systems with periodic boundary conditions. (See, for example, \cite{villani2006hypocoercive, constantin,kuramoto81,kuramoto88}.)
\subsection{State of the art}

%\textit{McKean-Vlasov SDEs and nonlinear Fokker-Planck equations.} 
%In \eqref{eq particles}, every individual particle $Y^{i,N}$ evolves according to some diffusive motion depending on the behaviour of others.
It is well-known that as the population size $N$ grows to infinity, \eqref{eq particles}, when subjected to I.I.D. initial conditions $\eta^1,\cdots,\eta^N$, behaves like the following McKean-Vlasov SDE 
%(an SDE in which the coefficients depend on the evolving law):
    \begin{equation} 
      X_t  = \eta + \int_0^t b\bigl(X_s, \rvlaw[X_s]\bigr) \, \ud s +   W_t, \quad t \geq 0\ ; \quad  
              \rvlaw[\eta]:= \text{Law}(\eta) = {\mu_{\text{init}}},  \label{eq:MVSDE} \end{equation}
where $(\eta,W)$ is a copy of $(\eta^1,W^1)$. Existence and uniqueness of a solution to 
   both \eqref{eq particles} and \eqref{eq:MVSDE}
   is known if $b$ is globally bounded and merely Lipschitz continuous in the measure argument with respect to the total variation distance  (see \cite{jourdain97,lacker18,mishuraverten} and the references therein). 
Moreover, the flow of marginal laws $(m(t \, ; {\mu_{\text{init}}}):=\rvlaw[X_{t}])_{t \geq 0}$ satisfies (at least in a distributional sense) the \textit{nonlinear} Fokker-Planck equation:
\begin{equation}
 \label{eq: forward eqn }
\partial_{t} m(t \, ; \mu) = \frac12 \Delta m(t \, ; \mu) - \textrm{\rm div} \Bigl[ m(t \, ; \mu) b\bigl(\cdot,m(t \, ; \mu)\bigr) \Bigr], \quad t \geq 0 \ ; \quad m(0 \, ; \mu) =\mu.
\end{equation}

%\textit{Propagation of chaos.}
Asymptotically, any finite subset
of particles becomes independent of each other. This phenomenon is known as \emph{propagation of chaos}. Precisely, 
on any finite time interval $[0,T]$ and, for any fixed $k \in \bN$,
$ (Y^{1,N}, \ldots, Y^{k,N}) \implies (X^1, \ldots, X^k)$,  as $N \to \infty,$ 
where $\{X^i \}_{i \in \bN}$ are i.i.d. copies of \eqref{eq:MVSDE} and `\hspace{-4pt}$\implies$\hspace{-4pt}' denotes weak convergence on the space $C([0,T], (\bT^d)^k)$.   The main reference in this direction is \cite{sznitman1991topics}, where 
propagation of chaos is proved by means of a coupling argument. 
The proof works for a jointly Lipschitz drift $b$ (w.r.t. ${\mathcal W}_1$ in the measure argument) 
and yields a quantitative convergence estimate which we describe in the next paragraph. Another result from \cite{sznitman1991topics} asserts that propagation of chaos  is equivalent to weak convergence of the measure-valued random variables $(\mu^N_{t})_{0 \le t \le T}$ to the limiting laws $(\cL(X_{t}))_{0 \le t \le T}$. This paves the way for another approach 
consisting in proving tightness of $(\pi^N_{t} := \rvlaw[ \mu^N_{t}] \in \mathcal P (\mathcal P (\bT^d)))_{0 \le t \le T}$, see \cite{gartner1988mckean,meleard1996asymptotic,sznitman1991topics}. 
%for classical results in this direction. %For sure, this strategy of proof does not reveal quantitative bounds regarding the rate of convergence  of \eqref{eq particles} to \eqref{eq:MVSDE}. 
\vskip 4pt 

\textit{Errors in finite time.}
The quantitative analysis of propagation of chaos can be carried out in various ways, whether the geometry of the state space is Euclidean or toroidal (it is only when it comes to the long time behaviour of 
\eqref{eq particles} 
and
\eqref{eq:MVSDE}
that compactness of the torus makes a substantial difference.)
 For instance, for a given metric on the space of probability measures, one may simply estimate the distance between the marginal law of $Y^{1,N}_t$ and the measure   
${\mathcal L}(X_t)$, for a fixed $t$ in some finite interval $[0,T]$. More generally, one may compare the joint law of the $k$ first particles $(Y^{1,N}_t,\cdots,Y_t^{k,N})$ with the product measure 
${\mathcal L}(X_t)^{\otimes k}$. Another approach is to compare the empirical measure $\mu^N_t$ with 
${\mathcal L}(X_t)$. For sure, one may also address the supremum of any of these distances over $t \in [0,T]$, which is very similar to what is done in the analysis of the strong error 
for discretisation schemes of SDEs. 
For instance, in the case where $b$ depends on the measure component linearly, i.e., is of the form $b(x, \mu) := \int_{\bR^d} B(x,y) \, \mu(\ud y),$ with $B$ being Lipschitz continuous in both variables, it  follows from a simple calculation 
(\cite{sznitman1991topics}) that $\sup_{t \in [0,T]} {\mathcal W}_2(\cL(Y^{1,N}_t), \cL(X_t))={\mathcal O}(N^{-1/2})$, 
where here and throughout 
${\mathcal O}(\cdot)$ stands for the big ${\mathcal O}$ Landau notation. 
This result has been improved in several contributions. Notably, 
the $1$-Wasserstein distance between ${\mathcal L}(Y_t^{1,N},\cdots,Y_t^{k,N})$ 
and 
${\mathcal L}(X_t)^{\otimes k}$
has been shown to be ${\mathcal O}((k/N)^2)$
in the recent work \cite{lacker2022hierarchies}, 
the proof relying on the analysis of the relative entropy between both laws and on the so-called BBGKY hierarchy. 
With similar tools, 
quantitative estimates are established for models with singular interactions
in \cite{Jabin}. Models with singular interactions have been also treated by means of the modulated energy method, in which the metric used for studying the convergence is adapted to the form 
of the interactions, see for instance \cite{Serfaty2017}. We refer to \cite{BJW} for combinations of relative entropy and modulated energy methods. 
When $b$ has a general (but regular) measure dependence, the rate of convergence deteriorates with the dimension $d$, since it is then needed to estimate the Wasserstein distance between the empirical law of I.I.D. samples and the limiting measure. This follows from results such as \cite{dereich2013constructive} or \cite{fournier2015rate} in which the dimension explicitly shows up. Dimension-free rates may be retrieved in this more general setting by requiring a strong form of smoothness of the drift $b$ with respect to the measure argument
(see \cite[Lemma 5.10]{delarue2019}
and \cite{szpruch2019antithetic}). 
%Certainly, the analysis of the strong error is not limited to models with smooth interactions; we refer for instance to the recent work 
% Also, it is worth saying that the strong error may be also quantified through a relevant form of central limit theorem (see \cite{meleard1996asymptotic,sznitman1985fluctuation,
%tanaka1981central}, or in the perspective of large deviation principle, see \cite{budhiraja-dupuis-fischer,
%dawson-gartner,sznitman1985fluctuation}). 

The rate of convergence can be also addressed by testing the statistical distribution of the empirical measure against real-valued functions 
defined on ${\mathcal P}(\bT^d)$, i.e., by estimating quantities of the form
\begin{equation}
\label{eq:intro:weak}
\Big| \bE[ \Phi(\mu^N_t)] -  \Phi\bigl({\mathcal L}(X_t)) \bigr) \Big|,  
\end{equation}
where $\Phi:\cP(\bT^d) \to \bR$ is a test functional chosen within a suitable class. 
For instance, $\Phi(\cdot) = {\mathcal W}_2(\cdot, {\mathcal L}(X_t))$ may be one such test functional, but more regularity may be required on $\Phi$ to obtain relevant bounds. 
Accordingly, 
\eqref{eq:intro:weak} should be understood as a weak error for the law of $\bar \mu^N_t$ when acting on a given class of test functionals $\Phi$. 
This direction of research has been introduced in independent works \cite{bencheikh2019bias,kolokoltsov2010nonlinear,mischler2013kac,mischler2015new}, for various forms of test functionals $\Phi$. Among others, $\Phi$ is a linear function in 
\cite{bencheikh2019bias}, i.e. $\Phi(\mu):=\int_{\bT^d}F(x)\mu(\ud x)$ for some function $F:\bT^d \rightarrow \bR$; $\Phi$ is a polynomial function in 
\cite{mischler2013kac,mischler2015new}, i.e. a product of linear functions; 
and, $\Phi$ is a quite general nonlinear function in 
\cite{kolokoltsov2010nonlinear}. Under appropriate smoothness conditions on the test functional and on the coefficients of \eqref{eq particles}, 
this gives a rate of convergence of ${\mathcal O}(1/N)$, plus the error due to the approximation of the functional of the initial law (as for the latter, see \cite[Lem. 4.6]{mischler2015new} for a dimension-dependent estimate and \cite[Th. 2.11]{chassagneux2019weak} for an ${\mathcal O}(1/N)$ bound). The key idea of the analysis (highlighted in  
\cite[Th. 9.2.1]{kolokoltsov2010nonlinear}
and
\cite[Th. 6.1]{mischler2015new}) is to work with a semigroup that acts on the space of functions of measures (expounded in the next section). A similar idea has been used in 
\cite{cardaliaguet2019master} in order to study the convergence problem for mean field games, up to the difference that the equation for the semigroup then becomes a nonlinear equation. 
Also, another 
recent work \cite{chassagneux2019weak} provides an extension of \cite{bencheikh2019bias,kolokoltsov2010nonlinear,mischler2015new} in the form 
of a weak error expansion
and the companion works
\cite{,MR4377993,CHAUDRUDERAYNAL20221} address cases with coefficients having lower H\"older regularity in the spatial variable. 
\vskip 4pt

\textit{Long time analysis.}
Propagation of chaos is said to be uniform whenever the quantitative estimates (of propagation of chaos) are uniform in time. This problem is hence more challenging as it involves two large parameters, $N$ and $t$. It is thus related to the long time behaviour of the McKean-Vlasov and non-linear Fokker-Planck equations 
\eqref{eq:MVSDE} and \eqref{eq: forward eqn } themselves. 

Actually, the ergodic analysis of McKean-Vlasov and corresponding nonlinear Fokker-Planck equations has been an intense topic of research on its own for more than twenty years. In the earlier papers 
\cite{benachourI,benachourII,benedetto,bolley,cattiaux,mccann}, convergence towards a (unique) invariant measure has been mostly studied for drifts of the form $b(x,\mu) = - \nabla V(x) - 
\int \nabla W(x-y) \mu( \ud y)$, with confinement and interaction potentials 
$V$ and $W$ satisfying suitable convexity properties, $W$ being symmetric.
In this framework, 
a key conceptual feature is that the Fokker-Planck equation \eqref{eq: forward eqn } can be regarded as a gradient flow on the space of probability measures.
However, as demonstrated 
in 
\cite{herrmann},  
uniqueness of the invariant measure may be easily lost under a small modification of the shape of the potentials, which obviously raises challenging questions about the long-time behaviour of the Fokker-Planck equation. In fact, regardless of the precise form of the drift $b$, a possible strategy to force uniqueness and in turn to get convergence towards the hence unique invariant measure is to assume that the mean field interaction is small enough (see for instance 
\cite{bogachev,butkovsky,eberle}). 
When convergence towards a unique invariant measure is no longer true, a case-by-case stability analysis of all the existing stationary solutions may be carried out, depending on the shape of the dynamics. 
We refer for instance to  
\cite{bertini:giacomin:pakdaman,giacomin:pakdaman:khashayar:pellegrin}
for results on the super-critical (toro\"idal) Kuramoto model, which we revisit in Section 
\ref{se:kuramoto}. More examples may be found in \cite{carrillo:gvalani:pavliotis:schlichting,degond,tugaut}.
 
Uniform in time propagation of chaos
is even more challenging. To wit, it may fail even in cases where there exists a globally attracting invariant measure to the Fokker-Planck equation \eqref{eq: forward eqn }. In the Euclidean setting, 
a well-known example by now may be found in \cite{MR1970276}. Therein, the center of mass of the $N$ particles in \eqref{eq particles} is shown to behave (for a suitable choice of $b$)  like a Brownian motion of intensity of $1/\sqrt{N}$, which becomes macroscopic in size for time 
$t$ larger than $N$. 
%A similar phenomenon may occur on the torus when the invariant measure is not the Lebesgue measure\footnote{\label{foo:invariant:leb}If the invariant measure is the Lebesgue measure, it is invariant by rotation. Accordingly, propagation of chaos may be uniform even when the center of mass is rotating fast.}; 
%once again, a prototype example is 
%the super-critical Kuramoto model (see for instance 
%\cite{bertini:giacomin:poquet}). 
That said, several positive results have already been proven under relevant conditions. 
For instance, the same example as in 
\cite{MR1970276}, but with an additional confining convex potential, 
is shown to satisfy 
 $\sup_{t \geq 0} {\mathcal W}_1(\cL(Y^{1,N}_t), \cL(X_t))={\mathcal O}(N^{-1/2})$, see 
 the earlier work
 \cite{MR1847094}
 together with the recent contribution 
 \cite{lacker2023sharp} in which the bound is improved into ${\mathcal O}(N^{-1})$. 
In this example, $b$ takes the aforementioned special form  
$b(x,\mu) = - \nabla V(x) - 
\int \nabla W(x-y) \mu(\ud y)$, 
with $V$ and $W$ satisfying strong convexity conditions.
In \cite{guillin2022systems,MR4564418}, the potential $W$ is even allowed to be singular, with the uniform in time estimates being possibly established under weaker distances and with weaker rates. 
In 
\cite{chen2023uniformintime}, the gradient flow structure is addressed in a more systematic manner, but assuming the potential on top of the gradient to be convex with respect to the measure argument in a functional sense. 
Outside the convex regime, quantitative bounds may be obtained under relevant conditions that force $W$ to be small enough, see  \cite{durmus2018elementary,salem2018gradient}. 
From a different perspective,  close to \cite{Jabin}, 
the authors of
\cite{guillin2021uniform}
have addressed on the torus the case when $\nabla W$ is replaced by a possibly singular divergence free vector field. 
Lastly, 
for uniform propagation of chaos when $b$ is a general drift with a small enough McKean-Vlasov dependence, 
we refer again to 
\cite{lacker2023sharp}, 
in which  
the authors extend 
\cite{lacker2022hierarchies}
and show 
that
$\sup_{t \geq 0} {\mathcal W}_1(\cL(Y^{1,N}_t,\cdots,Y_t^{k,N}), \cL(X_t)^{\otimes k})$ is  ${\mathcal O}((k/N)^2)$, 
and to 
 \cite{arnaudon2020}, in which  
 the error 
\eqref{eq:intro:weak}
is shown to be ${\mathcal O}(N^{-1} + N^{-1/d} \exp(-\lambda t))$, for some $\lambda >0$.

\subsection{Our contribution.}
Our contribution here is to address the weak error 
\eqref{eq:intro:weak}
when $\Phi$ is a general (smooth enough) test functional and 
to provide an (almost) ${\mathcal O}(N^{-1})$ bound for it under suitable generic assumptions on the long time behaviour of
\eqref{eq:intro:weak}, with a particular emphasis on the role of the invariant measures to \eqref{eq:intro:weak}. 

The first main statement in this regard is Theorem 
\ref{thm main result:2} 
below, which asserts among others that such a bound holds true for a bounded drift $b$ that is sufficiently smooth in the measure argument 
under the following two requirements: $(a)$ the equation 
\eqref{eq:intro:weak} has a (hence unique) globally attracting invariant measure $\nu_\infty$; $(b)$ the linearised 
version of  
\eqref{eq:intro:weak}
at 
$\nu_\infty$
 has ergodic properties, which are spelled out in the introduction of 
Section 
\ref{se:3}, see 
\hergo. 
We refer (for a tiny example) to 
\cite{cormier2023stability,dolbeaut,mischler2013kac,mischler2015new}
for
earlier (and related) uses 
of the linearised version of the nonlinear Fokker-Planck equation.   
We provide several examples of applications: $(i)$ when $b$ is a general function with a small enough 
dependence with respect to the measure argument $\mu$; $(ii)$ 
when 
 $b$ derives from an $H$-stable (periodic symmetric) potential, see \cite{ruelle}, i.e., all the Fourier coefficients of $W$ are non-negative;
 $(iii)$
 when $b$ is divergence-free in $x$. 
Obviously, 
similar cases have been already addressed, although in somewhat different contexts, in the aforementioned references 
(see \cite{durmus2018elementary,arnaudon2020} 
for $(i)$, 
\cite{chen2023uniformintime}
for $(ii)$, which is somehow the analogue of the convex case treated therein but on the torus, 
\cite{guillin2021uniform}
for $(iii)$). 
While this may seem a fair criticism of our work, it is interesting to note that our statement covers all three cases. 

What is more, our tools allow us to take the analysis a step further and to obtain local results under locally attracting properties of the invariant measures.
For instance, we obtain metastability bounds (i.e., quantitative propagation of chaos over time interval of any polynomial length in $N$), see
again 
Theorem \ref{thm main result:2} (together with 
Theorem \ref{prop:3:18} for a refined version), when 
$\nu_\infty$ in
the former condition $(a)$ is 
simply assumed to be an 
invariant measure, possibly not globally attracting (in fact, it becomes locally attracting under $(b)$). 
Moreover, in Theorem \ref{main:thm:kuramoto}, we provide a quantitative in time estimate for one specific model, namely the super-critical Kuramoto model, with an infinite set of invariant measures. 
It corresponds to $b(x,\mu)= - 
\int \nabla W(x-y) \mu( \ud y)$ with 
$W(x)=-2 \pi \kappa \cos(2 \pi x)$ for $\kappa>1$.
This case is challenging and does not fit any of the aforementioned situations $(i)$, $(ii)$ or $(iii)$. 
The invariant measures are the Lebesgue measure, which is unstable, and a collection of non-trivial measures obtained by rotating a common density profile on the torus, see 
\cite{bertini:giacomin:pakdaman,giacomin:pakdaman:khashayar:pellegrin}. Uniform propagation of chaos then fails, see 
\cite{bertini:giacomin:poquet} and the best result that has been proven so far is due to \cite{coppini2019long}: it says that, with high probability, the empirical distribution of 
\eqref{eq particles} stays close to the collection of non-trivial invariant 
measures up to times that are subexponential in $N$. 
Here, we prove that the weak error 
\eqref{eq:intro:weak}
is of order ${\mathcal O}(N^{-1})$, uniformly in time, for functionals $\Phi$ that are rotation invariant (see Definition \ref{def:rotation:invariant:function}) and for initial distributions $\mu_{\text{init}}$ that are at a positive distance from the Lebesgue measure. This result is new and, although it does not directly follow from our main statement Theorem \ref{thm main result:2} , it follows from the same approach. 

In all our results, the function $\Phi$ is a sufficiently smooth `nonlinear' functional on ${\mathcal P}(\bT^d)$. To make it clear,  $\Phi$ has two `linear functional' (or `flat') derivatives with respect to the measure argument (see Section \ref{se:2} for the details) that are H\"older continuous in the spatial variables (the derivatives of $\Phi$ at a measure $\mu$ are functions on the torus; H\"older continuity is thus interpreted in terms of the standard distance on the torus). 
When specialising to a linear function $\Phi(\mu)=\int_{\bT^d} F(x) \mu(\ud x)$, this says that 
$F$ has to be merely H\"older continuous. In contrast, $F$ is required to be twice differentiable (with bounded derivatives) in 
\cite{arnaudon2020}. Our strategy of allowing such weaker conditions combines two main ingredients. First, we exploit in a systematic manner the smoothing effect of the Laplace operator in \eqref{eq: forward eqn } (using Schauder's estimates, very like in the finite time analysis carried out in \cite{MR4377993,CHAUDRUDERAYNAL20221}); Second, 
we use a mollification method for $\Phi$, introduced in the first arXiv version \cite{DelarueTse-arXiv} of this work and then studied in a systematic manner in \cite{cecchin2022weak}, that permits to work with a smoother $\Phi$, provided that the resulting bounds for the weak error only depend on the regularity properties of the unmollified $\Phi$.
%\footnote{In order to save some space, we have decided to remove the exposition of this mollification method from this paper.}
In the end, a typical instance of nonlinear function is: 
$$\Phi(\mu)= \| \mu - \mu_{0} \|^2_{-(d/2+\varepsilon),2}, \quad \mu \in {\mathcal P}(\bT^d),$$ 
for some $\varepsilon>0$, where
$\mu_{0}$ is a fixed `target' probability measure on $\bT^d$
and  $\| \cdot \|_{-(d/2+\varepsilon),2}$ is the norm on the dual of the standard Sobolev space $H^{d/2+\varepsilon}(\bT^d)$ (see 
Proposition \ref{prop:4:phi:norm:-d}). We find it very useful, especially when $\mu_{0}$ is chosen 
as $m(t \, ; \mu)$ in \eqref{eq: forward eqn } or as $m(\infty \, ; \mu):=\lim_{t\nearrow \infty} m(t \, ; \mu)$. In particular, this permits to retrieve a dimension-free bound, at the price of working with a weaker distance than ${\mathcal W}_{1}$.
We refer to 
\cite{MR4564418} for a similar occurrence of this distance in related questions and 
to \cite{daudin2023optimal} for further mollification arguments that would permit to retrieve dimension-dependent rates in ${\mathcal W}_1$-distance.  
\vskip 4pt

\textit{Strategy of proof.}
Our approach is heavily based on the so-called \textit{master equation} satisfied by the semigroup generated by the McKean-Vlasov equation \eqref{eq:MVSDE}. Equivalently, the latter is the semigroup $(\textsf{P}_{t})_{t \geq 0}$ generated by the (deterministic) Fokker-Planck equation \eqref{eq: forward eqn }, whose action on a (bounded and measurable) test functional $\Phi  : {\mathcal P}(\bT^d) \rightarrow {\mathbb R}$
 reads:
\begin{equation}
\label{eq:intro:semigroup}
{\sf P}_{t} \Phi : {\mathcal P}(\bT^d) \ni \mu \mapsto \Phi\bigl(m(t \, ; \mu)\bigr), \quad t \geq 0,
\end{equation} 
where $(m(t \, ; \mu))_{t \geq 0}$ solves \eqref{eq: forward eqn }.

A key fact is that $(\textsf{P}_{t} \Phi)_{t \geq 0}$ is a classical solution of the aforementioned master equation
whenever $\Phi$ is smooth enough. In PDE theory, the trajectories $(m(t \, ; \mu))_{t\geq 0}$ should be regarded as the characteristics of the master equation. A systematic analysis of the smoothness of 
$(\textsf{P}_{t} \Phi)_{t \geq 0}$ is provided in \cite{buckdahn2017mean,CHAUDRUDERAYNAL20221,tse2019higher}  (see also 
\cite{cardaliaguet2019master} for a similar study in a nonlinear setting). The analysis of the weak error is then carried out in two steps. The first one is to `test' $(\textsf{P}_{t} \Phi)_{t \geq 0}$ onto the empirical distribution of 
the $N$-particle system; the resulting bound is shown to be ${\mathcal O}(1/N)$, with the leading constant
in the symbol ${\mathcal O}(\cdot)$ depending on the bounds for the derivatives of order 1 and 2
of $(\textsf{P}_{t} \Phi)_{t \geq 0}$. The second step is to provide uniform-in-time bounds for the derivatives of $(\textsf{P}_{t})_{t \geq 0}$, which is the main challenge in the proof. 

The derivatives of $(\textsf{P}_{t} \Phi)_{t \geq 0}$ may be explicitly computed by linearising the Fokker-Planck equation \eqref{eq: forward eqn }. This is a well-known fact in PDE theory: The derivatives of the solution of a transport equation 
can be expressed in terms of the derivatives of the corresponding characteristics with respect to the initial point. 
In our setting, the derivatives of the characteristics are indeed obtained by linearising the Fokker-Planck equation with respect to the measure argument. See for instance 
the results exposed in 
Subsection \ref{subse:second:order:spatial:derivatives}, which are mostly borrowed from 
\cite{tse2019higher}. 
Thus, not only does the long-time behaviour of the 
Fokker-Planck equation matter in our analysis, but also 
the long-time asymptotics of the linearised Fokker-Planck equation are important, which explains the formulation of our main 
Theorem 
\ref{thm main result:2}. 

Of course, this approach is reminiscent of the approach initiated in 
\cite{mischler2013kac,mischler2015new}. In particular, the proof of 
property (A4) in both papers is also based upon the linearised version of 
\eqref{eq: forward eqn }, but computed along any solutions. Here, Theorem 
 \ref{thm main result:2} is just formulated in terms of the properties of the linearised version of 
\eqref{eq: forward eqn } 
 at `the' (or `an' when uniqueness does not hold) invariant measure. In this respect, it is fair to say that
\cite[Sec. 6]{mischler2013kac} contains a similar use to ours of the existence
of a globally attracting invariant measure. One of our contribution is to make clear the 
role of the latter and, in turn, to extend the analysis to locally attracting invariant measures. Generally speaking, 
 the linearisation is of a richer structure (and thus of an easier study) at the invariant measure and this plays a key role in the examples that are treated below. For instance, our analysis of Kuramoto's model (as stated in Theorem 
 \ref{main:thm:kuramoto}, which is clearly out of reach of the results obtained in 
\cite{mischler2013kac,mischler2015new}) mostly relies on the properties of the linearised equations at each of the non-trivial invariant measures. From a different perspective,
the metastability properties provided by  
Theorems \ref{thm main result:2}
and
\ref{prop:3:18}
(which are no longer addressed in 
\cite{mischler2013kac,mischler2015new}) 
are also stated in terms of 
the linearised equation at a stable equilibrium, without any further global constraints.  
In the same vein, we believe that the application of Theorem 
\ref{thm main result:2}
to gradient and conservative systems provided 
in Subsection
\ref{subse:examples} is also greatly simplified by the fact the linearised equation is just computed at the invariant measure. 
On another matter, it is worth mentioning that the constraints we impose on the drift $b$ are weaker than those required in 
\cite{mischler2015new}. This requires an additional substantial effort to handle the small time singularities of the 
solutions to the various linearised equations under study. 
For instance, we here allow the drift to be merely bounded in the spatial position: to the best of our knowledge (and apart from the 
case by case singular examples
treated in 
 \cite{guillin2021uniform,guillin2022systems,MR4564418}), this is 
 something new. 
By comparison, the drift is assumed to be at least H\"older continuous in 
 \cite{MR4377993,CHAUDRUDERAYNAL20221}, in which propagation of chaos is just addressed
 in finite time. 

\vskip 4pt

\textit{Further prospects.} We are confident that similar results could be obtained in other cases, {including cases with a non-constant diffusion coefficient or defined in the Euclidean setting, provided that a suitable form of confining drift is added to the dynamics.
%
%
%
%.  Through combining an explicit representation of weak error established in \cite{chassagneux2019weak} in terms of the associated Fokker Planck PDE of \eqref{eq:MVSDE}, as well as the linear functional derivatives (to be discussed below) of test functionals of the solution to the Fokker-Planck PDE established in \cite{tse2019higher}, we are able to show our main result (Theorem \ref{thm main result}) that, for $b^{(\eps)}(x, \mu)= B(x, \mu) \eps$, where $B$, $\Phi$ are sufficiently regular and  $\eps$ is sufficiently small, we have
%\begin{equation} \sup_{t \geq 0} \Big|  \bE[ \Phi(\mu^{N}_t)] - \Phi(\rvlaw[X_t]) \Big| \leq \frac{C}{N}, \label{eq main result intro} 
%\end{equation} 
%for some constant $C>0$.  
\subsection{Organisation and notations}

 We start the paper by reviewing the theory of differential calculus in Wasserstein spaces and the master equation in Section \ref{se:2}. We also provide the key semi-group expansion that serves as bounding the weak error (see Lemma 
\ref{lem:general:lemma}). Subsequently, Section \ref{se:3} is dedicated to the statement and the proof of 
Theorem \ref{thm main result:2}, in which  
we derive uniform in time estimates for 
\eqref{eq:intro:weak}
from general `ergodic' properties 
of 
\eqref{eq: forward eqn }
and its linearised version. In Section  \ref{se:kuramoto}, we explore the special case related to the Kuramoto model in which the associated Fokker-Planck equation does not have a unique invariant measure (see Theorem 
\ref{main:thm:kuramoto}).  
\vskip 4pt

\textit{Useful Notations.}
 The scalar product between two vectors $a,b \in \bR^d$ is denoted by $a \cdot b$. For each $i \in \bR^d$, $e_i$ denotes the vector with $1$ in the $i$th component and 0 elsewhere. For any vector $x \in \bR^d$, $x_i$ denotes the $i$th component of $x$. For $a, b \in \bR$, $a \vee b$ denotes $\max \{a,b \}$ and $a \wedge b $ denotes $\min \{a, b \}$.
 The set ${\mathbb N}$ is the set of integers, including $\{0\}$. 
 For any real $s$, we call $\lfloor s \rfloor$ the floor part of $s$.
 For two probability measures $\mu$ and $\nu$ on $\bT^d$, we call 
 $\textrm{\rm dist}_{\textrm{\rm TV}}(\mu,\nu) = \sup_{\| f \|_{\infty} \leq 1} \int_{\bT^d}
 f(x) d (\mu-\nu)$ the total variation distance between $\mu$ and $\nu$, where 
 $\| f \|_{\infty}$
 is the essential sup norm of 
 $f :\bT^d  \rightarrow \bR$. 
 For $z \in {\mathbb C}$, $\overline z$ is the complex conjugate of $z$. Moreover, $\i$ is the complex number such that $\i^2=-1$. 
 %\textcolor{blue}{We denote by ${\mathbf 1}$ the constant function on ${\mathbb T}^d$ that is equal to $1$ everywhere}. 

\section{Main method of proof in this paper}
\label{se:2}
In this section, we introduce the main ingredients needed in our approach.
The space 
$\cP(\bT^d)$ is equipped with the ${\mathcal W}_{1}$ distance, where we recall that
convergence for ${\mathcal W}_{1}$ is equivalent to weak convergence on ${\mathcal P}(\bT^d)$. 
Moreover, the vector field $b$ is assumed to be at least bounded and measurable on 
$\bT^d \times {\mathcal P}(\bT^d)$. As a result, 
existence and uniqueness hold for \eqref{eq:MVSDE}
and the marginal law of the solution depends on the initial condition 
only through the statistical distribution of the latter. 
Then, we call $( m(t \, ; \mu) )_{t \geq 0}$ the evolving law of the process $X$ in \eqref{eq:MVSDE} when starting at law $\mu$; it solves 
\eqref{eq: forward eqn } in a distributional sense.

\subsection{Master equation}
Our framework of analysis relies on 
the so-called master equation for the semigroup $(\textsf{P}_{t} \Phi)_{t \geq 0}$ 
defined in 
\eqref{eq:intro:semigroup}. This requires 
a notion of differentiation w.r.t. measures in $\cP(\bT^d)$, called \emph{linear functional derivatives, } see \cite{cardaliaguet2019master,carmona2017probabilistic,chassagneux2019weak,delarue2019,wang}. A function 
$\cV : \cP(\bT^d) \rightarrow \R$ is said to 
have directional derivatives at some $m \in {\mathcal P}({\mathbb T}^d)$ if
there exists a bounded (measurable) function 
 $\frac{\delta \cV}{\delta m}(m,\cdot):  \bT^d \to \bR$
 such that, for any $m' \in {\mathcal P}( {\mathbb T}^d)$,
 \begin{equation*}
 \lim_{\epsilon \searrow 0} 
 {\mathcal V}\bigl( \epsilon m' + (1-\epsilon) m \bigr) =
 \int_{{\mathbb T}^d} 
 \frac{\delta \cV}{\delta m}(m,y)  \bigl( m'- m \bigr)(\ud y). 
 \end{equation*}
The function ${\mathcal V}$ is said to be 
 continuously differentiable if  
 it has directional derivatives at any $m \in {\mathcal P}({\mathbb T}^d)$ and the 
 resulting function 
 $\frac{\delta \cV}{\delta m}: \cP(\bT^d) \times \bT^d \to \bR$
 is (jointly) continuous, in which case it satisfies, 
 \begin{align}\label{eq de first order deriv}
\cV(m')- \cV(m) = \int_0^1 \int_{\bT^d} \frac{\delta \cV}{\delta m}\bigl( (1-s)m + sm',y\bigr) \,  (m'-m)(\ud y) \, \ud s, 
\quad 
m, m' \in \cP(\bT^d). 
 \end{align}
The function 
$\frac{\delta \cV}{\delta m}$
is said to be the \emph{linear functional derivative} of $\cV: \cP(\bT^d) \to \bR$. 
It is uniquely defined up to an additive constant, which is fixed by the convention 
\begin{equation}
    \intrd   \frac{\delta  \cV}{\delta m}(m,y)  m( \ud y) = 0.
    \label{eq: normalisation linear functional deriatives:p=1}
\end{equation} 
{If $[\delta \cV/\delta m](m,\cdot)$ is differentiable (w.r.t $y$), then 
we let $\partial_\mu \cV(m,y) = \partial_y [\delta \cV/\delta m](m,y)$.}
%Notice in particular that 
%\eqref{eq de first order deriv} may be reformulated as 
% \begin{align}
%\cV(m')- \cV(m) = \int_{\bT^d} \frac{\delta \cV}{\delta m}( m,y) \, (m'-m)(\ud y)  + o \bigl( {\mathcal W}_{1}(m,m') \bigr),
% \end{align}
%with the symbol $o(\cdot)$ being uniform with respect to $m,m' \in {\mathcal P}(\bT^d)$.  
By induction, we then introduce higher-order derivatives:
for any integer $p \geq 2$, $m, m' \in \cP(\bT^d)$ and $y \in (\bT^d)^{p-1}$,
\begin{align*}
 \frac{\delta^{p-1} \cV}{\delta m^{p-1}}(m',y) - \frac{\delta^{p-1} \cV}{\delta m^{p-1}}(m,y) = \int_0^1 \int_{\bT^{d}} \frac{\delta^p \cV}{\delta m^p}\bigl((1-s)m +sm',y,y'\bigr) \, (m'-m)( \ud y') \, \ud s,
\end{align*}
provided that the $(p-1)$-th order derivative is well defined. 
In order to ensure uniqueness, they are required to satisfy 
\begin{equation}
    \intrd   \frac{\delta^p \cV}{\delta m^p}(m,y_1, \ldots, y_p) \, m(\ud y_p) = 0, 
    \quad { y_{1},\cdots,y_{p-1} \in {\mathbb T}^{d-1}.} 
    \label{eq: normalisation linear functional deriatives}
\end{equation} 
There is a related notion called 
Wasserstein derivative $\partial_{\mu} \cV$ (\cite{ambrosiobook,buckdahn2017mean, cardaliaguet2010notes,carmona2017probabilistic}). 
In short, it is the gradient field of the linear functional derivative, i.e., 
$\partial_{\mu} \cV(m,y) = \partial_{y} [\delta V/\delta m] (m,y)$. 
%A more detailed review is however out of the scope of this paper, as we shall work with linear functional derivatives exclusively.
Propositions 5.48 and 5.51
in 
\cite{carmona2017probabilistic} serve as a dictionary to enable us to pass from one to the other.

For a bounded measurable function $\Phi : \cP(\bT^d) \rightarrow \bR$, 
we let (see \eqref{eq:intro:semigroup}):  
\begin{equation}  
\cU( t, \mu) :=\textsf{P}_{t} \Phi(\mu) = \Phi\bigl(m(t \, ; \mu)\bigr),
\quad t \geq 0,  \ \mu \in {\mathcal P}(\bT^d). \label{eq: defofflow}
\end{equation}
%where, for some $\bR^d$-valued random vector $\theta$ independent of $W$, $X^{\theta}$ satisfies the McKean-Vlasov SDE  $$X^{\theta}_t= \theta + \int_0^t b(X^{\theta}_s, \rvlaw[{X^{\theta}_s}]) \,ds +\sqrt{2}W_t,\quad \quad t\ge 0, \quad \quad \quad \cL(\theta) = \mu.$$
It is proven in Theorem 7.2 of  \cite{buckdahn2017mean} (see also Theorem 3.5 in \cite{CHAUDRUDERAYNAL20221}) that  $\cU$ satisfies the \emph{master equation}% given by
\begin{equation}  \label{eq pde measure} 
\begin{split} 
&\partial_t {\cU}(t, \mu) =\int_{\bT^d}  \biggl[ \sum_{i=1}^d\partial_{x_i} \ld[{\cU}] (t, \mu,x)  b_i( x, \mu)  + \tfrac12 \sum_{i=1}^d \partial^2_{x_{i} x_i}   \ld[{\cU}]  ( t,\mu, x)  \biggr] \, \mu (\ud x),
\end{split}
\end{equation} 
for $(t,\mu) \in [0,+\infty) \times {\mathcal P}({\mathbb T}^d)$, 
with 
${\cU}(0,\mu) = \Phi ( \mu)$ as boundary condition,      
provided that $b$ and $\Phi$ 
are smooth enough. 
While we could state properly the conditions required 
in \cite{buckdahn2017mean}, we feel useless to do so at this stage. 
We formulate below stronger sets of conditions on $b$ and $\Phi$ which subsume the conditions needed in 
\cite{buckdahn2017mean}.
In fact, 
the equation 
in \cite[Theorem 7.2]{buckdahn2017mean}
is set in the Euclidean setting, which requires, when translated in the periodic setting, to check that the various derivatives are periodic functions in $y$. We refer to 
the appendix in \cite{cardaliaguet2019master}.% and to \cite{marx} (the latter being more detailed, but restricted to the $1d$ case). 
 In this respect, it is worth noticing that in 
\cite[Theorem 7.2]{buckdahn2017mean}, 
the measure argument $\mu$ in \eqref{eq pde measure} is required to have a finite second moment. In our setting, such a restriction 
%on the integrability properties of the measure 
no longer exists since the torus is compact. 
Very importantly, the result 
of 
\cite[Theorem 7.2]{buckdahn2017mean}
says that:
\begin{proposition}
\label{eq:regularity:mathcalU}
If $b$ and $\Phi$ are smooth enough then, for every $i,j \in \{1, \ldots, d\}$, the derivatives
\[
\partial_{t} {\mathcal U}(t,\mu), \quad \partial_{(y_1)_i} \ld[{\mathcal U}](t,\mu,y_1),  \quad \partial_{(y_2)_j} \partial_{(y_1)_i} \sld[{\mathcal U}](t,\mu,y_1, y_2),  \quad \partial_{(y_1)_j} \partial_{(y_1)_i} \ld[{\mathcal U}](t, \mu,y_1)
\]
exist 
and 
are 
globally Lipschitz continuous w.r.t. $(\mu,y_{1},y_{2})$ for the Euclidean and ${\mathcal W}_1$ norms, uniformly in time $t$ in compact subsets, and are continuous in time (and hence jointly continuous).  
\end{proposition}
In fact, 
\cite[Theorem 7.2]{buckdahn2017mean}
is stated for the ${\mathcal W}_{2}$-distance.
The adaptation to the ${\mathcal W}_{1}$-distance may be found in 
\cite[Theorem 5.10]{CarmonaDelarue_book_II}, noticing that the equation for ${\mathcal U}$ has a more general nonlinear form. In particular, the restriction that $T$ has to be small enough (in the latter statement) can be easily removed in our linear setting. Also, our function ${\mathcal U}(t,\mu)$ corresponds
in the notation of  \cite{CarmonaDelarue_book_II} to $\int {\mathcal U}(t,x,\mu) \mu(\ud x)$. %\textcolor{red}{Chaudru}

\subsection{Expansion along the particle system}
The starting point of our analysis is to make use of the identity 
$\Phi(\mu) = {\mathcal U}(0,\mu)$, as given by 
 the initial condition of \eqref{eq: defofflow}. 
Recalling the notation 
$\mu^N_{t}$ from 
\eqref{eq particles},  this gives the decomposition
\begin{equation}
\begin{split}
\Phi(\nlaw[t])  - \Phi ( \cL(X_t) ) &=  \big( \cU(t,\nlaw[0]) - \cU(t,{\mu_{\text{init}}}) \big) + \big( \cU(0,\nlaw[t]) - \cU(t,\nlaw[0]) \big).
\end{split} 
\label{eq:errordecomp}
\end{equation}
To treat the last term, we define, for $t > 0$, the finite dimensional projection  ${U}_t:[0,t]\times(\bT^d)^N \to \bR$ by
\begin{equation*}
    {U}_t(s  ,x_1,\ldots,x_N):= {\cU} \bigg( t-s, \frac{1}{N}\sum_{i=1}^N\delta_{x_i} \bigg). 
\end{equation*} 
Then 
$${\cU}(0,\nlaw[t]) -{\cU}(t,\nlaw[0]) =  {U}_t(t, Y^{1,N}_t, \ldots, Y^{N,N}_t) -  {U}_t(0, Y^{1,N}_0, \ldots, Y^{N,N}_0).$$
We can now apply It\^{o}'s formula to this equality. By combining our Proposition 
\ref{eq:regularity:mathcalU}
with \cite[Proposition 3.1]{chassagneux2014probabilistic}, we can conclude that  ${U}_t$ is differentiable in the time component and twice-differentiable in the space components.  Very much in the spirit of 
\cite[(5.131)]{carmona2017probabilistic}, this allows us to use \eqref{eq pde measure} to obtain a cancellation of all the terms apart from one term which gives us the rate of convergence of $1/N$:
\begin{equation}
\label{eq:second:term:main result intro formula}
\begin{split}
{\mathbb E} \bigl[ \cU(0,\nlaw[t]) - \cU(t,\nlaw[0]) \bigr] 
&=  \frac{1}{N} \sum_{i=1}^d \int_0^t \bE \bigg[  \intrd \bigg( \partial_{(y_{2})_i} \partial_{(y_{1})_i}  \frac{ \delta^2 \cU}{\delta m^2} (t-s, \mu^{N}_s,z,z) \bigg)   \, \mu^{N}_s(\ud z) \bigg] \, \ud s.
\end{split}
 \end{equation}
 More importantly, this formula holds regardless of the assumptions on the initial 
 data $Y^{1,N}_{0},\cdots,Y^{N,N}_{0}$. In particular, they are not required to be I.I.D. 
 In fact, the I.I.D. assumption becomes useful in order to estimate the first term in 
 the expansion 
\eqref{eq:errordecomp}. Indeed, by \cite[Thm. 2.14]{chassagneux2019weak}
(which we can apply in our context thanks to 
Proposition \ref{eq:regularity:mathcalU}), we have 
\begin{equation}
\label{eq:second:term:main result intro formula:2}  
\begin{split}
&{\mathbb E} \bigl[ \cU(t,\nlaw[0]) \bigr] - \cU(t,{\mu_{\text{init}}})
=  \frac1{N}   \int_0^1 \int_0^1  \bE \bigg[ s \frac{\delta^2 \cU}{\delta m^2}(t, \tilde{\mu}^{N}_{s,s_1}, \tilde{\eta},\tilde{\eta})  - s \frac{\delta^2 \cU}{\delta m^2}(t, \tilde{\mu}^{N}_{s,s_1},\tilde{\eta},{ {\eta}_{1}})  \bigg] \ud s_1 \, \ud s, 
\end{split}
\end{equation}
{where 
$\eta_{1}$ is as 
in 
\eqref{eq particles}, $\tilde{\eta}$ 
is independent of $(\eta_{1},\cdots,\eta_{N})$} with law ${\mu_{\text{init}}}$ and 
$$ \tilde{\mu}^{N}_{s,s_1}:= \frac{ss_1}{N}(\delta_{\tilde{\eta}} - \delta_{{\eta_{1}}}) + {\mu_{\text{init}}} + s( \mu^N_0 - {\mu_{\text{init}}}), \quad \quad s, s_1 \in [0,1]. $$ 
By combining the above equation with 
\eqref{eq:errordecomp}
and 
\eqref{eq:second:term:main result intro formula}, we deduce the following 
lemma:
\begin{lemma}
\label{lem:general:lemma}
If $b$ and $\Phi$ are smooth enough (so that the master equation has a 
classical solution satisfying the conclusion of Proposition 
\ref{eq:regularity:mathcalU}), then
for any integer $N \geq 1$ such that 
$(Y_0^{1,N},\cdots,Y_0^{N,N})$ are I.I.D. with common law ${\mu_{\text{init}}} :={\mathcal L}(X_{0})$ and for any  
$t \geq 0$,
\begin{equation}
 \label{main result intro formula}
\begin{split}
\bE[ \Phi(\mu^{N}_t)] - \Phi(\rvlaw[X_t]) 
& =   \frac1{N}   \int_0^1 \int_0^1  \bE \bigg[ s \frac{\delta^2 \cU}{\delta m^2}(t, \tilde{\mu}^{N}_{s,s_1},\tilde{\eta},\tilde{\eta})  - s \frac{\delta^2 \cU}{\delta m^2}(t, \tilde{\mu}^{N}_{s,s_1},\tilde{\eta},{{\eta}_{1}})  \bigg] \ud s_1 \, \ud s 
\\
&\hspace{15pt} + \frac{1}{N} \sum_{i=1}^d \int_0^t \bE \bigg[  \intrd \bigg( \partial_{(y_{2})_i} \partial_{(y_{1})_i}  \frac{ \delta^2 \cU}{\delta m^2} (t-s, \mu^{N}_s,z,z) \bigg)   \, \mu^{N}_s(\ud z) \bigg] \, \ud s,
\end{split}
 \end{equation}
\end{lemma}
Equation \eqref{main result intro formula} is in fact a key in our analysis. 
In order to bound the left-hand side by $1/N$, uniformly in 
$t \geq 0$, we must be able: $(a)$
to bound  $\frac{\delta^2 \cU}{\delta m^2}(t, \cdot,\cdot, \cdot)$, uniformly in time; $(b)$
to bound $\partial_{(y_{2})_i} \partial_{(y_{1})_i}  \frac{\delta^2 \cU}{\delta m^2}(t, \cdot,\cdot, \cdot)$, with 
{an integrable decay
 in long time and, possibly, an integrable blow-up in small time}. 
Observe that, implicitly, those bounds are required to be uniform in space. Obviously, this is 
a very strong constraint, which we are however able to relax partially {in two different ways}: $(i)$ In Subsection 
\ref{subse:metastable}, where we 
just use local-in-time bounds for the derivatives (in the right-hand side in 
\eqref{main result intro formula})
and then obtain {bounds for the left-hand side 
up to times that are polynomial in $N$}; 
%an upper bound for the left-hand side up to times that are polynomial in $N$;
$(ii)$ In the analysis of the super-critical Kuramoto model
provided in Section \ref{se:kuramoto}, {where we use bounds for the deriavtives that are uniform away from an unstable equilibrium and where we obtain bounds for the left-hand side for a smaller class of test functionals $\Phi$.}

\subsection{Main assumptions} 

\subsubsection{Functional spaces}
We shall use two types 
of functional spaces in our analysis:
$W^{s,\infty}(\bT^d)$ spaces (and their duals), and $W^{s,2}(\bT^d)$
spaces (and their duals), for $s>0$. Following 
\cite{Brezis_Mironescu,dinezza}, we define the following notations.
\begin{enumerate}
\item For any integer $n \geq 0$, we call $W^{n,\infty}(\bT^d)$ the space of functions 
$f$ that are $(n-1)$-times differentiable and whose $(n-1)^{\rm th}$-derivative is Lipschitz continuous. 
The derivatives up to order $n-1$ are denoted by $(\nabla^k f)_{k=1,\cdots,n-1}$, 
with each $\nabla^k f$ taking values in $({\mathbb R}^d)^k$. 
The function $\nabla^{n-1} f$ itself has a generalised derivative $\nabla^n  f \in L^\infty(\bT^d;({\mathbb R}^d)^n)$. The $W^{n,\infty}(\bT^d)$-norm is written 
$\| f \|_{n,\infty}:=\sum_{k=0}^n \| \nabla^k f\|_{\infty}$. 
\item For any integer $n \geq 0$ and any real $\alpha \in (0,1)$, 
we call $W^{n+\alpha,\infty}(\bT^d)$ the space of functions that are 
$n$-times differentiable such that their $n^{\rm th}$-derivatives
are $\alpha$-H\"older continuous.
The $W^{n+\alpha,\infty}(\bT^d)$-norm is written as
$\| f \|_{n+\alpha,\infty}:=\sum_{k=0} \| \nabla^k f \|_{\alpha,\infty}$,
where $\| \cdot \|_{\alpha,\infty}$
is the standard H\"older norm 
$$\| f \|_{\alpha,\infty} = \sup_{x \in \bT^d} \vert f(x) \vert + \sup_{x,y \in \bT^d : x \not =y}
\frac{\vert f(x) - f(y)\vert}{\vert x-y \vert^{\alpha}}.$$
\item For any integer $n \geq 0$ and any real $\alpha \in (0,1)$, 
we call $(W^{n+\alpha,\infty}(\bT^d))'$ the dual
space of $W^{n+\alpha,\infty}(\bT^d)$.
The dual norm is denoted by $\| \cdot \|_{(n+\alpha,\infty)'}$.
Notice that 
$\| \cdot \|_{(0,\infty)'}$, {when restricted to the space of probability measures}, 
identifies with
$ \textrm{\rm dist}_{\rm TV}$. 
\end{enumerate}
We merely write $\| \cdot \|_{\infty}$ for $\| \cdot \|_{0,\infty}$. 
In the text, we make use of the following interpolation inequality:
\begin{equation}
\label{eq:interpolation}
\| \phi \|_{a+\eta,\infty} \leq \| \phi \|_{a,\infty}^{(\gamma-\eta)/\gamma}
\| \phi \|_{a+\gamma,\infty}^{\eta/\gamma},
\end{equation}
which holds for any $a\geq 0$, $\eta,\gamma \in [0,1]$ with $\eta \leq \gamma$.
Above, $\phi \in W^{a+\gamma,\infty}(\bT^d)$  (see
\cite{Brezis_Mironescu}).
\vskip 4pt

In order to introduce $W^{s,2}(\bT^d)$, we feel more convenient to use Fourier analysis. 
For a function $f \in L^2(\bT^d)$, we denote
its Fourier coefficients by 
\begin{equation*}
f^{\boldsymbol n} :=  \int_{\bT^d} f(x) e^{- \i 2 \pi {\boldsymbol n} \cdot x} \ud x, 
\quad {\boldsymbol n} \in {\mathbb Z}^d.
\end{equation*}
For $s>0$, we call $W^{s,2}(\bT^d)$ the space of functions $f \in L^2(\bT^d)$
such that $\| f \|_{s,2}^2:=\sum_{{\boldsymbol n} \in {\mathbb Z}^d} (1+n^2)^{s} \vert f^{\boldsymbol n} \vert^2 < \infty$. The $W^{s,2}(\bT^d)$-norm is $\| \cdot \|_{s,2}$. The dual space is identified with $W^{-s,2}(\bT^d)$, which is defined in a similar manner, by extending the notation 
$(q^{\boldsymbol n})_{{\boldsymbol n} \in {\mathbb Z}^d}$ for the Fourier coefficients 
of a Schwartz distribution $q$ (acting on smooth functions of $\bT^d$).
Then, 
$W^{-s,2}(\bT^d)$ is the space of distributions $q$
such that $\| q\|_{-s,2}^2:=\sum_{{\boldsymbol n} \in {\mathbb Z}^d} (1+n^2)^{-s} \vert q^{\boldsymbol n} \vert^2 < \infty$. 
The $W^{-s,2}(\bT^d)$-norm is $\| \cdot \|_{-s,2}$.
For brevity, we write $\| \cdot \|_{2}$ for $\| \cdot \|_{0,2}$. 
\vskip 4pt

For any vector field $f=(f^1,\cdots,f^d)$, we write 
$\| f \| = \max_{i=1,\cdots,d} \| f^i \|$ for any norm on the space in which 
the $f^i$'s are taken. 
Most of the time, the duality product between 
a function $f$ and a distribution $q$ is merely denoted by 
$\langle f,q \rangle$, with the spaces to which $f$ and $q$ belong
being implicitly understood.   
For $z \in {\bT}^d$, we define $D_{z}$ as the Dirac distribution at point $z$ and $D_{z}'$ for the
{opposite of its} derivative. In short, $\langle D_{z}',f\rangle = f'(z)$. 
For time-dependent function $f$ and distribution $q$, 
we often write $f(t,x)$ for the former and $q(t)$ for the latter. 
Finally, for any $n \in {\mathbb N} \setminus \{0\}$ and 
any $n$-time differentiable
function $\Phi$ on ${\mathcal P}(\bT^d)$,  we write $\frac{\delta^n \Phi}{\delta m^n} (\mu)(q_1, \ldots,q_n) $, 
for distributions $q_{1},\cdots,q_{n}$ on $\bT^d$,
 to denote
$ \frac{\delta^n \Phi}{\delta m^n} (\mu)(q_1, \ldots, q_n):=  \langle 
\frac{\delta^n \Phi}{\delta m^n}(\mu,\ldots), q_{1} \otimes q_{2} \otimes \ldots \otimes q_{n}\rangle,$
if the duality product makes sense, where we recall that the function in the left-hand side of the duality product is defined on $(\bT^d)^n$. Note that
 $q_{1} \otimes q_{2} \otimes \cdots \otimes q_{n}$ is the tensor product of $q_{1},\cdots,q_{n}$ (see \cite[Definition 40.3]{treves}). 
 The Lebesgue measure on ${\mathbb T}^d$ is denoted by
  $\textrm{\rm Leb}_{\bT^d}$. 
  The constant function on ${\mathbb T}^d$, equal to 1, is denoted by $\one$.  
%In particular, if $\mu_{1},\cdots,\mu_{n}$ are finite signed measures on 
%$\bT^d$, 
%we write
%\begin{equation*}
%\frac{\delta^n \Phi}{\delta m^n} (\mu)(\mu_1, \ldots, \mu_n)
% =
%\int_{\bT^d} \ldots \int_{\bT^d}  \frac{\delta^n \Phi}{\delta m^n} (\mu, x_1, \ldots, x_n) \, \mu_1(dx_1) \, \ldots \, \mu_n (dx_n),
%\end{equation*}
%if the iterated integral in the right-hand side is well-defined.

\subsubsection{Main assumptions} \label{section assumptions}
We use the following assumptions, with $n$ and $k$ denoting two integers and $\alpha$ a real in $[0,1)$:
\begin{description}
\item[\hspace{3pt} \hintb{n+\alpha}{k}] We say that $b$ satisfies \hintb{n+\alpha}{k} if, for any $i \in \{1,\cdots,d\}$, 
the function $b_{i}$ is $k$ times differentiable with respect to the measure argument 
$m$, and for any $m \in \cP(\bT^d)$ and $\ell \in \{0,\cdots,k\}$, the function 
\begin{equation*}
(\bT^d)^{\ell+1} \ni (x,y_{1},\ldots,y_{\ell})
\mapsto 
\frac{\delta^\ell b_{i}}{\delta m^\ell}(x,m,y_{1},\ldots,y_{\ell})
\end{equation*}
has crossed derivatives $\partial_{x}^{n_{0}} \partial_{y_{1}}^{n_1}
\ldots \partial_{y_{\ell}}^{n_{\ell}} \frac{\delta^\ell b_{i}}{\delta m^\ell}(x,m,y_{1},\ldots,y_{\ell})$
for any $n_{0},n_1,\ldots,$ $n_{\ell}$ in $\{0,\cdots,n\}$, with all these crossed derivatives being bounded w.r.t. $(x,y_{1},\ldots,y_{\ell})$, uniformly in $m$, {and $\alpha$-H\"older continuous w.r.t. 
$(x,y_{1},\ldots,y_{\ell})$, uniformly in $m$, if $\alpha >0$}.
\item[\hspace{3pt} \hlipb{n}{k}] We say that $b$ satisfies \hlipb{n}{k} if it satisfies 
\hintb{n}{k} and, for any $i \in \{1,\cdots,d\}$ and $\ell \in \{1,\cdots,k\}$,
for any $n_{0},n_1,\ldots,n_{\ell}$ in $\{0,\cdots,n\}$, 
the derivatives $\partial_{x}^{n_{0}} \partial_{y_{1}}^{n_1}
\ldots \partial_{y_{\ell}}^{n_{\ell}} \frac{\delta^\ell b_{i}}{\delta m^\ell}$
are Lipschitz continuous in $m$ with respect to ${\mathcal W}_{1}$. 
\end{description}
%Below, we do not require that \hintb{n}{k}
%and \hlipb{n}{k} hold for the same values of $n$ and $k$. 
%Typically, we require 
%\hintb{4}{2}
%and 
%\hlipb{3}{2}, noticing that the latter subsumes 
%the conditions stated in 
%\cite[Theorem 7.2]{buckdahn2017mean}
%to obtain 
%Proposition 
%\ref{eq:regularity:mathcalU}. Moreover, we often use 
%\hintb{n}{k}
%and \hlipb{n}{k}
%in the following manner.
%We indeed observe that \hintb{n}{k} implies that, for each $i \in \{1, \ldots, d\}$ and $\ell \in \{1, \ldots, k \}$, 
%\begin{equation}
%\sup_{m \in \cP(\bT^d)} \sup_{\substack{ \|q_1\|_{(n, \infty)'}, \ldots, \|q_{\ell}\|_{(n, \infty)'} \leq 1}} \bigg\|  \frac{\delta^{\ell} b_i}{\delta m^{\ell}} (\cdot,m)(q_1, \ldots, q_{\ell}) \bigg\|_{n, \infty} < + \infty. 
%\label{eq:Int-b-n-k}
%\end{equation}
%Similarly, \hlipb{n}{k} implies that, for each $i \in \{1, \ldots, d\}$ and $\ell \in \{1, \ldots, k-1\}$, 
%\begin{equation}
%\label{eq:Lip-b-n-k}
%\sup_{\substack{ \|q_j\|_{-(\na)'}  \leq 1}}  \sup_{{\mu_1 \neq \mu_2}} \big({\mathcal W}_1(\mu_1,\mu_2) \big)^{-1} \bigg\| \frac{\delta^{\ell} b_i}{\delta m^{\ell}}(\cdot,\mu_1)(q_1, \ldots, q_{\ell}) - \frac{\delta^{\ell} b_i}{\delta m^{\ell}}(\cdot,\mu_2)(q_1, \ldots, q_{\ell}) \bigg\|_{n, \infty } 
%\end{equation}
%is finite, where $j$ in the first supremum is between $1$ and $\ell$ and $\mu_1,\mu_2$ in the second supremum are in 
%$\cP(\bT^d)$
\vskip 4pt
%, the proof of which is as follows. We combine the identity 
%\begin{equation*}
% \frac{\delta^{\ell} b_i}{\delta m^{\ell}}(\cdot,\mu_1)(q_1, \ldots, q_{\ell}) - \frac{\delta^{\ell} b_i}{\delta m^{\ell}}(\cdot,\mu_2)(q_1, \ldots, q_{\ell})
% = \int_{0}^1
% \frac{\delta^{\ell+1} b_i}{\delta m^{\ell+1}}\bigl(\cdot, \lambda \mu_1 + (1-\lambda) \mu_{2}\bigr)(q_1, \ldots, q_{\ell},\mu_{1}-\mu_{2}).
%  \end{equation*} 
%with Kantorovich-Rubinstein duality, see \cite[Remark 6.5]{villani},
%which says that  
%the above right-hand side is less than 
%\begin{equation*}
%W_{1}(\mu_{1},\mu_{2}) \times \sup_{m \in \cP(\bT^d)} \sup_{z \in \bT^d} \sup_{\substack{ \|q_1\|_{(n, \infty)'}, \ldots, \|q_{\ell}\|_{(n, \infty)'} \leq 1}} \bigg\|  \frac{\delta^{\ell+1} b_i}{\delta m^{\ell+1}} (\cdot,m)(q_1, \ldots, q_{\ell},D_{z}') \bigg\|_{n, \infty}.
%\end{equation*} 

We proceed similarly with the test functional $\Phi: \cP(\bT^d) \to \bR$. 
For two integers $k$ and $n$, we define 
\hintphi{n+\alpha}{k}
by replacing $b_i(x,m)$ by $\Phi(m)$ 
in 
 \hintb{n+\alpha}{k}. In particular, the crossed derivatives $\partial_{y_{1}}^{n_1}
\ldots \partial_{y_{\ell}}^{n_{\ell}} \frac{\delta^\ell \Phi}{\delta m^\ell}(m,y_{1},\ldots,y_{\ell})$
for $n_1,\ldots,n_{\ell}$ in $\{0,\cdots,n\}$, are bounded, uniformly in $m$, 
{and $\alpha$-H\"older continuous w.r.t. 
$(y_{1},\ldots,y_{\ell})$, uniformly in $m$, if $\alpha >0$}.
%\item[\hregphi{n}{k}] We say that $\Phi$ satisfies \hregphi{n}{k} if it satisfies 
%\hregphi{n}{k} and, for any $i \in \{1,\cdots,d\}$ and $\ell \in \{1,\cdots,k\}$,
%for any $n_1,\ldots,n_{\ell}$ in $\{0,\cdots,n\}$, 
%the crossed derivatives $\partial_{y_{1}}^{n_1}
%\ldots \partial_{y_{\ell}}^{n_{\ell}} \frac{\delta^\ell \Phi}{\delta m^\ell}$
%is Lipschitz continuous in $m$ with respect to ${\mathcal W}_{1}$. 
%\end{description}

In the sequel, we will use several values
 of $(n,k)$
 and $\alpha$ in these assumptions. While we use higher values of 
 $(n,k)$
 for intermediary steps, we eventually recover 
the main results under the sole 
\hintb{0}{2}
and
\hintphi{\gamma}{2}, 
for some $\gamma \in (0,1]$, 
by 
a mollification argument.
%, which is proven in the appendix. 
%The analogue of 
%\eqref{eq:Int-b-n-k}, but 
%for
%$\phi$ satisfying  \hintphi{4}{3}, is straightforward. 
%The analogue of 
%\eqref{eq:Int-b-n-k}, but 
%for
%$\phi$ 
%satisfying 
%\hintphi{\alpha}{2}, requires some care. 
Importantly, we observe that, under the latter assumption, the mapping 
$y_{1} \mapsto [(\delta^2 \Phi/\delta m^2)(m,y_{1},\cdot) : y_{2} \mapsto 
(\delta^2 \Phi/\delta m^2)(m,y_{1},\cdot) \in W^{\gamma/2,\infty}(\bT^d)]$ is 
$\gamma/2$-H\"older continuous, from which we deduce that
\hintphi{\gamma}{2} implies
\begin{equation}
\sup_{m \in \cP(\bT^d)} \sup_{\substack{ \|q_1\|_{(\gamma/2, \infty)'}, \|q_{2}\|_{(\gamma/2, \infty)'} \leq 1}} \biggl|  \frac{\delta^{2} \Phi}{\delta m^{2}} (\cdot,m)(q_1,q_{2}) \biggr| < + \infty. 
\label{eq:Int-phi-alpha-k}
\end{equation}

\subsection{Examples}
\label{subse:examples}
%We now give some useful examples of $b$ and $\Phi$. 
\subsubsection{Linear interaction}
\label{subsubse:linear:interaction}
Let $n \in \bN$. Suppose that for each $i \in \{1, \ldots, d\}$, $F_i: \bT^d \times \bT^d \to \bR$
is $n$-times continuously differentiable and that $G: \bT^d \to \bR$  is $n$-times differentiable. We then define 
%functions $b_i: \bT^d \times \cP(\bT^d) \to \bR$ and $ \Phi: \cP(\bT^d) \to \bR$ by
$$ b_i(x, \mu):= \intrd F_i(x, y) \, \mu(\ud y),  \quad \quad \Phi(\mu):= \intrd G(y) \, \mu(\ud y).$$ 
It can be shown easily %by the definition of linear functional derivatives (along with the condition of normalisation) 
that, for any integer $k \geq 1$, 
\begin{equation*}
\begin{split}
&\frac{ \delta^{k} b_i}{\delta m^k} (x, \mu,y_1, \ldots, y_k)= (-1)^k \bigg( \intrd F_i(x, y) \, \mu(\ud y) - F_i(x, y_k) \bigg),
\end{split}
\end{equation*} 
and similarly for $\Phi$. 
It is easily seen (\cite{tse2019higher}) that  $b$ satisfies \hregb{n}{k} and 
\hlipb{n-1}{k} (if $n \geq 1$), whereas $\Phi$ satisfies \hintphi{n}{k}. %Note that $k$ is arbitrary in $\bN$ since the dependence on measure is linear for functions $b$ and $\sigma$. 

\subsubsection{A completely non-linear example}
The following `completely non-linear' example will be very useful. 

\begin{proposition}
\label{prop:4:phi:norm:-d}
For given $\alpha \in (0,1)$ and $\nu_{0} \in {\mathcal P}(\bT^d)$, 
the function $\Phi$ below satisfies
\textrm{\rm \hintphi{\alpha/4}{2}}:
\begin{equation*}
\Phi(\mu) = \bigl\| \mu - \nu_{0} \bigr\|_{-(d+ \alpha)/2,2}^2, \quad \mu \in {\mathcal P}(\bT^d). 
\end{equation*}
Moreover, there exists a real $C$ such that, for any $\mu \in {\mathcal P}({\mathbb T}^d)$, the 
$\alpha/4$-norm of $[\delta \Phi/\delta m](\mu,\cdot)$ is less than 
$C \sqrt{\Phi(\mu)}$ and 
the (joint) $\alpha/4$-norm 
of $[\delta^2 \Phi/\delta m^2](\mu,\cdot,\cdot)$
(w.r.t. the two dot arguments) 
is less than $C$. 
\end{proposition}
%One could increase the absolute value of the Sobolev index (towards a wider space of distributions) 
%to increase the smoothness of $\Phi$. Since this is nonetheless useless for our purpose, we find it better to restrict ourselves to this statement, which we use in the following analysis. 

\begin{proof}
We let $s:=(d+\alpha)/2$. 
Then, 
it is obvious to see that (writing $\bar z$ for the conjugate of $z$) 
\begin{equation}
\label{eq:Fourier:Phi}
\begin{split}
\Phi(\mu) &= \sum_{{\boldsymbol n} \in {\mathbb Z}^d} \frac1{(1+ \vert {\boldsymbol n} \vert^2)^s} 
\bigl( \mu^{\boldsymbol n} \bar \mu^{\boldsymbol n}+ \nu_{0}^{\boldsymbol n} \bar \nu_{0}^{\boldsymbol n} - \mu^{\boldsymbol n} \bar \nu_{0}^{\boldsymbol n} - \nu_{0}^{\boldsymbol n} \bar \mu^{\boldsymbol n}
\bigr).
\end{split}
\end{equation}
%\begin{equation*}
%\begin{split}
%\biggl\vert \int_{\bT}
%e^{i 2 \pi  n \theta} d\mu(\theta) - 
%\int_{\bT} e^{i 2 \pi n \theta} p_{0}(\theta) d \theta \biggr\vert^2
%\\
%&= \sum_{n \in {\mathbb N}} \frac1{n^2} \int_{\bT}
%e^{i 2 \pi  n (\theta- \theta')} d\mu(\theta) d\mu(\theta') 
%+ \sum_{n \in {\mathbb N}} \frac1{n^2} \int_{\bT}
%e^{i 2 \pi  n (\theta- \theta')} d\mu_{0}(\theta) d\mu_{0}(\theta')
%\\
%&\hspace{15pt}-
% \sum_{n \in {\mathbb N}} \frac1{n^2}
%\int_{\bT} e^{i 2 \pi n (\theta- \theta')} d\mu(\theta) d \mu_{0}(\theta')
%- 
% \sum_{n \in {\mathbb N}} \frac1{n^2}
%\int_{\bT} e^{i 2 \pi n (\theta- \theta')} d\mu_{0}(\theta) d \mu(\theta')
%\end{split}
%\end{equation*}
Throughout the proof, we use the fact that
\begin{equation}
\label{eq:series:convergent}
\sum_{{\boldsymbol n} \in {\mathbb Z}^d}
 \frac1{(1+ \vert {\boldsymbol n} \vert^2)^s} 
< \infty, \quad \sum_{{\boldsymbol n} \in {\mathbb Z}^d}
 \frac{ \vert {\boldsymbol n} \vert^{\alpha/2}}{(1+ \vert {\boldsymbol n} \vert^2)^s} 
< \infty. 
\end{equation}
Writing the product $\mu^{\boldsymbol n} \bar \mu^{\boldsymbol n}$ in the form 
$\int_{\bT^d} e^{-\i 2 \pi {\boldsymbol n} \cdot (\theta- \theta')}  \mu(\ud \theta)  \mu(\ud\theta')$
(and similarly for the other products) and using 
\eqref{eq:series:convergent}, we 
get that 
\begin{equation*}
\begin{split}
\frac{\delta \Phi}{\delta m}(\mu)(x) 
&= \Phi^{(1)}(\mu,x) - \int_{\bT^d} \Phi^{(1)}(\mu,y)  \mu(\ud y),
\\
\Phi^{(1)}(\mu,x)
&= \sum_{{\boldsymbol n} \in {\mathbb Z}^d} \frac1{(1+ \vert {\boldsymbol n} \vert^2)^s} 
 \int_{\bT^d}
\Bigl( e^{-\i 2 \pi {\boldsymbol n} \cdot (\theta -x)} + e^{-\i 2 \pi {\boldsymbol n} \cdot (x - \theta)} \Bigr)  \bigl( \mu - \nu_{0} \bigr)(\ud \theta), 
\quad x \in \bT^d. 
\end{split}
\end{equation*}
%Using the fact that the series $\sum_{n \in\bZ^d} \vert n \vert^2 (1+\vert n \vert^2)^{-k}$ is convergent, 
%we easily deduce that, for any $\mu$ as above, the function 
%$\bT^d \ni x \mapsto [\delta \Phi/\delta m](\mu)(x)$ 
%is twice differentiable up to the order $2$, with bounded derivatives, uniformly in $\mu$. 
We then compute the derivative $\delta^2 \Phi/\delta m^2$ in a similar manner. We have
\begin{equation*}
\begin{split}
&\frac{\delta^2 \Phi}{\delta m^2}(\mu)(x,x')
=  \Phi^{(2)}(x,x') - \int_{\bT^d} \Phi^{(2)}(x,y) \mu(\ud y)
- \Bigl( \Phi^{(1)}(\mu,x') -  \int_{\bT^d} 
\Phi^{(1)} (\mu,y) \mu(\ud y)
\Bigr). 
\\
&\Phi^{(2)}(x,x') =
\sum_{{\boldsymbol n} \in {\mathbb Z}^d}
\frac1{(1+\vert {\boldsymbol n} \vert^2)^{s}} 
\Bigl( e^{-\i 2 \pi {\boldsymbol n} \cdot (x -x')} + e^{-\i 2 \pi {\boldsymbol n} \cdot (x' - x)} \Bigr), \quad x,x' \in \bT^d. 
\end{split}
\end{equation*}
Then, by Cauchy-Schwarz inequality, we can find a constant $C$, depending on $\alpha$, such that for any $\mu \in \cP(\bT^d)$ 
and $x,x' \in \bT^d$, 
\begin{equation*}
\begin{split}
\bigl\vert 
\Phi^{(1)}(\mu,x)
-
\Phi^{(1)}(\mu,x') \bigr\vert 
&\leq C \vert x - x' \vert^{\alpha/4} \sum_{{\boldsymbol n} \in {\mathbb Z}^d}
\frac{\vert {\boldsymbol n} \vert^{\alpha/4}}{(1+\vert {\boldsymbol n} \vert^2)^{s}}
 \vert  ( \mu - 
\nu_0 )^{\boldsymbol n}
\vert  
\\
&\leq
C \vert x - x' \vert^{\alpha/4} 
\biggl( \sum_{{\boldsymbol n} \in {\mathbb Z}^d}
\frac{\vert {\boldsymbol n} \vert^{\alpha/2}}{(1+\vert {\boldsymbol n} \vert^2)^{s}}
\biggr)^{1/2}
\biggl( \sum_{{\boldsymbol n} \in {\mathbb Z}^d}
\frac{ \vert  ( \mu - 
\nu_0  )^{\boldsymbol n}
 \vert ^2}{(1+\vert {\boldsymbol n} \vert^2)^{s}}
\biggr)^{1/2}
\\
&\leq C  \sqrt{\Phi(\mu)} \vert x - x' \vert^{\alpha/4}.
\end{split}
\end{equation*} 
Proceeding in a similar way with $\Phi^{(2)}$, 
the conclusion easily follows. 
\end{proof}

%\subsection{About constants}
%Unless otherwise specified, $C$ is a generic constant that only depends on $n$,  $k$, $T$, $b$ and $\Phi$, whose value varies from line to line. The dimension $d$ is assumed to be constant, so any dependence  on it will not be indicated. If the dependence of $C$ needs to be explicit when necessary, we enclose $C$ with round brackets containing all the parameters on which $C$ depends.

\section{Uniform weak propagation of chaos for McKean-Vlasov equations}
\label{se:3}

This section is dedicated to the analysis of the general 
$d$-dimensional case. 
Most of our analysis is based upon the properties of the following linearised operator:
\begin{equation}
\label{eq:linearized:operator}
L_{m} q = \tfrac12 \Delta q - \textrm{\rm div} \bigl( b(\cdot,m) q \bigr) - \textrm{\rm div} \Bigl( m \frac{\delta b}{\delta m}(\cdot,m)(q) \Bigr),
\end{equation}
for a probability measure $m \in {\mathcal P}(\bT^d)$ and 
a distribution $q$ on $\bT^d$. For an initial condition $q_0 \in ( W^{k,\infty}(\mathbb{T}^d) )'$ and a source term $r \in \cap_{T>0} L^{\infty} ([0,T], ( W^{\beta, \infty}(\mathbb{T}^d) )' )$, 
for some $k \in [0,2)$ and $\beta \in [0,2)$, we denote by {\CLinear{$\mu$}{$q_0$}{$r$}}
the related Cauchy problem, 
 defined by
\begin{equation} 
    \partial_t q(t) - L_{m(t   ;\mu)} q(t)  -  r(t)  = 0, \quad t \geq 0 \ ; \quad
             q(0) = q_0,
 \label{eq q} \end{equation}
   interpreted in the weak sense. In particular, we study the behaviour of the master equation 
\eqref{eq pde measure}  under 
suitable ergodic properties of the operators 
$(L_{m(t;\mu)})_{\mu \in {\mathcal P}(\bT^d)}$, that are stated in terms of 
{\CLinear{$\mu$}{$q_0$}{$r$}}:
\begin{description}
\item[\hspace{3pt} \hherg{$\alpha$}{$\beta$}{$\gamma$}{$(C_k)_{0 \leq k < 2}$}{$\lambda$}{$\Xi$}] 
For 
$\mu \in {\mathcal P}({\mathbb T}^d)$, 
$\alpha,\beta \in [0,2)$ and 
non-negative constants $(C_k)_{0 \leq k < 2}$ and $\lambda$, we say that $L_{m(t;\mu)}$ in 
\eqref{eq:linearized:operator} satisfies \hherg{$\alpha$}{$\beta$}{$\gamma$}{$(C_k)_{0 \leq k < 2}$}{$\lambda$}{$\Xi$} if, 
for any 
$k \in [\alpha,2)$,
$q_0 \in (W^{k, \infty}(\bT^d))'$, with $\langle q_{0}, \one \rangle =0$ and $r \in \cap_{T>0} L^{\infty} ([0,T], ( W^{\beta, \infty}(\mathbb{T}^d) )' )$, 
with $\lev r(t),\one \rev= 0$, $t \geq 0$, 
 %= \lev 1, q_0 \rev_{{1,\infty}}, 
%\\
%\displaystyle \| r(t) \|_{(k,\infty)'} \leq \textcolor{blue}{\frac{K e^{- \gamma t} + K'}{1 \wedge t^{\beta}}}, 
%\end{cases}
%\quad \quad t \geq 0, 
%\label{decay ft} 
%\end{equation} 
%for some $K, K' \geq 0$, 
the unique solution in $\bigcap_{T >0} L^{\infty} ([0,T] ,(W^{k, \infty}( \bT^d))')$ of the {Cauchy problem} %\textcolor{red}{d\'efinir $L^{(\eps)}$}
{\CLinear{$\mu$}{$q_0$}{$r$}}
satisfies (existence and uniqueness are guaranteed by Lemma \ref{lions paper pde result} below):
   \begin{equation}
   \label{erg:hyp:3} 
%   \| q(t) \|_{(k-{\alpha})',\infty} \leq  {\frac{C_k}{1 \wedge t^{\alpha/2}}}
  % \Bigl[ 
  % e^{- \lambda t} \Bigl( K +  \| q_0 \|_{(k,\infty)'} \Bigr) + K' \Bigr] + \Xi, \quad \quad t >0. 
   \forall t >0, \quad
   \| q(t) \|_{(k-{\alpha})',\infty} \leq  
   C_k
      \biggl[ 
      {\frac{     \| q_0 \|_{(k,\infty)'}}{1 \wedge t^{\alpha/2}}}
e^{- \lambda t}
   +
   \int_{0}^t \frac{\| r(s) \|_{({\beta}, \infty)'} }{(t-s)^{{\beta}/2}}
   e^{-\lambda(t-s)} \ud s
   \biggr].  
  % \int_{0}^t \frac{\| r(s) \|_{(\textcolor{blue}{\beta}, \infty)'} }{(t-s)^{\textcolor{blue}{\beta}/2}} \ud s
   \end{equation} 
  \item[\hspace{3pt} \ahergo] We say 
that $b$ satisfies \ahergo \, if 
there exists $\lambda >0$ (depending on $b$), and, for any 
$\alpha,\beta \in [0,2)$, there exist
constants
$(C_k)_{0 \leq k < 2}$  (depending on $b$, $\alpha$ and $\beta$), such that 
 $L_{m(t;\mu)}$ satisfies
{ \hherg{$\alpha$}{$\beta$}{$\gamma$}{$(C_k)_{0 \leq k < 2}$}{$\lambda$}{$\Xi$} for all $\mu \in \cP(\bT^d)$}.
\end{description}

%
%In fact, our 
%assumption \ahergo \, should be regarded as an approximated version of the assumption 
%\hergo \, that we introduced 
%in the first version  \cite{DelarueTse-arXiv} of this work. 
%Here also, we make use of \hergo \,(whose definition is given right below), but the relaxed 
%assumption  
%\ahergo \, allows us to treat more examples. 
%With our standing notation,
%the definition of \hergo \, becomes very simple: in short, 
%\hergo \, is \ahergo \, but with $\Xi \equiv 0$, that is 
%
%
%
%\begin{description}
%\item[\hspace{3pt} \herg{$\alpha$}{$\beta$}{$\gamma$}{$(C_k)_{0 \leq k \leq 2}$}{$\lambda$}] 
%For $\mu \in {\mathcal P}({\mathbb T}^d)$, for $\alpha \in [0,2)$, {$\beta \in [0,1)$, $\gamma \geq 0$}
%and for non-negative constants $(C_k)_{0 \leq k \leq 2}$ and $\lambda \geq 0$, 
%we say that $L_{m(t;\mu)}$ in 
%\eqref{eq:linearized:operator} satisfies \herg{$\alpha$}{$\beta$}{$\gamma$}{$(C_k)_{0 \leq k \leq 2}$}{$\lambda$}
%if
%it satisfies 
% \hherg{$\alpha$}{$\beta$}{$\gamma$}{$(C_k)_{0 \leq k \leq 2}$}{$\lambda$}{$\Xi$}.
%  \item[\hspace{3pt} \hergo] We say 
%that $b$ satisfies \hergo \, if, for
%any $\gamma \geq 0$, there exists $\lambda \geq 0$ (depending on $b$ and $\gamma$), with $\lambda >0$ if $\gamma >0$, and, for any $\alpha \in [\textcolor{blue}{0},2)$ and $\beta \in [0,1)$, there exist constants 
%$(C_k)_{0 \le k \le 2}$ (depending on $b$, $\alpha$, $\beta$ and $\gamma$) such that 
% $L_{m(t;\mu)}$ satisfies
%\herg{$\alpha$}{$\beta$}{$\gamma$}{$(C_k)_{0 \leq k \leq 2}$}{$\lambda$}
% for any $\mu \in \cP(\bT^d)$.
%\end{description}

The following remarks are in order. First, we sometimes say that 
\hergo \, holds but only for a given probability measure $\mu$, in which case the 
property is just assumed `for this $\mu$' (and not `for all $\mu$', {as written in the last line of the definition}). 
Second, 
%it is worth observing that, the case $k=\alpha=2$ holds for free under 
%\hergo. It suffices to apply twice the inequality 
%\eqref{erg:hyp:3} (with $\Xi \equiv 0$), once at time $t/2$ with $(k,\alpha)=(2,1)$ and another time, also at time $t/2$, 
%but with $(k,\alpha)=(1,1)$ and with $q(t/2)$ as new initial condition. Last, 
we often use the notation 
\hergcoeff{$\alpha$}{$\beta$}{$\gamma$}
(resp. 
\ahergcoeff{$\alpha$}{$\beta$}{$\gamma$})
 to say that we invoke \hergo \, (resp. 
 \ahergo)
  for this choice of $\alpha$ and $\beta$, without specifying what the values of $(C_k)_{0 \le k \le 2}$
  {(which may depend on $(\alpha,\beta)$)}
   and $\lambda$ are in that case.

Here is the main statement of this section, which summarises several results that are proven next.  

\begin{theorem} \label{thm main result:2} 
Assume
that $b$ satisfies 
\emph{\hregb{\eta}{2}}
for some $\eta \in [0,1)$
and that 
$\Phi$ 
satisfies 
\emph{\hintphi{\gamma}{2}}, for some $\gamma \in (0,1)$.  
Assume that there exists a measure $\nu_\infty$ satisfying 
\emph{\hergo}
and attracting the solutions
of 
 \eqref{eq: forward eqn }, uniformly with respect to the initial point, i.e., 
 for any $\delta >0$, there exists 
$t \geq 0$ such that 
$\textrm{\rm dist}_{\rm TV}(m(t \, ; \mu),\nu_\infty) < \delta$ for any 
$\mu \in {\mathcal P}({\mathbb T}^d)$.
Then, there exists a collection of constants $(C_\delta)_{\delta \in [0,1)}$ such that, for any $N \geq 1$, 
\begin{equation}
\label{eq:thm:main:result:2}
 \sup_{t \geq 0} \Big|  \bE[ \Phi(\mu^{N}_t)] - \Phi(\rvlaw[X_t]) \Big| \leq
 \left\{ 
 \begin{array}{ll}
C_0 N^{-1} \quad &\textrm{\rm if} \quad \eta>0,
 \\
C_\delta N^{{-1+\delta}} \quad &\textrm{\rm if} \quad \eta=0, \quad \textrm{\rm for any} \  \delta >0.
\end{array}
\right.
\end{equation}
Moreover,  
if $\nu_\infty$
satisfies \emph{\hergo} but is not a global attractor, then the above two bounds remain true up to 
any time $t \in [0,N^p]$, for any integer $p$, in which case the constants $C_0$
and $C_\delta$ also depend on $p$. 
\end{theorem}

Notice 
that the second result is just local and that the invariant measure $\nu_\infty$ is not required to be unique.  
We call this regime `metastable'. 
As a corollary of this statement, we get  
that the distance between the empirical measure and 
the solution to the Fokker-Planck equation is typically of size $N^{-1/2}$ in the norm 
$\| \cdot \|_{-(d+\alpha)/2,2}$, 
see Corollary \ref{cor:4:phi:norm:-d}
for the globally attracting regime
and Theorem \ref{prop:3:18} for the metastable case (in which case the estimate holds true up to polynomial times).

It is worth noticing that similar bounds can be proven in finite time, without any further need to 
assume \hergo \ nor the existence of an invariant measure. The corresponding statement would share some similarities with the result obtained in  
\cite{MR4377993,CHAUDRUDERAYNAL20221}, and would even provide some 
improvement since 
the drift and its derivatives in $m$ are required to be H\"older continuous in 
\cite{MR4377993,CHAUDRUDERAYNAL20221}. 
In order to clarify the finite in time versions of our statements, it is useful to formulate the following local (in time) version of 
\hergo:

\begin{description}
\item[\hspace{3pt} \hergl{$\alpha$}{$\beta$}{$T$}{$(C_k)_{0 \leq k \leq 2}$}] 
For 
$\mu \in {\mathcal P}({\mathbb T}^d)$,  
$\alpha , \beta \in [0,2)$, $T \geq 0$ and non-negative constants $(C_k)_{0 \leq k \leq 2}$, we say that $L_{m(t;\mu)}$ in 
\eqref{eq:linearized:operator} satisfies \hergl{$\alpha$}{$\beta$}{$T$}{$(C_k)_{0 \leq k < 2}$}
if, for any $k \in [\alpha,2)$, 
   \eqref{erg:hyp:3} holds true under the same
   choice of $q$ and $r$, but with $t$ in $[0,T]$ and 
$\lambda=0$.
\vspace{3pt}

  \item[\hspace{3pt} \hergol] We say 
that $b$ satisfies \hergol \, if, for any $\alpha ,\beta \in [0,2)$ and $T \geq 0$,
 there exist $(C_k)_{0 \leq k < 2}$ (depending on $b$, $\alpha$, $\beta$, $T$), such that  
 $L_{m(t;\mu)}$ satisfies
 \hergl{$\alpha$}{$\beta$}{$T$}{$(C_k)_{0 \leq k < 2}$} for any $\mu \in \cP(\bT^d)$.   
\end{description}

%In our analysis, we prove that \hergol \
%holds 
%when 
%$b$ 
% satisfies 
% {\hregb{\eta}{2}}
%for some $\eta \in [0,1)$, see 
%Subsection
%\ref{ergodic bounds}. 
The rest of the section is organised as follows: 
In Subsection \ref{subse:second:order:spatial:derivatives}, we collect preliminary results on 
the master equation \eqref{eq pde measure}. 
In Subsection 
\ref{subse:tangent}, we explain how to use \hergo \ to get long time bounds on the derivatives of the master equation and we achieve a first step in the proof of Theorem 
\ref{thm main result:2}. 
Subsection 
\ref{ergodic bounds} addresses the verification of \hergol \ in finite time. The connection
between 
\hergo \, and the long time behaviour 
of 
 \eqref{eq: forward eqn }
is studied in  
Subsection \ref{subse:ERG},  
with a special 
treatment of the metastable case in 
Subsection \ref{subse:metastable}.
 Concrete examples are discussed in Subsection 
\ref{subse:examples}.

\subsection{Second order mixed spatial derivatives of the second order linear functional derivative of \texorpdfstring{$\cU$}{U}}
\label{subse:second:order:spatial:derivatives}

We first invoke a local estimate for forward Kolmogorov equations. The proof is an obvious variant of \cite[Theorem 2.3]{tse2019higher} (see also 
\cite[Subsection 3.3]{cardaliaguet2019master}). 
%(The main estimate below is stated 
%with an integral 
%of 
%$\| r (t) \|_{(n,\infty)'}$ over $t$ while it is stated with a $\sup$ over 
%$t$ in 
%\cite{tse2019higher}; the reader may check that the proof is the same, see \cite[(2.15)]{tse2019higher}.)

\begin{lemma}%[Bound for forward Kolmogorov equations] 
\label{lions paper pde result}
Assume \textrm{\rm \hlipb{4}{2}}
 and let $q_0 \in ( W^{n,\infty}(\mathbb{T}^d) )'$ and $r \in \cap_{T>0} L^{\infty} ([0,T], ( W^{n, \infty}(\mathbb{T}^d) )' )$, for some
 $n \in \{0,1,2,3\}$. Then, the Cauchy problem
\emph{\CLinear{$\mu$}{$q_0$}{$r$}} has a unique solution in the space $\cap_{T >0} L^{\infty} ([0,T], (W^{n, \infty}(\mathbb{T}^d))' )$ such that
$$  \sup_{t \in [0,T]}  \| q(t) \|_{(n,\infty)'} \leq C \Big( \| q_0 \|_{(n,\infty)'}  + 
{\sup_{t \in [0,T]}
\| r (t) \|_{(n,\infty)'}}  \Big), $$ 
for some constant $C>0$, independent of the inputs $q_{0}$ and $r$ (but depending on $T$ and on $b$). \end{lemma}

We also recall  \cite[Theorem 4.5]{tse2019higher}, which gives a representation of the second order linear functional derivative of $\cU$ in terms of solutions of forward Kolmogorov equations.
It is quite easy to see (see the arXiv version v1 of this work \cite{DelarueTse-arXiv}) that the assumptions stated below are enough to 
apply \cite{tse2019higher}. 

\begin{proposition}%[Theorem 4.5 from \cite{tse2019higher}] 
\label{recap theorem 4.5} 
Under \emph{\hlipb{4}{2}} and \emph{\hintphi{4}{3}}, 
${\mathcal U}$ is twice differentiable with respect to $m$ and the 
first and 
second-order derivatives 
 $\frac{\delta {\mathcal U}}{\delta m}$
 and $\frac{\delta^2 {\mathcal U}}{\delta m^2}$ are given by
\begin{align} 
\label{eq: main rep}
&\ld[\cU](t  ,\mu)(z) = \ld[\Phi](m(t \, ; \mu))\Big( m^{(1)} (t \, ; \mu, \delta_{z}) \Big),
\\
&\sld[\cU](t  ,\mu)(z_1, z_2) =  \sld[\Phi](m(t \, ; \mu))\Big( m^{(1)} (t \, ; \mu, \delta_{z_1}),m^{(1)}(t \, ; \mu, \delta_{z_2}) \Big) 
 +  \ld[\Phi](m(t \, ; \mu)) \Big( m^{(2)}(t \, ; \mu, \delta_{z_1}, \delta_{z_2}) \Big), 
 \nonumber
\end{align}
where, for any $\mu,\nu \in {\mathcal P}(\bT^d)$, $m^{(1)}(\cdot \, ;\mu,\nu)  \in \cap_{T>0} L^{\infty}([0,T],  {(W^{0, \infty}(\bT^d))'}) $ satisfies the Cauchy problem
\emph{\CLinear{$\mu$}{$\nu-\mu$}{0}},
%\begin{equation} 
%\begin{cases}
%&\partial_t m^{(1)}(t \, ; \mu,\nu) - L_{m(t  ; \mu)} m^{(1)}(t \, ; \mu,\nu)  = 0, \quad t \geq 0, 
%\\
%&m^{(1)}(0 \, ; \mu,\nu) = \nu-\mu, 
%\end{cases} 
%  \label{PDE linearised} \end{equation}
and, for any $\nu_{1},\nu_{2} \in {\mathcal P}(\bT^d)$, $m^{(2)}(\cdot \, ;  \mu, \nu_1 , \nu_2)\in \cap_{T>0} L^{\infty}([0,T],  {(W^{1, \infty}(\bT^d))'}) $ satisfies the Cauchy problem 
\emph{\CLinear{$\mu$}{$\mu-\nu_2$}{$r$}}, 
with $r=(r(t))_{t \geq 0}=\emph{\text{\CSource{$\mu$}{$m^{(1)}(\cdot ; \mu, \nu_1)$}{$m^{(1)}(\cdot ; \mu, \nu_2)$}}}$,
where 
\emph{\CSource{$\mu$}{$q_1$}{$q_2$}} is a generic notation for
\begin{equation}
%\left\{
\begin{split}
\emph{\text{\CSource{$\mu$}{$q_1$}{$q_2$}$(t)$}}
&:=
- \text{\emph{div}} \Big( q_1(t)   \,   \frac{ \delta b}{ \delta m} (\cdot, m(t \, ; \mu)) \big( q_2(t) \big) \Bigr)
 - \text{\emph{div}} \Big(  q_2(t)   \frac{ \delta b}{ \delta m} (\cdot, m(t \, ; \mu)) \big( q_1(t) \big) \Bigr)
 \\
&\hspace{15pt}  - \text{\emph{div}} \Big(        m(t \, ; \mu)  \frac{ \delta^{2} b}{ \delta m^{2}}(\cdot, m(t \, ; \mu)) \big( q_1(t),  q_2(t) \big)  \Big). 
\end{split} 
\label{PDE linearised 2}  \end{equation} 
%\begin{equation}
%%\left\{
%\begin{cases}
%       & \partial_t m^{(2)} (t \, ; \mu, \nu_1,\nu_2)- 
%      L_{m(t;\mu)} m^{(2)}(t \, ; \mu,\nu_{1},\nu_{2})
%\\
%      &\quad + \text{\emph{div}} \Big(  m^{(1)} (t \, ; \mu, \nu_{1})  \,   \frac{ \delta b}{ \delta m} (\cdot, m(t \, ; \mu)) \big( m^{(1)} (t \, ; \mu, \nu_2) \big) \Bigr)
%      \\
%      &\quad + \text{\emph{div}} \Big(  m^{(1)} (t \, ; \mu, \nu_{2})  \,   \frac{ \delta b}{ \delta m} (\cdot, m(t \, ; \mu)) \big( m^{(1)} (t \, ; \mu, \nu_1) \big) \Bigr)
%      \\
%&\quad + \text{\emph{div}} \Big(        m(t \, ; \mu)  \frac{ \delta^{2} b}{ \delta m^{2}}(\cdot, m(t \, ; \mu)) \big( m^{(1)}(t \, ; \mu, \nu_{1} ),  m^{(1)} (t \, ; \mu,  \nu_{2} ) \big)  \Big)  =0,  \quad t \geq 0,   
%      \\ &  m^{(2)} (0 \, ; \mu, \nu_1, \nu_2) = \nu_2- \mu.\\
%\end{cases} 
%\label{PDE linearised 2}  \end{equation} 
\end{proposition}
Intuitively, $m^{(1)}(t \, ;\mu,\nu)$ is equal to 
$[\ud / \ud \varepsilon]_{\vert \varepsilon=0+} 
m(t,(1-\varepsilon) \mu + \varepsilon \nu)$
and 
$m^{(2)}(t; \mu, \nu_1,\nu_2)$
  to 
$[\ud/ \ud \varepsilon]_{\vert \varepsilon=0+} 
m^{(1)}(t \, ; (1-\varepsilon) \mu + \varepsilon \nu_2,\nu_1)$.

{Notice that {\text{\CSource{$\mu$}{$q_1$}{$q_2$}$(t)$}} belongs to $(W^{k+1,\infty}({\mathbb T}^d))'$ when 
$q_1(t)$ and $q_2(t)$ belong to 
$(W^{k,\infty}({\mathbb T}^d))'$, for $k \in \{0,1,2\}$.}
We now differentiate each of the two terms on the right hand side of \eqref{eq: main rep} respectively.  
%\textcolor{red}{il faudrait mettre la formule pour la derivée $\partial_\mu$. On en a besoin.}
\begin{proposition}%[Spatial differentiation, part (i)] 
\label{diff thm 1}
Under \emph{\hlipb{4}{2}} and \emph{\hintphi{4}{3}}, for any $i,j \in \{1, \ldots, d \}$, $\mu \in \cP(\bT^d)$ and $z_{1},z_{2} \in \bT^d$, 
the derivative below exists and
\begin{equation*}
\begin{split}
&(\partial_{z_2})_j (\partial_{z_1})_i \biggl\{ \sld[\Phi](m(t \, ;\mu))\Big( m^{(1)} (t \, ; \mu, \delta_{z_1}),m^{(1)}(t \, ; \mu, \delta_{z_2}) \Big)\biggr\}
 = \sld[\Phi](m(t \, ;\mu))\Big( d^{(1)}_i(t \, ; \mu, z_1),d^{(1)}_j(t \, ; \mu, z_2) \Big), 
\end{split}
\end{equation*}
%it being understood that the argument in brackets in the left-hand side is indeed differentiable with respect to 
%$z_{1}$ and $z_{2}$
%and 
  where $d^{(1)}_i(\cdot \, ; \mu, {z}) \in \cap_{T>0} L^{\infty}([0,T], (W^{1, \infty}(\bT^d))') $ satisfies the Cauchy problem 
  \emph{\CLinear{$\mu$}{$(D'_z)_i$}{$0$}}, with the distribution
%  \begin{equation} 
%\begin{cases}
%&\partial_t d^{(1)}_{i}(t \, ; \mu,z) - L_{m(t;\mu)} d^{(1)}_{i}(t \, ;\mu,z)  = 0, \quad t \geq 0, 
%\\
%&d^{(1)}_i(0 \, ; \mu,z) = (D'_z)_i, 
%\end{cases} 
%  \label{PDE linearised diff} \end{equation}
%  and 
  $(D'_z)_i$ being defined by
  $\lev \xi, (D'_z)_i \rev = \partial_{x_i} \xi(z)$.
  \end{proposition}
As made clear in the proof below, 
$d^{(1)}_i(t \, ; \mu, z) 
= [d / d \varepsilon]_{\varepsilon =0+}
m^{(1)}(t \, ;\mu,\delta_{z+ \varepsilon e_i})$. Moreover,  
notice that, from the conservative form of 
\eqref{eq:linearized:operator}, the integral of $d^{(1)}(t\, ; \mu,z)$ with respect to the Lebesgue measure is zero, i.e. 
$\langle d^{(1)}(t\, ; \mu,z),\one\rangle = 0$, for all $t \geq 0$. 

  \begin{proof}
  Existence of solution to {\CLinear{$\mu$}{$(D'_z)_i$}{$0$}} in $\cap_{T>0} L^{\infty}([0,T], 
  (W^{1, \infty}(\bT^d))')$  is guaranteed by Lemma \ref{lions paper pde result}. 
  For given $T>0$ and $i \in \{1,\cdots,d\}$, we define, for 
  $t \in [0,T]$, $\mu \in \cP(\bT^d)$, $z \in \bT^d$ and $h \in {\mathbb R} \setminus \{0\}$, 
 \begin{equation}
 \label{eq:rho:1:notation} 
   \rho^{(1)}_i(t \, ; \mu,z,h):= \frac{1}{h} \Big( m^{(1)}(t \, ;\mu,\delta_{z+he_i}) - m^{(1)}(t \, ;\mu,\delta_{z}) \Big) - d^{(1)}_i(t\, ;\mu,z) ,   \quad t \in [0,T].  
   \end{equation} 
By linearity,  $\rho^{(1)}_i(\cdot ; \mu,z,h)$ solves
\CLinear{$\mu$}{$h^{-1} ( \delta_{z+he_i} - \delta_z ) -  (D'_z)_i$}{$0$}.
%
%\begin{equation*} 
%\begin{cases}
%&\partial_t \rho^{(1)}_i(t \, ; \mu,z,h)  - L_{m(t;\mu)} \rho^{(1)}_i(t\, ; \mu,z,h)   = 0, \quad t \in [0,T], 
%\\
%&\rho^{(1)}_i(0\, ; \mu,z,h)  = h^{-1} \bigl( \delta_{z+he_i} - \delta_z \bigr) -  (D'_z)_i .
%\end{cases}
% \end{equation*}
Obviously, 
\begin{equation}
\begin{split}
\| h^{-1} ( \delta_{z+he_i} - \delta_z ) -  (D'_z)_i \|_{{(2, \infty)'}} & =   \sup_{\| \xi \|_{2, \infty} \leq 1 } \bigg[ \frac{\xi(z+he_i)-\xi(z)}{h} - \partial_{x_i} \xi(z) \bigg] \nonumber 
\leq  \frac{\vert h \vert}{2}. 
\end{split}
\label{est h/2}
\end{equation}
By Lemma \ref{lions paper pde result} {(with $n=2$)}, 
we get 
\begin{equation}
\label{eq:rho1itmuzh}
 \lim_{h \to 0} 
\sup_{t \in [0,T]} \bigl\| \rho^{(1)}_i(t \, ; \mu,z,h) \bigr\|_{(2, \infty)'}  =0.
\end{equation}
Therefore, using the smoothness of $\delta^2 \Phi/\delta m^2$,
\begin{equation}
\begin{split}
& (\partial_{z_1})_i \biggl\{ \sld[\Phi](m(t \, ;\mu))\Big( m^{(1)} (t\, ; \mu, \delta_{z_1}),m^{(1)}(t\, ; \mu, \delta_{z_2}) \Big) \biggr\}  
\\
 &=  \lim_{h \to 0} \sld[\Phi](m(t\, ;\mu))\Big( \frac{m^{(1)}(t\, ; \mu, \delta_{z_1+he_i})-m^{(1)}(t\, ; \mu, \delta_{z_1})}h,m^{(1)}(t\, ; \mu, \delta_{z_2}) \Big)   \\ 
& =  \sld[\Phi](m(t\, ;\mu))\Big( d^{(1)}_i(t\, ; \mu, z_1),m^{(1)}(t\, ; \mu, \delta_{z_2})  \Big). 
\end{split}
\label{taking limit diff} 
\end{equation}
The result follows by repeating the same procedure on $z_2$.
\end{proof}

The following result {may be} proven in the same way.
% (see the arXiv version v1 of \cite{DelarueTse-arXiv} if needed). 

 \begin{proposition}%[Spatial differentiation, part (ii)] 
 \label{diff thm 2}
Under \emph{\hlipb{4}{2}} and \emph{\hintphi{4}{3}}, for any $i,j \in \{1, \ldots, d \}$,
$\mu \in \cP(\bT^d)$ and $z_{1},z_{2} \in \bT^d$,
the derivative below exists and
 $$ (\partial_{z_2})_j (\partial_{z_1})_i  \biggl\{ \ld[\Phi](m(t \, ; \mu)) \Big( m^{(2)}(t\, ; \mu, \delta_{z_1}, \delta_{z_2}) \Big) \biggr\} = \ld[\Phi](m(t\, ;\mu))\Big( d^{(2)}_{i,j}(t\, ; \mu, z_1,z_2) \Big), $$ 
  where $d^{(2)}_{i,j}(\cdot\, ; \mu, {z_1}, z_2) \in \cap_{T>0} L^{\infty}([0,T], (W^{2, \infty}(\bT^d))') $ satisfies the Cauchy problem 
    \emph{\CLinear{$\mu$}{$0$}{$r$}}
      with 
    $(r(t))_{t \geq 0}=$\emph{\CSource{$\mu$}{$d^{(1)}_i (\cdot ; \mu, z_1)$}{$d^{(1)}_j(\cdot ;\mu,z_2)$}}, see \eqref{PDE linearised 2}.
  \end{proposition}
  
  Similar to the interpretation of 
$d^{(1)}_i(t \, ; \mu, z)$ in Proposition 
\ref{diff thm 1}, 
we have 
the representation formula
$d^{(2)}_{i,j}(t \, ; \mu, z_1,z_2) 
= [\ud / \ud \varepsilon_1]_{\varepsilon_1 =0+}
[\ud / \ud \varepsilon_2]_{\varepsilon_2 =0+}
m^{(2)}(t \, ;\mu,\delta_{z+ \varepsilon_1 e_i},,\delta_{z+ \varepsilon_2 e_j})$.

  {  \begin{proof}
As before, existence of solution to 
    \CLinear{$\mu$}{$0$}{$r$} (with the same $r$ as in the statement)
     in the space $\cap_{T>0} L^{\infty}([0,T], 
    (W^{2, \infty}(\bT^d))')$  is guaranteed by Lemma \ref{lions paper pde result}. 
 By Lemma  2.2.4 in 
\cite{cardaliaguet2013long}, we can exchange $q_1$ and $q_2$ in 
\eqref{PDE linearised 2} when the latter two satisfy 
$\langle q_1(t),\one\rangle = \langle q_2(t),\one\rangle= 0$, for all $t \geq 0$.  
 
  Next,     
     for given $T>0$ and $i \in \{1,\cdots,d\}$, we define, for 
  $t \in [0,T]$, $\mu,\nu_{2} \in \cP(\bT^d)$, $z \in \bT^d$ and $h \in {\mathbb R} \setminus \{0\}$,  
      $$ \Pi_i(t\,; \mu,z,\nu_2,h):= \frac{1}{h} \Big( m^{(2)}(t\,;\mu,\delta_{z+he_i},\nu_2) - m^{(2)}(t\,;\mu,\delta_{z},\nu_2) \Big) - \Theta_i(t\,;\mu,z, \nu_2), $$ 
for $t \in [0,T]$,       
      where
  $\Theta_i(\cdot\,;\mu,z, \nu_2) \in  L^{\infty}([0,T], (W^{2, \infty}(\bT^d))') $  satisfies the Cauchy problem
  \CLinear{$\mu$}{$0$}{$r_\Theta$}, with 
    $r_\Theta(t)=${\CSource{$\mu$}{$m^{(1)} (\cdot \, ; \mu, \nu_2)$}{$d^{(1)}_i(\cdot \, ;\mu,z)$}}.
%\begin{equation}
%\label{PDE diff 2:thetai} 
%\begin{cases}
%      {} & \partial_t \Theta_i(t\,;\mu,z, \nu_2) - L_{m(t,\mu)} \Theta_i(t\,;\mu,z, \nu_2)  
%      \\
%&\quad   + \text{{div}} \Big(m^{(1)} (t\,; \mu, \nu_2) \big)   \,   \frac{ \delta b}{ \delta m} (\cdot, m(t\,; \mu)) \big( 
%      d^{(1)}_i(t\,;\mu,z) \big) \Bigr)
%      \\
%      & \quad {+ \text{{div}} \Big(d^{(1)}_i (t\,; \mu, z) \big)   \,   \frac{ \delta b}{ \delta m} (\cdot, m(t\,; \mu)) \big( 
%      m^{(1)}(t\,;\mu,\nu_2) \big) \Bigr)
%       } \\
%      &\quad +\text{{div}} \Big(  m(t\,; \mu)  \frac{ \delta^{2} b}{ \delta m^{2}}(\cdot, m(t\,; \mu)) \big(  d^{(1)}_i(t\,;\mu,z)  ,  m^{(1)} (t\,; \mu,  \nu_{2} ) \big)   \Big)  =0,  
%      \\
%      &   \Theta_i(0\,;\mu,z, \nu_2)  = 0.\\
%\end{cases} 
%\end{equation}
By linearity and
with the same notation as in 
\eqref{eq:rho:1:notation}
for $\rho^{(1)}_i(\cdot \, ;\mu,z,h)$, 
$ \Pi_i(\cdot\,; \mu,z,\nu_2,h) \in  L^{\infty}([0,T], (W^{2, \infty}(\bT^d))')$ satisfies 
\CLinear{$\mu$}{$0$}{$r_\Pi$}
with 
    $r_\Pi(t)=${\CSource{$\mu$}{$m^{(1)} (\cdot \, ; \mu, \nu_2)$}{$\rho^{(1)}_i(\cdot \, ;\mu,z,h)$}}.
%the Cauchy problem
%\begin{equation*}
%\begin{cases}
%& \partial_t \Pi_i(t \, ; \mu,z, \nu_2,h) - L_{m(t \, ; \mu)} \Pi_i(t\, ; \mu,z, \nu_2,h)   
%  \\
%&\quad  + \text{{div}}\Bigl( m^{(1)}  (t \, ; \mu, \nu_{2} )  \,   \frac{ \delta b}{ \delta m}(\cdot, m(t \, ; \mu)) \bigl(
%      \rho^{(1)}_i(t \, ; \mu,z,h)  
%             \bigr) \Bigr)  
%       \\
%       & \quad { +  \text{{div}}\Bigl( \rho^{(1)}_i(t \, ; \mu,z,h)    \,   \frac{ \delta b}{ \delta m}(\cdot, m(t \, ; \mu)) \bigl(
%      m^{(1)}  (t \, ; \mu, \nu_{2} ) 
%             \bigr) \Bigr)  } \\
%      &\quad + \text{{div}}\Bigl(  m(t \, ; \mu)  \frac{ \delta^{2} b}{ \delta m^{2}}(\cdot, m(t\, ; \mu)) \bigl(    
%            \rho^{(1)}_i(t\, ; \mu,z,h) 
%       ,  m^{(1)} (t \, ; \mu,  \nu_{2} ) \bigr) \Bigr)=0,    %\quad t \in [0,T], 
%       \\
%      &  \Pi_i(0 \, ; \mu,z, \nu_2,h)  = 0.\\
%\end{cases}
%\end{equation*}

By estimate \eqref{eq:rho1itmuzh}, along with condition {\hlipb{4}{2}},
\begin{equation}
\begin{split}
&\lim_{h \rightarrow 0} \sup_{t \in [0,T]} \Big\|  - \text{{div}} \Big[ 
 m^{(1)} (t \, ; \mu, \nu_2) 
     \frac{ \delta b}{ \delta m} \bigl(\cdot, m(t \, ; \mu)\bigr)
     \Bigl(
     \rho^{(1)}_i(t \, ; \mu,z,h)  
     %\frac{1}{h} \big( m^{(1)}(t, \mu, \delta_{z+ he_i}) - m^{(1)}(t, \mu, \delta_{z} ) \big) - d^{(1)}_i(t,\mu,z)
      \Big)
       \Bigr]\Bigr\|_{(3, \infty)'}  \\
& =\lim_{h \rightarrow 0}   \sup_{t \in [0,T]} \sup_{\| \xi \|_{3, \infty} \leq 1} \Bigl\langle \xi,  
  - \text{{div}} \Bigl[ 
  m^{(1)} (t \, ; \mu, \nu_2)  
 \,   \frac{ \delta b}{ \delta m} (\cdot, m(t \, ; \mu))
 \Bigl(
 \rho^{(1)}_i(t \, ; \mu,z,h) 
 \Bigr)
 %\frac{1}{h} \big( m^{(1)}(t, \mu, \delta_{z+ he_i}) - m^{(1)}(t, \mu, \delta_{z} ) \big) - d^{(1)}_i(t,\mu,z) \Big)   
 \Bigr] \Bigr\rangle
  \\
& = \lim_{h \rightarrow 0}  \sup_{t \in [0,T]}  \sup_{\| \xi \|_{3, \infty} \leq 1}  \Big\langle \frac{ \delta b}{ \delta m} \bigl(\cdot, m(t \, ; \mu)\bigr)
\Bigl(
\rho^{(1)}_i(t \, ; \mu,z,h) 
%\frac{1}{h} \big( m^{(1)}(t, \mu, \delta_{z+ he_i}) - m^{(1)}(t, \mu, \delta_{z} ) \big) - d^{(1)}_i(t,\mu,z) 
\Bigr)
 \cdot \nabla\xi,
  m^{(1)} (t \, ; \mu, \nu_2)  
 \Big\rangle = 0.
\end{split} 
 \label{ 4+alpha est 1}
\end{equation}
Similarly, one can show that
\begin{equation}
\begin{split}
&\lim_{h \rightarrow 0} \sup_{t \in [0,T]} \Bigl\| -\text{{div}} \Big[  m(t \, ; \mu)  \frac{ \delta^{2} b}{ \delta m^{2}}\bigl(\cdot , m(t \, ; \mu)\bigr) \Big(  \rho^{(1)}_i(t \, ; \mu,z,h)  ,  m^{(1)} (t \, ; \mu,  \nu_{2} ) \Big)
 \\
&\hspace{60pt}    { + \rho^{(1)}_i(t \, ; \mu,z,h)    \,   \frac{ \delta b}{ \delta m}\bigl(\cdot, m(t \, ; \mu)\bigr) \bigl(
      m^{(1)}  (t \, ; \mu, \nu_{2} ) 
             \bigr)  } \Bigr] \Bigr\|_{(3, \infty)'}   
=0.
\end{split}
\label{ 4+alpha est 2} 
\end{equation}
We then conclude by \eqref{ 4+alpha est 1}, \eqref{ 4+alpha est 2}  and Lemma \ref{lions paper pde result} that
$$ \lim_{h \to 0} \sup_{t \in [0,T]} \| \Pi_i(t \, ; \mu,z,\mu_2,h) \|_{(3, \infty)'}  =0.$$
Consequently, by repeating the same argument as \eqref{taking limit diff},
it follows that
\begin{equation}
\label{PDE diff 2:thetai:2} 
(\partial_{z_1})_i \ld[\Phi](m(t \, ; \mu)) \Big( m^{(2)}(t \, ; \mu, \delta_{z_1}, \delta_{z_2}) \Big) = \ld[\Phi](m(t \, ;\mu))\Big( \Theta_i(t \, ;\mu,z_{1}, \delta_{z_2}) \Big). 
\end{equation}
By repeating the same analysis on the variable $z_2$, the proof is complete. 
      \end{proof}
}

      By combining the above results (using the 
      notations from 
      Propositions 
\ref{diff thm 1}
and
\ref{diff thm 2}), we obtain:
% our main representation formula. 
      \begin{proposition} \label{diff second order L-deriv} 
      Under \emph{\hlipb{4}{2}}
      %, \emph{\hlipb{3}{2}}
       and \emph{\hintphi{4}{3}}, the derivative $(\partial_{z_2})_j (\partial_{z_1})_i \sld[\cU]$ exists
       for any $i,j \in \{1, \ldots, d \}$
        and, for any $\mu \in \cP_2(\bR^d)$, $(\partial_{z_2})_j (\partial_{z_1})_i\sld[\cU](t, \mu,z_1, z_2)$ is uniformly bounded in $t, z_1, z_2$ and Lipschitz continuous in $z_1$ and $z_2$, uniformly in time $t$ in segments. Moreover, it can be represented by
\begin{align*} 
%\label{eq: main rep diff} 
&(\partial_{z_2})_j (\partial_{z_1})_i\sld[\cU](t ,\mu,z_1, z_2) 
 =  \sld[\Phi](m(t\, ;\mu))\Big( d^{(1)}_i(t\, ; \mu, z_1),d^{(1)}_j(t \, ; \mu, z_2) \Big) 
 + \ld[\Phi](m(t\, ;\mu))\Big( d^{(2)}_{i,j}(t\, ; \mu) (z_1,z_2) \Big). 
\end{align*}
A similar statement holds true for 
$\partial_z \ld[\cU](t, \mu,z)$, 
$\partial_{z_1} \sld[\cU](t, \mu,z_1,z_2)$
and
$\partial_{z_2} \sld[\cU](t, \mu,z_1,z_2)$. 
{Namely, 
\begin{equation*} 
\begin{split} 
&(\partial_{z})_i\sld[\cU](t ,\mu,z) 
 = \ld[\Phi](m(t\, ;\mu))\big( d^{(2)}_{i}(t\, ; \mu, z) \big),
 \\
 & (\partial_{z_1})_i \sld[\cU](t  ,\mu)(z_1, z_2) =  \sld[\Phi](m(t \, ; \mu))\Big( d^{(1)}_i(t\, ; \mu, z_1),m^{(1)}(t \, ; \mu, \delta_{z_2}) \Big) 
 +  \ld[\Phi](m(t \, ; \mu)) \Big( \Theta_i(t \, ;\mu,z_{1}, \delta_{z_2}) \Big), 
 \end{split}
\end{equation*} 
with 
$\Theta_i(t \, ;\mu,z_{1}, \delta_{z_2})$
as in 
\eqref{PDE diff 2:thetai:2} 
(and similarly for the derivative w.r.t. $z_2$). 
}
\end{proposition}
%Existence of the derivative in the left-hand side of 
%\eqref{eq: main rep diff}
%is a consequence of 
%Proposition \ref{eq:regularity:mathcalU}. 
%%\textcolor{blue}{Using the fact that the cross spatial gradient of the spatial variables of $\sld[\cU]$ is in fact the L-derivative (see \cite{buckdahn2017mean} or \cite{chassagneux2019weak} for a detailed discussion) of $\cU$ (see Proposition 5.48 of \cite{carmona2017probabilistic}),} we note that the existence of $(\partial_{z_2})_j (\partial_{z_1})_i\sld[\cU](t\, ; \mu)(z_1, z_2)$ and its regularity follow from  assumptions {\diffb}  and {\diffphi}. 
%%This fact is proven in several works in the literature, such as Theorem 7.2 of \cite{buckdahn2017mean} and Theorem 2.15 of \cite{chassagneux2019weak} (for a more general version). 
%Formula \eqref{eq: main rep diff} follows from assumptions {\hregb{4}{2}}, {\hlipb{3}{2}} and {\hintphi{4}{3}}, by combining Theorem \ref{diff thm 1} and Theorem \ref{diff thm 2}.
%\end{proof}

\subsection{From ergodic estimates on the tangent processes to uniform propagation of chaos}
\label{subse:tangent}

The following two propositions illustrate how assumption \ahergo \ is used next. 

    \begin{proposition} 
    \label{global schauder 1}
    Assume 
    \emph{\hlipb{4}{2}}
     and \emph{\ahergo} and let $m^{(1)}(\cdot \, ; \mu, \nu)$ and $d^{(1)}_i(\cdot \, ; \mu, {z}) $ 
     be defined as in Propositions  
     \ref{recap theorem 4.5} 
     and 
     \ref{diff thm 1}. Then, for any  $\alpha \in [0,1]$,  
         \begin{equation*}
    \begin{split}
    &\sup_{\mu, \nu   \in \cP(\bT^d)}  \| m^{(1)}(t \, ; \mu, \nu) \|_{(0, \infty)'} \leq C_0^{(1)} e^{-\lambda_1 t},
    \ \ 
 \sup_{z \in \bT^d} \sup_{\mu  \in   \cP(\bT^d)}   \| d^{(1)}_i(t \, ; \mu, {z}) \|_{({1-\alpha,\infty})'} \leq \frac{C_\alpha^{(1)}}{1 \wedge t^{\alpha/2}}
         e^{-\lambda_1 t}, \quad t >0, 
         \end{split}
         \end{equation*}
         where $\lambda_1$ is chosen as
          $\lambda$ in
          \emph{\ahergo}, 
         $C^{(1)}_0$   as $C_0$  in \emph{\ahergcoeff{0}{0}{1}}
         and 
    $C^{(1)}_\alpha$  as 
    $C_1$ in \emph{\ahergcoeff{$\alpha$}{0}{1}}. 
              %\textcolor{brown}{The dependence of $C^{(1)}$ and $\lambda^{(1)}$ upon 
              %$b$ is only through $\sup_{m \in {\mathcal P}({\mathbb T}^d)} \| b(\cdot,m) \|_\infty$.}
    \end{proposition}
%\color{blue}
%It must be stressed that 
%    $C^{(1)}$ and $\lambda^{(1)}$ depend implicitly on $b$, but the precise form of the dependence upon $b$ is in fact hidden in the shape of the constants $(C_k)_{0 \leq k < 2}$ and $\lambda$ in 
%    \ahergo.
%    \color{black}
    
%\textcolor{red}{Je ne vois pas très bien pourquoi on demande à $C^{(1)}$ 
%et 
%$\lambda^{(1)}$ de dépendre de $b$. J'ai plutôt l'impression que cela ne dépend que des hypothèses de 
% \hergo \,. Les choses sont un peu différentes pour $d^{(2)}_{i,j}$ puisque le terme de reste dépend explicitement de $b$. Cela a une conséquence interéssante. L'estimation de $d^{(1)}$ ne dépend que de ce que l'on est
% capable dans \hergo \, .}  

    \begin{proof}
    This result is immediate from \ahergo \, applied to 
   $m^{(1)}( \cdot \, ; \mu,\nu)$
    (with $(k,\alpha,\beta)$ therein given by $(0,0,0)$)
    and
   $d^{(1)}( \cdot \, ; \mu,z)$  (with $(k,\alpha,\beta)$ therein given by $(1,\alpha,0)$). 
    \end{proof}
    \begin{proposition} \label{global schauder 2}
    Assume  \emph{\hlipb{4}{2}}
     and \emph{\ahergo} and 
     let
      $m^{(2)}(\cdot \, ;  \mu, \nu_1 , \nu_2)$ 
     and
     $d^{(2)}_{i,j}(\cdot \, ; \mu,{ {z}_1,z_2}) $ 
     be defined as in Propositions  
     \ref{recap theorem 4.5} 
and
\ref{diff thm 2}. Then, for any  
     {$\alpha \in [1,2)$ and any $\epsilon \in (0,1]$}, 
    \begin{equation*}
\begin{split}
&\sup_{\mu, \nu_1, \nu_2  \in \cP(\bT^d)}   \| m^{(2)}(t\, ; \mu, \nu_1, \nu_2) \|_{(0,\infty)'} \leq C^{(2)}_0   e^{-\lambda_2 t} ,
\   \sup_{z_1, z_2 \in \bT^d} \sup_{\mu \in \cP(\bT^d)}   \| d^{(2)}_{i,j}(t \, ; \mu, {z_1}, z_2) \|_{(0,\infty)'} 
\leq 
{C^{(2)}_{1,\epsilon} K_{\epsilon}  }%{1 \wedge t^{1/2}}
         e^{-\lambda_2 t},  
         \end{split}
         \end{equation*}
         for $t >0$, 
where
 $\lambda_2 >0$ is chosen as $\lambda$ in 
     \emph{\ahergo};
% \emph{\ahergcoeff{0}{0}{$\lambda^{(1)}$}} and  \emph{\ahergcoeff{$1$}{$1$}{$\lambda^{(1)}$}}; 
 $C^{(2)}_0$ only depends on 
$C_0$ in 
 \emph{\ahergcoeff{0}{1}{$\lambda^{(1)}$}}, $C^{(1)}_0$ in Proposition \ref{global schauder 1}
 and $K_0$ below; $C^{(2)}_{1,\epsilon}$ only depends on  $\epsilon$, 
$C_1$ in \emph{\ahergcoeff{$1$}{$1$}{$\lambda^{(1)}$}}, 
and 
$C^{(1)}_{1-\epsilon}$ and $C^{(1)}_1$ in Proposition \ref{global schauder 1};  and
          \begin{equation*}
\forall \eta \in [0,1),\quad         K_\eta:= \sup_{m \in {\mathcal P}({\mathbb T}^d)} \biggl[       \sup_{x \in {\mathbb T}^d}
           \Bigl\| \frac{\delta b}{\delta m}(x,m,\cdot) \Bigr\|_{\eta,\infty}
           +
     \sup_{x,y \in {\mathbb T}^d}
         \Bigl\| \frac{\delta^2 b}{\delta m^2}(x,m,y,\cdot) \Bigr\|_{\eta,\infty} \biggr].      
          \end{equation*}
    \end{proposition}

%\begin{remark}
%It is worth observing that the two bounds stated in Proposition 
%\ref{global schauder 1} can be extended to $\Theta_i(t \, ;\mu,z_{1}, \delta_{z_2})$
%in 
%\eqref{PDE diff 2:thetai}. 
%Duplicating the proof that is given below, we get 
%\begin{equation*}
%\sup_{z_1, z_2 \in \bT^d} \sup_{\mu \in \cP(\bT^d)} \sup_{t \geq 0} \| \Theta_i(t \, ;\mu,z_{1}, \delta_{z_2})  \|_{(0,\infty)'} 
%\leq \frac{\widetilde C^{(2)}_{1}}{1 \wedge t^{1/2}}
%K_{0} \Bigl[ 
%         e^{-\lambda^{(2)} t} 
%         +           \Xi^{(1)}_0 \Xi^{(1)}_{1}\Bigr]
%   + \widetilde \Xi_{1}^{(2)},
%\end{equation*}
%where $\lambda^{(2)}$ 
%and
%$K_0$ are 
%as in Proposition 
% \ref{global schauder 2},
%% is chosen as in 
%% \emph{\ahergcoeff{0}{0}{$\lambda^{(1)}$}} 
%% and
%% \emph{\ahergcoeff{1}{1/2}{$\lambda^{(1)}$}}
%% (and is thus the same exponent as in Proposition 
%% \ref{global schauder 2})
% $\widetilde C^{(2)}_1$  
%depends on 
%$C_1$ in 
% {\ahergcoeff{$1$}{$1/2$}{$\lambda^{(1)}$}} 
% and $C^{(1)}_0$ and $C^{(1)}_1$ 
%  in Proposition \ref{global schauder 1}, 
%  $\Xi_0^{(1)}$ and $\Xi_1^{(1)}$ are as in Proposition 
%   \ref{global schauder 2}
%   and $\widetilde X^{(2)}_1$
%   is chosen as $\Xi$ in 
%   {\ahergcoeff{$1$}{$1/2$}{$\lambda^{(1)}$}}.
%\end{remark}

      \color{black}
    \begin{proof}
    We just study $d^{(2)}_{i,j}( \cdot \, ; \mu, {z_1}, z_2)$, the proof being similar for $m^{(2)}(\cdot \, ; \mu, \nu_1, \nu_2)$
    (with the small difference that the initial condition of the latter is not equal to $0$ but is bounded 
    in $\| \cdot \|_{(0,\infty)'}$).  We notice 
    from Proposition 
    \ref{global schauder 1} ({applied twice, with $\alpha=1$ and  
    $\alpha=1-\epsilon$} respectively)
    that, for any $z_1,z_2 \in {\mathbb R}^d$ and $\mu \in {\mathcal P}({\mathbb T}^d)$,
    \begin{equation*}
    %\label{bound 2 product w22} 
    \begin{split}
 &  \bigg\|  -\text{div} \bigg[ d^{(1)}_i (t \ ;\mu, z_2)  \,   \frac{ \delta b}{ \delta m} (\cdot, m(t \, ; \mu)) \big( d^{(1)}_j (t \, ; \mu, z_1) \big) \bigg] \bigg\|_{({1},\infty)'}  
  \\
 & =    \sup_{\| \xi \|_{{{1},\infty}} \leq 1} \bigg| \lev  \frac{\delta b}{\delta m}( \cdot, m(t\, ; \mu) )\bigl(d^{(1)}_{ {j}}(t \, ; \mu, {z_1})\bigr) \cdot \nabla \xi(\cdot) ,  d^{(1)}_i (t \, ; \mu, z_2)\rev  \bigg|
\leq \frac{C^{(1)}_{1} C^{(1)}_{1-\epsilon} K_{\epsilon}}{{1 \wedge t^{1-\epsilon/2}}} e^{- 2\lambda_1 t}, %+ K_{\epsilon} \Xi^{(1)}_1 \Xi^{(1)}_{1-\epsilon},
 %\frac{C_{1}}{1 \wedge t^{\alpha}} e^{-C_2t}, 
    \end{split}
    \end{equation*}
 for $t>0$.  {(Notice that the exponent $1-\epsilon/2$ should be understood as $(1+(1-\epsilon))/2$, with $1$ corresponding to our first choice of $\alpha$ and $1-\epsilon$ to our second choice for $\alpha$.)}  The other terms 
appearing in \eqref{PDE linearised 2} can be handled in the same way. 
% 
% 
%% \textcolor{red}{pourquoi garder le deuxième terme ci dessous ?} 
%    % by Theorem \ref{global schauder 1},  there exists $\eps_0>0$ $($depending on $B)$ such that for every $\eps \in (0, \eps_0)$,
%    \begin{equation}
%    \begin{split}
%      &  \bigg\| -\text{div} \bigg[ m(t \, ; \mu)  \frac{ \delta^{2} b}{ \delta m^{2}}(\cdot, m(t \, ; \mu)) \big(  d^{(1)}_i (t \, ; \mu, z_1), d^{(1)}_j (t \, ; \mu, z_2) \big)   \\
%    & \hspace{15pt}  {+ d^{(1)}_i (t \ ;\mu, z_1)  \,   \frac{ \delta b}{ \delta m} (\cdot, m(t \, ; \mu)) \big( d^{(1)}_j (t \, ; \mu, z_2) \big) } \bigg] \bigg\|_{({\textcolor{blue}{1},\infty})'} 
%    \\
%    &\leq \frac{C^{(1)}_{1} C^{(1)}_{1-\epsilon} K_{\epsilon}}{\textcolor{blue}{1 \wedge t^{1-\epsilon/2}}} e^{-2 \lambda^{(1)} t},% + K_{\epsilon}  \Xi^{(1)}_1 \Xi^{(1)}_{1-\epsilon},
%      \end{split}
%      \label{bound 3 product w22} 
%    \end{equation}
%    for $t>0$. 
Noting that
$d^{(2)}_{i,j}(0 \, ; \mu, {z_1},z_2)   =0$
and
%  \label{bound 4 product w22}.
%    \end{equation}
%    and
%    \begin{eqnarray}
%    0 & = & \bigg\langle 1,  -\text{div} \bigg[ m(t, \mu)\frac{\delta b}{\delta m}( \cdot, m(t,\mu) )(d^{(2)}_{i,j}(t, \mu, {z_1},z_2)  ) + d^{(1)}_i (t, \mu, z_1)  \,   \frac{ \delta b}{ \delta m} (\cdot, m(t, \mu)) \big( d^{(1)}_j (t, \mu, z_2) \big) \nonumber \\
%    && \quad \quad  + m(t, \mu)  \frac{ \delta^{2} b}{ \delta m^{2}}(\cdot, m(t, \mu)) \big(  d^{(1)}_i (t, \mu, z_1), d^{(1)}_j (t, \mu, z_2) \big) \nonumber \\
%    && \quad \quad   + d^{(1)}_j (t, \mu, z_2)  \,   \frac{ \delta b}{ \delta m}(\cdot, m(t, \mu)) \big(  d^{(1)}_i (t, \mu, z_1) \big) \bigg]  \bigg\rangle_{{2,\infty}}.  \label{bound 5 product w22} 
%    \end{eqnarray}
    %Therefore, by   \eqref{bound 2 product w22}, \eqref{bound 3 product w22}, \eqref{bound 4 product w22}  and
    applying \ahergo \, 
    with $(k,\alpha,\beta)=(1,1,1)$ (note that the source term $r$,  {with an obvious choice for $r$ therein}, clearly satisfies 
    $\langle r(t),\one \rangle =0$  {and has an integrable singularity in zero}), the result follows. 
    \end{proof}

\begin{remark}
Under  
 \hergol,
Propositions
\ref{diff second order L-deriv} 
and
\ref{global schauder 1}
remain true, provided that $t$ is restricted to  $[0,T]$ for some $T>0$, 
and $\lambda_1$ and 
$\lambda_2$ are set equal to $0$. 
In this case, the various constants may depend on $T$. 
\end{remark}

We now return to the original problem of the weak error estimate between the particle system \eqref{eq particles}  and the equation \eqref{eq:MVSDE}.      
%The following result is a combination of various results in \cite{chassagneux2019weak}: the line after equation (2.18) in the proof of Theorem 2.9, Remark 2.10, Theorem 2.11 (and its proof) and Theorem 2.16.
%\begin{theorem}
%Assume \emph{\diffb}  and \emph{\diffphi}. Suppose that $\nu \in \cP_2(\bT^d)$. Then the weak error is given by
%\begin{eqnarray}  && \bE[ \Phi(\mu^{N}_t)] - \Phi(\rvlaw[X_t]) \nonumber \\
%& =  & \frac1{N}   \int_0^1 \int_0^1  \bE \bigg[ s \frac{\delta^2 \cU}{\delta m^2}(t, {m}^{N}_{s,s_1})(\tilde{\eta},\tilde{\eta})  - s \frac{\delta^2 \cU}{\delta m^2}(t, {m}^{N}_{s,s_1})(\tilde{\eta},{{\eta}})  \bigg] ds_1 \, ds \nonumber \\
%& &  + \frac{1}{N} \sum_{i,j=1}^d \int_0^t \bE \bigg[  \intrd \bigg( \partial_{(y_{2})_j} \partial_{(y_{1})_i}  \frac{ \delta^2 \cU}{\delta m^2} (t-s, \mu^{N}_s)(z,z) \bigg)   \, \mu^{N}_s(dz) \bigg] \, ds,
%\label{consequence master equation}  \end{eqnarray}
%where $\tilde{\eta}$ and $\eta$ are i.i.d. random variables with law $\nu$ and 
%$$ {m}^{N}_{s,s_1}:= \frac{ss_1}{N}(\delta_{\tilde{\eta}} - \delta_{\eta}) + \nu + s( \mu^N_0 - \nu), \quad \quad s, s_1 \in [0,1]. $$ 
%\end{theorem}
We are now in a position to prove 
a 
preliminary version of 
Theorem \ref{thm main result:2}. 

\begin{proposition} \label{thm main result} 
Assume
that the drift $b$ 
is bounded and
satisfies 
\emph{\hregb{0}{2}}. 
For
a sequence of mollifiers
$(\rho^n=n^d \rho (n \cdot))_{n \geq 1}$ on ${\mathbb R}^d$, 
with $\rho$ a smooth symmetric density on ${\mathbb R}^d$ with a compact support, 
define  
the drifts 
 $(b^n : {\mathbb T}^d \times {\mathcal P}({\mathbb T}^d) \ni (x,m) \mapsto b(\cdot,m * \rho^n) * \rho^n \in {\mathbb R}^d)_{n \geq 1}$ 
and assume that for any $n \geq 1$ and any $\alpha,\beta \in [0,2)$, 
$b^n$ satisfies 
\emph{\ahergcoeff{$\alpha$}{$\beta$}{$\gamma$}}
 with respect to constants 
$(C_k[\alpha,\beta ])_{0 \leq k < 2}$ and $\lambda$   independent of $n$.
%$(C_k)_{0 \leq k < 2}$ and $\lambda$  that are independent $n$ 
%and to constants $C_2^n$ and $\Xi^n=\xi/n$ that tends to $0$ as $n$ tends to $\infty$;
%\end{itemize}
Then, for a function 
$\Phi$ 
that
satisfies 
\emph{\hintphi{\gamma}{2}}
(for some $\gamma \in (0,1)$)
and
for any  $\epsilon \in (0,1)$, there exists 
a constant $C>0$, only depending on $\epsilon$, 
on $K_0$ in the statement of 
Proposition 
\ref{global schauder 2}, on 
the  
$L^\infty$ and H\"older bounds in 
\emph{\hintphi{\gamma}{2}}
and on
$\max_{\alpha=0,1} \max_{\beta=0,1/2,1-\epsilon/2} 
\max_{k=0,1} ( C_k[\alpha,\beta]),$
such that, for any $\mu_{\textrm{\rm init}} \in \cP(\bT^d)$
and any two integers $N,n \geq 1$, 
\begin{equation}
 \sup_{t \geq 0} \Big|  \bE\bigl[ \Phi(\mu^{N}_t)\bigr] - \Phi\bigl(\rvlaw[X_t^n]\bigr) \Big| \leq 
 { \frac{C}{\min(n^{1-\epsilon },N)}  {\bigl(1+ K_{\epsilon}^n\bigr)  }},% \Bigl[ 1 + \Xi^p N \Bigr],
 \label{eq:weak:error}
 \end{equation}
 where $K_\epsilon^n$ is defined as
 in the statement of Proposition 
\ref{global schauder 2}, but for $b^n$ instead of $b$, and 
$(X_t^n)_{t \geq 0}$ stands for the solution of 
\eqref{eq:MVSDE}
with $b^n$ as drift
and $\mu_{\rm init}^{\otimes N}$ as initial distribution. 
\end{proposition}

In brief, the impact of $K_\epsilon^n$ in \eqref{eq:weak:error} is clarified in the proof of Theorem 
\ref{thm main result:2}. When $b$ satisfies \hregb{\alpha}{2}, 
for some $\alpha \in (0,1)$, we choose $\epsilon  = \alpha$, in which case 
$K_\alpha^n$ can be bounded independently of $n$. 
Then, 
the right-hand side is 
less than ${\mathcal O}(N^{-1})$ when $n \geq N^{1/(1-\alpha)}$.
When $b$ just satisfies {\hregb{0}{2}}, 
$\epsilon$ is arbitrary and $K_\epsilon^n = {\mathcal O}(n^{\epsilon})$, in which case the right-hand side is 
less than ${\mathcal O}(N^{-1+2\epsilon})$ when $n=N$.
% is chosen as $n=N^p$ and for some integer 
%$p \geq 1$ and $\epsilon$ is replaced by $\epsilon/p$.
%The additional parameter $p$ is tuned in the proof of Theorem
%\ref{thm main result:2} 
 %below in order to speed up the convergence of 
%$\rvlaw[X_t^{N^p}]$ to 
%$\rvlaw[X_t]$ in the left-hand side of 
%\eqref{eq:weak:error}. 

\begin{proof} 
{In the proof, we make use of two statements from the appendix: Theorem \ref{thm:regul}, which is a regularisation result 
 for real-valued functions defined on ${\mathcal P}({\mathbb T}^d)$, and Lemma \ref{lem:annex:regularity:highd}, which provides estimates for the marginal densities of large
particle systems. For the time being, we notice} that 
\begin{equation}
\label{eq:deltabn:deltam}
\frac{\delta b^n}{\delta m}(x,m)(y) = \frac{\delta b}{\delta m}(\cdot,m*\rho^n)(\cdot)* \rho_n^{\otimes 2}(x,y), 
\quad 
{
\frac{\delta^2 b^n}{\delta m^2}(x,m)(y,z) = \frac{\delta^2 b}{\delta m^2}(\cdot,m*\rho^n)(\cdot,\cdot)* \rho_n^{\otimes 3}(x,y,z)}.
\end{equation}
For
each $n \geq 1$, 
$b^n$ satisfies  
{\hregb{4}{2}}, with constants depending on $n$, and 
{\hregb{0}{2}} independently of $n$. 
However, $b^n$ may not satisfy 
{\hlipb{4}{2}}
 {since the second order derivative in $m$ may not be Lipschitz in $m$}. We invoke Theorem  
\ref{thm:regul}  {below}, from which we deduce that, for any given $\varepsilon >0$ 
 {(which is distinct from $\epsilon$ in the statement)}
and each $n \geq 1$, there exists a new drift, denoted 
$\tilde b^n$, satisfying \hlipb{4}{2} w.r.t. constants depending on $(n,\varepsilon)$ and satisfying \hregb{0}2 independently of 
$(n,\varepsilon)$, such that
\begin{equation}
\label{eq:b-tildeb:n}
\sup_{n \geq 1} 
\Bigl[ 
\sup_{x \in {\mathbb T}^d} \sup_{\mu \in {\mathcal P}({\mathbb T}^d)}
 \vert (\tilde b^n - b^n)(x,\mu)\vert + 
 \sup_{x,y \in {\mathbb T}^d} \sup_{\mu \in {\mathcal P}({\mathbb T}^d)}
\bigl\vert \frac{\delta \tilde b^n}{\delta m}(x,\mu,y) - \frac{\delta b^n}{\delta m}(x,\mu,y) \bigr\vert 
\Bigr] \leq \varepsilon.
\end{equation}
Thanks to 
Theorem  
\ref{thm:regul},
$\tilde K^n_\epsilon$, defined as in Proposition 
\ref{global schauder 2} but for $\tilde b^n$, is less than 
$C K_\epsilon^n$, for  $C$ independent of $(n,\varepsilon)$. 
Importantly, the forthcoming 
Remark 
\ref{rem:stability:erg}
says that, for $\varepsilon$ small enough and 
  $\alpha,\beta \in [0,2)$, 
$\tilde b^n$ satisfies 
{\ahergcoeff{$\alpha$}{$\beta$}{$\gamma$}}
 w.r.t. constants 
 that depend only on 
$(C_k[\alpha,\beta ])_{0 \leq k <2}$ and $\lambda$ and not on $(n,\varepsilon)$. 

 {Next}, we can expand the difference $[b-b^n](x,m)$ into 
\begin{align}
&\bigl[ b-b^n \bigr](x,m) 
\nonumber
\\
&= \Bigl( b(x,m) - b(\cdot,m)* \rho^n (x) \Bigr) 
+ \Bigl( b(\cdot,m)*\rho^n - b(\cdot,m*\rho^n) * \rho^n \Bigr)(x)
\nonumber
\\
&=   b(\cdot,m) * \bigl( \delta_0 -   \rho^n  \bigr)(x) 
  + 
\int_{{\mathbb T}^d}
\biggl(  
 \int_0^1 \ud r
\int_{{\mathbb R}^d} 
\frac{\delta b}{\delta m}\Bigl(z, r m + (1-r) m* \rho^n ,y \Bigr)
\rho^n(x-z) \ud z
\biggr) 
\bigl[ m* (  \delta_0 - \rho^n) \bigr](\ud y)
\nonumber
\\
&=:B_0^n(x,m) +
\int_{{\mathbb T}^d} B_1^n(x,m,y) \, m(\ud y),
\label{eq:expansion:mollified:drift}
\end{align}
where 
\begin{equation}
\label{eq:B1n:**}
\begin{split}
&B_1^n(x,m,y) 
:= \Bigl[ b_1^n(x,m,\cdot) * (\delta_0- \rho^n ) \Bigr] (y),
\\ 
&b_1^n(x,m,y) :=  
 \int_0^1 
\int_{{\mathbb R}^d} 
\frac{\delta b}{\delta m}\Bigl(z, r m + (1-r) m* \rho^n ,y \Bigr)
\rho^n(x-z) \ud z.
\end{split}
\end{equation}
Clearly, 
$B_0^n$ and $B_1^n$  
are bounded and Lipschitz continuous in $m$ (w.r.t. $\textrm{\rm dist}_{\rm TV}$) uniformly in 
the other variables and in 
$n \geq 1$
(which follows from the fact that $b$ and $\delta b/\delta m$ 
are Lipschitz continuous in $m$, uniformly in the other variables{, with the latter being a consequence of the boundedness of 
$\delta^2 b/\delta m^2$}).
%\begin{equation}\label{eq:Bp0:Bp1:2} 
%\begin{split}
%&\sup_{m \in {\mathcal P}({\mathbb T}^d)} \| B^{n}_0(\cdot,m) \|_{(1,\gamma)'}
%+
%\sup_{m \in {\mathcal P}({\mathbb T}^d)} \bigl\| \frac{\delta B^{n}_0}{\delta m}(\cdot,m)(\cdot) \|_{(1,\gamma)'}
%\\
%&\hspace{15pt} + \sup_{m \in {\mathcal P}({\mathbb T}^d)} \| B^{n}_1(\cdot,m)(\cdot) \|_{(1,\gamma)'}
%+
%\sup_{m \in {\mathcal P}({\mathbb T}^d)} \bigl\| \frac{\delta B^{n}_1}{\delta m}(\cdot,m,\cdot)(\cdot) \|_{(1,\gamma)'}
% \leq c \delta_n^\gamma,  
% \end{split}
%\end{equation}

The value of $\epsilon$ 
in the statement
is fixed throughout the proof
and we denote by $C$ a generic constant as in the statement (whose value is allowed to vary from line to line).
Also, we let 
(with $\lambda_1$ and $\lambda_2$ being as 
in Propositions \ref{global schauder 1}
and
\ref{global schauder 2})
$\lambda := \min(1,\lambda_1,\lambda_2)$.
We first assume that
$\Phi$ satisfies \hintphi{4}{3}. 
This is only in the end that we relax this assumption,  
 just assuming \hintphi{\gamma}{2}
(in brief, the relaxation is possible because $C$ only depends on 
$\Phi$ through the bounds in  
\hintphi{\gamma}{2}). 
\vskip 4pt

\textit{First Step.} 
For each $n \geq 1$, we consider ${\mathcal U}^{n}$ the solution to  
\eqref{eq pde measure},
with $b$ 
replaced by $\tilde b^n$.
Accordingly, 
we call $m^{(1),n}$, $m^{(2),n}$, $d^{(1),n}$ and $d^{(2),n}$ the various 
 {functionals} in 
Propositions 
\ref{recap theorem 4.5}, 
\ref{diff thm 1}
and
 \ref{diff thm 2}.

Following the derivation of
Lemma  
\ref{lem:general:lemma}, 
we can apply 
It\^o's formula to the process $({\mathcal U}^{n} (  \mu^N_t ))_{t \geq 0}$, 
with $(\mu^N_t)_{t \geq 0}$ denoting the same 
flow of empirical measures as in 
\eqref{eq particles} (driven by the unmollified drift $b$). 
In comparison with
\eqref{main result intro formula}
in the statement of Lemma
\ref{lem:general:lemma}, 
we get an additional term coming from the fact that the drift of
\eqref{eq particles} is not the same as the drift of \eqref{eq pde measure}.
This additional term is 
\begin{equation}
\label{eq:Tadd}
T_{\textrm{\rm add}}(t) := \frac1{N} {\mathbb E} \int_0^t \sum_{i=1}^N 
\bigl[ b - \tilde b^{n} \bigr]\bigl( {Y_s^{i,N}}, \mu^N_s \bigr) \cdot \partial_{\mu} {\mathcal U}^{n} \bigl(s, \mu^N_s\bigr) (X_s^i)\, \ud s.
\end{equation}
\vskip 2pt

\textit{Second Step.} We address the right-hand side 
of 
 \eqref{main result intro formula}. By 
Propositions \ref{recap theorem 4.5}
 and
 \ref{global schauder 2}, 
 we have a bound for $\delta^2 \cU^n/\delta m^2$, independent of {$(n,\varepsilon)$ (as the quantities therein are controlled independently of $(n,\varepsilon)$}), 
 namely 
 $ \| [\delta^2/\delta m^2] {\mathcal U}^{n}(s,\cdot,\cdot) \|_\infty \leq  
C e^{-\lambda t}$. Thus, \begin{equation}
\label{eq:term1:lemma2.2}
\sup_{t \geq 0} \bigg| \int_0^1 \int_0^1  \bE \bigg[ s \frac{\delta^2 \cU^{n}}{\delta m^2}(t, \tilde{\mu}^{N}_{s,s_1})(\tilde{\eta},\tilde{\eta})  - s \frac{\delta^2 \cU^{n}}{\delta m^2}(t, \tilde{\mu}^{N}_{s,s_1})(\tilde{\eta},{{\eta}_{1}})  \bigg] \ud s_1 \, \ud s \bigg| \leq C.
\end{equation}
Next, we bound the second term 
in the right-hand side 
of 
 \eqref{main result intro formula}. By 
 Propositions
\ref{diff second order L-deriv}, 
    \ref{global schauder 1}
   (together with \eqref{eq:Int-phi-alpha-k}) 
 and \ref{global schauder 2},
we have
the following bound for 
 {$\partial_{y_2} \partial_{y_1} [ \delta^2 \cU^n/\delta m^2]$}:
 \begin{equation}
\label{eq:bound:partialz1:partialz2:U:theta:**}
\sup_{\mu \in \cP(\bT^d)} \sup_{y_1,y_2 \in \bT^d} \bigg|   (\partial_{y_1})_i
(\partial_{y_1})_j \sld[\cU^n](t,\mu,y_1, y_2) \bigg| \leq  \frac{C  {(1+ \tilde K_\epsilon^n)}}{1 \wedge t^{1-\gamma}}  e^{- \lambda t}
 \leq  \frac{C  {(1+ K_\epsilon^n)}}{1 \wedge t^{1-\gamma}}  e^{- \lambda t}.
\end{equation}
  that depends on $n$ 
through  {$\tilde K^n_{\epsilon}$} and thus 
 {$K^n_{\epsilon}$}. 
And then, 
\begin{equation}
\label{eq:term2:lemma2.2}
\begin{split}
&\bigg| \sum_{i=1}^d \int_0^t \bE \bigg[  \intrd \bigg( \partial_{y_{2}} \partial_{y_{1}}  \frac{ \delta^2 \cU^n}{\delta m^2} (t-s, \mu^{N}_s,z,z) \bigg)   \, \mu^{N}_s(\ud z) \bigg] \, \ud s \biggr\vert 
\leq 
C  {(1+  {K^n_{\epsilon}})}.
\end{split}
\end{equation}
\vspace{4pt}

\textit{Third Step.}
We now address $T_{\textrm{\rm add}}$ in 
\eqref{eq:Tadd}.
We recall 
Propositions
\ref{diff second order L-deriv} 
   and  \ref{global schauder 1}. 
We have the bound $\| \partial_\mu {\mathcal U}^{n}(s,\cdot,\cdot) \|_\infty \leq  
C e^{-\lambda t}$. 
%Therefore, 
%for some fixed $S>0$ and 
%for all $t \geq S$, 
%\begin{equation}
%\label{eq:term1:lemma2.2}
%\biggl\vert \frac1{N} {\mathbb E} \int_S^t \sum_{i=1}^N 
%\bigl[ b - b^{n} \bigr]\bigl(X_s^i,  \mu^N_s \bigr) \cdot \partial_{\mu} {\mathcal U}^{n}\bigl(s,  \mu^N_s\bigr) (X_s^i) \, \ud s 
%\biggr\vert \leq C  e^{-\lambda S}.
%\end{equation}
%The difficulty in this proof is thus to handle the integral for $t \leq S$.
{Next, 
as pointed out 
in the statement of Proposition 
 \ref{diff second order L-deriv},  
the analogue of 
Proposition 
 \ref{diff second order L-deriv}
 holds true}, but for 
  $\partial_{y_1}[ \delta^2 \cU^n/\delta m^2]$
  and, similar to 
  $\partial_{y_1} \partial_{y_2}[ \delta^2 \cU^n/\delta m^2]$,
  $\partial_{y_1}[ \delta^2 \cU^n/\delta m^2]$ can be bounded by 
\begin{equation}
\label{eq:bound:partialz1:partialz2:U:theta}
\sup_{\mu \in \cP(\bT^d)} \sup_{y_1,y_2 \in \bT^d} \bigg|   (\partial_{y_1})_i\sld[\cU^n](t,\mu,y_1, y_2) \bigg| \leq  \frac{C}{1 \wedge t^{1/2}}  e^{- \lambda t}.
\end{equation}
Intuitively, the exponent $1/2$ comes from the fact that there is only one derivative $\partial_{y_1}$ (and no derivative $\partial_{y_2}$). For the same reason, 
there is no need to add the additional factor $K_n$ in the right-hand side.  
{The details are as follows. 
In the formula 
for 
$\partial_{y_1}[ \delta^2 \cU^n/\delta m^2]$
displayed in
Proposition  
\ref{diff second order L-deriv}, the first term in the right-hand side gives the 
singular behaviour in small time. As for the second term in the right-hand side, 
it can be estimated by means of 
the definition of $\Theta_i$ in Proposition 
 \ref{diff thm 2}. 
Indeed, we know that 
  $\Theta_i(\cdot\,;\mu,z, \nu_2)$  satisfies the Cauchy problem
  \CLinear{$\mu$}{$0$}{$r_\Theta$}, with 
    $r_\Theta(t)=${\CSource{$\mu$}{$m^{(1)} (\cdot \, ; \mu, \nu_2)$}{$d^{(1)}_i(\cdot \, ;\mu,z)$}}. 
    Here, $\| m^{(1)} (t \, ; \mu, \nu_2) \|_{(0,\infty)'}$ 
    is bounded and decays exponentially fast in long time whereas 
$\| d^{(1)}_i (t \, ; \mu,z) \|_{(0,\infty)'}$
blows up like $1/\sqrt{t}$ in small time and decays exponentially fast 
in long time. To get the above inequality, it then suffices to 
recall that, 
for
each $n \geq 1$, 
$b^n$ satisfies  
{\hregb{0}{2}} with constants independent of $n$.}

We now use 
\eqref{eq:b-tildeb:n}
together with {the}
expansion 
\eqref{eq:expansion:mollified:drift}. 
By the bound for 
$\partial_\mu {\mathcal U}^n$, we have
\begin{equation}
\label{eq:pre:exchangeability}
\begin{split}
T_{\textrm{\rm add}}(t) &= \frac1{N} {\mathbb E} \int_0^t \sum_{i=1}^N 
B_0^{n} \bigl( {Y_s^{i,N}}, \mu^N_s \bigr)\cdot \partial_{\mu} {\mathcal U}^n\bigl(s , \mu^N_s,{Y_s^{i,N}}\bigr) \ud s
\hspace{-1pt} 
\\
&\hspace{15pt}+ \frac1{N^2} {\mathbb E} \int_0^t \sum_{i,j=1}^N 
B_1^{n}\bigl( {Y_s^{i,N}},  \mu^N_s, {Y_s^{j,N}}\bigr)
\cdot
\partial_{\mu} {\mathcal U}^n\bigl(s,  \mu^N_s,{Y_s^{i,N}}\bigr) \ud s 
 + {\mathcal O}( \varepsilon),
\end{split}
\end{equation}
where $\vert {\mathcal O}(\varepsilon)\vert \leq C \varepsilon$. By exchangeability and once again by the bound for 
$\partial_\mu {\mathcal U}^n$, 
\eqref{eq:pre:exchangeability} may be rewritten as
\begin{equation*}
\begin{split}
T_{\textrm {\rm add}}(t) &:=  {\mathbb E} \int_0^t  
\Bigl(
B^n_{0} \bigl( {Y_s^{1,N}}, \mu^N_s \bigr)  +
B^n_1\bigl({Y_s^{1,N}}, \mu^N_s,{Y_s^{2,N}}\bigr)
\Bigr)
\cdot \partial_{\mu} {\mathcal U}^n\bigl(s, \mu^N_s,{Y_s^{1,N}}\bigr)  \ud s
 + {\mathcal O}\Bigl(\frac{1}{N} + \varepsilon \Bigr),
\end{split}
\end{equation*}
where $\vert {\mathcal O}(1/N)\vert \leq C /N$, with $C$ as in the statement. 

The goal next is to replace 
$\mu^N_s$
 by 
${\mu}_s^{N-(1,2)}$, with the latter standing for the empirical measure of 
the $(N-2)$-vector $({Y_s^{3,N}},\cdots,{Y_s^{N,N}})$. 
To do so, 
we notice that, for any two $\mu,\nu \in {\mathcal P}({\mathbb T}^d)$,  
\begin{equation*}
\partial_{\mu} {\mathcal U}^n(s,\nu,y) 
- \partial_{\mu} {\mathcal U}^n(s,\mu,y)
= 
\int_0^1
\ud r
\int_{{\mathbb T}^d} 
\partial_{y} \frac{\delta^2 {\mathcal U}^n}{\delta m^2} 
\bigl(s, r \nu + (1-r) \mu ,y,z\bigr) \, 
\bigl( \nu - \mu \bigr)(\ud z),
\end{equation*}
for $\mu,\nu \in {\mathcal P}({\mathbb T}^d)$ and $y,z \in {\mathbb T}^d$. 
And then, by 
\eqref{eq:bound:partialz1:partialz2:U:theta},
\begin{equation}
\label{eq:bound:partialz1:partialz2:U:theta:appli}
\Bigl\vert 
\partial_{\mu} {\mathcal U}^n(s,\nu,y) 
- \partial_{\mu} {\mathcal U}^n(s,\mu,y)
\Bigr\vert \leq \frac{C}{1 \wedge t^{1/2}}  e^{- \lambda t}  \textrm{\rm dist}_{\rm TV}(\mu,\nu). 
\end{equation}
Noticing that 
$\textrm{\rm dist}_{\rm TV}({\mu}_s^{N-(1,2)},{\mu}_s^{N}) \leq 4/N$, we deduce from the Lipschitz property 
of $B_0^n$ and $B_1^n$ in the measure argument that
\begin{equation*}
%\label{eq:Tadd:1,2,N-2}
\begin{split}
T_{\textrm{\rm add}}(t) &:=  {\mathbb E} \int_0^t  
\Bigl(
B^n_0 \bigl( {Y_s^{1,N}}, \mu^{N-(1,2)}_s \bigr) 
+
B^n_1 \bigl({Y_s^{1,N}},  \mu^{N-(1,2)}_s, {Y_s^{2,N}}\bigr)
\Bigr)
\cdot \partial_{\mu} {\mathcal U}^n\bigl(s,\mu^{N-(1,2)}_s,{Y_s^{1,N}}\bigr)  ds
 \\
 &\hspace{15pt} + {\mathcal O}\Bigl(\frac{1}N + \varepsilon \Bigr).
\end{split}
\end{equation*}

\textit{Fourth Step.}
In order to handle the above display, we
denote by 
$[{\mathbb T}^d]^N 
\ni
{\boldsymbol x}=(x_1,\cdots,x_N) \mapsto 
p_t^N({\boldsymbol x})$ the marginal density at time 
$t$ of the particle system \eqref{eq particles}. 
We start with  {the analysis of} $B_1^n$ in 
\eqref{eq:B1n:**}. 
 
Denoting by 
$  \mu^{N-(1,2)}_{{\boldsymbol x}}$ the empirical measure on $\{x_3,\cdots,x_N\}$ when 
${\boldsymbol x}=(x_1,\cdots,x_N)$, we have
\begin{equation*}
\begin{split}
&{\mathbb E} \int_0^t  
B^n_1 \bigl({Y_s^{1,N}},  \mu^{N-(1,2)}_s, {Y_s^{2,N}}\bigr) \cdot \partial_{\mu} {\mathcal U}^n\bigl(s, \mu^{N-(1,2)}_s,{Y_s^{1,N}}\bigr) \, \ud s
\\
&=\int_{{\mathbb T}^d}  \int_0^t \biggl[ \int_{[{\mathbb T}^d]^N}  
b^n_1(x_1,  \mu^{N-(1,2)}_{{\boldsymbol x}},x_2) 
\cdot
\partial_{\mu} {\mathcal U}^n\bigl(s,
  \mu^{N-(1,2)}_{{\boldsymbol x}},x_1
\bigr)
\\
&\hspace{15pt} \times \bigl[ p_s^N(x_1,x_2,x_3,\cdots) - 
p_s^N(x_1,x_2+z,x_3,\cdots)
\bigr]
 \ud x_1 \cdots \ud x_n \biggr]  \rho^n(z) \, \ud s \, \ud z. 
\end{split}
\end{equation*}
{We now 
apply 
Lemma 
\ref{lem:annex:regularity:highd}, using the fact that $b_1$ has same smoothness in 
$(x,y)$ as $\delta b/\delta m$ and that 
the $m$-derivative of 
$\partial_\mu {\mathcal U}^n$ 
satisfies 
\eqref{eq:bound:partialz1:partialz2:U:theta}. 
In the statement of
Lemma 
\ref{lem:annex:regularity:highd}, we use 
$\rho=1-\epsilon$.} % and$\eta=\epsilon'$.} 
We get
\begin{equation}
\label{eq:term4:lemma2.2}
\begin{split}
&\biggl\vert {\mathbb E} \int_0^t  
B^n_1 \bigl( {Y_s^{1,N}},\mu^{N-(1,2)}_s,{Y_s^{2,N}} \bigr) \cdot \partial_{\mu} {\mathcal U}^{{n}}\bigl(s, \mu^{N-(1,2)}_s, {Y_s^{1,N}}\bigr) \, \ud s
\biggr\vert
\\
&\leq \frac{C}{n^{{1-\epsilon}}}     \int_0^t \Bigl( \frac1{1 \wedge s^{(1-\epsilon)/2}} + s \Bigr) 
\frac{e^{- \lambda s}}{1 \wedge s^{1/2}}     \, \ud  s
\leq 
\frac{C}{n^{{1-\epsilon}}}.
\end{split}
\end{equation} 
Proceeding in the same way with 
$B_0$ in \eqref{eq:expansion:mollified:drift}, we get 
%
%
%\begin{equation*}
%\begin{split}
%&{\mathbb E} \int_0^t  
%B^n_0 \bigl( {Y_s^{1,N}}, \mu^{N-(1,2)}_s \bigr) \cdot \partial_{\mu} {\mathcal U}^n\bigl(s, \mu^{N-(1,2)}_s, 
%{Y_s^{1,N}}\bigr) ds
%\\
%&=\int_{{\mathbb R}^d}  \int_0^t \biggl[ \int_{[{\mathbb T}^d]^N}  
%b^n_0(x_1, \mu^{N-(1,2)}_{{\boldsymbol x}}) 
%\cdot
%\bigl[ 
%\partial_{\mu} {\mathcal U}^n\bigl(s,
%\mu^{N-(1,2)}_{{\boldsymbol x}},x_1+z
%\bigr)
%p_s^N(x_1+z,x_2,x_3,\cdots) 
%\\
%&\hspace{45pt}  - 
%\partial_{\mu} {\mathcal U}^n\bigl(s,
% \mu^{N-(1,2)}_{{\boldsymbol x}},x_1
%\bigr)
%p_s^N(x_1,x_2,x_3,\cdots)
%\bigr]
% \ud x_1 \cdots \ud x_n \biggr]  \rho^n(z) \, \ud s \, \ud z,
%\end{split}
%\end{equation*}
%from which we get 
%by Lemma 
%\ref{lem:annex:regularity:highd}
%again that
%(together with 
%the bounds for the second order-derivatives):
\begin{equation}
\label{eq:term5:lemma2.2}
\begin{split}
&\biggl\vert {\mathbb E} \int_0^t  
B^n_0 \bigl( {Y_s^{1,N}},\mu^{N-(1,2)}_s \bigr) \cdot \partial_{\mu} {\mathcal U}^n\bigl(s, \mu^{N-(1,2)}_s,
{Y_s^{1,N}}\bigr)\, \ud s
\biggr\vert
 \leq 
\frac{C}{n^{{1-\epsilon}}}.% \bigl( 1 + K_{\epsilon}^n \bigr).%  \bigl( 1 +S^2 \Xi_p \bigr).  
\end{split}
\end{equation}

\textit{Conclusion.}
We 
now combine 
 {Lemma \ref{lem:general:lemma}
with} (\ref{eq:term1:lemma2.2})--(\ref{eq:term2:lemma2.2})--(\ref{eq:term4:lemma2.2})--(\ref{eq:term5:lemma2.2}). 
We get, for any $t \geq 0$, 
\begin{equation*}
%\label{eq:weak:error:bb}
\begin{split}
&\Big|  \bE[ \Phi(\mu^{N}_t)] - \Phi(\rvlaw[\tilde X_t^n]) \Big| 
 \leq
 { \frac{C}{\min(n^{1-\epsilon},N)}  
\bigl(1+ K_{\epsilon}^n\bigr)}  + {\mathcal O}( \varepsilon ),
%  \Bigl[ 1 + \min(n,N)   e^{-\lambda S}   \Bigr].
 \end{split}
\end{equation*}
with $(\tilde X_t^n)_{t \geq 0}$ standing for the solution of 
\eqref{eq:MVSDE}
with $\tilde b^n$ as drift. 
By letting $\varepsilon$ tend to $0$ in \eqref{eq:b-tildeb:n}
{(using for instance the arguments from 
\cite{lacker18} to pass to the limit in the McKean-Vlasov SDE)}, 
we obtain
\begin{equation}
\label{eq:weak:error:bb}
\begin{split}
&\Big|  \bE[ \Phi(\mu^{N}_t)] - \Phi(\rvlaw[X_t^n]) \Big| 
 \leq
 { \frac{C}{\min(n^{1-\epsilon },N)}  
\bigl(1+ K_{\epsilon}^n\bigr)}.
%  \Bigl[ 1 + \min(n,N)   e^{-\lambda S}   \Bigr].
 \end{split}
\end{equation}
This is the result for $\Phi$ smooth. We 
can extend it to \hintphi{\gamma}{2} thanks to Theorem 
\ref{thm:regul}. 
We can find a sequence $(\Phi_{k})_{k \geq 1}$ that converges uniformly to 
$\Phi$, such that each $\Phi_{k}$ satisfies \hintphi{4}{3} (for constants depending on 
$k$) and all the functions $\Phi_{k}$ satisfy 
\hintphi{\gamma}{2}  independently of $k$.
Applying 
\eqref{eq:weak:error:bb}
to each $\Phi_{k}$ and letting $k$ tend to $\infty$, we complete the proof (as $C$ above is then independent of $k$).
 \end{proof}
\color{black}

\subsection{Estimates in finite time}
\label{ergodic bounds} 
We here prove
\hergol \ under the sole assumption 
{\hlipb{4}{2}}.  
We start with  
\begin{lemma} 
\label{conjecture backward PDE } 
Let $t>0$, $\xi \in W^{1,\infty}(\bT^d)$ and $V$ be a vector field from $[0,t] \times \bT^d$ into ${\mathbb R}^d$ that is 
 H\"older continuous in time and space. Then the {Cauchy problem} 
\begin{equation}  
    \partial_s w(s,\cdot) + \tfrac12 \Delta_{x} w(s,\cdot) +  V(s, \cdot) \cdot \nabla_{x} w(s,\cdot)  =0, \quad \quad s \in [0,t]
    \, ;  \quad
             w(t,\cdot)  = \xi,
     \label{eq backward Cauchy}    
     \end{equation}
    admits a unique solution $(w(s, \cdot))_{0\leq s \leq t}$ 
    that is continuous on $[0,t] \times \bT^d$ and classical on $[0,t) \times \bT^d$. Moreover,
    there are constants  $C,\lambda>0$ ({only depending  on 
    $V$ through 
    $\sup_{s \in [0,t]}  \|V(s,\cdot)\|_{\infty}$}) such that
\begin{equation} \bigg\| w(s, \cdot) - \intrd w(s,y) \, \ud y \bigg\|_{ \infty} \leq C\|\xi\|_{ \infty} e^{-\lambda(t-s)}, \quad \quad \forall s \in [0,t], \label{W12 estimate est 1} 
\end{equation} 
and, for any {$\alpha \in [0,2)$ and $\beta \in [0,1)$ with $\alpha \leq 1+\beta$} (also allowing $C$  to depend on $\alpha,\beta$),  
\begin{equation}\big\|  \nabla_{x}  w(s, \cdot) \big\|_{\beta,\infty} \leq 
 {\frac{C}{1 \wedge (t-s)^{(1+\beta-\alpha)/2}}}
\|\xi\|_{ {\alpha},\infty}  e^{-\lambda(t-s)} , \quad \quad \forall s \in [0,t].
\label{eq: W12 estimate est 3} 
\end{equation}
\end{lemma}

Notice that, within the framework of Lemma \ref{conjecture backward PDE }, 
the $\alpha$-H\"older semi-norm  
of $w(s,\cdot)$, for $\alpha \in (0,1)$, 
is less than $C  (1 \wedge (t-s)^{\alpha/2})^{-1} \| \xi \|_\infty 
e^{-\lambda(t-s)}$, for $C,\lambda$ as in \eqref{eq: W12 estimate est 3} (with $\alpha=\beta = 0$ therein). 
This follows from an obvious interpolation argument 
(see \eqref{eq:interpolation})
combining 
\eqref{W12 estimate est 1} and
\eqref{eq: W12 estimate est 3} (with $(\alpha,\beta)=(0,0)$ therein).

\begin{proof}
The well-posedness of
\eqref{eq backward Cauchy}
in the classical sense is a standard fact  (see \cite[Thm. 5, Chap. 3]{friedman}). Estimate \eqref{W12 estimate est 1}  is a direct consequence of \cite[Lem. 7.4]{cardaliaguet2013long}. 
In order to prove
\eqref{eq: W12 estimate est 3}, 
we recall the following standard property. For {$\alpha \in [0,2)$ and $\beta
\in [0,1)$ with $\alpha \leq 1+\beta$}, 
\begin{equation}
\label{eq:Schauder:linfinity}
\bigl\| w(s,\cdot) \bigr\|_{1+\beta,\infty} \leq \frac{C}{({t-s})^{({1+\beta-\alpha})/2}} \bigl\| w\bigl( s+ ({t-s}) \vee 1,\cdot \bigr) 
\bigr\|_{\alpha,\infty},
\quad  {s <t},  
\end{equation}
for a constant $C$ only depending on 
$\alpha$, $\beta$ and 
$\sup_{s \in [0,t]} \| V(s,\cdot) \|_\infty$. 
This gives 
\eqref{eq: W12 estimate est 3} when $t-s\leq 1$.
When $t-s \geq 1$, 
we apply 
\eqref{eq:Schauder:linfinity}
to the function  {$(w(r,\cdot) - \int_{{\mathbb T}^d} w(s+1,y) \ud y)_{s \leq r \leq s +1}$}. 
By
\eqref{W12 estimate est 1}, 
\begin{equation}
\label{eq:Schauder:linfinity:000}
\bigl\| \nabla w(s,\cdot) \bigr\|_{\beta} \leq C \Bigl\| w\bigl( s+ 1,\cdot \bigr) 
-
\int_{{\mathbb T}^d} w(s+1,y) \ud y
\Bigr\|_{\infty}
\leq C e^{ - \lambda (t - s -1 ) }, 
\end{equation}
from which we complete the proof of 
\eqref{eq: W12 estimate est 3}.
\end{proof}

Lemma
\ref{conjecture backward PDE } 
allows us to check
\hergol \, in a general setting. 
%This statement serves us as a basis for the long time analysis 
%carried out in Subsection \ref{subse:ERG} below. 

\begin{proposition} \label{W1 decay} 
Let $b$ satisfy \emph{\hlipb{4}{2}}. %and \emph{\hintphi{4}{3}}. 
%For $k \in \{0,1,2\}$, 
%let  $q_0 \in (W^{k, \infty}(\bT^d))'$, with $\langle q_{0}, \one \rangle =0$ and $r \in L^{\infty} ([0,\infty), (W^{k,\infty}(\bT^d))')$ be a function satisfying, for some constants $K, \gamma >0$, 
%\begin{equation}
%\begin{cases}
%&\lev r(t),\one \rev= 0 %= \lev 1, q_0 \rev_{{1,\infty}}, 
%\\
%& \| r(t) \|_{-k,\infty} \leq K e^{- \gamma t}, 
%\end{cases}
%\quad \quad t \geq 0. 
%\label{averaging} 
%%\label{decay ft} 
%\end{equation} 
Then, 
it satisfies \emph{$\hergol$} and, for each $\alpha$, ${\beta}$
and $T$ as in the statement of 
\emph{\hlocalcoeff{$\alpha$}{$\beta$}{$T$}}, the constants $(C_k)_{0 \leq k < 2}$ 
in 
\emph{\hergl{$\alpha$}{$\beta$}{$T$}{$(C_k)_{0 \leq k \leq 2}$}}
depend on $b$ only through
the quantities
\begin{equation}
\label{eq:quantities:ergl}
\sup_{m \in {\mathcal P}({\mathbb T}^d)} \| b(\cdot,m)\|_{\infty},
\quad
  \sup_{x \in {\mathbb T}^d}
  \sup_{m \in {\mathcal P}({\mathbb T}^d)} 
  \bigl\|
   \frac{\delta b}{\delta m} \big( x, m \big)( \cdot)
  \bigr\|_{0,\infty}.
  \end{equation}
%When $k=2$, 
%the constant $C_2$ can be chosen of the form 
%$C_\epsilon \Gamma_\epsilon^3$, for
%\begin{equation}
%\label{eq:quantities:ergl:2}
%\Gamma_\epsilon=  1 + \sup_{m \in {\mathcal P}({\mathbb T}^d)} \| b(\cdot,m) \|_{\epsilon,\infty} 
%+ 
%  \sup_{m \in {\mathcal P}({\mathbb T}^d)} 
%  \bigl\|
%   \frac{\delta b}{\delta m} \big( x, m \big)( \cdot)
%  \bigr\|_{\epsilon,\infty}.
%\end{equation}
%for an auxiliary parameter $\epsilon \in (0,1]$
%and for $C_\epsilon$ depending on $b$ only through
%the same two quantities 
%as in 
%\eqref{eq:quantities:ergl}.
\end{proposition}

%\color{black}
%\textcolor{blue}{We refer to Remark 3.13 on how to improve the definition of threshold \eqref{eq:condition:epsilon0:petit}.}\textcolor{red}{De toutes les façons, cette remarque sera à déplacer, puisque maintenant, nous ne parlons plus, à ce stade du papier, que du cas en temps fini.}
%\textcolor{red}{De façon à raccourcir le texte, on pourrait faire sauter cette remarque, qui, manifestement n'a pas trop convaincu.}
%and $\mu \in {\mathcal P}(\bT^d)$,  the {Cauchy problem} %\textcolor{red}{d\'efinir $L^{(\eps)}$}
%\begin{equation*}  \begin{cases}
%    \partial_t q(t) - L^{(\eps)}_{m(t,\mu)}q(t)  -r(t)= 0, \quad \quad t \geq 0,\\
%             q(0) = q_0,
%    \end{cases}  \end{equation*}
% has a unique solution in $L^{\infty} ([0,\infty) ,(W^{k, \infty}( \bT^d))')$. Furthermore,
% for the same value of $\eps_{0}$,  
% for any $\alpha \in [0,k] \cap [0,2)$ (and 
% for any $\eps \in (0, \eps_0)$ and $\mu \in {\mathcal P}(\bT^d)$), 
%   \begin{equation} \| q(t) \|_{(k-{\alpha})',\infty} \leq C {\frac{1}{1 \wedge t^{\alpha/2}}}
%   e^{- \lambda t} \max\Big\{1, \| q_0 \|_{(k,\infty)'} \Big\}   , \quad \quad t >0, \label{decay W1, infty} 
%   \end{equation} 
%    for some constants $C, \lambda>0$ depending on $B$, $K$, $\gamma$ and $\eps$.
    %\end{theorem} 
    
        \begin{proof}
{For $\alpha$, $\beta$ and $T$ as in   
  \hlocalcoeff{$\alpha$}{$\beta$}{$T$},
  for $k \in [\alpha,2)$,
  for 
$r \in L^{\infty} ([0,\infty), (W^{k,\infty}(\bT^d))')$, 
with $\lev r(t),\one \rev= 0$, 
 and 
  for  
  $q_0 \in (W^{k, \infty}(\bT^d))'$, with $\langle q_{0}, \one \rangle =0$, 
 we consider the solution $q$ to 
 {\CLinear{$\mu$}{$q_0$}{$r$}}
 {(see  \eqref{eq q})}
within the space {$L^{\infty}([0,T] , (W^{k, \infty}(\bT^d))')$}}
(which exists under the standing assumption on $b$). 
\vskip 4pt

\textit{First Step.}
We provide a bound for $q$ in $L^\infty([0,{T}],(W^{k, \infty}(\bT^d))')$.
    We adopt a duality approach. For a smooth function $\xi$ on $\bT^d$ {and for $t \in [0,T]$}, we consider the following Cauchy problem
   \begin{equation}   
    \partial_s w +\tfrac12\Delta_{x} w + b\bigl(x,m(s \, ; \mu) \bigr) \cdot \nabla_{x} w = 0, \quad \quad (s,x) \in [0,t] \times \bT^d \ ; \qquad 
            w(t,x) = \xi(x).   \label{cauchy w} \end{equation}
%We shall first show that $q(t)$ is bounded over all $t \geq 0$ in the $W^{1,\infty}$ norm. 
The above problem fits the assumption of Lemma 
\ref{conjecture backward PDE }: 
since $b$ is bounded,  
the path 
$(m(s \, ; \mu))_{0 \leq s \leq t}$
is $1/2$-H\"older continuous in $s$ w.r.t. ${\mathcal W}_{1}$ and,  
by 
\hlipb{4}{2}, 
the transport coefficient in 
\eqref{cauchy w}
has time-space continuous derivatives. 
Since $b$ has derivatives in $x$ of order 1 and 2, 
$w$ has derivatives up to the order 4. 
Therefore, we can expand 
the duality product $\langle w(s,\cdot), q(s) \rangle$. 
By \eqref{eq:linearized:operator}
and 
\eqref{eq q}, 
    \begin{align}
\nonumber
      \langle \xi , q(t)   \rangle 
& =  \langle \bar w(0,\cdot)  , q(0)   \rangle
  +   \int_0^t \langle \bar w(s,\cdot) %-  \intrd w(s,y) \ud y 
  , r(s)   \rangle \, \ud s
 +    \int_0^t \int_{\bT^d}     \frac{\delta b}{\delta m} \big( x, m(s,\mu) \big)\bigl(q(s)\bigr) \cdot \nabla_{x} w(s,x)  \,  m(s, \mu)  (\ud x) \, \ud s
\\
&=: T_{1} + T_{2} + T_{3},
    \label{duality W1 infty} 
\end{align}
with $\bar w(s,\cdot) = w(s,\cdot) -   \intrd w(s,y) \ud y$
{(notice that we can write $T_1$ and $T_2$ in terms of $\bar w$ since $q(0)$ and 
$r$ are centred)}. 
    By 
  {Lemma}  
    \ref{conjecture backward PDE },
%    estimates \eqref{W12 estimate est 1} (which suffices if $k=0$ and $\alpha=0$), \eqref{eq: W12 estimate est 3} (the latter with
%    $(\alpha,\beta)$ therein given by: 
%      $(\alpha,0)$ if $k=1$ and $\alpha \in [0,1]$; 
%    $(\alpha-1,1)$ if $k=2$ and $\alpha \in [1,2)$)
% and
%    \eqref{W22 est 1} (with $(\alpha,\beta)$ therein given by $(\alpha,0)$ if $k=2$ and $\alpha \in [0,1)$), %\textcolor{red}{pr\'eciser la d\'ependance de $C$}
  \begin{equation}
 \big| 
 T_{1} \big| \leq 
 {
  \| \bar w(0,\cdot) \|_{k,\infty} \| q(0) \|_{(k,\infty)'}}   
 \leq
 \frac{C_{\alpha, {k},b}}{1 \wedge t^{\alpha/2}}
   \|\xi\|_{k-\alpha,\infty } \| q_0 \|_{(k,\infty)'},\label{eq: bound q est 1}
  \end{equation}
  where $C_{\alpha,{k},b}$ only depends on $\alpha$, {$k$} and $b$ 
   through the quantity 
   $\sup_{m \in {\mathcal P}({\mathbb T}^d)} \| b(\cdot,m)\|_\infty$.
   {As for $T_3$, we have, by 
 \eqref{eq: W12 estimate est 3} (with
 $(\alpha,\beta)$ therein given by  
 $( {0},0)$),}
\begin{equation}
\begin{split}
 \big| T_{3} \big| &  \leq   C_{\alpha,b} 
   \int_0^t  \| \nabla_x w(s, \cdot)  \|_{{\infty}}   \| q(s)  \|_{({0, \infty})'}
 \, \ud s
 \leq 
   C_{\alpha,b}    {\| \xi \|_{{\infty}}}  
 \int_0^t    \frac{ \| q(s)  \|_{({0, \infty})'}}{1 \wedge (t-s)^{1/2}}
 \, \ud s. 
\end{split}
\label{eq: bound q est 2}  
\end{equation}
 Similarly, 
\begin{equation}
\begin{split}
\bigl| T_{2} \bigr| &\leq C_{\alpha,\beta,b,T}  
\| \xi \|_{{\infty}}   
\int_{0}^t \frac{\| r(s) \|_{({\beta}, \infty)'} }{1 \wedge (t-s)^{{\beta}/2}} \ud s.
\end{split}
 \label{eq: bound q est 3}   
\end{equation} 
\vskip 4pt

\textit{Third Step.} 
By combining \eqref{duality W1 infty}, \eqref{eq: bound q est 1}, \eqref{eq: bound q est 2}  and \eqref{eq: bound q est 3},   
 we have
%for $(k,\alpha)=(0,0)$ or $(k,\alpha) \in \{1\} \times [0,1]$,  
\begin{equation} 
\begin{split}
\| q(t) \|_{(k-\alpha, \infty)'} &\leq   C_{\alpha,k,b} \biggl[ \frac{\| q_0 \|_{(k,\infty)'}}{1 \wedge t^{\alpha/2}}  +  
\int_0^t \frac{ \| q(s) \|_{(0,\infty)'}}{1 \wedge (t-s)^{1/2}} ds 
\biggr]
+
C_{\alpha,\beta,b,T}
\int_{0}^t \frac{\| r(s) \|_{({\beta}, \infty)'} }{1 \wedge (t-s)^{{\beta}/2}} \ud s.
\end{split}
\label{eq: bound q est 2021}
\end{equation} 
%Assuming that $k -\alpha=\delta$ (with $\alpha \in [1,2)$ if $k=2$),
%By an inequality of Gronwall type, we get the expected result,  
%namely 
%\begin{equation} 
%\begin{split}
%\| q(t) \|_{(0, \infty)'} &\leq   C_{0}(\alpha,b)  \| q_0 \|_{(0,\infty)'}
%+   C(\alpha,T,b)  \sup_{0 \leq s \leq t} \| r(s) \|_{(k, \infty)'}.
%\end{split}
%\label{eq: bound q est 2021:0}
%\end{equation} 
%
%We now address the case $k=\alpha=1$. 
Multiplying 
by $\exp(-ct) (\tau-t)^{-1/2}$, for some $c >0$, and integrating with respect to 
$t$ from $0$ to $\tau$, for some $\tau >0$,  
 we obtain, for $C=C_{\alpha,\beta,k,b,T,\tau}$,
{\begin{equation*}
%\label{eq: bound q est 2021:1:0}
\begin{split}
\int_0^{\tau} e^{-ct} \frac{ \| q(t) \|_{(k-\alpha, \infty)'}}{(\tau-t)^{1/2}} dt 
&\leq   C   \| q_0 \|_{(k,\infty)'}   
 +
  C 
 \int_0^{\tau} 
e^{-cs} \| q(s) \|_{(0, \infty)'} 
 \biggl( \int_{s}^\tau \frac{e^{-c(t-s)} }{(\tau-t)^{1/2} [1 \wedge (t-s)^{{1}/2}]} \ud t
 \biggr) \ud s.
 \\
&\hspace{15pt} +   
   C
 \int_0^{\tau} 
e^{-cs} \| r(s) \|_{({\beta}, \infty)'} 
 \biggl( \int_{s}^\tau \frac{e^{-c(t-s)} }{(\tau-t)^{1/2} [ 1 \wedge (t-s)^{{\beta}/2}]} \ud t
 \biggr) \ud s.
\end{split}
\end{equation*}
and then
%, using the fact that $e^{-c(t-s)} \leq C e^{-c(t-s)/2}/(t-s)^{1/2}$ 
(for a possibly new choice of $C$),}
 \begin{equation*}
%\label{eq: bound q est 2021:1:0}
\begin{split}
\int_0^{\tau} e^{-ct} \frac{ \| q(t) \|_{(k-\alpha, \infty)'}}{(\tau-t)^{1/2}} dt 
&\leq   C   \| q_0 \|_{(k,\infty)'}   
 +
  C 
 \int_0^{\tau} 
e^{-cs} \| q(s) \|_{(0, \infty)'} 
 \biggl( \int_{s}^\tau \frac{e^{-c(t-s)} }{(\tau-t)^{1/2} (t-s)^{{1}/2}} \ud t
 \biggr) \ud s.
 \\
&\hspace{15pt} +   
   C
 \int_0^{\tau} 
e^{-cs} \| r(s) \|_{({\beta}, \infty)'} 
 \biggl( \int_{s}^\tau \frac{e^{-c(t-s)} }{(\tau-t)^{1/2} (t-s)^{{\beta}/2}} \ud t
 \biggr) \ud s,
\end{split}
\end{equation*}
%where we used the fact that $e^{-c(t-s)} \leq C e^{-c(t-s)/2}/(t-s)^{1/2}$ (for a possibly new choice of $C$). 
%We observe that the two integrals from $t$ to $\tau$ that appear in the above 
%integrals can be (respectively)  
%bounded by $\kappa$ and $\kappa_\beta (\tau-s)^{-(\beta-1)/2}$.
Choosing $c$ large enough, the first integral from $s$ to $\tau$ in the right-hand side can be made small.
Then, \begin{equation*}
\label{eq: bound q est 2021:1:0}
\begin{split}
&\int_0^{\tau} e^{-ct} \frac{ \| q(t) \|_{(k-\alpha, \infty)'}}{(\tau-t)^{1/2}} \ud t \leq   C 
  \| q_0 \|_{(k,\infty)'}   
 +
\frac12 
 \int_0^{\tau} 
e^{-cs} \frac{\| q(s) \|_{(0, \infty)'} }{(\tau-s)^{ {1}/2}}  \ud s
+   
  C 
  \int_0^{\tau} 
e^{-cs} \frac{\| r(s) \|_{({\beta}, \infty)'} }{1 \wedge (\tau-s)^{\beta/2}}  \ud s.
\end{split}
\end{equation*}
Obviously, this gives a bound for the left-hand side when $k=\alpha=0$ (provided that it is finite).
Assuming first that $\| q_0 \|_{(0,\infty)'}$ and $\sup_{0 \leq t \leq \tau} \| r(t)\|_{(0,\infty)'}$
are finite 
(so that $\sup_{0 \leq t \leq \tau} \| q(t)\|_{(0,\infty)'}$
is finite thanks to Lemma \ref{lions paper pde result}), 
and inserting above bound in 
\eqref{eq: bound q est 2021}, 
we get 
$\hergol$
(for this type of initial condition). The assumptions on $q_0$ and $r$ can be easily dropped by mollification:
we get $\hergol$ 
for any initial condition. 
\end{proof}
 \color{black}

\subsection{Connection between the assumption \emph{\hergo}  and 
the long time behaviour of the McKean-Vlasov equation}
\label{subse:ERG}

We  establish connections between {\hergo} and  the long time behaviour of 
\eqref{eq:MVSDE}--\eqref{eq: forward eqn }. We start with 
\begin{proposition}
\label{eq:prop:dTV:exp}
Under 
\emph{\hregb{0}{2}}
and 
\emph{\hergo}, 
the equation \eqref{eq:MVSDE} has a unique invariant measure $\nu_{\infty}$ and it is exponentially stable, i.e., there exist two constants $C,\lambda >0$ such that, for any 
$\mu \in \cP(\bT^d)$,  
\begin{equation}
\label{eq:expo:stability}
\textrm{\rm dist}_{{\rm TV}} \Bigl( m(t \, ;\mu) ,  \nu_{\infty} \Bigr) \leq C e^{-\lambda t}, \quad t \geq 0.
\end{equation} 
\end{proposition}

\begin{proof}
For a smooth test function $f$ on the torus, we choose 
$\Phi(m)=\langle f,m\rangle$ in the statement of 
Proposition 
\ref{recap theorem 4.5}. 
And then,  
by 
Proposition \ref{global schauder 1}, we obtain that, for any two measures $\mu_{1},\mu_{2} \in 
{\mathcal P}(\bT^d)$,  
\begin{equation}
\label{eq:existence:invariant}
\begin{split}
\bigl\vert \langle f, m(t \, ; \mu_{1}) - m(t \, ; \mu_{2}) 
\rangle 
\bigr\vert &=  \biggl\vert \int_{0}^1 \biggl( \int_{{\mathbb T}^d} m^{(1)}\bigl(t \, ; \lambda \mu_{1}
+
(1-\lambda) \mu_{2} ,\delta_z \bigr)(f) 
\ud \bigl( \mu_2-\mu_1 \bigr)(z)
\biggr)
\ud \lambda
\biggr\vert
\\
& \leq C^{(1)}_0 \| f\|_{\infty} e^{-\lambda_1 t} \textrm{\rm dist}_{{\rm TV}}(\mu_1,\mu_2)
 \leq C^{(1)}_0 \| f\|_{1,\infty} e^{-\lambda_1 t} \textrm{\rm dist}_{{\rm TV}}(\mu_1,\mu_2). 
\end{split}
\end{equation}
%\textcolor{blue}{where we used the fact that 
%$m^{(1)}(t \, ; \lambda \mu_{1}
%+
%(1-\lambda) \mu_{2})
%= 
%\lambda 
%m^{(1)}(t \, ; \mu_{1}) 
%+
%(1-\lambda) 
%m^{(1)}(t \, ; \mu_{2}) $}. 
By choosing $\mu_{1}$ and 
$\mu_{2}$ as two candidates for being an invariant measure, 
this shows that an invariant measure (if it exists) must be unique. Existence follows by choosing $\mu_{1}=\mu$ and then $\mu_{2}=m(s \, ; \mu)$ in 
\eqref{eq:existence:invariant}. By the flow property of McKean-Vlasov dynamics, we have $m(t \, ; m(s \, ; \mu))=m(t+s \, ; \mu)$, and thus
\begin{equation*}
\begin{split}
\bigl\vert \langle f, m(t+s \, ; \mu) - m(t \, ; \mu) 
\rangle 
\bigr\vert \leq 2C_0^{(1)} \| f\|_{\infty} e^{-\lambda_1 t}. 
\end{split}
\end{equation*}
By completeness of $({\mathcal P}(\bT^d),{\mathcal W}_1)$, we deduce that 
$(m(t\, ; \mu))_{ t\geq 0}$ has a limit. We call it $\nu_{\infty}$. 
Writing the above left-hand side in the form 
$\vert \langle f, m (s \, ; m(t\, ;\mu)) - m(t \, ; \mu) 
\rangle 
\vert$ and letting $t$ tend to $\infty$, it is easy to deduce that $m(s\, ;\nu_{\infty})=\nu_{\infty}$. Finally, by choosing 
$\mu_{1}=\nu_{\infty}$ in 
\eqref{eq:existence:invariant}, we get that 
$\nu_{\infty}$
is exponentially stable. 
\end{proof}

{An interesting question is to determine to which extent the converse to 
Proposition 
\ref{eq:prop:dTV:exp}
holds true:}

\begin{proposition} \label{prop main result:2} 
Assume
that $b$ satisfies 
\emph{\hregb{0}{2}}
and that there exists an invariant measure $\nu_\infty$ satisfying 
\emph{\hergo}. For
$(\rho^n=n^d \rho (n \cdot))_{n \geq 1}$, 
with $\rho$ a smooth symmetric compactly supported density on ${\mathbb R}^d$, 
define  
the same drifts 
 $(b^n : {\mathbb T}^d \times {\mathcal P}({\mathbb T}^d) \ni (x,m) \mapsto b(\cdot,m * \rho^n) * \rho^n \in {\mathbb R}^d)_{n \geq 1}$
as in 
Proposition  
\ref{thm main result}.

Then, there exist $\lambda_0 >0$ and $C_0 \geq 0$, and for any $\alpha \in [0,2)$, there exist $a_\alpha >0$ and $(C_{\alpha,\beta} \geq 0)_{0<\beta<1} $, such that, for any 
$\mu \in {\mathcal P}({\mathbb T}^d)$ with $\| \mu - \nu_\infty\|_{(\alpha,\infty)'} \leq a_\alpha$, 
any integer $n \geq 1$ and any $\beta \in (0,1)$, 
\begin{equation*}
\forall t \geq 0, \quad \| m^n(t  \, ; \mu) - m(t \, ; \mu) \|_{(0,\infty)'}  
\leq \frac{C_{\alpha,\beta}}{n^\beta}, \quad 
 \| m^n(t  \, ; \mu) - \nu_\infty \|_{(0,\infty)'} 
\leq C_0 e^{-\lambda_0 t} 
+ {\frac{C_{\alpha,\beta}}{n^\beta}}, 
\end{equation*}
where $(m^n(t \, ; \mu))_{t \geq 0}$ is the solution to 
\eqref{eq: forward eqn }
 with $b^n$ as drift. 
And, 
for any large integer $n$ and any $\mu \in {\mathcal P}({\mathbb T}^d)$ 
with $\| \mu - \nu_\infty\|_{(\alpha,\infty)'} \leq a_\alpha$, 
$b^n$ satisfies \emph{\hergo} at $\mu$,  with  $(C_k)_{0 \leq k < 2}$ and $\lambda$  therein being uniform w.r.t. 
$\mu$.
\end{proposition}

\begin{proof} { \ }
We first recall the following preliminary result, see for instance 
\cite[Prop. 2.3]{jourdain:hal-03223426}.
For any $t>0$, $m(t \, ; \mu)$ has a density, denoted by 
$x \in {\mathbb T}^d \mapsto p(t,x \, ; \mu)$. For any $\beta \in (0,1)$, it satisfies 
\begin{equation}
\label{eq:preliminary:bound:density}
p(t,x\, ; \mu) \leq C g_\mu(t,x) \ ; \quad
\bigl\vert p(t,x\, ; \mu) - 
 p(t,x'\, ; \mu) \bigr\vert \leq \frac{C_\beta}{1 \wedge t^{\beta/2}} \vert x -x' \vert^\beta 
 \Bigl( g_\mu(t,x) + g_\mu(t,x') \Bigr),
 \end{equation}
 where $g_\mu(t,\cdot)$ is a density that satisfies $\| g_\mu(t,\cdot)\|_\varrho \leq C_\varrho \min(1,t)^{d(1/\varrho-1)/2}$, for 
 $\varrho \in (1,+\infty]$. 
 \vskip 4pt

\textit{First Step.} We  provide a bound for the distance between $(m^n(t \, ; \mu))_{t \geq 0}$ and $(m(t \, ; \nu))_{t \geq 0}$.  
 We write 
 \begin{equation*}
 \begin{split}
& \partial_t \bigl[ m^n(t \, ;\mu) - m(t \, ; \nu) \bigr] 
-  \frac12 \Delta \bigl( m^n(t \, ;\mu) - m(t \, ; \nu) \bigr) 
\\
&\hspace{5pt} + \textrm{\rm div} 
     \Bigl( b^n\bigl(\cdot,m^n(t\, ; \mu) \bigr) 
     \bigl[
     m^n(t \, ;\mu) - m(t \, ; \nu) \bigr] \Bigr)
     + \textrm{\rm div} 
     \Bigl( \bigl[ b^n\bigl(\cdot,m^n(t \, ; \mu)\bigr) - b \bigl(\cdot,m(t \, ;\nu) \bigr)   \bigr] m(t \, ;\nu) \Bigr) =0,
 \end{split}
 \end{equation*}
 with $\mu-\nu$ as initial condition. We rewrite the last term in the above left-hand side as 
 \begin{equation*}
 \begin{split}
&  b^n\bigl(\cdot ,m^n(t \, ;\mu)\bigr)- b \bigl(\cdot,m(t \, ;\nu) \bigr) 
 \\
 &=
  \Bigl[
  b^n\bigl(\cdot ,m^n(t \, ;\mu)\bigr)
  - b^n\bigl(\cdot,m(t \, ;\nu)\bigr) 
  \Bigr]
 +
 \Bigl[ 
 b^n\bigl(\cdot ,m(t \, ;\nu)\bigr)
 -
 b \bigl(\cdot ,m(t \, ;\nu) \bigr) 
 \Bigr]
 \\
 &=
  \biggl[   \frac{\delta b^n}{\delta m}\bigl(\cdot, m^n(t \, ;\mu)\bigr)\Bigl( m^n(t \, ;\mu) - m(t \, ;\nu)  \Bigr)
  \biggr]
  +
    \Bigl[ b^n\bigl(\cdot ,m(t \, ;\nu)\bigr)
 -
 b \bigl(\cdot,m(t \, ;\nu) \bigr) 
 \Bigr]
  \\
&\hspace{15pt} + \biggl[ \int_0^1   
\Bigl[ \frac{\delta b^n}{\delta m}\bigl(\cdot, r m^n(t \, ;\mu) + (1-r) m(t \, ;\nu) \bigr)
-
\frac{\delta b^n}{\delta m}\bigl(\cdot, m^n(t \, ;\mu) \bigr)  
    \Bigr]\Bigl( m^n(t \, ;\mu) - m(t \, ;\nu) \Bigr) \ud r \biggr]
 \\
&=:
    \frac{\delta b^n}{\delta m}\bigl(\cdot, m^n(t \, ;\mu)\bigr)\Bigl(m^n(t \, ;\mu) - m(t \, ;\nu)\Bigr)
    +\vartheta^n_1(t)(\cdot) 
+\vartheta^n_2(t)(\cdot).
  \end{split}
  \end{equation*}
We thus have 
 \begin{equation}
 \label{eq:equation:mn-n:stab}
 \begin{split}
& \partial_t \bigl[ m^n(t \, ;\mu) - m(t \, ; \nu) \bigr] 
- L^n_{m^n( t\, ; \mu)}  \bigl( m^n(t \, ;\mu) - m( t \, ; \nu) \bigr) 
+ {\rm div} \Bigl( \bigl[ \vartheta^n_1(t) + \vartheta^n_2(t) \bigr] m(t \, ;\nu) \Bigr) 
\\
&\hspace{15pt} + 
 { 
 {\rm div} \Bigl(
  \frac{\delta b^n}{\delta m}\bigl(\cdot, m^n(t \, ;\mu)\bigr)\Bigl(m^n(t \, ;\mu) - m(t \, ;\nu)\Bigr)
 \bigl( m(t \, ;\nu)
-
m^n(t \, ;\mu)
\bigr) 
\Bigr)} 
=0.
 \end{split}
 \end{equation}
 Above, $L^n_{m^n(t,\mu)}$ 
  denotes the linearised operator 
  \eqref{eq:linearized:operator}
 associated with $b^n$. 
By using the fact that $\delta^2 b^n/\delta m^2$ is bounded,  
% Writing 
%\begin{equation*}
%\vartheta^n_2(t)(\cdot)
%=  \int_0^1 \int_0^1  (r-1) 
% \frac{\delta^2 b^n}{\delta m^2}\bigl(\cdot,  (1 - s +s r) m^n(t \, ;\mu) + s (1-r) m(t\, ; \nu) \bigr)
%\Bigl( m^n(t \, ;\mu) - m(t \, ;\nu),m^n(t \, ;\mu) - m(t \, ;\nu) \Bigr)  
%\ud r \, \ud s,
%\end{equation*}
we get $\| \vartheta^n_2(t) \|_{\infty}\leq C \| m^n(t \, ;\mu) - m(t \, ; \nu) \|^2_{(0,\infty)'}$,
the constant $C$ being allowed (here and in the rest of the proof) to depend 
on the bounds in 
{\hregb{0}{2}}. We refer to the preliminary step in the proof of Proposition \ref{thm main result}  for the fact that 
each $b^n$ satisfies 
{\hregb{0}{2}} with constants that are independent of $n$. 
As for $\vartheta^n_1(t)$, we 
write it in the form:
\begin{equation*}
\begin{split}
\vartheta^n_1(t) 
&= 
 \Bigl[ b \bigl(\cdot,m(t \, ;\nu) * \rho^n \bigr) - 
 b\bigl(\cdot,m(t \, ;\nu)\bigr) \Bigr] * \rho^n
+
\Bigl[ b \bigl(\cdot,m(t \, ;\nu) \bigr) * \rho^n - 
 b\bigl(\cdot,m(t \, ;\nu)\bigr) \Bigr]
 =:\vartheta^n_{1,1}(t)+\vartheta^n_{1,2}(t).
\end{split}
\end{equation*}
Recalling 
\eqref{eq:preliminary:bound:density}, we obtain 
\begin{equation*}
\begin{split}
\| \vartheta^n_{1,1}(t) \|_{\infty} \leq C \textrm{\rm dist}_{\rm TV} \bigl(m(t \, ;\nu) , m(t \, ;\nu)*\rho^n \bigr)
&\leq C \int_{[{\mathbb T}^d]^2}  \vert p(t,x ;\nu) - p(t,x-y ;\nu)  \vert \rho^n(y)  \ud x \, \ud y
\leq \frac{C_\beta n^{-\beta}}{(1\wedge  t^{\beta/2})}.
\end{split}
\end{equation*}
Moreover, by the same argument, 
\begin{equation*}
\begin{split}
\Bigl\| \textrm{\rm div} \Bigl( \vartheta_{1,2}^n(t) m(t \, ;\nu) \Bigr) \Bigr\|_{(1+\beta,\infty)'} 
&\leq 
\sup_{\| \xi \|_{1+\beta,\infty} \leq 1} 
\Bigl\langle 
\Bigl[ \Bigl( b \bigl(\cdot,m(t \, ;\nu) \bigr) * \rho^n - 
 b\bigl(\cdot,m(t \, ;\nu)\bigr) \Bigr) 
 m(t\, ; \nu) 
 \Bigr] ,
\nabla \xi
\Bigr\rangle 
\\
&=
\sup_{\| \xi \|_{1+\beta,\infty} \leq 1} 
\Bigl\langle 
b \bigl(\cdot,m(t \, ;\nu) \bigr),  \bigl( m(t\, ; \nu)  \nabla \xi \bigr) * \rho^n - 
m(t\, ; \nu)  \nabla \xi
\Bigr\rangle \leq 
 \frac{C_\beta n^{-\beta}}{(1\wedge  t^{\beta/2})}.
\end{split}
\end{equation*}
Applying 
{\hergol}
(i.e., the finite time version of 
\eqref{erg:hyp:3})
to 
\eqref{eq:equation:mn-n:stab}
 three times, 
 once
with $k=\alpha$, 
$\mu - \nu$ as initial condition and $0$ as remainder,
once 
  with $k=0$, $0$ as initial condition and $\textrm{\rm div}(\vartheta^n_1 m(t \, ; \nu))$ as remainder
(which is in $(W^{(1+\beta,\infty)}({\mathbb T}^d))'$)
 and another time with $k=0$, $0$ as initial condition and $\textrm{\rm div}(\vartheta^n_2 m(t \, ; \nu))$ as remainder
(which is in $(W^{(1,\infty)}({\mathbb T}^d))'$), we get, for any time $S \geq 0$ and any $t \in [0,S]$, 
%\textcolor{red}{il y a une subtilité ici dans l'application de (ERG), puisqu'on applique de deux façons différentes : une fois avec $r=0$ pour conserver la décroissance exponentielle et une fois avec $q_0=0$ pour chasser la condition initiale. La linéarité fait le reste. C'est certainement une remarque intéressante que l'on peut faire directement à la suite de la condition (ERG). }
\begin{equation}
\label{eq:mn-m*}
\begin{split}
\bigl\| m^{n}(t \, ;\mu) - m(t\, ; \nu) \bigr\|_{(0,\infty)'} 
& \leq 
C_{\alpha,S}
\frac{\| \mu  - \nu \|_{(\alpha,\infty)'} }{1 \wedge t^{\alpha/2}}
+ C_{S} \int_0^t  
\frac{\bigl\| m^{n}(s\, ;\mu) - m(s \, ;\nu) \bigr\|_{(0,\infty)'}^2
}{1 \wedge (t-s)^{1/2}} 
\, 
\ud s
 + \frac{C_{\beta,S}}{n^{\beta}}.
\end{split}
\end{equation}

\textit{Second Step.} 
By upper bounding 
$\| m^{n}(s\, ;\mu) - m(s \, ;\nu) \|_{(0,\infty)'}^2$ by 
$2 \| m^{n}(s\, ;\mu) - m(s \, ;\nu) \|_{(0,\infty)'}$, 
we recover an inequality very similar to 
\eqref{eq: bound q est 2021}, from which we deduce that 
\begin{equation}
\label{eq:mn-m:rate:shorttime}
\bigl\| m^{n}(t\, ;\mu) - m(t \, ; \nu)  \|_{(0,\infty)'} \leq 
C_{\alpha,S}
\frac{\| \mu  - \nu \|_{(\alpha,\infty)'} }{1 \wedge t^{\alpha/2}}
 + \frac{C_{\beta,S}}{n^{\beta}}, \quad t \in (0,S], 
\end{equation}
By choosing $\mu=\nu$, this provides the rate of convergence of $m^n(\cdot \, ; \mu)$ to $m(\cdot \, ; \mu)$  {in finite time}.

Now, we choose $\nu=\nu_\infty$. 
For a fixed $a >0$, we can choose $t_0$ small enough and $n_0$ large enough such that 
the sum of the last two terms on  
\eqref{eq:mn-m*} (with $t=t_0$ and $n \geq n_0$) is less than $a/2$. Next, for this $t_0$, we
can choose 
$\| \mu -\nu_\infty\|_{(\alpha,\infty)'}$ small enough such that 
the first term on the right-hand side is also less than 
$a/2$. 
We deduce that $\|m^n(t_0 \, ; \mu) -\nu_\infty\|_{(0,\infty)'} \leq a$ for $n \geq n_0$. 
This says that, when 
$\mu$ is close to $\nu_\infty$ for
the norm $\| \cdot \|_{(\alpha,\infty)'}$, 
$m^n(\cdot \, ; \mu)$
is close to $\nu_\infty$ for $\textrm{\rm dist}_{\textrm{\rm TV}}$ at time $t_0$. Equivalently, by 
restarting $m^n$ at time $t_0$, we can assume that
the initial condition is close to $\nu_\infty$  for $\textrm{\rm dist}_{\textrm{\rm TV}}$ (and not only for the norm 
 $\| \cdot \|_{(\alpha,\infty)'}$). 
 Hence, we now study the long time behaviour of 
 $({\rm dist}_{\rm TV}( m^n(t \, ; \mu) ,\nu_\infty))_{t \geq 0}$
 under the assumption 
 that 
  ${\rm dist}_{\rm TV}(\mu, \nu_\infty)$ is small. The strategy is to repeat the analysis of the first step, but exchanging the roles of $b^n$ and $b$, and using \hergo \, for $b$ (instead of using 
 \hergol \, for $b^n$).  {This is the key point to get long time estimates, but this barely changes the proof. In fact,} this leads to the following variant of 
\eqref{eq:mn-m*}: 
\begin{equation}
\label{eq:mn-m*:glob}
\begin{split}
\bigl\| m^{n}(t \, ;\mu) - \nu_\infty \bigr\|_{(0,\infty)'} 
& \leq C
 \| \mu  - \nu_\infty \|_{(0,\infty)'} e^{-\lambda t} 
+ C   \int_0^t  
\frac{\| m^{n}(s\, ;\mu) - \nu_\infty \|_{(0,\infty)'}^2}{1 \wedge (t-s)^{1/2}} 
e^{-\lambda(t-s)}
\ud s
 + \frac{C_{\beta}}{n^{\beta}}.
\end{split}
\end{equation}
We choose $\beta=1/2$. 
For some $a>0$ (whose value is fixed next and which has nothing to do with the parameter $a$ used in the previous paragraph), we now assume that
$\mu$ and $n$ satisfy
\begin{equation}
\label{eq:condition:a:small}
C \textrm{\rm dist}_{\rm TV}( \mu, \nu_\infty 
)
+ \frac{C_{1/2}}{n^{1/2}} \leq a.
\end{equation}
We argue by contradiction to show that 
$\textrm{\rm dist}_{\rm TV}( \mu, \nu_\infty )$ remains small. We thus assume that the hitting time $t_0:=\inf\{ t \geq 0 : \| m^{n}(t \, ;\mu) - \nu_\infty \|_{(0,\infty)'} 
\geq 2a\}$
is finite.
Then, by continuity (in $s$) of 
$\| m^{n}(s \, ;\mu) - \nu_\infty \|_{(0,\infty)'}$, we get 
from \eqref{eq:mn-m*:glob} that 
$2 a \leq a + 4 Ca^2 \int_0^{\infty} e^{-\lambda s}/{(1 \wedge s^{1/2})} \ud s,$
which is impossible if $a$ is small enough. This says that, 
for $a$ small enough and under the prescription 
\eqref{eq:condition:a:small}, we have 
$\sup_{t \geq 0} \| m^n(t \, ; \mu) - \nu_\infty \|_{(0,\infty)'}
\leq 2a$.
By a similar analysis (or letting 
$n$ tend to $\infty$ in 
\eqref{eq:mn-m:rate:shorttime}{, an argument invoked in the proof of 
Proposition \ref{thm main result}}), 
we also have 
$\sup_{t \geq 0} \| m(t \, ; \mu) - \nu_\infty \|_{(0,\infty)'}
\leq 2a$ when $\| \mu- \nu_\infty \|_{(0,\infty)'}=\textrm{\rm dist}_{\rm TV}( \mu, \nu_\infty ) < a/C$. 

This explains how to obtain 
a global in time version of 
\eqref{eq:mn-m:rate:shorttime}. 
Assume indeed for a while that, for some $\mu$ satisfying 
\eqref{eq:condition:a:small}
and for any $n$ large enough, 
\hergo \,  holds 
true for 
the version of
 \eqref{eq: forward eqn } driven by $b^n$ and for constants $(C_k)_{0 \leq k < 2}$
 and $\lambda$ that are independent of $n$ (this is proven in the two steps below). 
Then, 
for two measures $\mu$ and $\nu$ satisfying 
\eqref{eq:condition:a:small} and for 
$n$ large, 
we obtain the 
following
long time version
of \eqref{eq:mn-m*} 
(whose derivation is similar to 
\eqref{eq:mn-m*:glob}):
\begin{equation*}
\begin{split}
\bigl\| m^{n}(t \, ;\mu) - m(t\, ; \nu) \bigr\|_{(0,\infty)'} 
& \leq 
 \| \mu  - \nu  \|_{(0,\infty)'} e^{-\lambda t}
+ C \int_0^t  
\frac{\| m^{n}(s\, ;\mu) - m(s \, ;\nu) \|_{(0,\infty)'}^2}{1 \wedge (t-s)^{1/2}} 
e^{-\lambda(t-s)}
\, 
\ud s  + \frac{C_{\beta}}{n^{\beta}}.
\end{split}
\end{equation*}
Inserting
the two bounds 
$\sup_{t \geq 0} \| m^n(t \, ; \mu) - \nu_\infty \|_{(0,\infty)'}
\leq 2a$
and
$\sup_{t \geq 0} \| m(t \, ; \nu) - \nu_\infty \|_{(0,\infty)'}
\leq 2a$ and using in addition the inequality
$e^{-\lambda t}/(1 \wedge t^{1/2}) \leq C e^{-\lambda t/2}/t^{1/2}$, we obtain, for $t \geq 0$, 
\begin{equation}
\label{eq:mn-m*:3}
\begin{split}
\bigl\| m^{n}(t\, ;\mu) - m(t\, ;\nu) \bigr\|_{(0,\infty)'}
&\leq C  \| \mu  - \nu \|_{(0,\infty)'} e^{-\lambda t} 
 +
C a 
\int_{0}^t  
\frac{\| m^{n}(s\, ;\mu) - m(s \, ;\nu) \|_{(0,\infty)'}}
{ (t-s)^{1/2}}
e^{-\lambda (t-s)/2} \, 
\ud s + 
\frac{C_\beta}{n^{\beta}}.
\end{split}
\end{equation}
Multiplying by 
$e^{-\lambda (\tau-t)/2}/(\tau-t)^{1/2}$,  
integrating w.r.t. $t \in [0,\tau]$, we obtain (for a new value of $C$)
\begin{equation*}
\begin{split}
&\int_{0}^{\tau} 
\frac{\| m^{n}(t\, ;\mu) - m(t\, ;\nu) \|_{(0,\infty)'}}{(\tau-t)^{1/2}}
e^{-\lambda(\tau-t)/2}  \ud t
\\
&\leq C \| \mu  - \nu \|_{(0,\infty)'} e^{-\lambda \tau/2}  
 +
C a 
\int_{0}^\tau
\bigl\| m^{n}(s\, ;\mu) - m(s \, ;\nu) \|_{(0,\infty)'}
\Bigl( \int_{s}^\tau
\frac{e^{-\lambda (\tau-s)/2} }{ (t-s)^{1/2}  (\tau-t)^{1/2}} 
\ud t
\Bigr)
\, 
\ud s
 + 
\frac{C_\beta}{n^{\beta}} 
\\
&\leq C \| \mu  - \nu \|_{(0,\infty)'} e^{-\lambda \tau/2}   
 +
C a 
\int_{0}^\tau
\frac{\| m^{n}(s\, ;\mu) - m(s \, ;\nu) \|_{(0,\infty)'}
}{ (\tau-s)^{1/2}}
e^{-\lambda (\tau-s)/2} 
\ud s  + 
\frac{C_\beta}{n^{\beta}} .
\end{split}
\end{equation*}
For $C a \leq 1/2$ (recalling 
\eqref{eq:condition:a:small} for $\mu$, $\nu$ and $n$), we get 
a bound for the left-hand side.  
Back to 
\eqref{eq:mn-m*:3},
\begin{equation}
\label{eq:mn-m*:4}
\bigl\| m^{n}(t\, ;\mu) - m(t\, ; \nu)  \|_{(0,\infty)'} \leq C  \| \mu  - \nu \|_{(0,\infty)'} e^{-\lambda t/2}    + 
\frac{C_\beta}{n^{\beta}}, \quad t \geq 0. 
\end{equation}
Choosing $\nu=\mu$, we get the first inequality in the statement (noticing from 
\eqref{eq:mn-m:rate:shorttime} that 
the bound also holds true in small time% when 
%the distance
%between $\mu$ and $\nu_\infty$ is 
%measured under 
%$\| \cdot \|_{(\alpha,\infty)'}$
). Choosing 
$\nu=\nu_\infty$, we get the second claim. 
 \vskip 4pt

\textit{Third Step.} 
We now show that 
$m^n(\cdot \, ; \mu)$ (whose dynamics  are driven by $b^n$) satisfies \hergo \, when 
$\| \mu - \nu_\infty\|_{(\alpha,\infty)'}$ is small, which property we have just used to prove 
\eqref{eq:mn-m*:3}
and
\eqref{eq:mn-m*:4}. By Proposition 
\ref{W1 decay}, we already know that it satisfies \hergol \, 
on any finite interval $[0,T]$
w.r.t. constants that depend on $T$ but are independent of $n$
(since the quantities in 
\eqref{eq:quantities:ergl}, with 
$b$ replaced by $b^n$, are independent of $n$). 
The point is mainly to establish 
   \eqref{erg:hyp:3} for time indices $t$ greater than or equal to 
 $1$.  In fact, by applying 
    \eqref{erg:hyp:3} at fixed $t_0>0$, we already have a bound for 
    $\| q(t_0) \|_{(0,\infty)'}$, and by
    the paragraph below 
    \eqref{eq:mn-m:rate:shorttime}, we also know 
 that $\| m(t_0\, ; \mu) - \nu_\infty\|_{(0,\infty)'}$ is small
 when 
 $\| \mu - \nu_\infty\|_{(\alpha,\infty)'}$ is small
 and $n$  and $t_0$ are large. 
%This says that it mostly suffices to treat the case 
  %  $k=0$ under the assumption 
    %that $\| \mu - \nu_\infty\|_{(0,\infty)'}$ is small. 
% Lastly, , it is also important to notice that, whatever the initial condition 
%$\mu$ for 
%\eqref{eq:MVSDE}, the solution 
% to \eqref{eq: forward eqn }
% satisfies $\| m(1\, ; \mu) \|_{1/2,\infty} 
% \leq c$ for a constant $c$ that only depends on 
% $\sup_{m \in {\mathcal P}({\mathbb T}^d)} \| b(\cdot,m)\|_\infty$. 
% This follows from Lemma 
%\ref{lem:annex:regularity:highd}
%in the appendix. 
%In particular, when restarting the dynamics from time $t=1$, we can assume that 
%$\sup_{t \geq 0} \| m(t \, ; \mu) \|_{1/2,\infty} 
% \leq c$
% and
%$\sup_{t \geq 0} \| m^n(t \, ; \mu) \|_{1/2,\infty} 
% \leq c$. 
%    \vskip 4pt
%We recall that 
%\begin{equation*}
%\frac{\delta b^n}{\delta m}(x,m)(y) = \frac{\delta b^n}{\delta m}(\cdot,m*\rho^n)(\cdot)* \rho_n^{\otimes 2}(x,y), 
%\quad n \geq 1. 
%\end{equation*}
The main goal is thus to prove 
 \eqref{erg:hyp:3}
 for $k=\alpha=0$ under   \eqref{eq:condition:a:small}, which implies 
 $\sup_{t \geq 0} \| m^n(t \, ; \mu) - \nu_\infty \|_{(0,\infty)'}
\leq 2a$. 
 
 For an initial condition $q_0$ and for a source term as in 
 \eqref{erg:hyp:3}, we rewrite the solution of the equation
 \begin{equation}
 \label{eq:3rdstep:mn:q}  
    \partial_t q(t) - L^n_{m^n(t;\mu)}q(t)  -r(t)= 0, \quad \quad t \geq 0,
    \end{equation}
as $q(t) := q_1(t) + q_2(t)$ with
 \begin{equation*}  
    \partial_t q_1(t) - L_{\nu_\infty} q_1(t) - r(t)= 0, \quad   t \geq 0
   \,; \quad q_1(0)=q_0 \, ; \quad 
    \partial_t q_2(t) - L_{\nu_\infty} q_2(t) - r_2^n(t)= 0, \quad t \geq 0
   \,; \quad q_2(0)=0, 
    \end{equation*} 
    where 
$ r^n_2(t) =  [  L^n_{m^n(t;\mu)} - L_{\nu_\infty} ] q(t)$.
By \hergo \, (for $\nu_\infty$), it is easy to estimate 
 $q_1(t)$. Next, we rewrite 
$r_2^n(t)$ as
\begin{equation}
\begin{split}
r^n_2(t) &=  
  \textrm{\rm div} 
     \Bigl( \bigl[ b^n \bigl(\cdot,m^n(t \, ;\mu) \bigr) - b(\cdot,\nu_\infty) \bigr] q(t) \Bigr)
     + 
     \textrm{\rm div}\Bigl(
           \frac{\delta b^n}{\delta m}\bigl(\cdot,m^n(t;\mu)\bigr)
     \bigl(q(t)\bigr)
     \bigl[ m^n(t;\mu)
-      \nu_\infty\bigr] \Bigr)
\\
&\hspace{15pt}    
+
   \textrm{\rm div}\Bigl(
     \Bigl[ 
      \frac{\delta b^n}{\delta m}\bigl(\cdot,m^n(t;\mu)\bigr)
     \bigl(q(t)\bigr)
-      \frac{\delta b}{\delta m}(\cdot,\nu_\infty)\bigl(q(t)\bigr) \Bigr] \nu_\infty\Bigr). 
\label{eq:rn2}
\end{split}
     \end{equation}
The difficulty here is that we have no $L^\infty$ estimates on $b^n-b$ (since $b$ may not be continuous in the spatial variable). 
We proceed as follows. 
For  $\delta>0$, we 
let $p_n(\delta):=\textrm{\rm Leb}_d(\{ x \in {\mathbb T}^d : 
\vert b^n(x,\nu_\infty) - b(x,\nu_\infty) \vert \geq \delta\})+
\textrm{\rm Leb}_{2d}(\{ (x,y) \in {\mathbb T}^{2d} : 
\vert [\delta b^n/\delta m](x,\nu_\infty,y) - [\delta b/\delta m](x,\nu_\infty,y) \vert \geq \delta\})$, which 
tends to $0$ as $n$ tends to $\infty$ (see \eqref{eq:deltabn:deltam}). 
Then, for $\| \nabla \xi \|_\infty \leq 1$ (and observing from \eqref{eq:preliminary:bound:density} that $\nu_\infty$ has a bounded density), 
\begin{align}
\label{eq:add:proof:l1}
&\biggl\vert \int_{{\mathbb T}^d}
\Bigl[\frac{\delta b^n}{\delta m}\bigl(x,\nu_\infty\bigr) -\frac{\delta b}{\delta m}(x,\nu_\infty)\Bigr]\bigl(q(t)\bigr) \cdot 
\nabla \xi(x) \, \ud \nu_\infty(x) \biggr\vert
\\
&= \biggl\vert \biggl\langle 
\Bigl[\frac{\delta b^n}{\delta m}\bigl(\cdot,\nu_\infty,\cdot\bigr) -\frac{\delta b}{\delta m}(\cdot,\nu_\infty,\cdot)\Bigr],
\bigl( \nabla \xi \cdot \nu_\infty \bigr) \otimes q(t)
\biggr\rangle\biggr\vert
\leq  C \delta \|q(t)\|_{(0,\infty)'}
+ C \sup_{\| \zeta \|_\infty \leq 1, \| \zeta \|_1 \leq p_n(\delta)}
\bigl\langle q(t), \zeta \bigr\rangle,
\nonumber
\end{align}
with $\zeta$ in the supremum being  
a real-valued measurable function on ${\mathbb T}^d$.
{To obtain the second line right-above, we used the distributional version of Fubini's theorem together with the fact that 
\begin{equation*}
\begin{split} 
&\biggl\langle 
\Bigl[\frac{\delta b^n}{\delta m}\bigl(\cdot,\nu_\infty,\cdot\bigr) -\frac{\delta b}{\delta m}(\cdot,\nu_\infty,\cdot)\Bigr],
\bigl( \nabla \xi \cdot \nu_\infty \bigr) \otimes q(t)
\biggr\rangle
\\
&= \biggl\langle 
\Bigl[\frac{\delta b^n}{\delta m}\bigl(\cdot,\nu_\infty,\cdot\bigr) -\frac{\delta b}{\delta m}(\cdot,\nu_\infty,\cdot)\Bigr]
\one_{A_n} ,
\bigl( \nabla \xi \cdot \nu_\infty \bigr) \otimes q(t)
\biggr\rangle
\\
&\hspace{15pt} +
\biggl\langle 
\Bigl[\frac{\delta b^n}{\delta m}\bigl(\cdot,\nu_\infty,\cdot\bigr) -\frac{\delta b}{\delta m}(\cdot,\nu_\infty,\cdot)\Bigr]
\one_{A_n^{\complement}},
\bigl( \nabla \xi \cdot \nu_\infty \bigr) \otimes q(t)
\biggr\rangle,
\end{split} 
\end{equation*} 
where we let $A_n := \{ (x,y) \in {\mathbb T}^{2d} : 
\vert [\delta b^n/\delta m](x,\nu_\infty)(y) - [\delta b/\delta m](x,\nu_\infty)(y) \vert \geq \delta\}$. 
In order to get the very last term in 
\eqref{eq:add:proof:l1}, we let 
\begin{equation*} 
\zeta(x) := \int_{{\mathbb T}^d} 
\one_{A_n^{\complement}}(y,x)  
\Bigl[\frac{\delta b^n}{\delta m}\bigl(y,\nu_\infty,x\bigr) -\frac{\delta b}{\delta m}(y,\nu_\infty,x)\Bigr]
 \cdot \nabla \xi(y)   \nu_\infty(\ud y), 
\end{equation*} 
which satisfies $\| \zeta \|_{\infty} \leq C$ and $\| \zeta \|_1 \leq C p_n(\delta)$.  The claim 
\eqref{eq:add:proof:l1} easily follows.}

Proceeding in the same way for the other two terms in the expression of
$r^n_2(t)$, 
we have
(using 
the bound  $\sup_{t \geq 0} \| m^n(t \, ; \mu) - \nu_\infty \|_{(0,\infty)'}
\leq 2a$ in order to handle the last term in the definition 
\eqref{eq:rn2}
of $r_2^n(t)$)
\begin{equation}
\label{eq:rn2:2}
\begin{split}
\bigl\| r^n_2(t) \bigr\|_{(1,\infty)'} 
&\leq 
C
\|q(t)\|_{(0,\infty)'}
\bigl( \delta + 
\bigl\| m^{n}(t\, ;\mu) - \nu_\infty \|_{(0,\infty)'}
\bigr) +
C  \sup_{\| \zeta \|_\infty \leq 1, \| \zeta \|_1 \leq p_n(\delta)}
\bigl\langle q(t), \zeta \bigr\rangle
\\
&\leq C 
\|q(t)\|_{(0,\infty)'}
\bigl(  \delta
+a \bigr) + C \sup_{\| \zeta \|_\infty \leq 1,\| \zeta \|_1 \leq p_n(\delta)}
\bigl\langle q(t), \zeta \bigr\rangle.
\end{split}
\end{equation}
By \hergo \, at $\nu_\infty$ (applied twice, 
once to $q_1$  and once to $q_2$, see 
the decomposition of \eqref{eq:3rdstep:mn:q}  ), we get 
\begin{equation*}
\begin{split}
&\| q(t) \|_{(0,\infty)'} \leq C e^{-\lambda t} \| q(0)\|_{(0,\infty)'} 
+ 
C \int_0^t \| r(s) \|_{(\beta,\infty)'} \frac{e^{-\lambda(t-s)}}{1 \wedge (t-s)^{\beta/2}} \ud s
\\
&\hspace{15pt} + 
C \int_0^t \bigl( \delta + a  \bigr) 
\| q(s) \|_{(0,\infty)'} 
 \frac{e^{-\lambda(t-s)}}{1 \wedge (t-s)^{1/2}} \ud s
  + 
C \int_0^t 
\sup_{\| \zeta \|_\infty \leq 1 ,\| \zeta \|_1 \leq p_n(\delta)}
\bigl\langle q(s), \zeta \bigr\rangle 
\frac{e^{-\lambda(t-s)}}{1 \wedge (t-s)^{1/2}} \ud s.
\end{split}
\end{equation*}
%Obviously, 
%\eqref{eq:q0:bound:time} says that, for any $S>0$, the above is true for $t \in [0,S]$
%without any further need to include the term on the second line. This yields
%\begin{equation*}
%\begin{split}
%\| q(t) \|_{(0,\infty)'} &\leq C e^{-\lambda t} \| q(0)\|_{(0,\infty)'} 
%+ 
%C \int_0^t \| r(s) \|_{(1+\beta,\infty)'} \frac{e^{-\lambda(t-s)}}{1 \wedge (t-s)^{\beta/2}} ds
%\\
%&\hspace{15pt} + 
%C 
%\bigl( e^{-\lambda S/2} + a + \frac1{n^{1/2}} \bigr) 
%\int_S^{t \vee S}  
%\| q(s) \|_{(0,\infty)'} 
% \frac{e^{-\lambda(t-s)}}{1 \wedge (t-s)^{\beta/2}} ds
%\\
%&\hspace{15pt} + 
%C \int_0^t 
%\sup_{\| \zeta \|_\infty \leq 1 ,\| \zeta \|_1 \leq p_n(a)}
%\bigl\langle q(s), \zeta \bigr\rangle 
%\frac{e^{-\lambda(t-s)}}{1 \wedge (t-s)^{1/2}} ds.
%\end{split}
%\end{equation*}
We now proceed as in the second step. We multiply by 
$e^{-\lambda(\tau-t)/2}/(\tau-t)^{1/2}$, for some $\tau > 0$ and then integrate 
w.r.t $t \in [0,\tau]$. 
Assuming w.l.o.g that $\delta+a$ is small enough (compared to $1/C$),
we obtain 
\begin{equation}
\begin{split}
\| q(t) \|_{(0,\infty)'} &\leq C \| q(0)\|_{(0,\infty)'} e^{-\lambda t/2}
+ 
C \int_0^t \| r(s) \|_{(\beta,\infty)'} \frac{e^{-\lambda(t-s)/2}}{1 \wedge (t-s)^{\beta/2}} \ud s
\\
&\hspace{15pt} + 
C \int_0^t 
\sup_{\| \zeta \|_\infty \leq 1, \| \zeta \|_1 \leq p_n(\delta)}
\bigl\langle q(s), \zeta \bigr\rangle 
\frac{e^{-\lambda(t-s)/2}}{1 \wedge (t-s)^{1/2}} \ud s.
\end{split}
\label{eq:q:mn:third:step}
\end{equation}
\vskip 2pt

\textit{Fourth Step.} 
The problem now is to handle the last term in 
\eqref{eq:q:mn:third:step}. This asks us to revisit the proof of 
Proposition \ref{W1 decay}. 
For
 $t >0$, we consider the solution   to  
\eqref{cauchy w} 
with $w(t,\cdot) = \zeta$ as terminal condition, where
$\| \zeta \|_\infty \leq 1$ and $\| \zeta \|_1 \leq p_n(\delta)$. Then, 
we already have the two bounds 
\eqref{W12 estimate est 1} 
and
\eqref{eq: W12 estimate est 3}, but we can improve them using the assumption on $\| \zeta\|_1$. 
Following
\eqref{eq:preliminary:bound:density}, 
for $s \in [(t-1) \vee 0,t]$,  
for any real $\varrho > 1$, 
\begin{equation*}
\begin{split}
\| w(s,\cdot) \|_\infty &\leq C 
\sup_{x \in {\mathbb T}^d}
\int_{{\mathbb T}^d} \zeta(y) g_{\delta_x}(t-s,y) 
\ud y
\leq 
C 
\| \zeta \|_{\varrho}
\sup_{x \in {\mathbb T}^d}
\| g_{\delta_x}(t-s,\cdot) \|_{\varrho/(\varrho-1)}
\leq C_{\varrho} \frac{p_n(\delta)^{1/\varrho}}{ (t-s)^{d/(2\varrho)}}.
\end{split}
\end{equation*}
and then, by treating the case $s \leq t-1$ by regarding $w(t-1,\cdot)$ as a new initial condition, we deduce from 
 \eqref{W12 estimate est 1} that 
 \begin{equation*}
 \Big\| w(s, \cdot) - \intrd w(s,y) \, \ud y \Big\|_{ \infty} \leq  C \frac{p_n(\delta)^{1/\varrho}}{1 \wedge (t-s)^{d/(2\varrho)}}  e^{-\lambda(t-s)},
 \end{equation*}
 first for $t-s \geq 1$ and, then, for $t-s > 0$. 
 Letting $\epsilon:=1/\varrho$ and repeating 
\eqref{eq:Schauder:linfinity}
and
\eqref{eq:Schauder:linfinity:000}, we obtain
%, for $\varrho \geq 2$, 
\begin{equation*}
\| \nabla_x w(s,\cdot) \|_{\infty} \leq C \frac{p_n(\delta)^{\epsilon} }{ 1 \wedge (t-s)^{1/2+\epsilon d /2}} e^{-\lambda(t-s)}. 
\end{equation*}
Then, 
assuming {$1/2+  \epsilon d/2 <1$} (which is always doable by choosing $\varrho$ large enough) 
and
expanding $(\langle w(s),q(s) \rangle)_{0 \leq s \leq t}$ as
in the proof of 
Proposition
\ref{W1 decay}, 
we can insert 
the above estimate in the analysis of $T_3$ in
\eqref{eq: bound q est 2}  
  (whereas \eqref{eq: bound q est 1}
  and \eqref{eq: bound q est 3}   
  do not change). 
  We get 
  {\begin{equation*}
\begin{split}
 \big| T_{3} \big| &  \leq   C    \int_0^t  \| \nabla_x w(s, \cdot)  \|_{{\infty}}   \| q(s)  \|_{({0, \infty})'}
 \, \ud s
 \leq 
   C   p_n(\delta)^{\epsilon}  
   \int_0^t  
\frac{\| q(s) \|_{(0,\infty)'}}{ 1 \wedge (t-s)^{1/2+\epsilon d/2}} e^{-\lambda(t-s)}
\ud s. 
\end{split}
%\label{eq: bound q est 2}  
\end{equation*}
and then,}
\begin{equation*}
\begin{split}
\langle \zeta,q(t) \rangle &\leq 
 C  e^{-\lambda t}  \| q(0)\|_{(0,\infty)'}
 +
 \int_{0}^t \frac{\| r(s) \|_{(\beta,\infty)'}}{ 1 \wedge (t-s)^{\beta/2}} e^{-\lambda(t-s)} \ud s
 +
 C 
 p_n(\delta)^{\epsilon}
 \int_{0}^t 
 \frac{\| q(s) \|_{(0,\infty)'}}{ 1 \wedge (t-s)^{1/2+\epsilon d/2}} e^{-\lambda(t-s)} \ud s.
\end{split}
\end{equation*}
We complete the proof by inserting the above estimate in 
\eqref{eq:q:mn:third:step}. 
Proceeding as 
in 
\eqref{eq:mn-m*:3} and choosing $n$ large enough so that $p_n(\delta)$ becomes small enough, 
we get 
%\begin{equation*}
%\begin{split}
%\langle \zeta,q(t) \rangle &\leq 
% C  e^{-\lambda t}  \| q(0)\|_{(0,\infty)'}
% +
% \int_{0}^t \frac{\| r(s) \|_{(\beta,\infty)'}}{(t-s)^{\beta/2}} e^{-\lambda(t-s)} ds.
%\end{split}
%\end{equation*}
%%and then we complete the proof as before by using the fact that 
%%we can choose $p_n(a)$ small. 
%We complete the proof by inserting the above estimate in 
%\eqref{eq:q:mn:third:step}, which gives 
\begin{equation}
\begin{split}
\| q(t) \|_{(0,\infty)'} &\leq C \| q(0)\|_{(0,\infty)'} e^{-\lambda t/4}
+ 
C \int_0^t \| r(s) \|_{(\beta,\infty)'} \frac{e^{-\lambda(t-s)/4}}{1 \wedge (t-s)^{\beta/2}} \ud s,
\end{split}
\label{eq:q:mn:third:step:bb}
\end{equation}
which is 
\eqref{erg:hyp:3} 
when    
$q(0)$ is in $(W^{0,\infty}({\mathbb T}^d))'$.

Now, we must explain what happens when we just have a bound for 
$q(0)$ in $(W^{k,\infty}({\mathbb T}^d))'$ for some $k \in (0,2)$. 
As we already said, we can apply \eqref{eq:q:mn:third:step:bb} but for the dynamics restarted at some
small $t_0>0$. In fact, $t_0$ can be chosen in some interval $[S/2,3S/2]$, for $S$ small. 
We then 
have \eqref{eq:q:mn:third:step:bb}
but with 
$\| q(0)\|_{(0,\infty)'}$ replaced by 
$\| q(t_0)\|_{(0,\infty)'}$ and next, by averaging w.r.t. 
$t_0$, 
we have \eqref{eq:q:mn:third:step:bb}
but with  $\| q(0)\|_{(0,\infty)'}$ 
replaced by $S^{-1} \int_{S/2}^{3S/2} \| q(t_0)\|_{(0,\infty)'} \ud t_0$. It remains to see that, by applying Proposition
 \ref{W1 decay} 
to \eqref{eq:3rdstep:mn:q}, we already have a bound in finite time, which writes:
\begin{equation}
\label{eq:q:mn:third:step:bb:ddd}
\begin{split}
\| q(t_0) \|_{(0,\infty)'} &\leq \frac{C_S}{t_0^{k/2}} \|q(0)\|_{(k,\infty)'}
+
C_S \int_0^{t_0} 
\frac{ \| r(s) \|_{(\beta,\infty)'} }{1 \wedge (t_0-s)^{\beta/2}}
\ud s, \quad t_0 \in \bigl[ \frac{S}2,\frac{3S}2\bigr], 
\end{split}
\end{equation}
for a constant $C_S$ depending on the time index $S$. By integrating w.r.t $t_0$
over $[S/2,3S/2]$, we get 
\begin{equation}
\label{eq:q:mn:third:step:bb:eee}
\begin{split}
\frac1S \int_{S/2}^{3S/2} \| q(t_0) \|_{(0,\infty)'} \ud t_0 &\leq C_S \|q(0)\|_{(k,\infty)'}
+
C_S \int_0^{3S/2} 
 \| r(s) \|_{(\beta,\infty)'} 
\ud s, 
\end{split}
\end{equation}
which can be inserted in 
\eqref{eq:q:mn:third:step:bb}, 
when
$\| q(0)\|_{(0,\infty)'}$ 
therein is replaced by $S^{-1} \int_{S/2}^{3S/2} \| q(t_0)\|_{(0,\infty)'} \ud t_0$. 
This gives 
\eqref{erg:hyp:3}
for $t \geq S$. 
The result, for $t < S$, is a direct consequence of 
\hergol. 
\end{proof}

\begin{remark}
\label{rem:stability:erg}
The reader will observe that the proof of Proposition 
\ref{prop main result:2} 
{relies on a stability argument. As such, it} 
can be easily adapted to prove the following: 
if the original dynamics 
 \eqref{eq: forward eqn }
 satisfies \hergo \, at a given $\mu$ and if 
 $\tilde b$ is another drift, satisfying
 \hregb{0}{2}, such that  
 \begin{equation*}
\sup_{x \in {\mathbb T}^d} \sup_{\mu \in {\mathcal P}({\mathbb T}^d)}
 \vert (\tilde b - b)(x,\mu)\vert + 
 \sup_{x,y \in {\mathbb T}^d} \sup_{\mu \in {\mathcal P}({\mathbb T}^d)}
\bigl\vert \frac{\delta \tilde b}{\delta m}(x,\mu,y) - \frac{\delta b}{\delta m}(x,\mu,y) \bigr\vert 
  \leq \varepsilon,
\end{equation*}
for some $\varepsilon >0$,  then, for $\varepsilon$ small enough, 
 the dynamics driven by 
 $\tilde b$ satisfy \hergo \, at $\mu$ w.r.t. to constants that are independent of 
 $\varepsilon$. 
 We used this observation in the proof of Proposition \ref{thm main result}.  
 
 Also, in the sequel, we sometimes refer to 
 \eqref{eq:mn-m:rate:shorttime}
as a stability estimate in finite time. Importantly, it 
is proven by means of the sole assumption 
{\hergol}
for the drift $b^n$, which follows from 
Proposition \ref{W1 decay}. 
In particular, letting $n$ tend to $\infty$ in  \eqref{eq:mn-m:rate:shorttime}, we obtain
a bound for 
$\| m(t\, ;\mu) - m(t \, ; \nu)  \|_{(0,\infty)'}$, but it blows-up  as $t$ tends to $0$.
 {To make it clear,  
 \begin{equation*}
\bigl\| m(t\, ;\mu) - m(t \, ; \nu)  \bigr\|_{(0,\infty)'} \leq 
C_{\alpha,S}
\frac{\| \mu  - \nu \|_{(\alpha,\infty)'} }{1 \wedge t^{\alpha/2}}, \quad t \in (0,S].
\end{equation*}}%Of course, the right-hand side blows up as $t$ tends to $0$ because we use a stronger norm in the left-hand side. 
To overcome this drawback, one may use the same norm on the two sides of the inequality. In fact, by the same argument, 
we can easily show the following inequality: 
 \begin{equation}
\label{eq:mn-m:rate:shorttime:37}
\bigl\| m(t\, ;\mu) - m(t \, ; \nu) \bigr\|_{(\alpha,\infty)'} \leq 
C_{\alpha,S}
\| \mu  - \nu \|_{(\alpha,\infty)'}, \quad t \in [0,S],
\end{equation}
where $\alpha >0$. 
{It suffices to return to 
\eqref{eq:mn-m*} and to estimate the left-hand side therein by means of 
$\| \cdot \|_{(\alpha,\infty)'}$. 
Whereas, in the derivation of \eqref{eq:mn-m*}, we applied 
{\hergol}
 one
 first time
with $k=\alpha$, 
$\mu - \nu$ as initial condition and $0$ as remainder
and with $\alpha$ as parameter in 
  \eqref{erg:hyp:3}, we here replace 
  the latter choice of 
  $\alpha$ as parameter by $0$ in 
  \eqref{erg:hyp:3}. This gives 
\begin{equation*}
\begin{split}
\bigl\| m^{n}(t \, ;\mu) - m(t\, ; \nu) \bigr\|_{(\alpha,\infty)'} 
& \leq 
C_{\alpha,S}
\| \mu  - \nu \|_{(\alpha,\infty)'}  
+ C_{S} \int_0^t  
\frac{\bigl\| m^{n}(s\, ;\mu) - m(s \, ;\nu) \bigr\|_{(0,\infty)'}^2
}{1 \wedge (t-s)^{1/2}} 
\, 
\ud s
 + \frac{C_{\beta,S}}{n^{\beta}},
\end{split}
\end{equation*}
from which 
\eqref{eq:mn-m:rate:shorttime:37}
follows.}

 \end{remark}

We can now complete:
\begin{proof}[Proof of Theorem \ref{thm main result:2}, first part.]
The first part of the statement 
of Theorem \ref{thm main result:2} (the case when $\nu_\infty$ is a global attractor) 
follows from the combination of 
Propositions \ref{thm main result} and 
\ref{prop main result:2}, provided we can prove that $b^n$ satisfies 
\hergo \, at any $\mu \in {\mathcal P}({\mathbb T}^d)$. 
It suffices to notice that, by assumption, 
there exists $t_0>0$ such that, for any 
$\mu \in {\mathcal P}({\mathbb T}^d)$, 
$\| m(t_0,\mu) - \nu_\infty \|_{(0,\infty)'} \leq a_0$, 
for $a_0$ as in the statement of Proposition \ref{prop main result:2}. 
By 
Proposition \ref{prop main result:2} itself, 
the same holds with $m^n(t_0,\mu)$ instead of 
$m(t_0,\mu)$, namely 
$\| m^n(t_0,\mu) - \nu_\infty \|_{(0,\infty)'} \leq a_0$, 
for $n$ large enough. 
By Proposition \ref{prop main result:2} again, 
$(m^n(t \, ; m^n(t_0,\mu)))_{t \geq 0}$ 
and $b^n$
(with the former being also  
$((m^n(t + t_0 ,\mu))_{t \geq 0}$)
satisfies \hergo \, (for $n$ large enough). 
The proof of \hergo \ for $b^n$ at $\mu$ and on the whole $[0,+\infty)$ is then achieved as (the end of) the fourth step of the proof of Proposition 
\ref{prop main result:2},
using in particular 
\eqref{eq:q:mn:third:step:bb:ddd}
and
\eqref{eq:q:mn:third:step:bb:eee}. 
{The details are as follows. 
 For an initial condition $q_0$ and for a source term as in 
 \eqref{erg:hyp:3}, we rewrite the solution (after time $t_0$) 
 of the equation
 \begin{equation*}
    \partial_t q(t) - L^n_{m^n(t;\mu)}q(t)  -r(t)= 0, \quad \quad t \geq 0,
    \end{equation*}
as
 \begin{equation*}
    \partial_t q(t+t_0) - L^n_{m^n(t;m^n(t_0,\mu))}q(t+t_0)  -r(t+t_0)= 0, \quad \quad t \geq 0.
    \end{equation*}
    This yields
    \begin{equation*}
\begin{split}
\| q(t+t_0) \|_{(0,\infty)'} &\leq C \| q(t_0)\|_{(0,\infty)'} e^{-\lambda t}
+ 
C \int_0^t \| r(t_0+s) \|_{(\beta,\infty)'} \frac{e^{-\lambda(t-s)}}{1 \wedge (t-s)^{\beta/2}} \ud s,
\end{split}
\end{equation*}
for $t \geq 0$, or equivalently, 
\begin{equation*}
\begin{split}
\| q(t) \|_{(0,\infty)'} &\leq C \| q(t_0)\|_{(0,\infty)'} e^{-\lambda (t-t_0)}
+ 
C \int_{t_0}^t \| r(s) \|_{(\beta,\infty)'} \frac{e^{-\lambda(t-s)}}{1 \wedge (t-s)^{\beta/2}} \ud s,
\end{split}
\end{equation*}
for $t \geq t_0$. 
At this point, we can invoke 
\eqref{eq:q:mn:third:step:bb:ddd}
and
\eqref{eq:q:mn:third:step:bb:eee}
    in order to conclude.}
    
    The remaining difficulty is to estimate 
    (in the statement of  
Proposition \ref{thm main result})
\begin{equation*} 
\Bigl\vert 
\Phi\bigl(\rvlaw[X_t^n]\bigr)
-
\Phi\bigl(\rvlaw[X_t]\bigr)
\Bigr\vert.
\end{equation*} 
Since $\delta \Phi/\delta m$ is bounded, it suffices to 
give a bound for 
$d_{\rm TV}(\rvlaw[X_t^n],\rvlaw[X_t])$, but this a consequence of 
Proposition \ref{prop main result:2}  again: when $t \leq t_0$, we get 
$C n^{-\beta}$ 
as a direct consequence
of 
Remark 
\ref{rem:stability:erg}, 
and this 
without using 
\hergo \,  (for sure, $C$ depends on $t_0$); when 
$t>t_0$, we get the same bound, but now using 
\hergo \, and restarting the two dynamics (with $b$ and $b^n$ as respective drifts) 
from
two initial conditions in 
 the neighbourhood of $\nu_\infty$, which framework fits 
exactly 
\eqref{eq:mn-m*:4}.
Here, $\beta$ is a fixed parameter in $(0,1)$: 
when 
$b$ satisfies \hregb{\alpha}{2}, 
for some $\alpha \in (0,1)$, we choose $\beta  = 1- \alpha$
and  $n = N^{1/(1-\alpha)}$ in 
Proposition \ref{thm main result};
when $b$ just satisfies {\hregb{0}{2}}, 
we choose $\beta= 1-\epsilon$ and $n=N$.
\end{proof}
By combining 
the last two statements with 
Proposition 
\ref{prop:4:phi:norm:-d}, we get the following result: 
\begin{corollary}
\label{cor:4:phi:norm:-d}
Assume
that $b$ satisfies 
\emph{\hregb{\eta}{2}}
for some $\eta \in [0,1)$. 
Assume that there exists a measure $\nu_\infty$ satisfying 
\emph{\hergo}
that attracts the solutions
of 
 \eqref{eq: forward eqn }, uniformly w.r.t. the initial point. Then, there exists $\lambda>0$ and, for any $\alpha \in (0,1)$, there exists also a collection of constants $(C_\delta)_{\delta \in [0,1)}$ such that, for any $N \geq 1$, 
 \begin{equation*}
 \forall t \geq 0, \quad
{\mathbb E}
\Bigl[ \bigl\| \mu^N_{t} - \nu_{\infty} \bigr\|_{-(d+ \alpha)/2,2}^2
\Bigr] \leq 
 \left\{ 
 \begin{array}{ll}
C_0 \Bigl( N^{-1} + e^{-\lambda t} \Bigr) \quad &\textrm{\rm if} \quad \eta>0,
 \\
C_\delta \Bigl( N^{-1+\delta} + e^{-\lambda t}\Bigr) \quad &\textrm{\rm if} \quad \eta=0, \quad \textrm{\rm for any} \  \delta >0.
\end{array}
\right. 
\end{equation*}
\end{corollary}
\begin{proof}
It suffices to choose $\Phi$ as in Proposition 
\ref{prop:4:phi:norm:-d} and to observe from Proposition
\ref{prop main result:2} that $(m(t \, ; \mu))_{t \geq 0}$ converges exponentially fast to $\nu_\infty$.  
\end{proof}
\begin{subsection}{Metastability}
\label{subse:metastable}

As in the second part of the statement of 
Theorem \ref{thm main result:2}, 
we now consider an invariant measure 
$\nu_\infty$ to 
 \eqref{eq: forward eqn }
that 
satisfies {\hergo}
{but that may not be a global attractor.
So is the case throughout the subsection}. 
%Noticeably, 
% \eqref{eq: forward eqn }
% may have other invariant measures 
% (whilst 
% $\nu_\infty$ is the unique invariant measure in the framework of 
% Theorem \ref{thm main result:2}). 
%Some relevant examples are given in the forthcoming Subsection 
%\ref{subse:examples} below. 

We recall now the Markov property satisfied by the empirical distribution
$(\mu^N_t)_{t \geq 0}$, which takes values in ${\mathcal P}_{N}(\bT^d)$, the collection of probability measures  $\mu$ 
that are uniformly distributed 
on some finite state $\{x_{1},\cdots,x_{N}\}$, with ${\boldsymbol x}=(x_{1},\cdots,x_{N})\in ({\mathbb T}^d)^N$.
The Markov property of 
$(\mu^N_t)_{t \geq 0}$
 follows from the exchangeable structure of the particle system
\eqref{eq particles}: 
for any 
permutation $\sigma$ on $\{1,\cdots,N\}$,
$(Y_t^{1,N},\cdots,Y_{t}^{N,N})_{t \geq 0}$ starting from $(x_{1},\cdots,x_{N})$
has the same law as 
$(Y_t^{1,N},\cdots,Y_{t}^{N,N})_{t \geq 0}$ starting from $(x_{\sigma(1)},\cdots,x_{\sigma(N)})$. Accordingly, 
for any (bounded measurable) test functional $\Phi : {\mathcal P}({\mathbb T}^d) \rightarrow {\mathbb R}$, 
the conditional expectation 
${\mathbb E} [ \Phi( {\mu}_{t}^N ) \vert {\mathcal F}_{s} ]$
must be a symmetric function of $(Y^{1,N}_{s},\cdots,Y^{N,N}_{s})$
and thus a function of $\mu^N_s$. 
{Details are as follows. Letting 
$u^N(t,{\boldsymbol x}) = 
{\mathbb E} [ \Phi( {\mu}_{t}^N ) \vert 
(Y_0^{1,N},\cdots,Y_0^{N,N})={\boldsymbol x}]$, 
we 
deduce from the exchangeable structure that 
$u^N(t,\cdot)$ is symmetric. Also, since the dynamics 
\eqref{eq particles} are time-homogeneous, 
we have 
${\mathbb E} [ \Phi( {\mu}_{t}^N ) \vert {\mathcal F}_{s} ]
= u^N(t-s,Y_s^{1,N},\cdots,Y_s^{N,N})$. And then, 
${\mathbb E} [ \Phi( {\mu}_{t}^N ) \vert {\mathcal F}_{s} ]$ is indeed a symmetric function of 
 $(Y^{1,N}_{s},\cdots,Y^{N,N}_{s})$. The same argument holds when $s$ is replaced by a stopping time.}

Below, we often distinguish between two types of initial condition for 
$(\mu^N_t)_{t \geq 0}$. Some properties are stated under the initial condition
$\mu^N_0=\mu^N_{\boldsymbol x}:=N^{-1} \sum_{i=1}^N \delta_{x_i}$ for some tuple 
${\boldsymbol x}=(x_1,\cdots,x_N)\in ({\mathbb T}^d)^N$, i.e. ${\boldsymbol Y}^N_0:=(Y^{1,N}_0,\cdots,Y_0^{N,N})=
{\boldsymbol x}$, in which case we write the corresponding 
probabilities (resp. expectations) in the form ${\mathbb P}(\cdot \, \vert \, {\boldsymbol Y}^N_0 = {\boldsymbol x})$
(resp. ${\mathbb E}[ \cdot \, \vert \, {\boldsymbol Y}_0^N={\boldsymbol x} ]$). Alternatively, 
we may work with  a random initial condition 
${\boldsymbol Y}^N_0 \sim \nu^{\otimes N}$, for some $\nu \in {\mathcal P}({\mathbb T}^d)$, in which case we write 
probabilities in the form ${\mathbb P}(\cdot \, \vert \, {\boldsymbol Y}^N_0 \sim \nu^{\otimes N})$,
and similarly for the expectations. 
With these notations, 
the expectation in 
 \eqref{eq:weak:error}
 (main display in the statement of Proposition 
\ref{thm main result}) should be rewritten 
$ \bE[ \Phi(\mu^{N}_t) \vert {\boldsymbol Y}_0^N \sim \muinit^{\otimes N}
]$. Part of the 
analysis below relies on the fact that 
the arguments in Proposition 
\ref{thm main result} also hold true under 
the initial condition
${\boldsymbol Y}^N_0 = {\boldsymbol x}$. 
We clarify this in the text. 
Here is a primer. 
{The first step in the adaptation of 
Proposition 
\ref{thm main result}
to the case 
${\boldsymbol Y}^N_0 = {\boldsymbol x}$
is to come back to \eqref{eq:errordecomp}. 
We throw away the first term in the right-hand side of the identity and
we just retain the following simpler version of  
\eqref{eq:errordecomp}:
\begin{equation*}
\begin{split}
\Phi(\nlaw[t])  - \cU(t,\nlaw[0]) &=   \cU(0,\nlaw[t]) - \cU(t,\nlaw[0]) .
\end{split} 
\end{equation*}
Under the probability 
${\mathbb P}(\cdot \, \vert \, {\boldsymbol Y}^N_0 =  {\boldsymbol x})$, 
$\nlaw[0]$ is obviously equal to $\mu_{\boldsymbol x}^N$ and the left-hand side is equal to 
$\Phi(\nlaw[t])  - \cU(t,\mu_{\boldsymbol x}^N)$. Then, 
the 
display 
\eqref{eq:second:term:main result intro formula:2}  
becomes useless and one may just focus on 
\eqref{eq:second:term:main result intro formula}.
The objective is then to give an upper bound 
for $\vert 
{\mathbb E}[\Phi(\nlaw[t]) \vert Y_0^N] = {\boldsymbol x} ] - \cU^n(t,\mu_{\boldsymbol x}^N)
\vert$ by adapting the proof of 
Proposition \ref{thm main result}
(we here use $\cU^n$ and not $\cU$ to be consistent with the statement of 
Proposition 
\ref{thm main result}, in which $\Phi$ is computed at ${\mathcal L}(X_t^n)$). In the proof, 
\eqref{eq:term1:lemma2.2} is clearly useless, whereas 
\eqref{eq:term2:lemma2.2} remains necessary. 
As for the term 
$T_{\textrm{\rm add}}(t)$
in 
\eqref{eq:Tadd}, the analysis can be carried out in a similar way, except for one main thing:
in the third step of the proof of 
Proposition \ref{thm main result}, one can no longer invoke the exchangeability property of the particle system. 
Instead, one must estimate 
the summand in
\eqref{eq:pre:exchangeability}
for each pair $(i,j)$ of indices with $i \not = j$. Proceeding as in the fourth step of the proof, one gets
the two bounds
\eqref{eq:term4:lemma2.2}
and
\eqref{eq:term5:lemma2.2}
for any $i \not =j$. The end of the proof is similar. We retrieve the same bound 
as in 
Proposition \ref{thm main result}
but for 
$\vert 
{\mathbb E}[\Phi(\nlaw[t]) \vert Y_0^N] = {\boldsymbol x} ] - \cU^n(t,\mu_{\boldsymbol x}^N)
\vert$.}
\vskip 4pt

\begin{lemma}
\label{lem:metastab}
For any integer $p \in {\mathbb N}$, any 
$\alpha,\epsilon \in (0,1)$ and any 
$t \geq 0$, there exists a constant $C$ such that, 
for any integer $N \geq 1$, any 
probability measure $\muinit \in {\mathcal P}({\mathbb T}^d)$
and any tuple ${\boldsymbol x} \in ({\mathbb T}^d)^N$, 
\begin{equation*}
{\mathbb E} \Bigl[ \| \mu_t^N - m(t \, ; \muinit) \|_{-(d+\alpha)/2,2}^{p} 
\, \vert \, 
{\boldsymbol Y}_0^N \sim \muinit^{\otimes N}
\Bigr] +
{\mathbb E} \Bigl[ \| \mu_t^N - m(t \, ; \mu^N_{\boldsymbol x}) \|_{-(d+\alpha)/2,2}^{p} 
\, \vert \, 
{\boldsymbol Y}_0^N = {\boldsymbol x}
\Bigr]
\leq
C  N^{-p/2+\epsilon}.
\end{equation*}
\end{lemma}

\begin{proof}
We first introduce the following mollification of the norm 
$\| \cdot \|_{-(d+\alpha)/2,2}$, letting, for some $k >0$, 
\begin{equation}
\label{eq:def:aleph:k} 
\aleph_k(q) = \sum_{\vert {\boldsymbol n} \vert \leq k} \frac1{(1+ \vert {\boldsymbol n} \vert^2)^s}  \vert q^{\boldsymbol n}\vert^2,
\quad q \in W^{-(d+\alpha)/2,2}({\mathbb T}^d), 
\end{equation} 
which definition is very similar to 
\eqref{eq:Fourier:Phi}. 
The role of this mollification is very similar to the 
role of the mollification used in 
the last step of the proof of Proposition 
\ref{thm main result}.
 The key point is that all the estimates below are independent of $k$, which allows us to send 
$k$ to $\infty$ in the end. 
Intuitively, this comes from the fact that 
the function 
${\mathcal P}({\mathbb T}^d) \ni 
\mu \mapsto 
\aleph_k (\mu - \nu_0)$, for some $\nu_0 \in {\mathcal P}({\mathbb T}^d)$,
satisfies the conclusion of Proposition
\ref{prop:4:phi:norm:-d} uniformly in $k$ (taking into account the fact that the sum is truncated in the definition of $\aleph$).
By the way,
it is easy to notice that  
$\aleph_k (q) \rightarrow \| q \|_{-(d+\alpha)/2,2}^2$ as $k \rightarrow \infty$. 
In the rest of the proof (except when indicated), $k$ is fixed. For this reason, we omit it in the notation 
$\aleph_k$ and we just write $\aleph$.  

With this convention, and for  given $t_0>0$, $p \in {\mathbb N}$ and $\nu \in {\mathcal P}({\mathbb T}^d)$, we let 
\begin{equation}
\label{eq:def:phipnmu}
\Phi_p^n(\mu) := \Bigl[ \aleph \Bigl( \mu - m^n(t_0 \ ; \nu) \Bigr) \Bigr]^{p},
\end{equation}
with
 $(m^n(t \, ; \nu))_{t \geq 0}$
solving 
 \eqref{eq: forward eqn }
 when the latter is driven by $\tilde b^n$, for 
$\tilde b^n$ as in the proof of Proposition 
\ref{thm main result}, see in particular \eqref{eq:b-tildeb:n}. 
The reason why we force 
 $(m^n(t \, ; \mu))_{t \geq 0}$ to be driven by 
 $\tilde b^n$ and not by $b^n$ 
 (which is what is done in the proof of Proposition 
 \ref{thm main result})
 is explained in the second step below. Moreover, 
the precise choice of $\nu$ is specified next. 
Following \eqref{eq: defofflow}, 
we call ${\mathcal U}_p^n(t,\mu) := \Phi_p^n(m^n(t \, ; \mu))$.
We first establish the bound for the second term in the main inequality of the statement, 
but with $\| \cdot \|_{-(d+\alpha)/2,2}$ replaced by $\aleph$. 
As the resulting bound is shown to be independent of $k$, this also implies 
the expected bound for the second term in the statement. 
%This is only in the end that we obtain a bound for the first one. 
\vskip 4pt

\textit{First Step.} 
We start with $p=1$. 
Then, 
for another integer $q \in {\mathbb N}$
(with $(1-\epsilon) q>1$) 
and 
$n=N^q$ 
in Proposition 
\ref{thm main result} (so that $K^{N^q}_\epsilon={\mathcal O}(N^{\epsilon q})$), the same proof as therein 
(but just using 
\eqref{eq:second:term:main result intro formula} without invoking 
\eqref{eq:second:term:main result intro formula:2} and using \hergol \ instead of \hergo) yields   
\begin{equation*}
\begin{split}
\Bigl\vert 
{\mathbb E} \Bigl[ \bigl\| \mu^N_{t_0} - m^{N^q}(t_0 \, ; \nu)  \bigr\|^{2}_{-(d+\alpha)/2,2}
- 
{\mathcal U}_1^{N^q}( {t_0}, \mu_0^N)  
\Bigr] 
\Bigr\vert 
&\leq \sup_{0 \leq t \leq t_0} 
\Bigl\vert 
{\mathbb E} \Bigl[
{\mathcal U}_1^{N^q}(t , \mu_t^N)
- 
{\mathcal U}_1^{N^q}(t_0 , \mu_0^N)  
\Bigr] 
\Bigr\vert 
\leq \frac{C_{\epsilon,q,t_0}}{N^{1- {(q+1)}\epsilon}}.
\end{split}
\end{equation*}
(The constant right above also depends on $\alpha$, but we feel useless to indicate the dependence on $\alpha$, which is fixed.)
%with probability 1, with the constant $C_{t_0}$ being non-decreasing with $t_0$. 
Replacing $\epsilon$ by $\epsilon/(q+1)$, choosing $\nu=\mu_{\boldsymbol x}^N$ for some 
${\boldsymbol x} \in ({\mathbb T}^d)^N$, working under 
the initial condition $\mu^N_0=\mu^N_{\boldsymbol x}$ (so that 
${\mathcal U}_1^{N^q}(t_0 , \mu_0^N)=0$)
and using in addition 
Remark
\ref{rem:stability:erg}
to compare $m^{N^q}(t_0 \, ; \nu)$ and 
$m(t_0 \, ; \nu)$, 
we get the second inequality in the statement with $p=1$. 
\vskip 4pt

\textit{Second Step.}
We now work with an integer $p \geq 2$.
We repeat the proof of 
Proposition 
\ref{thm main result}, with $n=N^q$ in 
\eqref{eq:term4:lemma2.2}
and
\eqref{eq:term5:lemma2.2}, except that 
we do not upper bound the left-hand side in 
\eqref{eq:term2:lemma2.2}.
 We obtain
 \begin{align}
 \label{eq:metastab:main:inequality}
&\sup_{0 \leq t \leq t_0} 
  \Bigl\vert 
  {\mathbb E} \bigl[
 {\mathcal U}_p^{N^q}(t_0-t,\mu^N_t) 
  - {\mathcal U}_p^{N^q}(t_0,\mu^N_0)
  \bigr]
   \Bigr\vert
\leq \frac{C_{\epsilon,p,q,t_0}}{N^{{q(1-\epsilon)}}} %\bigl( 1 + K_{\epsilon}^{N^q} \bigr)
\\
&\hspace{15pt} +  
 \frac{1}{N}
 %\bigl( 1 + K_{\epsilon}^{N^q} \bigr)
 \sum_{i=1}^d \int_0^{t_0} \biggl\vert \bE \biggl[ 
  \intrd \bigg( \partial_{(y_{2})_i} \partial_{(y_{1})_i}  \frac{ \delta^2 \cU_p^{N^q}}{\delta m^2} (t_0-s, \mu^{N}_s,z,z) \bigg)   \, \mu^{N}_s(\ud z) \bigg] \biggr\vert \, \ud s
  +
    {\mathcal O}\bigl( \varepsilon \bigr). 
 \nonumber
 \end{align}
In the right-hand side, the last term on the first line 
corresponds to 
\eqref{eq:term4:lemma2.2}
and
\eqref{eq:term5:lemma2.2}. The last term on the second line corresponds to  
 the remainder term in 
\eqref{eq:pre:exchangeability}, 
 with $\varepsilon$ as in 
 \eqref{eq:b-tildeb:n}.
 As for the first term on the second line, it corresponds to 
\eqref{eq:second:term:main result intro formula}. 

Now, we 
 use 
Proposition
\ref{diff second order L-deriv}
 to represent 
 $[\partial_{(y_2)_i} \partial_{(y_1)_i} \delta^2 \cU_p^{N^q}/\delta m^2] (t_0-s, \mu,\cdot,\cdot)$
in the right-hand side. 
As made clear in the latter statement, these derivatives can be represented in terms of 
$[\delta^2 \Phi_p^{{N^q}}/\delta m^2](m^{N^q}(t_0-s \, ; \mu),\cdot,\cdot)$
and 
$[\delta \Phi_p^{{N^q}}/\delta m](m^{N^q}(t_0-s \, ; \mu),\cdot)$ 
and  
$d^{(1)}_i$ and $d^{(2)}_{i,j}$ 
in Propositions 
\ref{diff thm 1}
and
\ref{diff thm 2}. 
Then, 
by 
Proposition
\ref{prop:4:phi:norm:-d}, we
observe that, for any 
$\mu \in  {\mathcal P}({\mathbb T}^d)$, 
\begin{equation}
\label{eq:bound:deltaphi_p}
\begin{split}
&
\Bigl\|
\frac{\delta  \Phi^{{N^q}}_p}{\delta m}(\mu,\cdot)
\Bigr\|_{\alpha/4,\infty}
\leq
C_{p}
\Phi_{p-1}^{N^q} \bigl(\mu \bigr)
\sqrt{\Phi^{N^q} \bigl(\mu \bigr)}, 
\quad 
\Bigl\|
\frac{\delta^2  \Phi^{{N^q}}_p}{\delta m^2}(\mu,\cdot,\cdot)
\Bigr\|_{\alpha/4,\infty} \leq 
C_{p}
\Phi_{p-1}^{N^q} \bigl(\mu \bigr).
\end{split}
\end{equation}
Moreover, 
$d^{(1)}_i$ and $d^{(2)}_{i,j}$ 
in Propositions 
\ref{diff thm 1}
and
\ref{diff thm 2}
can be estimated by means of {\hergol}. 
Similar to 
\eqref{eq:bound:partialz1:partialz2:U:theta:**}, we get, for $t \in [0,t_0]$, 
 \begin{equation}
\label{eq:bound:partialz1:partialz2:U:theta:**:aleph}
\sup_{\mu \in \cP(\bT^d)} \sup_{y_1,y_2 \in \bT^d} \bigg|   (\partial_{y_1})_i
(\partial_{y_1})_j \sld[\cU^{N^q}](t,\mu,y_1, y_2) \bigg| \leq    \frac{C_{\epsilon,p,t_0}  {(1+ K_\epsilon^{N^q})}}{1 \wedge t^{1-\alpha/4}} 
\Phi_{p-1}^{N^q} \bigl( m^{N^q}(t \, ; \mu) \bigr).
\end{equation}
Then,
 choosing as before $\nu=\mu_{\boldsymbol x}^N$
 in the definition of $\aleph$ and noticing 
 from the fact that 
 $(m^{N^q}(t \, ; \cdot))_{t \geq 0}$ 
 is
 driven by $\tilde{b}^{N^q}$ 
 that 
$\Phi^{{N^q}}_{p-1}( m^{N^q}(t_0-s \, ; \mu) ) 
= {\mathcal U}_{p-1}^{N^q}(t_0-s,\mu)$, 
for 
$0 \leq s \leq t_0$, 
 \eqref{eq:metastab:main:inequality}
 yields
 \begin{equation}
 \label{eq:bound:deltaphi_p:2} 
 \begin{split}
  &\sup_{0 \leq t \leq t_0} 
  {\mathbb E} \Bigl[
 {\mathcal U}_p^{N^q}(t_0-t,\mu^N_t) 
 \, \vert \, 
% \mu^N_0 = \mu^N_{\boldsymbol x}
  {\boldsymbol Y}_0^N = {\boldsymbol x}
  \Bigr]
  \\
&\hspace{15pt} \leq \frac{C_{\epsilon,p,t_0}}{N^{{q(1-\epsilon)}}} %\bigl( 1 + K_{\epsilon}^{N^q} \bigr)
  +  
 \frac{C_{\epsilon,p,t_0}}{N}
  \bigl( 1 + K_{\epsilon}^{N^q} \bigr)
 \int_0^{t_0} 
\frac{  {\mathbb E} [
 {\mathcal U}_{p-1}^{N^q}(t_0-s,\mu^N_s) 
 \, \vert \, 
{\boldsymbol Y}_0^N = {\boldsymbol x}
  ]}{1 \wedge (t_0-s)^{1-\alpha/4}}
 \, \ud s   +
    {\mathcal O}\bigl( \varepsilon \bigr).
 \end{split}
 \end{equation}
% {the integral on the second line being obtained by inserting  the   representation of 
%  $[\delta^2 \Phi_p^{ {N^q}}/\delta m^2](m^{N^q}(t_0-s \, ; \mu),\cdot,\cdot)$ in the  second line 
%  \eqref{eq:metastab:main:inequality}}. 
We then proceed by induction  {on $p$},
assuming that $ {\mathbb E} [
 {\mathcal U}_{p-1}^{N^q}(t_0-t,\mu^N_t) 
 \, \vert \, 
{\boldsymbol Y}_0^N = {\boldsymbol x}
  ] \leq C_{\epsilon,p, t_0} N^{-p+1+\epsilon} + 
   {\mathcal O}( \varepsilon )$, 
   for any $q>p$ (here, the shape of ${\mathcal O}( \varepsilon )$
  does not really matter because $\varepsilon$ is sent to to $0$ in the end, but a careful inspection would show that it is of the form 
   $\vert {\mathcal O}( \varepsilon )
    \vert \leq 
   C_{\epsilon,p, t_0} \varepsilon$). 
  We get,  {for $q>p$}, 
 \begin{equation*}
 \begin{split}
  &\sup_{0 \leq t \leq t_0} 
  {\mathbb E} \Bigl[
 {\mathcal U}_p^{N^q}(t_0-t,\mu^N_t) 
 \, \vert \, 
{\boldsymbol Y}_0^N = {\boldsymbol x}
  \Bigr]
 \leq \frac{C_{\epsilon,p,t_0}}{N^{ {q(1-\epsilon)}}} %\bigl( 1 + K_{\epsilon}^{N^q} \bigr)
  +  
 \frac{C_{\epsilon,p,t_0}}{N^{p-\epsilon}}
  \bigl( 1 + K_{\epsilon}^{N^q} \bigr)   +  
C_{\alpha,\epsilon,p,t_0}  \varepsilon. 
 \end{split}
 \end{equation*}
 Recalling that $q>p$ and that 
 $\vert  K_{\epsilon}^{N^q} \vert = \vert {\mathcal O}(N^{\epsilon q}) \vert \leq C_{\epsilon} N^{q \epsilon}$
and changing the value of $\epsilon$ in terms of $q$, 
 we prove the induction hypothesis at rank $p$. Choosing 
 $t=t_0$, 
we get 
\begin{equation*}
{\mathbb E} \Bigl[ 
\Bigl[ \aleph
\Bigl( 
 \mu_{t_0}^N - m^{N^q}(t_0 \, ; \mu^N_{\boldsymbol x}) \Bigr)\Bigr]^{p} 
\, \vert \, 
{\boldsymbol Y}_0^N = {\boldsymbol x}
\Bigr]
\leq
C_{\epsilon,p,t_0} N^{-p+\epsilon} + {\mathcal O}(\varepsilon).
\end{equation*}
%Following
% Proposition 
% \ref{prop main result:2}, 
% but just in finite time, 
% we can upper bound ${\rm dist}_{\rm TV}(m^{N^q}(t_0 \, ; \mu),m(t_0 \, ; \mu))$
% by $C/{N^{q/2}}+\delta_\varepsilon(N)$, where $\delta_\varepsilon(N )$ tends to $0$ 
% with $\varepsilon$ for $N$ fixed (Proposition 
% \ref{prop main result:2}  is stated for 
% $\mu$ close to $\nu_\infty$, but the result obviously 
% holds true for any $\mu$ when stated in finite time).
% The presence of the additional $\delta_{\varepsilon}(N)$ 
% comes from the fact that 
%  $(m^{N^q}(t \, ; \cdot))_{t \geq 0}$ 
% is
% driven by $\tilde{b}^{N^q}$, which requires to measure the distance between 
% $b$ and  $\tilde{b}^{N^q}$. 
  Letting $\varepsilon$ tend to $0$, then $k$ to $\infty$ and finally $q$ to $\infty$, we complete the proof
 of the bound of the second term in the main inequality of the statement.  
 \vskip 4pt
 
 \textit{Third Step.}
 We now complete the proof of the bound of the first term in the main inequality of the statement. 
 We rewrite the conclusion of the second step in the form 
 \begin{equation*}
 {\mathbb E} \Bigl[ \bigl\| \mu_{t_0}^N - m^{{N^q}}(t_0,\mu_0^N) \bigr\|^{2p}_{-(d+\alpha)/2,2} 
 \Bigr] \leq  \frac{C_{\epsilon,p,t_0}}{N^{p-\epsilon}} + {\mathcal O}(\varepsilon),
 \end{equation*}
{which follows by conditioning on the value of $Y_0^N$.}
 It thus remains to bound 
 \begin{equation}
 \label{eq:metastab:2}
 {\mathbb E} \Bigl[ \bigl\| m^{{{N^q}}}(t_0,\mu_0^N) - m^{{{N^q}}}(t_0,\muinit)  \bigr\|^{2p}_{-(d+\alpha)/2,2} 
 \, \vert \, {\boldsymbol Y}_0^N \sim \muinit^{\otimes N}
 \Bigr].
 \end{equation}
With the same notation as in the previous step, the strategy is to bound 
 \begin{equation}
 \label{eq:metastab:2:aleph}
 {\mathbb E} \Bigl[ \Bigl[ \aleph_k \Bigl( m^{{{N^q}}}(t_0,\mu_0^N) - m^{{{N^q}}}(t_0,\muinit)  \Bigr)
 \Bigr]^p
  \, \vert \, {\boldsymbol Y}_0^N \sim \muinit^{\otimes N}
 \Bigr],
 \end{equation}
independently of the parameter $k$. We rewrite the above quantity as 
 ${\mathbb E}[{\mathcal U}_p^{N^q}(t_0,\mu_0^N) \, \vert \, {\boldsymbol Y}_0^N \sim \muinit^{\otimes N}]$ when $\nu=\muinit$ in the definition of $\Phi_p$
 in \eqref{eq:def:phipnmu}. 
When $p=1$, it suffices to use  
\eqref{eq:second:term:main result intro formula:2}, which directly says that 
\eqref{eq:metastab:2:aleph}
 is bounded by 
$C_{t_0}/N$. 
When $p \geq 2$,   
we need a bound for  
 the second-order 
 derivatives $[\delta^2 {\mathcal U}^{N^q}_p/\delta m^2](t_0,\cdot)$ (as the latter appears
 in the formula 
 \eqref{eq:second:term:main result intro formula:2}). Proceeding as in the second step
 {(the framework  being  simpler since there is no derivatives in $(y_1)_i$ and 
 $(y_2)_i$)}, 
 we have the following analogue of 
 \eqref{eq:bound:partialz1:partialz2:U:theta:**:aleph}:
 \begin{equation}
\label{eq:bound:partialz1:partialz2:U:theta:**:aleph:2}
 \Bigl\|
\frac{ \delta^2 {\mathcal U}^{N^q}_p}{\delta m^2}(t_0,\mu,\cdot,\cdot) \Bigr\|_{\infty}
\leq C_{p} {\mathcal U}^{N^q}_{p-1}(t_0,\mu), \quad \mu \in {\mathcal P}({\mathbb T}^d). 
\end{equation} 
Back to the formula
\eqref{eq:second:term:main result intro formula:2}, we thus must compute 
${\mathcal U}^{N^q}_{p-1}(t_0,\tilde{\mu}^{N}_{s,s_1})$.
By the definition
\eqref{eq:def:aleph:k} of $\aleph$, we obtain
(with the same notation as in the formula) 
\begin{align}
\nonumber
&{\mathcal U}^{N^q}_{p-1}\bigl(t_0,\tilde{\mu}^{N}_{s,s_1}\bigr) 
 = \Phi^{ {N^q}}_{p-1} \bigl( m^{N^q}( t_0 ,
\tilde{\mu}^{N}_{s,s_1}) \bigr)
\\
&= \Bigl[ \aleph \Bigl( 
m^{N^q}( t_0 ,
\tilde{\mu}^{N}_{s,s_1})
- 
m^{N^q}( t_0 ,\muinit)
\Bigr) \Bigr]^{p-1} 
\label{tilde:mu:s:s1:N}
\\
&\leq C_p 
\Bigl( 
\Bigl[
\aleph
\Bigl(
m^{N^q}( t_0 ,
\tilde{\mu}^{N}_{s,s_1})
- 
m^{N^q}( t_0 , s \mu_0^N + (1-s) \muinit
) 
\Bigr) 
\Bigr]^{p-1} 
 +
\Phi^{ {N^q}}_{p-1} \bigl( m^{N^q}( t_0 ,
s \mu_0^N + (1-s) \muinit \bigr)
\Bigr) \nonumber
\\
&\leq  C_p 
\Bigl( 
\bigl\|
m^{N^q}( t_0 ,
\tilde{\mu}^{N}_{s,s_1})
- 
m^{N^q}( t_0 , s \mu_0^N + (1-s) \muinit
) 
\bigr\|^{2(p-1)}_{-(d+\alpha)/2,2} 
 +
 {\mathcal U}^{N^q}_{p-1}\bigl(t_0,
s \mu_0^N + (1-s) \muinit 
 \bigr) 
\Bigr). \nonumber
\end{align}
By {\hergol} (which holds true for $\tilde b^{N^q}$ w.r.t. constants that are independent of $N$
 and which enforces a form of stability, see \eqref{eq:mn-m:rate:shorttime}) {and by Sobolev embedding (which says that 
$\| \cdot \|_{-(d+\alpha)/2,2}\leq \| \cdot \|_{(0,\infty)'}$)}, 
\begin{equation}
\label{eq:3rdstep:local}
\begin{split}
{\mathcal U}^{N^q}_{p-1}\bigl(t_0,\tilde{\mu}^{N}_{s,s_1}\bigr) 
&\leq  C_{p,t_0} 
\Bigl( 
\bigl\| 
\tilde{\mu}^{N}_{s,s_1}
- 
\bigl( s \mu_0^N + (1-s) \muinit
\bigr) \bigr\|_{(0,\infty)'}^{2(p-1)} 
 +
 {\mathcal U}^{N^q}_{p-1}\bigl(t_0,
s \mu_0^N + (1-s) \muinit 
 \bigr) 
\Bigr)
\\
&\leq  C_{p,t_0} 
\Bigl( 
\frac1{N^{2(p-1)}}
 +
 {\mathcal U}^{N^q}_{p-1}\bigl(t_0,
s \mu_0^N + (1-s) \muinit
 \bigr) 
\Bigr).
\end{split}
\end{equation}
Therefore, by  formulas 
 \eqref{eq:second:term:main result intro formula:2} and 
 \eqref{eq:bound:partialz1:partialz2:U:theta:**:aleph:2},
\begin{equation}
\label{eq:lemn:2.17:cl}
\begin{split}
&{\mathbb E} \Bigl[ {\mathcal U}^{N^q}_{p}\bigl(t_0,\mu_0^N\bigr)
 \, \vert \, {\boldsymbol Y}_0^N \sim \muinit^{\otimes N}
\Bigr] 
\\
&\hspace{15pt} \leq  C_{p,t_0} 
\Bigl( 
\frac1{N^{2p-1}}
 +
\frac1N \sup_{0 \leq s \leq 1} {\mathbb E} \Bigl[ {\mathcal U}^{N^q}_{p-1}\bigl(t_0,
s \mu_0^N + (1-s) \muinit 
 \bigr)
  \, \vert \, {\boldsymbol Y}_0^N \sim \muinit^{\otimes N}
  \Bigr] \Bigr).
  \end{split}
\end{equation}
In fact, we can easily replace in the analysis 
$\mu_0^N$ by any $\lambda \mu_0^N + (1-\lambda) \muinit$. This amounts to repeat the computation with 
$\cU^{ {N^q}}_p(t_0,\lambda \mu + (1-\lambda) \muinit)$
in place of
 $\cU^{ {N^q}}_p(t_0,\mu)$ in formula 
  \eqref{eq:second:term:main result intro formula:2}, which is licit thanks to 
  \cite[Thm. 2.14]{chassagneux2019weak}. 
  {Then, the bound 
 \eqref{eq:bound:partialz1:partialz2:U:theta:**:aleph:2} 
 with $\mu$ replaced by 
 $\lambda \mu + (1-\lambda) \muinit$ remains true. 
 Denoting by 
$\tilde{\mu}^{N}_{s,s_1}[\lambda]$ the 
corresponding version of 
$\tilde{\mu}^{N}_{s,s_1}$
  in 
  \eqref{eq:second:term:main result intro formula:2}, i.e., 
  $$ \tilde{\mu}^{N}_{s,s_1}[\lambda]
  := \frac{ss_1}{N}(\delta_{\tilde{\eta}} - \delta_{{\eta_{1}}}) + {\mu_{\text{init}}} + s \lambda (\mu^N_0 - {\mu_{\text{init}}}), \quad \quad s, s_1 \in [0,1],$$ 
  we can repeat 
  \eqref{tilde:mu:s:s1:N}, which yields
  \begin{align}
\nonumber
{\mathcal U}^{N^q}_{p-1}\bigl(t_0,\tilde{\mu}^{N}_{s,s_1}[\lambda]\bigr) 
&\leq  C_p 
\Bigl( 
\bigl\|
m^{N^q}\bigl( t_0 ,
\tilde{\mu}^{N}_{s,s_1}[\lambda]\bigr)
- 
m^{N^q}( t_0 , s\lambda \mu_0^N + (1-s \lambda) \muinit
) 
\bigr\|^{2(p-1)}_{-(d+\alpha)/2,2} 
\\
&\hspace{15pt} +
 {\mathcal U}^{N^q}_{p-1}\bigl(t_0,
s\lambda \mu_0^N + (1-s \lambda) \muinit 
 \bigr) 
\Bigr). \nonumber
\end{align}
And then,
\eqref{eq:3rdstep:local}
becomes
\begin{equation*}
\begin{split}
{\mathcal U}^{N^q}_{p-1}\bigl(t_0,\tilde{\mu}^{N}_{s,s_1}[\lambda]\bigr) 
&\leq C_{p,t_0} 
\Bigl( 
\frac1{N^{2(p-1)}}
 +
 {\mathcal U}^{N^q}_{p-1}\bigl(t_0,
s \lambda \mu_0^N + (1-s \lambda) \muinit
 \bigr) 
\Bigr).
\end{split}
\end{equation*}}In the end, we can replace $\mu_0^N$ in the left-hand side of 
\eqref{eq:lemn:2.17:cl}
  by $\lambda \mu_0^N + (1-\lambda) \muinit$. 
Taking the supremum over $\lambda$, we obtain 
\begin{equation*}
\begin{split}
&
\sup_{0 \leq \lambda \leq 1}
{\mathbb E} \Bigl[ {\mathcal U}^{N^q}_{p}\bigl(t_0, \lambda \mu_0^N + (1-\lambda) \muinit \bigr)
 \, \vert \, {\boldsymbol Y}_0^N \sim \muinit^{\otimes N}
\Bigr] 
\\
&\hspace{15pt} \leq  C_{p,t_0} 
\Bigl( 
\frac1{N^{2(p-1)}}
 +
\frac1N \sup_{0 \leq s \leq 1} {\mathbb E} \Bigl[ {\mathcal U}^{N^q}_{p-1}\bigl(t_0,
s \mu_0^N + (1-s) \muinit 
 \bigr)
  \, \vert \, {\boldsymbol Y}_0^N \sim \muinit^{\otimes N}
  \Bigr] \Bigr).
  \end{split}
\end{equation*}
We deduce by induction (on $p$) that 
the left-hand side is less than $C_{p,t_0} N^{-p}$. 
 \end{proof}

We now provide  metastability properties of an invariant measure 
$\nu_\infty$ 
satisfying \hergo. Before we state the results, we claim that, from 
Proposition 
\ref{prop main result:2}, for any $\alpha \in (0,1)$, there exists $a>0$ such that, for 
\begin{equation}
\label{eq:basin:a}
\lim_{t \rightarrow \infty}
\sup_{\| \mu - \nu_\infty\|_{-(d+\alpha)/2,2} \leq a} 
\bigl\| m(t \, ; \mu) - \nu_\infty \bigr\|_{-(d+\alpha)/2,2} =0.
\end{equation} 
The proof relies on the following observation. Fix $\delta >0$. Then, for any $\varepsilon >0$,  
we can find {$\Gamma_\varepsilon >0$} such that, 
for any $\varphi$ in the unit ball of $W^{(\delta,\infty)}({\mathbb T}^d)$, 
there exists $\psi \in W^{((d+\alpha)/2,\infty)}({\mathbb T}^d)$ with 
$\| \psi \|_{(d+\alpha)/2,2} \leq \| \psi \|_{(d+\alpha)/2,\infty} \leq  {\Gamma_\varepsilon}$ and 
$\| \varphi - \psi \|_{0,\infty} \leq \varepsilon$, 
which proves that, for any distribution $q$, 
\begin{equation}
\label{eq:embedding:0,alpha,d/2}
\| q \|_{(\delta,\infty)'}
\leq \varepsilon 
\| q \|_{(0,\infty)'} 
+ 
\Gamma_\varepsilon 
\| q \|_{-(d+\alpha)/2,2}. 
\end{equation}
Display \eqref{eq:basin:a} follows by choosing $\varepsilon$ small and by noting that 
$\| q \|_{(0,\infty)'} \leq 2$ when $q$ is the difference of two probability measures. 
Therefore, for $a_\alpha$ as in 
Proposition 
\ref{prop main result:2}, 
we can find $a \in (0,a_\alpha)$ such that 
$\| \mu  -  \nu_\infty \|_{-(d+\alpha)/ {2},2} \leq a$
implies
$\| \mu  -  \nu_\infty \|_{(\alpha,\infty)'} \leq a_\alpha$, which allows one to use Proposition 
\ref{prop main result:2}.

\begin{corollary}
\label{cor:3:16}
Consider an invariant measure 
$\nu_\infty$ to 
 \eqref{eq: forward eqn }
at which \emph{\hergo} holds (for the drift $b$) and, for some $\alpha \in (0,1)$, let 
$a$ be as in \eqref{eq:basin:a}. 
Then, for 
any integer $p \geq 1$, 
there exist $\varepsilon >0$ and a constant $C_{p}$ such that, 
for any integer $N \geq 1$,
any $\muinit \in {\mathcal P}({\mathbb T}^d)$ with 
$\| \muinit - \nu_\infty\|_{-(d+\alpha)/2,2} \leq \varepsilon$
and any ${\boldsymbol x}_0
\in ({\mathbb T}^d)^N$ 
with
$\| \mu_{{\boldsymbol x}_0}^N - \nu_\infty\|_{-(d+\alpha)/2,2} \leq \varepsilon$,
\begin{equation*}
\begin{split}
& {\mathbb P}\Bigl( 
 A^{(N)}
\, \vert \, 
{\boldsymbol Y}_0^N \sim \muinit^{\otimes N} 
\Bigr)
 +
 {\mathbb P}\Bigl( 
 A^{(N)}
\, \vert \, 
{\boldsymbol Y}_0^N = {\boldsymbol x}_0
\Bigr)
\leq \frac{C_{p}}{N^p}, 
\quad 
{\rm with} \ \ A^{(N)} := 
\Bigl\{
 \sup_{t \leq N^p}
 \| \mu_t^N - \nu_\infty \|_{-(d+\alpha)/2,2} \geq \frac{a}2
\Bigr\}. 
 \end{split}
\end{equation*}
\end{corollary}
\begin{proof} { \ }

\textit{First Step.} 
For $a$ as in \eqref{eq:basin:a}, 
we can find $t_0>0$ such that 
$\sup_{\| \mu - \nu_\infty \|_{-(d+\alpha)/2,2} \leq a} \|  m(t_0 \, ; \mu) - \nu_\infty \|_{-(d+\alpha)/2,2} \leq a/2$.
Then, 
for any integer
$k \in {\mathbb N}$, 
we 
let
$A_k :=\{ \| \mu_{k t_0}^N - \nu_\infty \|_{-(d+\alpha)/2,2} \geq 3a/4
\}$. For $k \geq 2$, 
and 
with ${\mathbb P}$ being implicitly understood as 
${\mathbb P}( \cdot \, \vert \, {\boldsymbol Y}_0^N \sim \muinit^{\otimes N} )$
or 
${\mathbb P}( \cdot \, \vert \, {\boldsymbol Y}_0^N = {\boldsymbol x}_0 )$, 
we
have
\begin{equation*}
\begin{split}
&{\mathbb P}\bigl( A_k
\bigr) 
\leq 
{\mathbb P}\bigl( A_{k}
\, \vert \, 
A_{k-1}^{\complement}
\bigr)
{\mathbb P}\bigl( 
A_{k-1}^{\complement}
\bigr) +
{\mathbb P}\bigl(
A_{k-1}
\bigr).
\end{split}
\end{equation*}
%$${\mathbb P}\Bigl( 
%\| \mu_{t_0}^N - \nu_\infty \|_{(2+\alpha,\infty)'} \geq 3a/4
%\vert 
%\| \mu_{0}^N - \nu_\infty \|_{(2+\alpha,\infty)'} \leq 3a/4
%\Bigr) \leq \frac{C}{N^p}.$$
By 
the homogeneous 
Markov 
structure of $(\mu_t^N)_{t \geq 0}$, we 
have
\begin{equation*}
\begin{split}
{\mathbb P}\bigl( A_k
\, \vert \,  A_{k-1}^{\complement}
\Bigr)
&\leq \sup_{{\boldsymbol x} : \| \mu^N_{\boldsymbol x} - \nu_\infty \|_{-(d+\alpha)/2,2} \leq 3a/4}
 {\mathbb P}\Bigl( \Bigl\{ \bigl\| \mu_{t_0}^N - \nu_\infty \bigr\|_{-(d+\alpha)/2,2} \geq 3a/4
\Bigr\}
\, \vert \, \mu_0^N = \mu^N_{\boldsymbol x}
\Bigr).
\end{split}
\end{equation*}
For ${\boldsymbol x}$ as in the supremum right above, 
$\| m(t_0\, ; \, \mu^N_{\boldsymbol x}) - \nu_\infty \|_{-(d+\alpha)/2,2}
\leq a/2$ (by definition of $a$). Therefore, 
\begin{equation*}
\begin{split}
{\mathbb P}\bigl( A_k
\, \vert \,  A_{k-1}^{\complement}
\Bigr)
&\leq \sup_{{\boldsymbol x} \in ({\mathbb T}^d)^N}
{\mathbb P}\Bigl( \Bigl\{ \bigl\| \mu_{t_0}^N - m(t_0 \, ; \mu_{\boldsymbol x}^N) 
\bigr\|_{-(d+\alpha)/2,2} \geq a/4
\Bigr\}
\, \vert \,  
\mu_0^N = \mu_{\boldsymbol x}^N
\Bigr)
\leq \frac{C_{p}}{N^{p}},
\end{split}
\end{equation*}
with the second bound following from 
Lemma \ref{lem:metastab}.
Provided that $\muinit$ or $\mu_{{\boldsymbol x}_0}^N$ in the statement is close enough to 
$\nu_\infty$, we can handle 
${\mathbb P}(A_1)$ in a similar way and prove it to be less than $C_{p}/N^{p}$. By iteration, 
we obtain
${\mathbb P}( A_k) \leq k C_{p} / N^{p}$.
\vskip 4pt

\textit{Second Step.}
For $k \in {\mathbb N}$ and $t_0$ as above, we now compute 
${\mathbb P}(\{ \| \mu_{t}^N - \nu_\infty \|_{-(d+\alpha)/2,2} \geq 3a/4\} )$ 
for $t \in [kt_0,(k+1)t_0]$
under the assumption that 
$\| \muinit - \nu_\infty\|_{-(d+\alpha)/2,2}$ 
(resp. 
$\| \mu^N_{{\boldsymbol x}_0} - \nu_\infty\|_{-(d+\alpha)/2,2}$ )
is small enough
and under the initial condition 
${\boldsymbol Y}_0 \sim \muinit^{\otimes N}$ (resp. 
${\boldsymbol Y}_0 = {\boldsymbol x}_0$). 
By Sobolev embedding
and
by the stability property 
\eqref{eq:mn-m:rate:shorttime:37},
there exists a real $\delta(\alpha) \in (0,1)$ such that 
\begin{equation*}
\begin{split}
\bigl\| \mu_{t}^N - \nu_\infty \bigr\|_{-(d+\alpha)/2,2}
&\leq 
\bigl\| \mu_{t}^N - m(t -kt_0 \, ; \mu_{kt_0}^N) \bigr\|_{-(d+\alpha)/2,2}
+
\bigl\|  m(t -k t_0 \, ; \mu_{k t_0}^N) -  \nu_\infty \bigr\|_{-(d+\alpha)/2,2}
\\
&\leq 
\bigl\| \mu_{t}^N - m(t -kt_0 \, ; \mu_{kt_0}^N) \bigr\|_{-(d+\alpha)/2,2}
+
\bigl\|  m(t -k t_0 \, ; \mu_{k t_0}^N) -  m(t -k t_0 \, ; \nu_\infty) \bigr\|_{(\delta(\alpha),\infty)'}
\\
&\leq
\bigl\| \mu_{t}^N - m(t -kt_0\, ; \mu_{kt_0}^N) \bigr\|_{-(d+\alpha)/2,2}
+
C_{t_0}
\bigl\| \mu_{kt_0}^N -  \nu_\infty \bigr\|_{(\delta(\alpha),\infty)'},
\end{split}
\end{equation*}
where we used 
$m(t - kt_0 \, ; \nu_\infty)=\nu_\infty$ in the second line. 
By 
\eqref{eq:embedding:0,alpha,d/2}, we  can find $a' \in (0,a)$ such that 
$\| \mu_{kt_0}^N -  \nu_\infty \|_{-(d+\alpha)/2,2} \leq a'$
implies
$
C_{t_0}
\| \mu_{kt_0}^N -  \nu_\infty \|_{(\delta(\alpha),\infty)'} \leq a/4$.
Then, proceeding as in the first step, Markov property yields
\begin{equation*}
\begin{split}
{\mathbb P}\Bigl( \Bigl\{ \| \mu_{t}^N - \nu_\infty \|_{-(d+\alpha)/2,2} \geq 3a/4
\Bigr\} \Bigr) 
& \leq 
 {\mathbb P}\Bigl( 
\Bigl\{  \bigl\| \mu_{k t_0}^N - \nu_\infty \bigr\|_{-(d+\alpha)/2,2} \geq a'
\Bigr\}
\Bigr)
\\
&\hspace{-10pt} + \sup_{{\boldsymbol x} \in ({\mathbb T}^d)^N}
{\mathbb P}\Bigl( \Bigl\{ \bigl\| \mu_{t- k t_0}^N - m(t- k t_0 \, ; \mu_{{\boldsymbol x}}^N) \bigr\|_{-(d+\alpha)/2,2} \geq a/2
\Bigr\}
\, \vert \,  
\mu_0^N=
\mu^N_{{\boldsymbol x} }
\Bigr).
\end{split}
\end{equation*}
The   term on the second line is handled by means of
Lemma 
\ref{lem:metastab}
whilst the term on the first line is treated by means of the first step, by modifying the value of 
$a$. We get, for $k \geq 1$  and $t \in [kt_0,(k+1)t_0)$, 
\begin{equation*}
{\mathbb P}\Bigl( \Bigl\{ \| \mu_{t}^N - \nu_\infty \|_{-(d+\alpha)/2,2} \geq 3a/4
\Bigr\}
\Bigr) 
\leq \frac{C_{p}(k+1)}{N^{p}}.
\end{equation*}
If $k=0$, we can proceed in a similar way in order to upper bound 
${\mathbb P}(\{ \| \mu_{t}^N - \nu_\infty \|_{-(d+\alpha)/2,2} \geq 3a/4\})$
for $t \in [0,t_0]$. One just needs  to assume that 
$\| \mu - \nu_\infty\|_{-(d+\alpha)/2,2}$ (resp. 
$\| \mu^N_{{\boldsymbol x}_0} - \nu_\infty\|_{-(d+\alpha)/2,2}$) is small enough. 
\vskip 4pt

\textit{Third Step.}
We now complete the proof.  For $p$ as above, with {$p/2$ even}, 
and we consider the mesh $(s_j=j/N)_{0 \leq j \leq  {N^{p/4}}}$. Then, 
as a consequence of the second step, 
\begin{equation*}
{\mathbb P}
\biggl( \bigcup_{j=1}^{N^{ {p/4}}}  \Bigl\{ \| \mu_{s_j}^N - \nu_\infty \|_{-(d+\alpha)/2,2} \geq 3a/4\Bigr\}
\biggr) {\leq}  \frac{C_{p}}{N^{ {p/2}}}. 
\end{equation*}
Now, back to \eqref{eq particles}, it is standard to prove that, for each 
$j \in \{0,\cdots,N^{{p/4}}-1\}$, 
\begin{equation*}
\frac1N {\mathbb E} \Bigl[ \sup_{s_j \leq s \leq s_{j+1}} \sum_{i=1}^N \vert
Y^{i,N}_s -  
Y^{i,N}_{s_j}\vert^p \Bigr] \leq \frac{C_p}{N^{p/2}}, 
\end{equation*}
from which we deduce that 
\begin{equation*}
{\mathbb E} \Bigl[ \sup_{s_j \leq s \leq s_{j+1}} 
 \| \mu_{s}^N - \mu_{s_j}^N \|_{-(d+\alpha)/2,2}^p
 \Bigr] \leq 
 {\mathbb E} \Bigl[ \sup_{s_j \leq s \leq s_{j+1}} 
 \| \mu_{s}^N - \mu_{s_j}^N \|_{(\delta(\alpha),\infty)'}^p
 \Bigr]
 \leq
\Bigl( \frac{C_p}{N^{p/2}} \Bigr)^{\delta(\alpha)}, 
\end{equation*}
and then, we easily deduce that (changing $p$ into $2p/\delta(\alpha)$ in the above bound)
\begin{equation*}
{\mathbb P} \Bigl( \Bigl\{ \sup_{s_j \leq s \leq s_{j+1}} 
 \| \mu_{s}^N - \mu_{s_j}^N \|_{-(d+\alpha)/2,2} \geq \frac{a}4
  \Bigr\} 
  \Bigr)
  \leq \frac{C_{p}}{N^{p}}, 
\end{equation*}
and then 
\begin{equation*}
\begin{split}
&{\mathbb P}
\biggl(  {\bigcup_{j=1}^{N^{{p/4}}}}  \Bigl(  \Bigl\{ \| \mu_{s_j}^N - \nu_\infty \|_{-(d+\alpha)/2,2} \geq 3a/4
\Bigr\} \cup \Bigl\{ 
  \sup_{s_j \leq s \leq s_{j+1}} 
 \| \mu_{s}^N - \mu_{s_j}^N \|_{-(d+\alpha)/2,2} \geq a/4
 \Bigr\}
 \Bigr) 
\biggr)
  \leq  { \frac{C_{p}}{N^{p/2}}}
   \leq  { \frac{C_{p}}{N^{p/4}}},
\end{split} 
\end{equation*}
from which the result easily follows {(changing $p/4$ into $p$)}. 
\end{proof}

We deduce the following statement, which says that 
for any polynomial time in $N$, 
the empirical measure stays at distance of order $N^{-1/2}$ 
(for $\| \cdot \|_{-(d+\alpha)/2,2}$) to the solution of the Fokker-Planck equation. 
%It also says that when the empirical measure is in a macroscopic neighborhood of the invariant measure, 
%it spends in fact a negligible time at a macroscopic distance from the invariant measure (and, equivalently, 
%it must be very close to it). 

\begin{theorem}
\label{prop:3:18}
Consider an invariant measure 
$\nu_\infty$ to 
 \eqref{eq: forward eqn }
at which \emph{\hergo} holds (for the drift $b$). Then, for  $\alpha \in (0,1)$, 
there exists  
$a >0$ such that, for any $p \geq 1$ and $\epsilon >0$, there is a constant 
$C$ satisfying, for any integer $N$,  
\begin{equation}
\label{eq:prop:3:18:1}
\begin{split}
&\sup_{\| \mu_{\textrm{\rm init}} - \nu_\infty \|_{-(d+\alpha)/2,2} \leq a} \, 
\sup_{t \leq N^p} {\mathbb E} \Bigl[ \| \mu_t^N - m(t \, ; \muinit) \|_{-(d+\alpha)/2,2}^{p} 
\, \vert \, 
{\boldsymbol Y}_0^N \sim \muinit^{\otimes N}
\Bigr]
\\
&\hspace{15pt} +
\sup_{{\boldsymbol x} : 
\| \mu^N_{\boldsymbol x} - \nu_\infty \|_{-(d+\alpha)/2,2} \leq a} \, 
\sup_{t \leq N^p} {\mathbb E} \Bigl[ \| \mu_t^N - m(t \, ; \mu^N_{\boldsymbol x}) \|_{-(d+\alpha)/2,2}^{p} 
\, \vert \, 
{\boldsymbol Y}_0^N = {\boldsymbol x}
\Bigr]
 \leq C N^{-p/2+\epsilon}.
 \end{split}
\end{equation}

%Moreover, there exists $a>0$ and, for any integer $p \in {\mathbb N}$, there
%exists
%a constant $C_{p}$ such that, for 
% any initial distribution $\mu_{\textrm{\rm init}}$
% satisfying $\| \muinit - \nu_\infty \|_{-(d+\alpha)/2,2} \leq a$,  any integer $N \in {\mathbb N}$ and any real $c \in (0,1)$, 
%\begin{equation}
%\label{eq:prop:3:18:2}
%\sup_{t \geq 0} 
%\frac1t  {\mathbb E} \int_0^t {\mathbf 1}_{\{c a  \leq  \| \mu^N_s - \nu_\infty \|_{-(d+\alpha)/2,2} \leq a \}} 
%\ud s 
%\leq \frac{C_p}{N^p}. 
%\end{equation}
%
%
\end{theorem}

\begin{proof}
The main idea is to localise 
the arguments in 
 \eqref{main result intro formula} (the proof of which relies on a semi-martingale expansion) 
 and Proposition 
\ref{thm main result}, using the stopping time $\tau^N := \inf\{ s>0 : \| \mu_s^N - \nu_\infty \|_{-(d+\alpha)/2,2} \geq a\} \wedge t_0$, where 
$t_0$ is a fixed time in $[0,N^p]$ that plays the role of 
$t$ in 
 \eqref{main result intro formula}
and
$a$ is chosen in such a way that \hergo \ holds true for 
$\mu$ satisfying 
$ \| \mu - \nu_\infty \|_{-(d+\alpha)/2,2} \leq a$ (and for the dynamics driven by the mollified drifts $\tilde b^n$, for 
$n \geq N$).
Indeed, by
\eqref{eq:embedding:0,alpha,d/2}, we can render 
$ \| \mu - \nu_\infty \|_{(\alpha,\infty)'}$ small enough by choosing 
$a$ small enough. 
Then, Proposition
\ref{prop main result:2} 
and Remark 
\ref{rem:stability:erg}
 guarantee that, for $N$ large enough and $n \geq N$,  
 $\tilde b^n$ satisfies 
\hergo \ holds at any $\mu$ satisfying
$ \| \mu - \nu_\infty \|_{(\alpha,\infty)'}$ small enough and thus for $a$ small enough. 
 
The next step is to follow the proof of Lemma 
\ref{lem:metastab} 
 (with the same notation for ${\mathcal U}_p^{N^q}$), recalling that we are now looking for constants independent of $t_0$
 (and also of $k$ in \eqref{eq:def:aleph:k}). In this respect, \hergo \ is crucial. 
Indeed, 
thanks to the localisation, 
$\mu^N_s$ stays in the basin where \hergo \, holds true as long as 
$s \leq \tau^N$. 
In the end, 
by expanding 
${\mathcal U}_{p}^{N^q}(t_0-s,\mu_s^N)$ in \eqref{eq:second:term:main result intro formula}
for $s$ between $0$ and $t \wedge \tau^N$, 
\eqref{eq:metastab:main:inequality}
(with
\eqref{eq:def:phipnmu})
becomes
\begin{equation*}
\begin{split}
&\sup_{0 \leq t \leq t_0} 
  \Bigl\vert 
  {\mathbb E} \bigl[
 {\mathcal U}_p^{N^q}\bigl(t_0-t \wedge \tau^N,\mu^N_{t \wedge \tau^N}\bigr) 
  - {\mathcal U}_p^{N^q}(t_0,\mu^N_0)
  \bigr]
   \Bigr\vert
  \\
&\leq \frac{C_{ {\epsilon},p,q}}{N^{ {q(1-\epsilon)}}} \bigl( 1 + K_{\epsilon}^{N^q} \bigr) 
 +  
 \frac{1}{N}%\bigl( 1 + K_{\epsilon}^{N^q} \bigr)
 \sum_{i=1}^d \bE\biggl[  \int_0^{t_0 \wedge \tau^N}  \biggl\vert  
  \intrd \bigg( \partial_{(y_{2})_i} \partial_{(y_{1})_i}  \frac{ \delta^2 \cU_p^{N^q}}{\delta m^2} (t_0-s, \mu^{N}_s)(z,z) \bigg)   \, \mu^{N}_s(\ud z)  \biggr\vert \bigg] \, \ud s +       {\mathcal O}\bigl( \varepsilon \bigr),
\end{split}
\end{equation*}
where $q>p$. 
The first key point is that 
 ${\mathbb P}(\{\tau^N <t_0\}) \leq  {C_{p} N^{-p}}$, 
which follows from Corollary 
\ref{cor:3:16}.
The second one is that, in 
the above right-hand side, the constant $C_{\epsilon,p,q}$ is independent of $t_0$, which follows from 
\eqref{eq:term4:lemma2.2}
and
\eqref{eq:term5:lemma2.2}
together with the fact that 
all the involved derivatives of $\cU_p^{N^q}$
 feature an extra exponential decay (in time) at any 
 probability measure belonging to the 
basin where \hergo \, holds true. 
In particular, so is the case when those derivatives are computed at $\mu_s^N$ for $s \leq \tau^N$. As a result, and this is our third point, 
the available bounds for the derivatives
$[\partial_{(y_2)_i} \partial_{(y_1)_i} \delta^2 \cU_p^{N^q}/\delta m^2] (t_0-s, \mu^N_s,\cdot,\cdot)$
(that appear on the second line) also
feature an extra exponential decay $\exp(-\lambda(t_0-s))$.

By
inserting this extra exponential decay in 
\eqref{eq:bound:partialz1:partialz2:U:theta:**:aleph}, we get the following variant of 
 \eqref{eq:bound:deltaphi_p:2}
 (with $\nu = \mu_{\boldsymbol x}^N$ in 
 \eqref{eq:def:phipnmu} and $\| \mu_{\boldsymbol x}^N - \nu_\infty \|_{-(d+\alpha)/2,2} \leq a$): 
  \begin{align}
  &\sup_{0 \leq t \leq t_0} 
  {\mathbb E} \Bigl[
 {\mathcal U}_p^{N^q}(t_0-t,\mu^N_t) 
 \, \vert \, 
{\boldsymbol Y}_0^N = {\boldsymbol x}  \Bigr]
    \label{eq:thm:3:19:1}
  \\
&\hspace{15pt} \leq 
\frac{C_{\epsilon,p,q}}{N^{q(1-\epsilon)}}
  +  
 \frac{C_{\epsilon,p,q}}{N}
  \bigl( 1 + K_{\epsilon}^{N^q} \bigr)
 \int_0^{t_0} 
\frac{  {\mathbb E}[
 {\mathcal U}_{p-1}^{N^q}(t_0-s,\mu^N_s) 
 \, \vert \, 
{\boldsymbol Y}_0^N = {\boldsymbol x}
 ]}{1 \wedge (t_0-s)^{1-\alpha/4}} e^{-\lambda(t_0-s)}
 \, \ud s   +
    {\mathcal O}\bigl( \varepsilon \bigr). 
    \nonumber
 \end{align}
 By induction, 
we get $ {\mathbb E} [
 {\mathcal U}_{p}^{N^q}(t_0-t,\mu^N_t) 
 \, \vert \, 
{\boldsymbol Y}_0^N = {\boldsymbol x}
  ] \leq C_{\epsilon,p} N^{-p+\epsilon} + 
   {\mathcal O}( \varepsilon )$, 
   with $\vert {\mathcal O}( \varepsilon ) \vert \leq 
   C_{\epsilon,p,q} \varepsilon$.
   Together with Proposition 
   \ref{prop main result:2},
   this yields the bound for the second term in \eqref{eq:prop:3:18:1}. 
  The bound for the first term in \eqref{eq:prop:3:18:1} is obtained as in the third step of the proof of 
  Lemma
  \ref{lem:metastab}, except that we use {\hergo} instead of 
  {\hergol} in    
   \eqref{eq:3rdstep:local}.   
  By 
{\hergo}, 
  the right-hand side of 
  \eqref{eq:bound:partialz1:partialz2:U:theta:**:aleph:2} 
 features an extra exponential decay
 and the two constants 
$C_{p,t_0}$ in \eqref{eq:3rdstep:local}
and
\eqref{eq:lemn:2.17:cl}
are independent of $t_0$, which suffices to conclude. 
\end{proof}

We can now complete:
\begin{proof}[Proof of Theorem \ref{thm main result:2}, second part.]
To prove the second part of 
of Theorem \ref{thm main result:2} (when $\nu_\infty$ is not a global attractor), it suffices to repeat the localisation argument used in 
 \eqref{eq:thm:3:19:1}, except that 
 $\Phi$ in  
 \eqref{eq:def:phipnmu}
 is now taken as a general test functional satisfying \hintphi{4}{3} and that 
 $p$ is implicitly taken as $1$ (in clear, there is no induction, which makes the proof very much simpler). 
 As in the proof of 
 Theorem 
 \ref{prop:3:18}, the key point is to use Corollary 
\ref{cor:3:16}
in order to upper bound   
 ${\mathbb P}(\{\tau^N <t_0\})$ by $C_{q} N^{-q}$, 
 for any $q \geq 1$.
\end{proof}
\end{subsection}

\begin{subsection}{Examples}
\label{subse:examples}

\subsubsection{Small case interaction}

\begin{proposition} \label{W1 decay:small} 
 {Assume that
$b$ satisfies \emph{\hregb{\eta}{2}}, for some $\eta \in [0,1)$}. 
Then, there exists $\eps_0>0$ $($only depending on  $\sup_{m \in \cP(\bT^d)} \| b(\cdot,m)\|_{0,\infty})$
such that 
\eqref{eq:thm:main:result:2}
holds true 
 if
\begin{equation}
\label{eq:condition:epsilon0:petit}
 \sup_{m \in \cP(\bT^d)} \Bigl\| \ld[b] (\cdot, m)(\cdot) \Bigr\|_{0, \infty} < \epsilon_{0}.
\end{equation}
\end{proposition}

\begin{proof}
The strategy is just to repeat the proof of  {finite in time result, see} Proposition \ref{W1 decay} (with $b$ replaced by $b^n$, with 
$b^n$ as in the proof of Proposition \ref{thm main result}),  {and then} to apply 
Proposition \ref{thm main result}, noticing that  
the constant $C_{\alpha,b}$
in 
\eqref{eq: bound q est 2} 
(which derives from the analysis of $T_3$ in   \eqref{duality W1 infty})
can be made small if 
$\epsilon_0$ in \eqref{eq:condition:epsilon0:petit} 
 is small
(which in turn implies that 
$\sup_{m \in \cP(\bT^d)} \| [\delta b^n/\delta m] (\cdot, m)(\cdot) \|_{0, \infty} < \epsilon_{0}$). 
We then conclude 
as in the proof of 
\eqref{eq:q:mn:third:step} that each $b^n$ satisfies {\hergo} (for any probability measure $\mu$): the role played by 
$\delta + a$ in 
\eqref{eq:q:mn:third:step}  {and the line before} 
is here played by $\epsilon_0$; 
moreover, the last term in
the display 
\eqref{eq:q:mn:third:step}
does not
appear in this analysis, 
see for instance
    \eqref{duality W1 infty} 
    where the only difficulty comes from $T_3$. 
    The fact that Proposition
    \ref{thm main result} holds true permits to conclude directly, as in the proof of 
    the first part of 
    Theorem \ref{thm main result:2}.
\end{proof}

\subsubsection{Conservative case}

\begin{proposition}
\label{prop:div:free}
 {Assume that
$b$ satisfies \emph{\hregb{\eta}{2}}}, for some $\eta \in [0,1)$, and that, for any $m \in {\mathcal P}({\mathbb T}^d)$, $b(\cdot,m)$ is divergence free 
in the sense of distribution. 
Then, 
\eqref{eq:thm:main:result:2}
holds true.
\end{proposition}

A prototype for $b$ being divergence free  is $b(x,m) = (B*m)(x)$, for a bounded vector field $B$ from 
${\mathbb T}^d$ into ${\mathbb R}^d$ with zero divergence (in the sense of distribution). 

\begin{proof}
Obviously, the Lebesgue measure is invariant for  \eqref{eq: forward eqn }. 
We prove that it is exponentially stable. 
It suffices to observe that the solution to the Fokker-Planck equation 
writes (in the sense of distribution):
\begin{equation*}
\partial_t m(t \, ; \mu) - \tfrac12 \Delta m(t \, ; \mu) + b\bigl(\cdot , m(t \, ; \mu)\bigr) \cdot \nabla_x m(t \, ; \mu) 
= 0.
\end{equation*}
Since $m(t \, ; \mu)$ is known to have a bounded density for any $t>0$, it is clear that the solution to 
the above equation has a continuous gradient (in $x$) at any positive time $t>0$. 
By Lemma 
\ref{conjecture backward PDE }, it satisfies 
\begin{equation*}
\bigl\| m(t \, ; \mu) - 1 \bigr\|_\infty
=
\Bigl\| m(t \, ; \mu) - \int_{{\mathbb T}^d} m(t \, ; \mu)(x) dx \Bigr\|_\infty \leq C e^{-\lambda t}. 
\end{equation*}
As for the linearized operator
\eqref{eq:linearized:operator}
 at $\textrm{\rm Leb}_{{\mathbb T}^d}$, 
 we notice that 
 $[\delta b/\delta m](\cdot,m)(q)$ is divergence free.
 This is shown %to be true
 (first at any $m$ with a positive density and then at any $m \in {\mathcal P}({\mathbb T}^d)$) 
  by passing to the limit in 
  \begin{equation*}
\frac1{\epsilon} \int_{{\mathbb T}^d} b\bigl(x,  m + \epsilon q \bigr) \cdot \nabla \varphi(x) \, \ud x =0,
 \end{equation*}
 for 
any smooth  $q : {\mathbb T}^d \rightarrow {\mathbb R}$ with $\int_{{\mathbb T}^d} q(x) \ud x = 0$ and any
 smooth $\varphi : {\mathbb T}^d \rightarrow {\mathbb R}$. 
Therefore, 
\eqref{eq:linearized:operator}
 becomes
 \begin{equation*}
L_{\textrm{\rm Leb}_{{\mathbb T}^d}} q = \tfrac12 \Delta q - \textrm{\rm div} \bigl( b(\cdot,\textrm{\rm Leb}_{{\mathbb T}^d}) q \bigr),
\end{equation*}
which is local (whilst 
$L_m$ in 
\eqref{eq:linearized:operator}
is nonlocal in $q$ because of the third term therein). 
Accordingly, the term $T_3$ in the proof of 
Proposition \ref{W1 decay} 
becomes null, which makes it possible to prove 
\hergo \ at $\textrm{\rm Leb}_{{\mathbb T}^d}$.   
%  Theorem \ref{thm main result:2} applies. 
\end{proof}

\subsubsection{Gradient systems}
\label{main result extension}
We now the study the case when $b$ derives from 
a symmetric potential $W$, namely 
\begin{equation}
\label{eq:B}
b(x,m) = - \kappa \int_{\bT^d} \nabla W(x-y)  m(\ud y), \quad x \in \bT^d, \ m \in {\mathcal P}(\bT^d),
\end{equation}
for a positive constant $\kappa$ and a twice continuously differentiable potential $W : {\bT}^d \rightarrow {\mathbb R}$ that is 
 coordinate-wise even, i.e.
$W(x_{1},\cdots,-x_{i},\cdots,x_{d})=W(x_{1},\cdots,x_{i},\cdots,x_{d})$, for 
$(x_{1},\cdots,x_{d}) \in {\mathbb T}^d$.
This example has received a lot of attention in the literature. Below, we borrow several results from 
\cite{carrillo:gvalani:pavliotis:schlichting}. 

The uniform distribution $\textrm{\rm Leb}_{\bT^d}$ is an invariant measure. 
This follows from the simple fact that $b(x,\textrm{\rm Leb}_{\bT^d})=0$. The linearised operator $L_{\textrm{\rm Leb}_{{\bT}^d}}$ at $\textrm{\rm Leb}_{{\bT}^d}$ has the simple form (see \cite[Subsection 3.2]{carrillo:gvalani:pavliotis:schlichting}):
\begin{equation}
\label{eq:L:leb:H}
L_{\textrm{\rm Leb}_{\bT^d}}(\cdot) = \tfrac12 \Delta (\cdot) + 
\kappa \Delta \bigl( W \star \cdot \bigr),
\end{equation}
where $\star$ stands for the convolution product. 
Since $W$ is coordinate-wise even, $L_{\textrm{\rm Leb}_{\bT^d}}$ is symmetric. 
 
The following two results follow the analysis carried out in 
\cite[Section]{carrillo:gvalani:pavliotis:schlichting}
(with $\beta=2$, $L=1$). 
\begin{proposition}
\label{prop:cas:potential}
 Assume that 
$\textrm{\rm Leb}_{\bT^d}$ is the unique invariant measure
 and 
 that 
 $-1-2 \kappa \inf_{{\boldsymbol n} \in {\mathbb Z}^d} \widehat{W}^{\boldsymbol n} <0$, 
then
\eqref{eq:thm:main:result:2}
holds true.
\end{proposition}

For instance, if the potential $W$ satisfies 
$\widehat W^{\boldsymbol n} \geq 0$ for any 
${\boldsymbol n} \in {\mathbb Z}^d$, then it satisfies the assumptions of Proposition \ref{prop:cas:potential}, see
\cite[Section 3]{carrillo:gvalani:pavliotis:schlichting}. The latter condition is equivalent to 
\begin{equation*}
\int_{\bT^d} \int_{\bT^d} W(x-y)  \eta(\ud x)   \eta(\ud y) \geq 0,
\end{equation*}
for any finite measure $\eta$ on 
${\mathbb T}^d$, 
see again
\cite{carrillo:gvalani:pavliotis:schlichting}. 
{The above condition 
says the potential 
${\mathcal P}({\mathbb T}^d) \ni m \mapsto 
\int_{\bT^d} \int_{\bT^d} W(x-y)  m(\ud x)   m(\ud y)$ is convex in the functional sense, 
which is consistent with the recent result 
obtained in \cite{chen2023uniformintime}.} 
More generally,
letting
\begin{equation*}
W_u(x) := - \sum_{{\boldsymbol n} \in {\mathbb Z}^d} \bigl( \widehat{W}^{\boldsymbol n} \bigr)_- \cos \bigl( 2 \pi {\boldsymbol n} \cdot x), 
\quad x \in {\mathbb T}^d, 
\end{equation*}
it is shown in 
\cite{carrillo:gvalani:pavliotis:schlichting}
that 
$\textrm{\rm Leb}_{\bT^d}$  is
the unique invariant measure if
$2 \pi^2 > \kappa \sup_{x \in {\mathbb T}^d} \| \Delta W_u \|_\infty$, but this condition is not sharp. For instance, the Kuramoto model addressed in the next section 
corresponds to 
$d=1$
and
$W(x)=-\cos(2 \pi x)$ (for $x \in {\mathbb T}$). 
In this case, 
$2 \pi^2 > \kappa \sup_{x \in {\mathbb T}^d} \| \Delta W_u \|_\infty$
if and only if $\kappa <1/2$. 
However, the analysis carried out in 
\cite{giacomin:pakdaman:khashayar:pellegrin}
(see in particular 
Proposition 4.2 therein) 
shows that the proof of Proposition
\ref{prop:cas:potential}
still works when $\kappa \in [1/2,1)$.

\begin{proof}
\textit{First Step.}
The first step is to show that 
$L_{\textrm{\rm Leb}_{\bT^d}}$ satisfies 
\hergo. This follows from 
the analysis performed in 
\cite{carrillo:gvalani:pavliotis:schlichting}.
The 
non-trivial eigenfunctions of the operator $L_{\textrm{\rm Leb}_{{\mathbb T}^d}}$
are the non-trivial Fourier functions. They form an orthonormal basis of the space $\{ f \in {\mathbb L}^2(\bT^d) : \langle f,\one\rangle=0\}$ 
and, under the assumption of the statement,
 all the corresponding eigenvalues $(\lambda_{{\bm k}})_{{\bm k} \in {\mathbb Z}^d \setminus \{0\}}$ are strictly negative, with a non-zero spectral gap, i.e.
$\sup_{{\bm k} \in {\mathbb Z}^d \setminus \{0\}} \lambda_{\bm k} <0$.

We then adapt 
Lemma 
 \ref{conjecture backward PDE }
 to our setting in order 
 to prove 
 \eqref{erg:hyp:3}.
 To do so, we use the existence of a spectral gap. It says that, for any  smooth $u : {\bT}^d \rightarrow {\mathbb R}$ with $\langle u, \one \rangle=0$,
$- \langle L_{\textrm{\rm Leb}_{\bT^d}} u,u \rangle \geq \lambda
\langle u,u \rangle.$
Now, for
 a smooth function $\xi$ on $\bT^d$ and for $t>0$,  
 we consider 
 {the solution} 
 $(w(s))_{0 \leq s \leq t}$ to the equation 
\begin{equation}
\label{eq:L:leb:w}
\partial_s w(s,\cdot) + L_{\textrm{\rm Leb}_{{\bT}^d}}w(s,\cdot) = 0, \quad s \in [0,t] ; \qquad 
             w(0) = \xi.
\end{equation}
Then $\langle w,\one \rangle$ is constant and, by the {existence of a} spectral gap, 
\begin{equation}
\label{eq:expo:decay:w:L2}
\bigl\| w(s) - \langle w,\one\rangle \bigr\|_{2} \leq \| \xi \|_{2} e^{-\lambda (t-s)}, \quad s \in [0,t]. 
\end{equation} 
{Noticing} 
that $\int_{\bT^d} \Delta W(x-y) \ud y=0$, 
we deduce that 
%\textcolor{red}{on peut quantifier : j'ai l'impression qu'il faut d\'eriver 4 fois, donc \c{c}a ferait $C^6$ pour $W$. On pourrait faire mieux, mais est ce n\'ecessaire ?}
\begin{equation}
\label{eq:L:leb:derivative}
\sup_{x \in \bT^d} \bigl\vert  \Delta (W \star w)(s,x) \bigr\vert
=
\sup_{x \in \bT^d} \bigl\vert  \Delta  {\bigl(W \star \bigl[ w - \langle w,\one \rangle \bigr] \bigr)}(s,x) \bigr\vert
 \leq {C} \| \xi \|_{\infty} e^{-\lambda(t-s)}.
\end{equation}
%and, in turn, by the maximum principle, 
%\begin{equation*} 
%\sup_{(s,x) \in [0,t] \times \bT^d} \vert w(s,x) \vert \leq C \| \xi \|_{\infty}. 
%\end{equation*}
%We now split $w$ into $w:=w_{1}+w_{2}$, with 
%\begin{equation*}
%\begin{cases}
%\partial_s w_{1}(s,\cdot) + \tfrac12 \Delta w_{1}(s,\cdot) = 0, \quad \quad s \in [0,t],\\
%             w_{1}(0) = \xi,
%\end{cases}
%\end{equation*}
%and
%\begin{equation*}
%\begin{cases}
%\partial_s w_{2}(s,\cdot) + \tfrac12 \Delta w_{2}(s,\cdot)
%+  \tfrac12 \bigl( \Delta W \star w\bigr)(s,\cdot) = 0, \quad \quad s \in [0,t],\\
%             w_{2}(0) = 0.
%\end{cases}
%\end{equation*}          
%Obviously, $w_{1}$ satisfies all the conclusions of    
% Theorem 
% \ref{conjecture backward PDE }. 
%  In particular, 
%$w_{1}$ satisfies
%\eqref{W12 estimate est 1}
%and thus $w_{2}$ also satisfies
%\eqref{eq:expo:decay:w:L2}.
% This makes it possible to treat $w_{2}$. It
%  suffices to 
%%  differentiate four times the PDE (which is possible thanks to the smoothness of $W$) and next to use 
%% \eqref{eq:L:leb:derivative} to recover the conclusion of
%% Theorem 
%% \ref{conjecture backward PDE }, 
%represent $w_{2}$ in terms of the underlying (Gaussian) transition kernel on an interval of length $\tau$, by
%following the arguments of \eqref{eq:decomposition:transition:kernel}. 
%The new point here comes from 
%the presence of a source term, but the latter satisfies \eqref{eq:L:leb:derivative}. 
By
 {recalling the shape of
$L_{\textrm{\rm Leb}_{{\bT}^d}}$ in 
\eqref{eq:L:leb:H} 
and} regarding the equation 
\eqref{eq:L:leb:w} as the heat equation plus a source term that decays exponentially fast, it is quite standard to show that 
$w$ satisfies the conclusion of
 Lemma
 \ref{conjecture backward PDE }.  It remains to adapt the proof of
 Proposition \ref{W1 decay}. 
 We start from \eqref{eq q}, {when driven by $L_{\textrm{\rm Leb}_{{\mathbb T}^d}}$}. 
 Instead of considering 
 $w$ as the solution of  
 \eqref{cauchy w}, we choose $w$ as the solution of 
 \eqref{eq:L:leb:w}. This leads to a new expansion in 
     \eqref{duality W1 infty}
with $T_{3}=0$ (because we included the term 
$\kappa \Delta ( W \star w )$ in 
\eqref{eq:L:leb:w}). It then suffices to let appear the exponential decay of 
$w$ and $\nabla_x w$ in the estimates of $T_1$ and $T_2$. 
\vskip 4pt

\textit{Second Step.}
The next step is to prove that 
$L_{\textrm{\rm Leb}_{\bT^d}}$ is 
uniformly attracting (as in the statement of Theorem \ref{thm main result:2}). This follows from the
variational structure of the problem and from the additional assumption that 
$L_{\textrm{\rm Leb}_{\bT^d}}$ is the unique invariant measure. 
Indeed, the McKean-Vlasov equation
 may be regarded as a gradient flow, with potential
\begin{equation}
\label{eq:potential:mathcalF}
{\mathcal F}(m) =  \frac12 \int_{{\mathbb T}^d}
\frac{dm}{dx}(x) \ln \bigl( \frac{dm}{dx}(x) \bigr) \ud x + \frac{\kappa}2  
\int_{{\mathbb T}^d} \int_{{\mathbb T}^d} W(x-y) m(x)m(y)  \ud x \ud y, 
\end{equation}
if $m$ is absolutely continuous with respect to 
$L_{\textrm{\rm Leb}_{\bT^d}}$,
and 
${\mathcal F}(m)=+\infty$ if it is not absolutely continuous. 

Recall that the solution of 
\eqref{eq: forward eqn } has a density $p(t,x;\mu)$ in time $t>0$.
Since $\nabla W$ is Lipschitz continuous, 
$p(t,x\, ;\mu)$ is continuously differentiable in $x$ and belongs to a compact subset 
${\mathcal K} \subset {\mathcal C}^1({\mathbb T}^d)$, independent of $\mu$, when $t \geq 1$. 
 In particular, 
$ {\mathcal F}(m(t\, ; \mu)) < \infty$
for any $ t >0$. 
Moreover, by the gradient flow structure, 
\begin{equation}
\label{eq:gradient:flow}
\begin{split}
\forall 0<t_{1}<t_2, \quad &{\mathcal F} \bigl( m(t_{2} ;\mu) \bigr) 
- {\mathcal F} \bigl( m(t_{1} ;\mu) \bigr) 
\\
&= - \int_{t_{1}}^{t_{2}} 
\int_{{\bT}^d}
\Bigl\vert 
\kappa \int_{{\bT}^d} \nabla W \bigl( x - x' \bigr) m(r  ;\mu)(\ud x')
+ \frac12 \frac{\nabla_x p(r,x ;\mu)}{p(r,x  ;\mu)} \Bigr\vert^2 m(r \, ; \mu)(\ud x) \ud r. 
\end{split}
\end{equation}
For $\delta >0$, let ${\mathcal K}_{\delta} = \{ p \in {\mathcal K} : \| p- \one \|_{{\mathcal C}^1} \geq \delta \}$. Clearly, ${\mathcal K}_{\delta}$ is a compact subset of 
${\mathcal C}^1({\mathbb T}^d)$. 
More importantly, the quantity 
\begin{equation*}
\int_{\bT}
\Bigl\vert 
 \kappa \int_{\bT} \nabla W \bigl( x - x' \bigr) p(x') \ud x' 
+\frac12 \frac{\nabla p}{p}(x) \Bigr\vert^2 p(x) \ud x
\end{equation*}
cannot vanish on ${\mathcal K}_{\delta}$, as otherwise there would exist
a non-trivial stationary solution of 
\eqref{eq: forward eqn }. 
By compactness of ${\mathcal K}_{\delta}$ and by a straightforward continuity argument, we deduce that there exists 
a constant $c>0$ such that, for all $p \in {\mathcal K}_{\delta}$, 
the above quantity is greater than $c$. 
In particular, if we take $t_{1}=1$ 
and $t_{2} = \inf\{ t \geq 1 : p(t,\mu) \not \in {\mathcal K}_{\delta} \}$ in 
\eqref{eq:gradient:flow}, we obtain that
\begin{equation}
\label{eq:gradient:flow:222}
{\mathcal F}\bigl( p(t_{2},\mu) \bigr)  \leq   {\mathcal F}\bigl( p(t_{1},\mu) \bigr)  - c (t_{2} - t_{1})
\leq C - c(t_2-t_1). 
\end{equation}
By the same compactness argument, ${\mathcal F}$ must be lower-bounded on ${\mathcal K}_{\delta}$, 
from which 
we deduce that, for any $\mu$, $t_{2}$ is finite and that, most of all, 
there exists $T<\infty$, independent of $\mu$, such that $t_{2} \leq T$. 
\end{proof}

We conclude with the following statement, which shows that the results proven in the metastable regime cover 
some of the examples 
considered in 
\cite{carrillo:gvalani:pavliotis:schlichting}:

\begin{proposition}
\label{prop:cas:potential:2}
We can find potentials $W$ 
for which there are several invariant measures
 and 
for which  $-1-2 \kappa \inf_{{\boldsymbol n} \in {\mathbb Z}^d} \widehat{W}^{\boldsymbol n} <0$. 
In this case, 
$\textrm{\rm Leb}_{\bT^d}$
 satisfies the second part of
Theorem 
\ref{thm main result:2}  (in the metastable regime) and 
Theorem 
\ref{prop:3:18}.
\end{proposition}

\begin{proof}
This follows from (3.5),
Definition 5.1 {and Proposition 5.8} in \cite{carrillo:gvalani:pavliotis:schlichting}. With the same notations as therein, 
it suffices to have $\kappa_c<\kappa_\sharp$ and to choose $\kappa \in (\kappa_c,\kappa_\sharp)$. 
Examples 
 {of potentials $W$ for which 
$\kappa \in (\kappa_c,\kappa_\sharp)$}
are provided by 
Corollaries 5.13 and 5.14 
and Proposition 6.2
in \cite{carrillo:gvalani:pavliotis:schlichting}.
\end{proof}

\end{subsection}

 \begin{section}{Model without a unique invariant measure}
 \label{se:kuramoto}

The purpose of this section is to address the Kuramoto model.
%, which has received a lot of attention in statistical physics and in neurosciences. 
It is in fact a particular case of example
\eqref{eq:B}, with $d=1$ and $W(x)=-\cos(2 \pi x)$, i.e. 
%the following interaction term:
\begin{equation}
\label{eq:b:kuramoto}
b(y,\mu) = - 2 \pi \kappa \int_{{\mathbb T}} \sin(2\pi (y-x))  \mu(\ud x).
\end{equation} 
%Here, the normalisation constant $2 \pi$ in front of $\kappa$ comes from the fact that
%the standard form of the Kuramoto model is usually defined on ${\mathbb R}/(2 \pi {\mathbb Z})$
%(whilst our state space is ${\mathbb R}/{\mathbb Z}$). By an obvious change of variable, we can easily 
%pass from one model to the other.   
%
%
%has a slightly different form:
%\begin{equation*}
%\ud \tilde Y_{t}^{i,N} = \tilde \eta_{i} - \frac{\kappa}{N} \sum_{j=1}^N   \sin\bigl( \tilde Y_{t}^{i,N} - \tilde Y_{t}^{j,N} \bigr) 
%\ud t + \ud \tilde W_{t}^i,
%\end{equation*}
%for a new collection of initial conditions $(\tilde \eta_{i})_{1 \leq i \leq N}$
%and a new collection of Brownian motions $((\tilde W_{t}^i)_{t \geq 0})_{1 \leq i \leq N}$. 
%Obviously, the connection between the above equation and 
%\eqref{eq particles} is obtained by observing that, for a convenient choice of the initial condition, 
%$((2 \pi)^{-1} \tilde Y_{4\pi^2 t})_{t \geq 0}$ solves 
%\eqref{eq particles}
%with $b$ as in \eqref{eq:b:kuramoto}. 
%Having the right normalisation in \eqref{eq:b:kuramoto} for the constant $\kappa$ is important, since 
Interestingly, Kuramoto's model exhibits a phase transition when 
$\kappa=1$ (see for instance \cite{bertini:giacomin:pakdaman}). When $\kappa \leq 1$, the Fokker-Planck equation 
\eqref{eq: forward eqn }
has a unique invariant probability measure, which is given by 
${\textrm{\rm Leb}}_{{\mathbb T}^d}$ (and 
the method of proof of
Proposition 
\ref{prop:cas:potential} covers the case $\kappa <1$). 
When $\kappa>1$, it has an infinite number of invariant measures,
namely ${\textrm{\rm Leb}}_{{\mathbb T}^d}$ and 
a collection of non-trivial ones, all of them being obtained by rotation of a  non-constant density
$p_\infty$ (i.e., $p_{\infty,\psi}:=p_{\infty}(\cdot - \psi)$ is an invariant measure for any 
$\psi \in {\mathbb T}$). 

In the rest of the section, we focus on the regime $\kappa>1$. In 
that case, 
propagation of chaos cannot hold at time of order $t=N$   
(see \cite{bertini:giacomin:poquet} together with
\cite{MR1970276}
for a similar phenomenon in the Euclidean setting). 
In fact, the result of  
\cite{bertini:giacomin:poquet}
has just been revisited by 
\cite{coppini2019long} (in the even more complex case when the interactions are subjected to a
 non-complete graph). 
The main idea therein is to show that, even though the particle system may strongly deviate from an invariant profile in time of order 
$N$, it stays close to the whole collection\footnote{Most of the time, we shall identify the densities that belong to ${\mathcal I}$ together with the probability measures that are driven by those densities.} ${\mathcal I}:=\{p_{\infty,\psi}, \psi \in {\mathbb T}\}$ for a time period that is nearly exponential in 
$N$. However, the rate of convergence is not addressed in \cite{coppini2019long}. Using the techniques developed in the previous section, we show here that, if the initial condition is not ${\textrm{\rm Leb}}_{{\mathbb T}^d}$, we can retain a uniform
weak error of size $1/N$ provided that we force the test functional  
$\Phi$ in 
\eqref{main result intro formula}
to be invariant by rotation (see Theorem 
\ref{main:thm:kuramoto}
below).  
We stress that the latter requirement on $\Phi$ is fully consistent 
with the point of view used in 
\cite{coppini2019long}. 
Our improvement is thus twofold: Not only do we get an explicit rate for the weak error, but we also manage to get a bound that 
holds uniformly in time (not only up until times that are exponential in $N$). Notice however that, in 
\cite{coppini2019long}, the convergence is understood for the sup norm over the trajectory (which is stronger).
%, whilst we focus over here on the marginal distributions. 

Things become more subtle whenever the system is initialized from the invariant measure ${\textrm{\rm Leb}}_{{\mathbb T}^d}$ since 
% By propagation of chaos in a finite time, 
%the empirical measure of the particle system then stays close to ${\textrm{\rm Leb}}_{{\mathbb T}^d}$
%over any finite interval, with the related weak error still being of the order $1/N$ for functionals $\Phi$ as in Theorem \ref{thm main result}. 
%However, this error deteriorates with time. To wit (and this is an important ingredient in our proof), we 
%show in 
Lemma \ref{lem:exit:time:cos}
below shows that the empirical measure leaves, with a large probability, any sufficiently small neighbourhood 
of the uniform distribution in a time that is at most polynomial in $N$. As a result, another study would be necessary to handle this case specifically.
%Once again, our main motivation here is mostly to prove that our approach is robust enough to accommodate cases when 
%invariant measures are not unique, whence our choice to focus on the Kuramoto model with $\kappa>1$.  
\vskip 4pt

Below, we use freely the same general notations as in the previous section. 
In particular,
 $m(\cdot \, ;\mu)$ denotes the solution to \eqref{eq: forward eqn }
with $\mu$ as initial condition and with $b$ as in 
\eqref{eq:b:kuramoto}. Here, it takes the form
\begin{equation}
\label{eq:se:5:FKP} 
     \partial_t m(t \, ; \mu) - \tfrac12 \partial^2_{xx} m(t \, ; \mu) - \partial_{x} \Bigl( m(t \, ; \mu) \bigl( J \star m(t\, ; \mu) \bigr)  \Bigr) =0,
\end{equation}
where $\star $ denotes the standard convolution product  and 
$J( \cdot) = 2 \pi \kappa \sin(2 \pi \cdot)$. 

Of course, 
$m(\cdot \, ;\mu)$ is absolutely continuous in positive time, we therefore let $p(t \, ;\mu)= (\ud /\ud x)m(t ; \mu) : x \mapsto p(t \, ; \mu)$ (which we also write $p(t ,x\, ;\mu)$) be the density of 
$m(t\, ;\mu)$, for $t >0$. The function $p(t\, ;\mu)(\cdot)$ is a smooth function of $x$, uniformly in $t \in [t_{0},\infty)$ for any $t_{0}>0$.  
%Similarly, 
%we still denote by 
%$\mu^N$ the flow of empirical distributions defined in 
%\eqref{eq particles}. 
Also, we recall that each $p_{\infty,\psi}$ is (strictly) positive. 
%{Lastly, except when it is initialised with another condition (in which case it is explicitly stated), 
%the (common) law of 
%the I.I.D. initial positions 
%$\eta_{1},\cdots,\eta_{N}$
%in 
%\eqref{eq particles}
%is denoted by $\mu_{\textrm{\rm init}}$.} 

\subsection{Main result} 
We focus on initial conditions that are away from ${\textrm{\rm Leb}}_{{\mathbb T}^d}$, namely we let, for any 
$\eta \in (0,1)$,
\begin{equation*}
{\mathcal Q}_{\eta} = \bigl\{ \mu \in {\mathcal P}({\mathbb T}) : \vert \mu^1 \vert \geq \eta \bigr\}, 
\quad \mu^1 :=  \int_{{\mathbb T}}
\exp( - \i  2 \pi \theta) \mu(\ud \theta). 
\end{equation*}
The next result (proven at the end of the subsection) shows that 
${\mathcal Q}_{\eta}$ is attracted   by ${\mathcal I}$.
\begin{proposition}
\label{prop:5:10}
For any $\eta \in (0,1)$,
and any integer $k \geq 1$, 
there exist an exponent $\beta >0$ and a constant $C$, both depending on $\kappa$, $\eta$ and $k$, such that  
\begin{equation*}
\forall t \geq 1, \quad \sup_{\mu \in {\mathcal Q}_{\eta}}
\inf_{\psi \in {\mathbb T}} \, 
\bigl\|  p(t \, ; \mu) - p_{\infty,\psi} \bigr\|_{k,\infty} \leq C \exp(-\beta t).
\end{equation*}
\end{proposition}
%Of course, the notation $d m(t,\mu)/dx$ is used to denote the density of the measure $m(t,\mu)$ with respect to the Lebesgue measure. The latter exists in positive time (if needed, the result may refer to the proof of Proposition 
%\ref{prop:appendix:1} below). 
Proposition \ref{prop:5:10} plays a key role in our analysis. 
Notice that the constraint $t \geq 1$ may be easily changed into $t \geq t_{0}$ for any $t_{0}>0$, in which case the 
constant $C$ may depend on $t_{0}$ as well. 
%This does not make any conceptual difference, since the main interest of the result is about the long time
%behaviour of $p(\cdot\, ;\mu)$.
  
In order to state our main result precisely, we need the following 
additional definition.
\begin{definition}
\label{def:rotation:invariant:function}
We say that a function $\Phi : {\mathcal P}(\bT) \rightarrow {\mathbb R}$
is rotation invariant if, for any 
$\mu \in {\mathcal P}(\bT)$ and 
$\psi \in \bT$, 
$\Phi ( 
 \mu \circ \tau_{\psi}^{-1}) = \Phi ( \mu )$,
where 
$ \mu \circ \tau_{\psi}^{-1}$ is the image of $\mu$ by 
the translation $\tau_{\psi} : \bT \ni x \mapsto x + \psi$. 
\end{definition}

We now have all the ingredients to formulate the main theorem of this section.
\begin{theorem}
\label{main:thm:kuramoto}
Assume
that 
$\Phi$ is rotation invariant and 
satisfies 
\emph{\hintphi{\gamma}{2}}
for some $\gamma \in (0,1]$. Then, 
for any $\eta \in (0,1)$, 
there exists a constant $C>0$ such that, 
for any $\mu_{\textrm{\rm init}} \in {\mathcal Q}_{\eta}$ and any $N \geq 1$, 
$$ \sup_{t \geq 0} \Big|  \bE\bigl[ \Phi(\mu^{N}_t)\bigr] - \Phi\bigl(m(t\, ;{\mu_{\text{\emph{init}}}})\bigr) \Big| \leq \frac{C}{N}.$$
%The constant $C$ only depends on $\Phi$ through the quantities  
%$\| \Phi \|_{\infty}$, 
%$\sup_{m \in {\mathcal P}({\mathbb T})}
%\sup_{\| q \|_{-2,\infty} \leq 1}
%[\delta \Phi/\delta m](m,q)$ and 
%$\sup_{m \in {\mathcal P}({\mathbb T})}
%\sup_{\| q_{1} \|_{-1,\infty},\| q_{1} \|_{-1,\infty} \leq 1}
%[\delta^2 \Phi/\delta m^2](m,q_{1},q_{2})$
\end{theorem}
%To the best of our knowledge, this is the first convergence result for Kuramoto's model 
%that holds true globally in time. For sure, we here succeeded to do so by weakening the underlying notion of convergence. 

Very much in the spirit of 
Proposition \ref{prop:4:phi:norm:-d}, 
the most useful example for $\Phi$ is

\begin{proposition}
\label{prop:choose:Phi:kuramoto}
Let $p_{\infty,+}$ denote the unique element of ${\mathcal I}$ whose first Fourier coefficient $p_{\infty,+}^1$ is a positive real and let $\mu_{\infty,+}:=p_{\infty,+} \cdot \textrm{\rm Leb}_{\mathbb T}$. 
For $\varepsilon \in (0,1]$ and for a smooth non-decreasing cut-off function $\varphi : [0,1] \rightarrow [0,1]$ that is equal to $0$ on 
$[0,\delta/2]$ and  $1$ on $[\delta,1]$, for some $\delta \in (0,1)$, let $\Phi$ be defined by
\begin{equation*}
\Phi(\mu) := 
\varphi \bigl( \vert \mu^1 \vert \bigr) 
\Bigl\| \mu \circ \tau_{\frac{\mu^1}{\vert \mu^1\vert}}^{-1} - \mu_{\infty,+} \Bigr\|_{-(1+\varepsilon)/2,2}^2 + 1 - \varphi \bigl( \vert \mu^1 \vert \bigr), 
\end{equation*}
where $\| \cdot \|_{-(1+\varepsilon)/2,2}$ is defined  
as in Proposition 
\ref{prop:4:phi:norm:-d}
with $d=1$ therein. 
Then $\Phi$ satisfies the assumption of 
Theorem 
\ref{main:thm:kuramoto}. 
\end{proposition}

We add a few words about the meaning of $p_{\infty,+}$. The elements of ${\mathcal I}$ are obtained by rotation. Therefore, the collection of their first Fourier coefficients coincides with a circle, whose radius is non-zero (see \cite[Subsection 1.2]{bertini:giacomin:poquet}). Consequently, we may indeed choose $p_{\infty,+} \in {\mathcal I}$ such that 
$p_{\infty,+}^1>0$. 
Accordingly, we also notice that, in the notation $\mu \circ \tau_{\mu^1/\vert \mu^1 \vert}^{-1}$, 
we identify $ \mu^1/\vert \mu^1 \vert$ with the unique element $\psi \in {\mathbb T}$ such that 
$\mu^1= \vert \mu^1 \vert \exp(\i 2 \pi \psi)$. In particular, 
the first Fourier coefficient $(\mu \circ \tau_{\mu^1/\vert \mu^1\vert}^{-1})^1$ of 
$\mu \circ \tau_{\mu^1/\vert \mu^1\vert }^{-1}$, which is 
equal to $\exp(- \i 2 \pi \psi) \mu^1= \vert \mu^1 \vert$, is 
positive (when $\mu^1 \not = 0$), which explains why we compare $\mu \circ \tau_{\mu^1/\vert \mu^1\vert}^{-1}$ with $\mu_{\infty,+}$. We refer to 
\cite[Lemma 2.8]{MR3689966} for another 
projection onto ${\mathcal I}$.

Moreover, we notice that this is precisely the role of the cut-off function $\varphi$ in the definition of $\Phi$  to remove the measures $\mu$ for which $\mu^1 =0$. In fact, the cut-off function has no real consequence on our result. 
Actually, what matters is that $\Phi(\mu)$ is small if and only if $\mu$ is close enough to ${\mathcal I}$. Precisely, we can find a constant $c>0$ such that 
\begin{equation}
\label{eq:lower:bound:Phi}
\Phi(\mu) \geq c \inf_{\psi \in \bT} \bigl\| \mu - \mu_{\psi} \|_{-(1+\varepsilon)/2,2}^2. 
\end{equation}
The proof of the above lower bound is quite easy. It suffices to prove it for $\vert \mu^1\vert$ bounded away from zero. To do so, we may observe that 
$\|  \mu \circ \tau_{\mu^1/\vert \mu^1\vert}^{-1} - \mu_{{\infty,+}}\|_{-(1+\varepsilon)/2,2}^2$
is lower-bounded by 
$\inf_{\psi \in \bT} \bigl\| \mu - \mu_{\psi} \|_{-(1+\varepsilon)/2,2}^2$.
The following is a straightforward corollary.
\begin{corollary}
\label{corol:H-3/2}
For any $\eta \in (0,1)$ and $\varepsilon \in (0,1]$, 
there exist two (positive) constants $c$ and $C$ such that, for $\mu_{\textrm{\rm init}}$ in ${\mathcal Q}_{\eta}$, 
\begin{equation*}
\forall t \geq 0, \quad {\mathbb E} \Bigl[\inf_{\psi \in \bT} \bigl\| \mu_{t}^N - \mu_{\psi} \|_{-(1+\varepsilon)/2,2}^2 \Bigr] \leq \frac{C}{N} + C \exp(-ct).
\end{equation*}
\end{corollary}
\begin{proof}[Proof of Corollary \ref{corol:H-3/2}]
We take for granted the statements of Theorem 
\ref{main:thm:kuramoto} and 
Propositions 
\ref{prop:5:10} and \ref{prop:choose:Phi:kuramoto}. 
{By 
\eqref{eq:lower:bound:Phi} and with $\Phi$ as in Proposition \ref{prop:choose:Phi:kuramoto}}, it suffices to prove that 
$\Phi(m(t\, ;\mu))$ decays exponentially fast for any $\mu \in {\mathcal Q}_{\eta}$. 
By Proposition 
\ref{prop:5:10}, there exists $\psi \in \bT$ such that 
$\| p(t\, ; \mu) - p_{\infty,\psi}\|_{2} \leq C \exp (-ct)$, for $t \geq 1$. 
Hence, it is enough to show that 
\begin{equation*}
\Bigl\| p_{\infty,\psi} - p_{\infty,+} \circ \tau_{\frac{(p(t,\mu))^1}{\vert (p(t,\mu))^1 \vert}}
\Bigr\|_{2} \leq C \exp(- ct),
\end{equation*}
or equivalently that 
\begin{equation*}
\Bigl\| p_{\infty,+} \circ \tau_{\frac{(p_{\infty,\psi})^1}{\vert (p_{\infty,\psi})^1 \vert}} - 
p_{\infty,+}
\circ \tau_{\frac{(p(t,\mu))^1}{\vert (p(t,\mu))^1} \vert}
\Bigr\|_{2} \leq C \exp(- ct),
\end{equation*}
at least for $t$ sufficiently large. 
Since $p^1_{\infty,+}>0$, we know that
$\vert (p(t,\mu))^1 \vert$ is lower-bounded by a positive constant, uniformly over all $t$ greater than some $t_{0}>0$. 
In turn, we have 
\begin{equation*}
\biggl\vert 
\frac{(p(t,\mu))^1}{ \vert (p(t,\mu))^1 \vert} 
- 
\frac{(p_{\infty,\psi})^1}{ \vert (p_{\infty,\psi})^1 \vert} 
\biggr\vert \leq C \exp(-ct),
\end{equation*}
for some possibly new value of $C$, which gives the expected result.  
\end{proof}

\subsubsection*{Proofs of the auxiliary Propositions 
\ref{prop:5:10}
and
\ref{prop:choose:Phi:kuramoto}}
The reader may skip the proofs of
Propositions 
\ref{prop:5:10}
and
\ref{prop:choose:Phi:kuramoto}  
 ahead on a first reading. 

\begin{proof}[Proof of Proposition \ref{prop:choose:Phi:kuramoto}.]
%In the framework of Section \ref{se:kuramoto}, things are not so obvious because the set of invariant measures 
%does not reduce a singleton; the synchronized ones are given by $p_{\psi}$, for $\psi$ running in $\bT$. 
%Accordingly, we may think of redefining the functional $\Phi$ and of working instead with 
%\begin{equation*} 
%\Phi(\mu) = \inf_{\psi \in \bT} \bigl\| \mu - \mu_{\psi} \|_{-k}^2,
%\end{equation*}
%where $\mu_{\psi}$ is implicitly understood as the probability measure having $p_{\psi}$ as density. 
%The difficulty with this formulation is that $\Phi$ cannot be proven to be smooth enough. So, another approach is to 
%define $\Phi$ differently. 
%For $\mu$ such that $\mu^1 \not =0$, we can indeed write 
%\begin{equation*}
%\mu^1 = \varrho^1 \exp(i 2 \pi \psi^1),
%\end{equation*}
%with $\rho^1 >0$ and $\psi^1 \in \bT$. 
%Then, the first Fourier coefficient of $\tau_{-\psi^1} \mu^1$
%is $\varrho^1$ and is hence a positive real: 
%we then observe that, within the family $\{ \mu_{\psi}, \ \psi \in \bT\}$, there is a unique element 
%$\mu_{0}$ such that $\mu_{0}^1 \in ]0,+\infty[$, which hence prompts us to compare
%$\tau_{-\psi^1} \mu^1$ with $\mu_{0}$. 
%In order to ease the notations, we will write $\tau_{z}$ for the translation $\tau_{\psi}$ whenever 
%$z$ is the complex number $z= \exp(i 2 \pi \psi)$. 
%We then consider the following functional 
%\begin{equation*}
%\Phi(\mu) = \Bigl\| \tau_{-\frac{\vert \mu^1\vert}{\mu^1}} \mu - \mu_{0} \Bigr\|_{-k}^2, 
%\end{equation*}
%which is well-defined if $\mu^1 \not =0$. 
We use a simplified notation $\tilde \mu$ for 
$\mu \circ \tau_{{\mu^1}/{\vert \mu^1\vert}}^{-1}$. 
The tricky point is then to study the smoothness of the mapping 
\begin{equation*}
\tilde{\Phi}(\mu) = \bigl\| \tilde \mu - \nu_{0} \bigr\|_{-(1+\varepsilon)/2,2}^2,
\end{equation*}
at least when $\mu^1$ stays away from $0$, for  $\nu_{0}$  {fixed}. 
We observe that we have the following: 
\begin{equation*}
\tilde\mu^n
= \int_{\bT} e^{-\i 2\pi n x} d \tilde \mu(x)
= \Bigl( \frac{\vert \mu^1 \vert}{\mu^1} \Bigr)^n
\int_{\bT} e^{-\i 2\pi n x} d \mu(x) =
\Bigl( \frac{\vert \mu^1 \vert}{\mu^1} \Bigr)^n  \mu^n. 
\end{equation*}
Therefore, following 
\eqref{eq:Fourier:Phi}, 
we get 
\begin{equation*}
\Phi(\mu) = \sum_{n \in {\mathbb N}} \frac1{(1+n^{2})^{(1+\varepsilon)/2}} 
\biggl[ \mu^n \bar \mu^n + \nu_{0}^n \bar \nu_{0}^n - \Bigl( \frac{\vert \mu^1 \vert}{\mu^1} \Bigr)^n \mu^n \bar \nu_{0}^n - \Bigl( \frac{\vert \mu^1 \vert}{\bar \mu^1} \Bigr)^n \nu_{0}^n \bar \mu^{n}
\biggr],
\end{equation*}
which we can then rewrite in the form 
\begin{equation*}
\Phi(\mu) = \sum_{n \in {\mathbb N}}   \frac1{(1+n^{2})^{(1+\varepsilon)/2}} 
\biggl[ \mu^n \bar \mu^n + \nu_{0}^n \bar \nu_{0}^n - \bigl( \Psi(\mu) \bigr)^n \mu^n \bar \nu_{0}^n - \bigl( 
\overline{\Psi}(\mu)
 \bigr)^n \nu_{0}^n \bar \mu^{n}
\biggr],
\end{equation*}
with $\Psi(\mu) = \vert \mu^1 \vert / \mu^1$. 
On the open subset $\{\mu^1 \not =0\}$ (for the 
${\mathcal W}^1$ topology), the function $\Psi$ is infinitely differentiable with respect to 
$\mu$. The power $n$ creates additional factors that are handled in the same way as in the proof of Proposition 
\ref{prop:4:phi:norm:-d}. As a result, we get that, for the same values of $k$, 
$\Phi$ satisfies the same properties as in Proposition \ref{prop:4:phi:norm:-d}, but on any domain where 
$\mu^1$ stays away from $0$. 
\end{proof}

We close this subsection with the following proof.

\begin{proof}[Proof of Proposition \ref{prop:5:10}.]
The proof is achieved in two steps, which are mostly adapted from 
\cite{bertini:giacomin:pakdaman} and \cite{giacomin:pakdaman:khashayar:pellegrin}.
\vskip 4pt

\textit{First step.}
The first step is to show that
\begin{equation*}
\lim_{t \rightarrow \infty}
\sup_{\mu \in {\mathcal Q}_{\eta}}
\inf_{\psi \in {\mathbb T}}
\
\bigl\| p(t\, ;\mu) - p_{\infty,\psi} \bigr\|_{2}  =0. 
\end{equation*}
In order to do so, we follow the proof of Proposition 1.7 in \cite{bertini:giacomin:pakdaman} and of 
Proposition \ref{prop:cas:potential}. 
We recall indeed that the McKean-Vlasov equation
\eqref{eq:se:5:FKP}
 may be regarded as a gradient flow, with potential
 ${\mathcal F}$ in 
\eqref{eq:potential:mathcalF} (with $W(x)=-\cos(2 \pi x)$). 
%Recall that the solution of 
%\eqref{eq:se:5:FKP} has a smooth density in positive time. In particular, %we can write, with a slight abuse of notation:
%\begin{equation*}
%\forall t >0, \quad {\mathcal F}\bigl(p(t\, ; \mu)\bigr) < \infty.
%\end{equation*}
%Moreover, the gradient flow structure says that, for any $0<t_{1}<t_2$, 
%\begin{equation}
%\label{eq:gradient:flow}
%\begin{split}
%&{\mathcal F} \bigl( p(t_{2} ;\mu) \bigr) 
%- {\mathcal F} \bigl( p(t_{1} ;\mu) \bigr) 
%\\
%&= - \int_{t_{1}}^{t_{2}} 
%\int_{\bT}
%\Bigl\vert 
%2 \pi \kappa \int_{\bT} \sin\bigl( 2\pi (x - x') \bigr) m(r  ;\mu)(\ud x')
%+ \frac12 \frac{p'(r,x ;\mu)}{p(r,x  ;\mu)} \Bigr\vert^2 m(r ; \mu)(\ud x) \ud r. 
%\end{split}
%\end{equation}
%
Then, we know from Proposition 4.4 in \cite{giacomin:pakdaman:khashayar:pellegrin} that, after some time 
$t_{\eta,\epsilon}$ (independent of the choice of $\mu \in {\mathcal Q}_{\eta}$), $p(t \,  ;\mu) \not \in B_{L^2(\bT)}(\one,\epsilon)$ for a given $\epsilon>0$ (here,  $B_{L^2(\bT)}(\one,\epsilon)$ is the $L^2(\bT)$-ball of center $\one$ and radius $\epsilon$). 
Also, after the same time $t_{\eta,\epsilon}$, we know that 
$p(t \, ; \mu)$ belongs to a compact subset ${\mathcal K}$ of ${\mathcal C}^1({\mathbb T},(0,+\infty))$, which may be chosen independently of $\mu$.
% (the lower bound on $p(t \, ; \mu)$  is a mere consequence of Harnack's inequality, which can be found in, e.g., Proposition 7.37 of
%\cite{Lieberman1996}). 
%It is clear that the function ${\mathcal F}$ is continuously differentiable on ${\mathcal K}$. 

For $\delta >0$ as above, let ${\mathcal K}_{\delta} = \{ p \in {\mathcal K} : \| p- \one \|_{{\mathcal C}^1} \geq \delta, \,
\inf_{\psi \in \bT} \|p - p_{\infty,\psi} \|_{{\mathcal C}^1} \geq \delta \}$, which is a compact subset of 
${\mathcal C}^1$. Next, we follow
\eqref{eq:gradient:flow}. 
If 
%
%{eq:gradient:flow:222}
%More importantly, the quantity 
%\begin{equation*}
%\int_{\bT}
%\Bigl\vert 
%2 \pi \kappa \int_{\bT} \sin\bigl( 2\pi (x - x') \bigr) p(x') \ud x' 
%+\frac12 \frac{p'}{p}(x) \Bigr\vert^2 p(x) \ud x
%\end{equation*}
%cannot vanish on ${\mathcal K}_{\delta}$, as otherwise there would exist $p \in {\mathcal K}_{\delta}$ such that 
%(recalling that $p \in {\mathcal K} \Rightarrow \inf_{\bT} p >0$)
%\begin{equation*}
%2 \pi \kappa \int_{\bT} \sin\bigl( 2\pi (x - x') \bigr) p(x') \ud x' 
%+\frac12  \frac{p'}{p}(x) =0, \quad x \in \bT.
%\end{equation*}
%However, by multiplying by $p(x)$ and then by taking the derivative with respect to $x$, we would get 
%\begin{equation*}
%\frac{\ud}{\ud x} \biggl( 2 \pi \kappa\,  p(x) \int_{\bT} \sin\bigl( 2\pi (x - x') \bigr) p(x') \ud x'  \biggr)
%+ \frac12 p''(x) =0, \quad x \in \bT,
%\end{equation*}
%namely $p$ would be a stationary solution of 
%\eqref{eq: forward eqn }, which is impossible as there are only two types of stationary solutions, $\one$ or $(p_{\infty,\psi})_{\psi \in \bT}$, and none of them belongs to ${\mathcal K}_{\delta}$. 
%
%By compactness of ${\mathcal K}_{\delta}$ and by a straightforward continuity argument, we deduce that there exists 
%a constant $c>0$ such that, for all $p \in {\mathcal K}_{\delta}$, 
%\begin{equation*}
%\int_{\bT}
%\Bigl\vert 
%2 \pi \kappa \int_{\bT} \sin\bigl( 2\pi (x - x') \bigr) p(x') \ud x' 
%+ \frac12 \frac{p'}{p}(x) \Bigr\vert^2 p(x) \ud x \geq c. 
%\end{equation*}
%In particular, if 
we take $t_{1}=t_{\eta,\epsilon}$ 
and $t_{2} = \inf\{ t \geq t_{\eta,\epsilon} : p(t,\mu) \not \in {\mathcal K}_{\delta} \}$, we obtain the analogue of
\eqref{eq:gradient:flow:222}. 
By the same  argument as therein,  we deduce that, for any $\mu \in {\mathcal Q}_{\eta}$, $t_{2}$ is finite and that 
there exists $T<\infty$, independent of $\mu$, such that $t_{2} \leq T$. 

Now, if we choose $\delta \leq \epsilon$, then we cannot have $\| p(t\, ;\mu) - \one \|_{{\mathcal C}^1} \leq \delta$ for $t \geq t_{1}$, as otherwise {$p(t\, ;\mu)$} would belong to $B_{L^2(\bT)}(\one,\epsilon)$. Therefore, the condition that $t_{2}< \infty$ implies 
$\inf_{\psi \in \bT} \| p(t_{2}\, ;\mu) - p_{\infty,\psi} \|_{{\mathcal C}^1} < \delta$. 
All in all, we deduce that, for any $\mu \in {\mathcal Q}_{\eta}$, there exists $t_{2}$ (possibly depending on 
$\mu$), less than $T$ (which is independent of $\mu$), such that 
$\inf_{\psi \in \bT} \|p(t_{2}\, ;\mu) - p_{\infty,\psi} \|_{{\mathcal C}^1} \leq \delta$. 
\vskip 4pt

\textit{Second step.}
By choosing $\delta$ as small as we want, we deduce that, after the same time $t_{2} \leq T$ as above, 
$p(t_{2}\, ;\mu)$, regarded as the new initial condition for the Fokker-Planck equation, satisfies 
the assumption of Theorem 4.6 in \cite{giacomin:pakdaman:khashayar:pellegrin}. 
It now remains to see that the exponential decay in \cite[Theorem 4.6]{giacomin:pakdaman:khashayar:pellegrin} is uniform 
on $B_{L^2(\bT)}(p_{\infty,\psi},\delta)$, provided that $\delta$ is small enough (independently of the value of $\psi$), which is exactly \cite[Lemma 4.9]{giacomin:pakdaman:khashayar:pellegrin}. 
By combining with the smoothing estimate in 
\cite[Lemma 2.2]{giacomin:pakdaman:khashayar:pellegrin}, this completes the proof. 
\end{proof}

\subsection{Linearised operator}
\label{subse:5.1}
We now address the linearisation of the nonlinear Fokker-Planck equation
at 
a (non-trivial) invariant measure (which hence has the form 
%$p_{\infty,\psi}$ for some phase $\psi \in \bT$). 
%Accordingly, non-trivial invariant measures appear very often in the sequel. For that reason, we prefer to write $p_{\psi}$ instead of 
%$p_{\infty,\psi}$. 
%Although the global architecture of the proof is quite similar to the one 
%implemented in Section 
%\ref{ergodic bounds}, we need to revisit each of the statement therein to adapt it to the current framework. 
%As we already argued in Introduction, it indeed looks rather difficult to have statements
%that are sufficiently general to cover the two cases at the same time. 
%To wit, 
%most of the result 
%below make use of the properties of the linearized operator, with the peculiarity that we perform the linearization at a (non-trivial) invariant measure (which hence has the form 
$p_{\infty,\psi}$ for some phase $\psi \in \bT$). 
%Accordingly, non-trivial invariant measures appear very often in the sequel. For that reason, we prefer to 
For simplicity, we write $p_{\psi}$ instead of 
$p_{\infty,\psi}$. 
The linearised version
\eqref{eq:linearized:operator} may be written in the form 
\begin{equation*}
\partial_{t} q(t) - L_{m(t  ;\mu)} q(t) = 0,
\end{equation*}
with the convenient notation
\begin{equation*}
  L_{m} q  
= \tfrac12 \partial^2_{xx} q + \partial_{x} \Bigl( q \bigl( J \star m \bigr)   + m  \bigl( J \star q \bigr)   
\Bigr),
 \end{equation*}
for any two distributions $m$ and $q$ acting on smooth functions on the torus (notice that
$L_{m} q$ always makes sense as a distribution since 
$J \star m$ and $J \star q$ themselves should be smooth functions).  
When $m=p_{\psi}$ for some element $\psi \in \bT$, we merely write 
$L_{\psi}$ for $L_{m}$. 
 Notice in particular that, by choosing $m_{t}=p_{\psi}$ (which is hence independent of $t$)
 in  
\eqref{eq:se:5:FKP} and then by taking the derivative with respect to $\psi$ (which coincides with the derivative in $x$),  
\begin{equation}
\label{eq:se:5:ppsiprime}
L_{\psi} p_{\psi}'=0,
\end{equation}
where $p_{\psi}'=(\ud/\ud x) p_{\psi}$. 
The above identity was already used in \cite{bertini:giacomin:pakdaman} and in the subsequent works of the same authors. It plays a key role here in our analysis as well. 

We start with 
the long-run analysis of the linearised operator, which is the most demanding step. In fact, the proof is made easier by all the existing results on the Kuramoto model, but the reader must realise that this preliminary step is the cornerstone of the whole analysis in this subsection. 
\begin{proposition}
\label{prop:Q:estimates}
For a smooth initial condition $q_{0}$ on $\bT$, with $\langle q_{0},\one \rangle=0$, 
let $q$ denote the solution of 
\begin{equation}
\label{eq:prop:Q:estimates:0}
 \partial_t q - L_{\psi} q  =0, \quad \quad t \geq 0,  
\end{equation}
and, for any $t \geq 0$, let $Q(t,\cdot)$ denote the (unique) periodic primitive 
of $q(t,\cdot)$ satisfying 
$\int_{\bT} Q(t,x) \ud x=0$. 

Then there exist two positive constants $\lambda$ and $C$, only depending on $\kappa$, 
together with a constant $\bar q_{1/2}$, depending on $q(0,\cdot)$ and $\psi$, such that, for any $t \geq 0$,  
\begin{align}
&\int_{\mathbb T} \vert Q(t,x) \vert^2 
\ud x
\leq 
C \int_{\mathbb T} \vert Q(0,x) \vert^2
\ud x, 
\quad \int_{\mathbb T} \vert Q(t,x)- \bar q_{1/2} (p_{\psi}(x)-1) \vert^2 
\ud x
\leq 
C e^{- \lambda t} 
\int_{\mathbb T} \vert Q(0,x) \vert^2
\ud x. \label{eq:prop:Q:estimates:2}
\end{align}
\end{proposition}

\begin{remark}
\label{rem:5:2}
In fact, the proof shows that 
$\bar q_{1/2}^2 \leq C 
\int_{\bT}  \vert Q(0,x) \vert^2 \ud x$. \end{remark}

\begin{proof}
The fact that 
\eqref{eq:prop:Q:estimates:0} has a unique solution is a  consequence of the analysis performed in the previous section. 
Since $\int_{\mathbb T} q(t,x) \ud x=0$ for any $t \geq 0$, it makes sense to define
$Q(t,\cdot)$ as in the statement. 

Thanks to 
Remark
\ref{rem:5:2}, it suffices to focus on the proof of
\eqref{eq:prop:Q:estimates:2}. 
The proof mostly relies on the work of 
\cite{bertini:giacomin:pakdaman}.
Following the notation introduced in 
\cite[(1.23)]{bertini:giacomin:pakdaman}, we indeed let
\begin{equation}
\label{eq:inner:h-1}
\llangle u,v \rrangle_{\psi} := \int_{\bT} \frac{\bar U(x) \bar V(x)}{p_{\psi}(x)} \ud x,
\end{equation}
for any two distributions $u$ and $v$ on $\bT$ and any two $\bar U$ and $\bar V$ in $L^2(\bT)$ such that $\bar U'=u$, $\bar V'=v$ and 
$\langle \bar U,p_{\psi}^{-1} \rangle = \langle \bar V,p_{\psi}^{-1} \rangle =0$. 
From 
\cite[(2.14), (2.16), (2.37)]{bertini:giacomin:pakdaman}, there exists a constant $\lambda$, only depending on $\kappa$, such that,
for any $t \geq 0$, 
\begin{equation}
\label{eq:gronwall:decay}
\frac{\ud}{\ud t} \bigl\llangle q(t,\cdot),q(t,\cdot) \bigr\rrangle_{\psi}
+ \lambda \big\llangle q(t,\cdot) - \bar q_{1/2}(t) p_{\psi}',q(t,\cdot) - \bar q_{1/2}(t) p_{\psi}'\bigr\rrangle_{\psi} \leq 0, 
\quad
 {\rm with} 
 \ 
\bar q_{1/2}(t) = \frac{\big\llangle q(t,\cdot),p_{\psi}' \big\rrangle_{\psi}}{\big\llangle p_{\psi}',p_{\psi}' \big\rrangle_{\psi}}. 
\end{equation}
By \cite[(2.14)]{bertini:giacomin:pakdaman} again, we observe that 
\begin{equation*}
\frac{\ud}{\ud t}
\big\llangle q(t,\cdot),p_{\psi}' \big\rrangle_{\psi}
= 
\big\llangle L_{\psi} q(t,\cdot),p_{\psi}' \big\rrangle_{\psi} =
\big\llangle  q(t,\cdot),L_{\psi} p_{\psi}' \big\rrangle_{\psi} = 0, 
\end{equation*}
with the last equality following from 
\eqref{eq:se:5:ppsiprime}. Hence, we can write $\bar{q}_{1/2}(t)$ as $\bar{q}_{1/2}$. 
In particular, 
\begin{equation} 
\label{eq:barq1/2}
\vert \bar q_{1/2}
\vert = 
\biggl\vert 
\frac{\big\llangle q(0,\cdot),p_{\psi}' \big\rrangle_{\psi}}{\big\llangle p_{\psi}',p_{\psi}' \big\rrangle_{\psi}}
\biggr\vert 
\leq C 
\biggl( \int_{\mathbb T} 
\vert Q(0,x) \vert^2 \ud x \biggr)^{1/2}.  
\end{equation} 
Moreover, 
applying \eqref{eq:gronwall:decay} to $q(t,\cdot) - \bar{q}_{1/2} p_{\psi}'$, we deduce that, for $t \geq 0$, 
\begin{equation}
\label{eq:exponential:decay}
\big\llangle q(t,\cdot) - \bar q_{1/2} p_{\psi}',q(t,\cdot) - \bar q_{1/2} p_{\psi}'\bigr\rrangle_{\psi}
\leq 
\big\llangle q(0,\cdot) - \bar q_{1/2} p_{\psi}',q(0,\cdot) - \bar q_{1/2} p_{\psi}'\bigr\rrangle_{\psi}
e^{ - \lambda t }.
\end{equation}

By 
\eqref{eq:inner:h-1}, 
the right-hand side in \eqref{eq:exponential:decay}
reads
\begin{equation*}
\big\llangle q(0,\cdot) - \bar q_{1/2} p_{\psi}',q(0,\cdot) - \bar q_{1/2} p_{\psi}'\bigr\rrangle_{\psi}
= \int_{\bT} \frac{\vert Q(0,x) - \bar q_{1/2} p_{\psi}(x) -  k(0) \vert^2}{p_{\psi}(x)} \ud x,
\end{equation*}
where $k(0)$ is a centring constant that forces the mean of $(Q(0,\cdot) - \bar q_{1/2} p_{\psi} - k(0))/p_{\psi}$  
to be zero. In particular, 
the left-hand side is less than 
\begin{align}
\big\llangle q(0,\cdot) - \bar q_{1/2} p_{\psi}',q(0,\cdot) - \bar q_{1/2} p_{\psi}'\bigr\rrangle_{\psi}
\leq 
\int_{\bT} \frac{\vert Q(0,x) - \bar q_{1/2} p_{\psi}(x) \vert^2}{p_{\psi}(x)} \ud x
&\leq 
C \int_{\bT} \vert Q(0,x)  \vert^2  \ud x,
\label{eq:exponential:decay:RHS}
\end{align}
where the last inequality is obtained by expanding the square and by invoking 
Remark
\ref{rem:5:2}, with the constant $C$ depending only on $\kappa$. 
Back to \eqref{eq:exponential:decay}, we write in a similar manner:
\begin{equation*}
\big\llangle q(t,\cdot) - \bar q_{1/2} p_{\psi}',q(t,\cdot) - \bar q_{1/2} p_{\psi}'\bigr\rrangle_{\psi}
= \int_{\bT} \frac{\vert Q(t,x) - \bar q_{1/2} p_{\psi}(x) -  k(t) \vert^2}{p_{\psi}(x)} \ud x,
\end{equation*}
for a new centring constant $k(t)$. Using now an upper bound for 
$p_{\psi}$ and assuming w.l.o.g. that the constant $C$ right above is large enough (as long as it only depends on $\kappa$), 
we obtain that
\begin{equation*}
\big\llangle q(t,\cdot) - \bar q_{1/2} p_{\psi}',q(t,\cdot) - \bar q_{1/2} p_{\psi}'\bigr\rrangle_{\psi}
\geq C^{-1} \int_{\bT} \vert Q(t,x) - \bar q_{1/2} p_{\psi}(x) -  k(t) \vert^2  \ud x.
\end{equation*}
Using the fact that $Q(t,\cdot) - \bar q_{1/2} (p_{\psi} -1)$ has zero mean, we deduce that 
\begin{equation}
\label{eq:exponential:decay:LHS}
\big\llangle q(t,\cdot) - \bar q_{1/2} p_{\psi}',q(t,\cdot) - \bar q_{1/2} p_{\psi}'\bigr\rrangle_{\psi}
\geq C^{-1} \int_{\bT} \vert Q(t,x) - \bar q_{1/2} (p_{\psi}(x) -  1) \vert^2  \ud x.
\end{equation}
Combining 
\eqref{eq:exponential:decay},
\eqref{eq:exponential:decay:RHS}
and 
\eqref{eq:exponential:decay:LHS}, we get
\eqref{eq:prop:Q:estimates:2}. 
\end{proof}

The following result 
is the analogue of
Lemma 
\ref{conjecture backward PDE } in 
Section  
\ref{ergodic bounds}.

\begin{proposition}
\label{prop:5:3}
Let $t>0$, $\psi \in \bT$ and $\xi \in W^{1,\infty}(\bT)$ and let
$L_{\psi}^*$ denote the adjoint of $L_{\psi}$ on $L^2(\bT)$. Then the  problem
\begin{equation} 
\label{eq:w:prop:5:3} 
    \partial_s w + L_{\psi}^* w  =0, \quad s \in [0,t] ; \quad
 w(t,x)  = \xi(x),
    \end{equation}
admits a unique classical solution $(w(s,\cdot))_{0 \leq s \leq t}$.  Moreover,
    there exist constants  $C,\lambda>0$ $($only depending  on 
    $\kappa$ and hence independent of $t)$ such that
\begin{equation} 
\label{eq:prop:5:3:1}
\bigg\| w(s, \cdot) - \int_{\bT} w(s,y)\ud y \bigg\|_{\infty} \leq C
\Bigl( \|\xi\|_{\infty} e^{-\lambda(t-s)} + \vert \langle \xi, p_{\psi}'  \rangle
 \vert \Bigr), \quad \quad \forall s \in [0,t],
\end{equation}  
and,  for any $\alpha,\beta \in [0,1]$
$($allowing, in addition, the constants $C$ and $\lambda$ to depend on 
$\alpha,\beta)$,
\begin{equation} 
\label{eq:prop:5:3:2}
\begin{split}
&\big\| \partial_{x} w(s, \cdot)  \big\|_{\beta,\infty} \leq 
{\frac{C}{1 \wedge (t-s)^{(\alpha+\beta)/2} }}
\Bigl( \|\xi\|_{1-\alpha,\infty} e^{-\lambda(t-s)}
+ 
 \vert  \langle \xi, p_{\psi}'  \rangle
 \vert \Bigr), \quad \quad \forall s \in [0,t].
%\\
%&\big\| \partial_{x}^{k+1} w(s, \cdot)  \big\|_{\infty} \leq 
%\textcolor{red}{\frac{C}{\min((t-s)^{1/2},1)}}
% \|\xi\|_{k-\alpha,\infty}, \quad \quad \forall s \in [0,t], %\label{se:5:W12 estimate est 3} 
\end{split}
\end{equation} 
%Now suppose that, for $k \in \{2,3\}$, $\xi \in W^{k,\infty}(\bT)$. 
%Then, for any $\alpha,\beta \in [0,1]$, 
%there exist constants $C,\lambda >0$, depending only 
%$\alpha$, $\beta$ and $\kappa$, such that 
%    \begin{itemize}
%        \item[$(iii)$]  
%\begin{equation*} 
% {  \big\|  \partial^k_{x} w(s, \cdot) \big\|_{\beta,\infty} \leq 
%   {\frac{C}{ 1 \wedge (t-s)^{(\alpha+\beta)/2}}} 
%     \Bigl( \|\xi\|_{{k- \alpha,{\infty}}}  e^{-\lambda(t-s)}} + 
%     \vert  \langle \xi, p_{\psi}' \rangle \vert
%     \Bigr), \quad \quad \forall s \in [0,t]. \quad \quad  
%\begin{split}
%&\big\| \partial_{x} w(s, \cdot)  \big\|_{\beta,\infty} \leq 
%{\frac{C}{1 \wedge (t-s)^{(\alpha+\beta)/2} }}
%\Bigl( \|\xi\|_{1-\alpha,\infty}
%+ 
% \bigl\vert \bigl\langle \xi, p_{\psi}'(t \, ; \mu) \bigr\rangle
%\bigr\vert \Bigr), \quad \quad \forall s \in [0,t].
%%\\
%%&\big\| \partial_{x}^{k+1} w(s, \cdot)  \big\|_{\infty} \leq 
%%\textcolor{red}{\frac{C}{\min((t-s)^{1/2},1)}}
%% \|\xi\|_{k-\alpha,\infty}, \quad \quad \forall s \in [0,t], %\label{se:5:W12 estimate est 3} 
%\end{split}
%\end{equation*} 
%\end{itemize}
\end{proposition}

\begin{proof}
Existence and uniqueness of a classical solution to the Cauchy problem is 
standard. 
%If needed, it may be 
%checked along the lines of
%Proposition  
%\ref{prop:appendix:1}. 
The rest of the proof is  similar to 
the proof of Lemma
\ref{conjecture backward PDE }, but with some differences that we clarify below.  
\vskip 4pt

\textit{First Step.}
The first step is to  prove the bounds  when 
$t-s \leq 1$. 
%The proof follows from the same argument as the one used in the 
%proof of Theorem  
%\ref{thm:3:17:H}.
By expanding the operator $L_{\psi}^*$, %we indeed observe that 
the equation satisfied by 
$w$ may be rewritten in the form 
\begin{equation}
\label{eq:expanded:form:}
\partial_{s} w(s,\cdot) + \tfrac12 \partial^2_{xx} w(s,\cdot) + V_1  \partial_{x} w(s,\cdot) + \int_{\bT} V_2(\cdot,y) \bar w(s,y) \ud y = 0,
\quad s \in [0,t],
\end{equation}
where 
$\bar w(s,x) = w(s,x) - \int_{\bT} w(s,y) \ud y$
and $V_1$ and $V_2$ are smooth functions on $\bT$ and $\bT^2$ respectively, 
whose derivatives up to any order are bounded in terms of $\kappa$ only.  
By a standard application of the maximum principle combined with Gronwall's lemma, 
it is easy to show that   
\begin{equation}
\label{eq:prop:4:8:bound:infini:w}
\sup_{\max(0,t-1) \leq s \leq t}\| w(s,\cdot) \|_{\infty} \leq C \| \xi \|_{\infty}. 
\end{equation}
Since $V_2$ is smooth, this provides a bound for the derivatives of any order 
of the third term in the left-hand side of 
\eqref{eq:expanded:form:}. 
We then split $w$ into $w=w_{1}+w_{2}$, 
with 
\begin{equation}
\label{eq:expanded:form:n}
\begin{split}
&\partial_{s} w_{1}(s,\cdot) + \tfrac12 \partial^2_{xx} w_{1}(s,\cdot) + V_1  \partial_{x} w_{1}(s,\cdot)  = 0,
\quad s \in [0,t] \, ; \quad w_{1}(t,\cdot) = \xi,
\\
&\partial_{s} w_{2}(s,\cdot) + \tfrac12 \partial^2_{xx} w_{2}(s,\cdot) 
+ V_1  \partial_{x} w_{2}(s,\cdot) + 
\int_{\bT} V_2(\cdot,y) \bar w(s,y) \ud y =0, 
\quad s \in [0,t] \, ; \quad w_{2}(t,\cdot) = 0.
\end{split}
\end{equation}
By Lemma
\ref{conjecture backward PDE }, 
we get all the required bounds on 
$w_{1}(s,\cdot)$ and its derivatives,
at least for $t-s \leq 1$. 
 
Now, in the equation for $w_{2}$,  the source term 
 (with $\bar w$ being frozen)
is smooth. The solution $w_{2}(s,\cdot)$ thus has  bounded (spatial) derivatives 
of any order, for $t-s \leq 1$, with the bounds being independent of $t$. 
\vskip 4pt

\textit{Second Step.}
The rest of the proof is dedicated to the case $t \geq 1$.
We start with the proof of \eqref{eq:prop:5:3:1},
%Following the statement of 
%Proposition 
%\ref{prop:5:10}, we 
%may call 
%$\psi \in \bT$ such that $p_{\psi}$ is the limit of $p(t\, ;\mu)$ as $t$ tends to $\infty$.  
%\textcolor{red}{attention: $\psi$ est donn\'e dans l'\'enonc\'e. A mon avis, il suffit de regarder 
%$\xi$ orthogonal \`a $p_{\psi}$; cela simplifie la preuve.}
%Also, 
using the same notations 
 $q(0,\cdot)$, $q$ and $\bar q_{1/2}$ 
as in the statement and the proof of 
Proposition
\ref{prop:Q:estimates}.
By Proposition 
\ref{prop:Q:estimates} and Remark 
\ref{rem:5:2},
 there exist $\lambda$ and $C$ as in the statement (but the values of which are allowed to vary from line to line) such that 
\begin{equation*}
\begin{split}
\bigl\vert 
\bigl\langle \xi, q(t,\cdot) \bigr\rangle
\bigr\vert
&\leq 
\bigl\vert 
\bigl\langle
\xi, 
\bigl( q(t,\cdot) - \bar q_{1/2} p_{\psi}' \bigr) 
\bigr\rangle  \bigr\vert
+
\bigl\vert 
\bar q_{1/2}
\bigr\vert 
\, \bigl\vert 
\bigl\langle \xi,    p_{\psi}'  %-p'(t\, ;\mu) 
\bigr\rangle \bigr\vert
%+
%\bigl\vert \bar  q_{1/2}
%\bigr\vert 
%\bigl\vert 
%\bigl\langle 
%\xi, p'(t\, ;\mu)  \bigr\rangle  \bigr\vert
\\
&\leq \bigl\vert 
\bigl\langle \xi',  \bigl( Q(t,\cdot) - \bar q_{1/2}  [p_{\psi}-1] \bigr) \bigr\rangle  \bigr\vert
+ C 
\| Q(0,\cdot) \|_{2}
%  \| \xi \|_{\infty} \exp(-\lambda t) 
%+ 
 \bigl\vert \bigl\langle \xi,p_{\psi}' \bigr\rangle \bigr\vert  %\bigr)
 \leq 
C 
 \| Q(0,\cdot) \|_{2} \bigl( \| \xi \|_{1,\infty} e^{-\lambda t} 
+  \bigl\vert \bigl\langle \xi,p_{\psi}' \bigr\rangle \bigr\vert  \bigr). 
\end{split}
\end{equation*}
%Proposition \ref{prop:Q:estimates} provides a bound for the first term on the second line. 
Next, 
we use the Sobolev
bound
\begin{equation*}
\begin{split}
\| Q(0,\cdot)\|_{2} 
%&=
%\biggl( 
%\int_{\bT} \biggl\vert Q(0,x) 
%- 
%\int_{\bT} Q(0,y) \ud y
%\biggr\vert^2 dx \biggr)^{1/2}
\leq 
\biggl( 
\int_{\bT} \int_{\bT} \bigl\vert Q(0,x)  - Q(0,y) \bigr\vert^2 
\ud x \, \ud y \biggr)^{1/2}
\leq C \| q(0,\cdot)\|_{1}, 
\end{split}
\end{equation*}
where we used the 
equality $\int_{\bT} Q(0,x) \ud x=0$
together with the
obvious identity 
$Q(0,x) - Q(0,y) = \int_{y}^x q(0,z) \ud z$.  
By the same duality argument as in the proof of Proposition 
\ref{W1 decay}, 
we observe that
%, for
%$q(0,\cdot)$, $q$ and $\bar q_{1/2}$ as
%in the statement of 
%Proposition  
%\ref{prop:Q:estimates},
\begin{equation}
\label{eq:prop:4:8:duality}
\begin{split}
\frac{\ud}{\ud s} \langle w(s,\cdot) ,q(s,
\cdot) \rangle  =0,
\end{split}
\end{equation}
which implies that 
$\langle w(0,\cdot), q(0,\cdot) \rangle = \langle \xi, q(t,\cdot) \rangle$. 
%Since $\langle \xi,p_{\psi}'\rangle =0$, we can subtract $\bar q_{1/2} p_{\psi}'$ to $q(t,\cdot)$ in the latter term, 
%from which we get (together with an integration by parts)
%\begin{equation}
%\label{eq:centering}
%\begin{split}
%\langle w(0,\cdot), q(0,\cdot) \rangle=
%\bigl\langle \xi, q(t,\cdot) - \bar q_{1/2} p_{\psi}'\bigr\rangle
%=
%\bigl\langle \xi', Q(t,\cdot) - \bar q_{1/2} \bigl(p_{\psi}-1\bigr) \bigr\rangle, 
%\end{split}
%\end{equation}
%with $Q(t,\cdot)$ as in the statement of Proposition  
%\ref{prop:Q:estimates}. 
%By \eqref{eq:prop:Q:estimates:2}, 
We therefore deduce that 
\begin{equation*}
\bigl\vert \langle w(0,\cdot), q(0,\cdot) \rangle \bigr\vert
 \leq C  
 \| q(0,\cdot)\|_{1}
 \bigl( \| \xi \|_{1,\infty} e^{-\lambda t} 
+  \bigl\vert \bigl\langle \xi,p_{\psi}' \bigr\rangle \bigr\vert  \bigr).
\end{equation*}
%Using the fact that $\int_{\bT} Q(0,y) dy=0$, we deduce that 
%\begin{equation*}
%\bigl\vert \langle w(0,\cdot), q(0,\cdot) \rangle \bigr\vert \leq C  \exp(-\lambda  t) \| \xi \|_{1,\infty} \biggl( 
%\int_{\bT^2} \vert Q(0,x) - Q(0,y) \vert^2 dx dy \biggr)^{1/2}
%\leq C  \exp(-\lambda t) \| \xi \|_{1,\infty} \| q(0,\cdot) \|_{1},
%\end{equation*}
%where we used the obvious identity $Q(0,x)-Q(0,y) = \int_{x}^y q(0,z) dz$. 
By choosing $q(0,x)=q^n(x)-1$, where $(q^n)_{n \geq 1}$ is a standard mollifier of the Dirac
mass at some point $x_{0} \in \bT$, and letting $n$ tend to $\infty$ and then taking the supremum over $x_{0}$, we obtain %that
\begin{equation}
\label{eq:prop:4:8:w0-meanw}
\sup_{x \in \bT} \biggl\vert w(0,x) - \int_{\bT} w(0,y) \ud y  \biggr\vert 
\leq 
C  
 \bigl( \| \xi \|_{1,\infty} e^{-\lambda t} 
+  \bigl\vert \bigl\langle \xi,p_{\psi}' \bigr\rangle \bigr\vert  \bigr).
\end{equation}
The above bound does not exactly fit \eqref{eq:prop:5:3:1}. 
The first point to recover \eqref{eq:prop:5:3:1} is to replace 
$\| \xi \|_{1,\infty}$ by $\| \xi \|_{\infty}$. 
We apply 
\eqref{eq:prop:4:8:w0-meanw},
but on the interval $[0,t-1]$ and with $\xi=w(t-1,\cdot)$ itself. 
For a new value of $C$,% (that depends on the same parameters as before), 
\begin{equation}
\label{bound:w0:infini:4:8}
\begin{split}
\sup_{x \in \bT} \biggl\vert w(0,x) - \int_{\bT} w(0,y) \ud y  \biggr\vert 
&\leq 
C  
 \Bigl( \| w(t-1,\cdot) \|_{1,\infty}  e^{-\lambda t} 
+  \bigl\vert \bigl\langle w(t-1,\cdot),p_{\psi}' \bigr\rangle \bigr\vert  \Bigr)
  \leq 
C  
 \Bigl( \| \xi \|_{\infty} e^{-\lambda t}
+  \bigl\vert \bigl\langle \xi,p_{\psi}' \bigr\rangle \bigr\vert  \Bigr),
\end{split}
\end{equation}
where we used   the fact that $\langle \xi, p_{\psi}' \rangle 
= \langle w(t-1,\cdot),p_{\psi}' \rangle$, 
which follows 
from 
\eqref{eq:se:5:ppsiprime}
and then from the same duality argument as in 
\eqref{eq:prop:4:8:duality}. 
In the above, we  also used the bound 
$\| w(t-1,\cdot) \|_{1,\infty} \leq C \| \xi \|_{\infty}$, which follows from 
the first step. 
%$(ii)$ (in small time, as already studied in the first step). 
%It now remains to invoke Proposition 
%\ref{prop:5:10}
%in order to replace $p_{\psi}'$ by $p'(t;\mu)$. 
\vskip 4pt

\textit{Third Step.}
By Lemma \ref{conjecture backward PDE }, 
 $w_{1}$ in 
\eqref{eq:expanded:form:n} satisfies all the required bounds (in long time). 
 In particular, $w_{2}=w-w_{1}$
 satisfies 
\eqref{bound:w0:infini:4:8}. 
By interior estimates for 
the second equation in \eqref{eq:expanded:form:n},
we  obtain 
\eqref{eq:prop:5:3:2}.
\end{proof}

The following proposition is one key step in our proof. 

\begin{proposition}
\label{prop:5:5}
For fixed $\kappa >1$ and $\eta \in (0,1)$,
the drift \eqref{eq:b:kuramoto}
satisfies 
\textrm{\rm \hergo} up to the change that, in
\textrm{\rm\hergcoeff{$\alpha$}{$\beta$}{$\gamma$}}, 
$\mu$ is taken in ${\mathcal Q}_{\eta}$
and that
\begin{enumerate}
\item \eqref{erg:hyp:3}
holds  when $t \leq 1$;
\item   
%with $\lambda=0$
%when $t >0$ and $\gamma =0$;
%\item 
%\eqref{erg:hyp:3} 
%holds 
%with $\lambda >0$
when $t \geq 1$,  
the following variant of 
\eqref{erg:hyp:3}
holds true:
   \begin{equation}
   \label{erg:hyp:3:bbb} 
%   \| q(t) \|_{(k-{\alpha})',\infty} \leq  {\frac{C_k}{1 \wedge t^{\alpha/2}}}
  % \Bigl[ 
  % e^{- \lambda t} \Bigl( K +  \| q_0 \|_{(k,\infty)'} \Bigr) + K' \Bigr] + \Xi, \quad \quad t >0. 
  \bigl\| q(t) - q_{\infty}
p'(t \, ;\mu) \bigr\|_{(k-{\alpha})',\infty} \leq  
   C_k
   e^{- \lambda t}
      \bigl[ 
        \| q_0 \|_{(k,\infty)'}
   +
   K \bigr], \ 
  \textrm{\rm with}   \ K:=\sup_{t \geq 0} \bigl[ \| r(t) \|_{(\beta,\infty)'} \exp(\lambda_0 t)\bigr], 
  % \int_{0}^t \frac{\| r(s) \|_{(\textcolor{blue}{\beta}, \infty)'} }{(t-s)^{\textcolor{blue}{\beta}/2}} \ud s
   \end{equation} 
under the additional assumption that 
$\| r(t) \|_{(\beta,\infty)'}$ decays exponentially fast as $t$ tends to $+\infty$ (at a rate $\lambda_0$, 
 {on which $C_k$ and 
$\lambda$ may depend})
and
where
$q_{\infty}$ is a real number depending on the input of the Cauchy problem 
 \textrm{\rm{\CLinear{$\mu$}{$q_0$}{$r$}}} (see \eqref{eq q}), but independent of $t$,  and
is bounded by 
$C( \|q_{0}\|_{(k,\infty)'} + K )$, for $C$ only depending
on $k$, $\eta$ and $\kappa$. 
\end{enumerate}
\end{proposition}
{Of course, the bound  \eqref{erg:hyp:3}, which 
holds for $t \leq 1$, 
implies the bound \eqref{erg:hyp:3:bbb}
for $t \leq 1$.}

\begin{proof} 
The result for $t \leq 1$ is a direct consequence of {\hergol}. 
We focus on the case $t >1$. 
\vskip 4pt
  
 \textit{First Step.}
 We first choose $\mu = p_\psi$ for some $\psi \in {\mathbb T}$. 
%In finite time, the bounds for $q(t)$ follow from 
%Proposition 
%\ref{W1 decay}.
%% remains true on the time interval $[0,\tau]$, {whether 
%%\eqref{eq:condition:epsilon0:petit}
%%holds or not}. This says that, in 
%%\herg{$K$}{$\gamma$}{$\alpha$}, 
%% \eqref{erg:hyp:3} 
%% holds  (under 
%%  \eqref{erg:hyp:1} 
%%  and
%%  \eqref{erg:hyp:2})
%%  for $t \in (0,\tau]$, with the constants being independent of the choice of $\mu$. 
%\vskip 4pt
%
%\textit{Second Step.}
  We  start from  \eqref{eq q} (with the same solution $q$), but 
instead of considering 
 $w$ as the solution of  
 \eqref{cauchy w}, we choose $w$ as the solution of 
 \eqref{eq:w:prop:5:3}. %with $\xi$ therein being orthogonal to $p_{\psi}'$. 
Following the proof of Proposition 
\ref{W1 decay},   
 this leads to a new expansion in 
     \eqref{duality W1 infty}: $T_{1}$ and $T_{2}$ are the same, but $T_{3}$ 
     is zero. 
Following 
\eqref{eq: bound q est 1} and \eqref{eq: bound q est 3}, 
we deduce from 
\eqref{eq:prop:5:3:1}
    and
    \eqref{eq:prop:5:3:2}
    that, for $\alpha,\beta \in [0,2)$ and $k \in [\alpha,2)$, 
    \begin{equation}
    \label{eq:prop:5:3:0:00}
\| q(t) \|_{(k-\alpha,\infty)'} \leq  \frac{C}{1 \wedge t^{\alpha/2}}
 \| q_0 \|_{(k,\infty)'} %e^{-\lambda t} 
   + C  
\int_{0}^t \frac{\| r(s) \|_{({\beta}, \infty)'} }{1 \wedge (t-s)^{{\beta}/2}} \ud s,
\quad t >0. 
%  e^{-\lambda (t-s)} \ud s, 
%\quad t >0, 
\end{equation}
Notice that there is no exponential decay at this stage, due to to fact that we have not assumed yet 
that $\xi$ in 
 \eqref{cauchy w}
satisfies 
$\langle \xi,p_\psi' \rangle =0$ 
(which term appears in \eqref{eq:prop:5:3:1}
    and
    \eqref{eq:prop:5:3:2}). 
However, 
since $\| r(t) \|_{(\beta,\infty)'}$ decays exponentially fast, this says that 
$\| q(t) \|_{(0,\infty)'} \leq C( \| q_0  
\|_{(k,\infty)'}
+ K) $, for $t \geq 1$. 

We now assume that $\xi$ in 
 \eqref{eq:w:prop:5:3}
satisfies 
$\langle \xi,p_\psi' \rangle =0$. 
With the same duality argument, \eqref{eq:prop:5:3:0:00} becomes 
    \begin{equation}
    \label{eq:newproof:prop:4:9:1}
\Bigl\|
 q(t) - \frac{ \langle q(t),p_{\psi}' \rangle}{\langle p_{\psi}',p_{\psi}'\rangle}
p_{\psi}'
\Bigr\|_{(k-\alpha,\infty)'}   \leq \frac{C}{1 \wedge t^{\alpha/2}}
   \| q_0 \|_{(k,\infty)'} e^{-\lambda t} + C  
\int_{0}^t \frac{\| r(s) \|_{({\beta}, \infty)'} }{1 \wedge (t-s)^{{\beta}/2}} e^{-\lambda (t-s)} \ud s.  
\end{equation} 
Then, for a given $t_0 \geq 1$ and for $t\geq t_0$, we  expand $q(t)$ as $q(t)=q_1(t)+q_2(t)$, with $q_1$
solving 
\textrm{\rm{\CLinear{$\psi$}{$q_0$}{$0$}}}
on $[t_0,+\infty)$, i.e. $q_1(t_0)=q_0$,
 and $q_2$ solving 
\textrm{\rm{\CLinear{$\psi$}{$0$}{$r$}}} on $[t_0,+\infty)$, i.e. $q_2(t_0)=0$. 
Applying 
\eqref{eq:prop:5:3:0:00}
at time $t_0$ instead of $0$ and using the exponential decay of 
$\| r(t) \|_{(\beta,\infty)'}$, we get
$\|q_2(t) \|_{(0,\infty)'} \leq C 
K \exp(-\lambda_0 t_0)$, for $t \geq t_0$.
Next, we 
prove that 
$\langle q_1(t),p_{\psi}' \rangle/\langle p_{\psi}',p_{\psi}'\rangle$
converges exponentially fast to a constant $q_{\infty}$. 
We write
$\langle q_1(t),p_{\psi}' \rangle=
-\langle Q_1(t,\cdot),p_{\psi}'' \rangle$, with $Q_1$ being the primitive of $q_1$ (in space) with a zero mean. Invoking Proposition 
\ref{prop:Q:estimates}, 
$\langle Q_1(t,\cdot),p_{\psi}'' \rangle$ converges exponentially fast to some 
constant. 
The rate of  convergence, as given by 
\eqref{eq:prop:Q:estimates:2}, 
 depends on 
$\| Q_1(t_0,\cdot)\|_{2}  \leq C \|q(t_0,\cdot) \|_{(0,\infty)'}$, from which we deduce that, for some $q_\infty$, 
\begin{equation*}
\bigl\vert \langle q_1(t),p_\psi' \rangle  {-q_\infty} \bigr\vert
\leq C \|q(t_0,\cdot) \|_{(0,\infty)'} e^{-\lambda (t-t_0)}. 
\end{equation*}
Adding the bound for $q_2(t)$, we get, for any $t \in [2t_0-1,2t_0+1]$,  
$\vert \langle  {q(t)} ,p_\psi' \rangle  {-q_\infty} \vert
\leq C ( \| q(t_0,\cdot)
\|_{(0,\infty)'}
+ K ) e^{-\lambda (t-t_0)}$,
 for 
$\lambda$ depending on $\lambda_0$. 
Equivalently, we can first fix $t > 3$ and then choose $t_0 \in [(t-1)/2,(t+1)/2]$. We obtain
\begin{equation}
\label{eq:prop:5:3:0:001}
\bigl\vert \langle {q(t)},p_\psi' \rangle  {-q_\infty} \bigr\vert
\leq C\bigl( \| q(t_0,\cdot)
\|_{(0,\infty)'}
+ K\bigr) e^{-\lambda t/2},
\end{equation}
 for 
$t_0 \in [(t-1)/2,(t+1)/2]$. 
Choosing $t=t_0$ and $k=\alpha$ in 
    \eqref{eq:prop:5:3:0:00}, we can 
replace $\| q(t_0,\cdot)\|_{(0,\infty)'}$ by 
$\| q_0\|_{(k,\infty)'}$
in 
\eqref{eq:prop:5:3:0:001}. 
Back to 
\eqref{eq:newproof:prop:4:9:1}
 {(and replacing 
$q_\infty$ by 
$q_\infty 
\langle p_{\psi}',p_{\psi}'\rangle$)}, we obtain, for $t \geq 1$, 
\begin{equation}
\label{eq:newproof:prop:4:9:2}
\bigl\|
 q(t) - 
 q_{\infty}
p_{\psi}'
\bigr\|_{(k-\alpha,\infty)'}   \leq 
C\bigl( \| q_0
\|_{(k,\infty)'}
+ K \bigr) e^{-\lambda t/2}. 
\end{equation} 
Since $p(t\, ; \mu)=p_\psi$, this  is  the  result for $t \geq 1$. 
The bound for $q_\infty$ follows by letting $t$ 
tend to $+\infty$ in 
    \eqref{eq:prop:5:3:0:00}. 
%Inserting the bound for 
%$q_\infty$, this yields
%\begin{equation}
%\label{eq:newproof:prop:4:9:3}
%\bigl\|
% q(t) 
%\bigr\|_{(k-\alpha,\infty)'}   \leq  
%C
%   \| q_0 \|_{(k,\infty)'}   + C  
%\int_{0}^t \frac{\| r(s) \|_{({\beta}, \infty)'} }{1 \wedge (t-s)^{{\beta}/2}} e^{-\lambda (t-s)} \ud s, \quad t \geq 1. 
%\end{equation} 
\vskip 4pt

\textit{Second Step.}
When 
$\mu$ is taken in ${\mathcal Q}_{\eta}$, we know from 
Proposition 
\ref{prop:5:10}
that $p(t):=p(t\, ; \mu)$ converges exponentially fast to 
$p_{\psi}$ for some $\psi \in {\mathbb T}$.
This allows us to repeat the third step of the proof of Proposition 
\ref{prop main result:2}, see in particular
\eqref{eq:rn2}. Here, the point is to 
rewrite 
\eqref{eq q}
as
$\partial_t q(t) - L_{ {\psi}} q(t) - (r(t)+r_2(t))= 0$, for $t \geq 0$,
$ r_2(t) =  [  L_{m(t;\mu)} - L_{\psi}] q(t)$.
Following 
\eqref{eq:rn2:2}, we obtain 
$\| r_2(t)\|_{(1,\infty)'} \leq C \| q(t)\|_{(0,\infty)'} \exp(- \lambda t)$. 
By 
\eqref{eq:prop:5:3:0:00},
 {but replacing 
$r$ by $r+r_2$ (in order to force the equation to be driven by $L_\psi$)}, we have, for $t>0$, 
\begin{equation}
\label{eq:newproof:prop:4:9:33}
\bigl\|
 q(t) 
\bigr\|_{(k-\alpha,\infty)'}   \leq  
\frac{C \| q_0 \|_{(k,\infty)'}}{1 \wedge t^{\alpha/2}}
      + C  
\int_{0}^t \frac{\| r(s) \|_{({\beta}, \infty)'} }{1 \wedge (t-s)^{{\beta}/2}}  \, \ud s
 + C  
\int_{0}^t \frac{\| q(s) \|_{(0, \infty)'} }{1 \wedge (t-s)^{1/2}}   e^{-\lambda s} \, \ud s, \quad t >0.
\end{equation} 
Fix $t_0>0$. By {\hergol}, we get a bound for $\| q(s)\|_{(0,\infty')}$ for $s \in (0,t_0]$. Inserting this bound in 
the above display, we get 
\begin{equation}
\label{eq:newproof:prop:4:9:34}
\bigl\|
 q(t) 
\bigr\|_{(k-\alpha,\infty)'}   \leq  
C_{t_0}
   \| q_0 \|_{(k,\infty)'}   + C_{t_0}  
\int_{0}^t \frac{\| r(s) \|_{({\beta}, \infty)'} }{1 \wedge (t-s)^{{\beta}/2}}   \ud s
 + C  
\int_{t_0}^t \frac{\| q(s) \|_{(0, \infty)'} }{1 \wedge (t-s)^{1/2}}  e^{-\lambda s} \ud s, \quad t \geq t_0.
\end{equation} 
Choosing $k=\alpha$,  we get  
\begin{equation*}
\sup_{t \geq t_0} 
\bigl[ \|
 q(t) 
\|_{(0,\infty)'}
\bigr]
\leq  
C_{t_0}
   \| q_0 \|_{(k,\infty)'}   + C_{t_0} K 
+ C  
 e^{-\lambda t_0}
\sup_{s \geq t_0} 
\bigl[ \|
 q(s) 
\|_{(0,\infty)'} \bigr].
\end{equation*} 
For $t_0$ large enough, we obtain a bound for the left-hand side.  
In
\eqref{eq:newproof:prop:4:9:2}, 
 {we replace 
$r$ by $r+r_2$}.
Recalling that 
$\| r_2(t)\|_{(1,\infty)'} \leq C \| q(t)\|_{(0,\infty)'} \exp(- \lambda t)$
and using the above bound, we complete the proof. 
\end{proof}

\subsection{Estimates of the tangent processes}
We   estimate the  processes 
$m^{(1)}$,
$m^{(2)}$, 
$d^{(1)}$ 
and $d^{(2)}$
in 
Propositions
\ref{recap theorem 4.5}, 
\ref{diff thm 1}
and 
\ref{diff thm 2}
respectively. Since the model is one-dimensional, we may remove the indices $i$ and $(i,j)$
in $d^{(1)}$ and $d^{(2)}$ respectively. 

\begin{proposition}
\label{prop:tangent:proc:kuramoto}
For any 
$\alpha \in [0,1]$,
$\beta \in [1,2)$ 
and
$\eta \in (0,1)$,
there exist two positive constants $\lambda$ and $C$, only depending on $\kappa$, $\alpha$, $\beta$ and $\eta$, such that, for any $\mu$ in  ${\mathcal Q}_{\eta}$, 
$\nu$, $\nu_{1}$
and $\nu_2$ in 
${\mathcal P}(\bT)$
and
$z$, $z_{1}$ and 
$z_{2}$ in ${\mathbb T}$, 
the following bound holds:
\begin{equation}
\label{eq:tangent:proc:kuramoto:m}
\bigl\|m^{(1)}(t \, ; \mu,\nu)
\bigr\|_{(0,\infty)'} + 
\bigl\|m^{(2)}(t\, ; \mu,\nu_{1},\nu_{2})
\bigr\|_{(0,\infty)'}
\leq C,
%\\
%&\bigl\|
%\tilde m^{(1)}(\mu,\nu)  (t,\cdot)  
%\bigr\|_{-1+\alpha ,\infty} + 
%\bigl\|
%\tilde m^{(2)}(\mu,\nu_{1},\nu_{2})  (t,\cdot)  
%\bigr\|_{-2+2\alpha,\infty}
%\leq C
%\exp(-\lambda  t),
%\end{cases}
\qquad t \geq 0.
\end{equation}
Moreover, 
we can find two real numbers 
$q_{\infty}^{(1)}(\mu, z)$ and $q_{\infty}^{(2)}(\mu , z_{1},z_{2})$ such that 
\begin{equation}
\label{eq:tangent:proc:kuramoto:d}
\begin{split}
& \min(1,t^{\alpha/2}) \bigl\|d^{(1)}(t \, ; \mu,z)
\bigr\|_{(1-\alpha,\infty)'} +
\min(1,t^{\beta/2}) 
\bigl\|d^{(2)}(t \, ; \mu,z_{1},z_{2})
\bigr\|_{(2-\beta,\infty)'}
\leq C,
\\
&\min(1, t^{\alpha/2}) \bigl\|
\tilde d^{(1)}(t\, ; \mu,z)  
\bigr\|_{(1-\alpha ,\infty)'} + 
\min(1,t^{\beta/2})
\bigl\|
\tilde d^{(2)}(t\, ; \mu,z_{1},z_{2})    
\bigr\|_{(2-\beta,\infty)'}
\leq C
e^{-\lambda  t},
\end{split}
\end{equation}
for $t>0$, where
\begin{equation}
\label{eq:tangent:proc:kuramoto:d:def}
\begin{split}
&\tilde d^{(1)}(t\, ; \mu,z)  = d^{(1)}(t \, ; \mu,z)  - q_{\infty}^{(1)}(\mu,z) p'(t\, ; \mu), 
\\
&\tilde d^{(2)}(t\, ; \mu,z_{1},z_{2})=d^{(2)}(t \, ; \mu,z_{1},z_{2}) 
-
q_\infty^{(1)}(\mu,z_{2}) \partial_{x} d^{(1)}(t\, ; \mu,z_{1})
- 
q_\infty^{(2)}(\mu,z_{1},z_{2}) p'(t \, ; \mu). 
\end{split}
\end{equation}
\end{proposition}
 {The reader may observe that the estimate for $d^{(2)}$ in 
\eqref{eq:tangent:proc:kuramoto:d}  is not formulated as in the statement of Proposition 
\ref{global schauder 2} (basically, it is worse). In fact, we chose to give it in this form in order to make it consistent with the estimate of $\tilde d^{(2)}$ in the second line of \eqref{eq:tangent:proc:kuramoto:d}.}

\begin{proof}
%The proof
%of 
%\eqref{eq:tangent:proc:kuramoto:m}
%is the same as in 
%Propositions 
%    \ref{global schauder 1}
%    and
%        \ref{global schauder 2}, since we can use
%        \textrm{\rm \herg{$K$}{$\gamma$}{$\alpha$}}
%        with $\gamma=0$ (no exponential decay is required in the bound). 
%The proof of 
%    \eqref{eq:tangent:proc:kuramoto:d:def}    
% is also very similar 
%to the proofs of Propositions 
%    \ref{global schauder 1}
%    and
%        \ref{global schauder 2}, but they require some care since 
%        Proposition 
%        \ref{prop:5:5} (in (2))
%        just provides a weaker form of
%        \textrm{\rm \herg{$K$}{$\gamma$}{$\alpha$}} when $\gamma>0$.
%\vskip 4pt
%
\textit{First Step.}
We start with the proof of 
\eqref{eq:tangent:proc:kuramoto:d}.
Throughout the proof, we fix $\eta \in (0,1)$. 
By Proposition \ref{diff thm 1}, 
we know that, for any $\mu \in {\mathcal Q}_{\eta}$ and $z \in \bT$,
%, for any $\mu \in {\mathcal P}(\bT)$, %(with the property that $\mu$ is sufficiently close to 
%the set $\{p_{\psi}, \ \psi \in \bT\}$ for $\| \cdot \|_{-1}$) 
% and $z \in \bT$,  
$d^{(1)}(t\, ; \mu,z)$ solves the linearised equation 
\begin{equation}
\label{eq:Kuramoto:eq:d1}
\partial_{t} d^{(1)}(t\, ; \mu,z) - L_{m(t  ; \mu)} d^{(1)}(t \, ; \mu,z)  = 0,
\end{equation}
with $d^{(1)}(0 \, ; \mu,z) =  D_{z}' = - \partial_{x}(\delta_{z})$ and $m(0 \, ; \mu)=\mu$. Then, 
by 
item (2) in Proposition \ref{prop:5:5}, with $k=1$, 
we get the two bounds for 
$d^{(1)}(t \, ; \mu,z)$ and 
$\tilde d^{(1)}(t \, ;\mu,z)$,  with $q_{\infty}^{(1)}(\mu,z)$ being given by 
Proposition \ref{prop:5:5}. 
%we know that, for any 
%$z \in \bT$, we can find two constants $C$ and $\lambda$, only depending on $K$, $\alpha$ and $\eta$, and a real $q_{\infty}^{(1)}(z,\mu)$ such that
%\begin{equation}
%\label{eq:estimates:d1}
%\begin{cases}
%&t^{\alpha/2} \bigl\|d^{(1)}(\mu,z)(t,\cdot)  
%\bigr\|_{-1+\alpha,\infty}
%\leq C,
%\\
%&t^{\alpha/2} \bigl\|
%\tilde d^{(1)}(\mu,z)  (t,\cdot)
%\bigr\|_{-1+\alpha,\infty}
%\leq C
%\exp(-\lambda  t),
%\end{cases}
%\qquad t > 0, 
%\end{equation}
%where we have let 
%\begin{equation}
%\label{eq:kuramoto:widetilded1}
%\tilde d^{(1)}(\mu,z)(t,\cdot) = d^{(1)}(\mu,z)(t,\cdot) - q_{\infty}^{(1)}(z,\mu) p'(t,\cdot). 
%\end{equation}

As for $d^{(2)}$, things are more complicated. 
By Proposition \ref{diff thm 2}, we indeed know that, for any $\mu \in {\mathcal Q}_{\eta}$
and any  $z_{1},z_{2} \in {\mathbb T}^d$, $d^{(2)}(t \, ; \mu,z_{1},z_{2})$ solves the equation
\begin{equation}
\label{eq:kuramoto:d2}
\begin{split}
&\partial_{t} d^{(2)}(t\,;\mu,z_{1},z_{2}) - L_{m(t;\mu)} d^{(2)}(t \, ; \mu,z_{1},z_{2})
\\
&\hspace{5pt}- \partial_{x} \Bigl( d^{(1)}(t\,;\mu,z_{1})  \bigl( J \star 
d^{(1)}(t \, ; \mu,z_{2})  \bigr) + d^{(1)}(t \, ; \mu,z_{2}) \bigl( J \star d^{(1)}(t\, ; \mu,z_{1}) \bigr) 
\Bigr) = 0,
\end{split}
\end{equation}
with $d^{(2)}(0\, ; \mu,z_{1},z_{2}) = 0$. 
The two bounds in finite time follow from {\hergol} 
(noticing that the estimate for $d^{(1)}$ gives a bound 
for $\partial_x d^{(1)}$ in $(W^{2-\alpha,\infty}({\mathbb T}^d))'${, which is needed to estimate $\tilde d^{(2)}$, and that a similar estimate holds for $p'(\cdot ; \mu)$}).
In order to apply 
 (2) in Proposition 
\ref{prop:5:5}  {and} get bounds in long time, %with $k=2$ in \eqref{erg:hyp:3}, 
we let% the following in\eqref{erg:hyp:2}:
\begin{equation*}
r(t) = \partial_{x} \Bigl( d^{(1)}(t \, ; \mu,z_{1})  \bigl( J \star 
d^{(1)}(t\, ; \mu,z_{2})  \bigr) + d^{(1)}(t \, ; \mu,z_{2}) \bigl( J \star d^{(1)}(t \, ; \mu,z_{1}) \bigr) 
\Bigr).
\end{equation*}
However, we cannot prove that 
$\| r(t) \|_{(1,\infty)'}$ decays exponentially 
fast, since 
\eqref{eq:tangent:proc:kuramoto:d}
just provides an exponential bound for 
$\tilde d^{(1)}(t\, ; \mu,z_{i})$, and not for $d^{(1)}(t\,;\mu,z_{i})$ (with $i=1,2$). 
Instead, we focus on 
\begin{equation}
\label{eq:bar:d2}
\bar{d}^{(2)}(t\, ; \mu,z_{1},z_{2}) := d^{(2)}(t\, ; \mu,z_{1},z_{2}) - 
q_\infty^{(1)}(\mu,z_{2}) \partial_{x} d^{(1)}(t\, ; \mu,z_{1}), 
\end{equation}
for $q_{\infty}^{(1)}(\mu,z)$ as in \eqref{eq:tangent:proc:kuramoto:d:def}. 
We easily see that 
$(\partial_{x} d^{(1)}(t\, ; \mu,z))_{t \geq 0}$ solves (in a weak sense) the equation
\begin{equation}
\label{eq:Kuramoto:partialxd1}
\begin{split}
&\partial_{t} \partial_{x}d^{(1)}(t\, ; \mu,z) - L_{m(t \, ; \cdot)} \partial_{x} d^{(1)}(t\, ; \mu,z)
\\
&\hspace{15pt} - \partial_{x} \Bigl( d^{(1)}(t\, ; \mu,z)  \bigl( J \star 
m'(t \, ; \mu)  \bigr) + m'(t\, ; \mu) \bigl( J \star d^{(1)}(t\, ; \mu,z) \bigr)   
\Bigr) = 0,
\end{split}
\end{equation}
in the space
$\cap_{T>0} L^{\infty}([0,T], (W^{2, \infty}(\bT^d))')
\cap_{T>1} L^{\infty}([1/T,T], (W^{2-2\alpha, \infty}(\bT^d))')$, 
with $-\partial^2_{x}(\delta_{z})$ as initial condition.
Choosing $z=z_{1}$
in 
\eqref{eq:Kuramoto:partialxd1},
multiplying 
by 
$q_{\infty}^{(1)}(\mu,z_{2})$
and then subtracting  to 
\eqref{eq:kuramoto:d2}, we obtain 
\begin{equation*}
\begin{split}
&\partial_{t} \bar d^{(2)}(t\, ; \mu,z_{1},z_{2}) 
- L_{m(t;\mu)} \bar d^{(2)}(t\, ; \mu,z_{1},z_{2})  -r(t) = 0, 
\end{split}
\end{equation*}
with 
\begin{equation}
\label{eq:r:m:d:tilde}
r(t)  = 
\partial_{x} \Bigl( d^{(1)}(t \, ; \mu,z_{1})   \bigl( J \star 
\tilde d^{(1)}(t\, ; \mu,z_{2})   \bigr) + \tilde d^{(1)}(t \, ; \mu,z_{2})  \bigl( J \star d^{(1)}(t\, ; \mu,z_{1})  \bigr)   
\Bigr).
\end{equation}
We first evaluate the $(W^{1,\infty}(\bT))'$ norm of $r(t)$. 
Using the fact that the convolution kernel $J$ is odd, we get, for any $\xi \in W^{1,\infty}(\bT)$, 
\begin{equation*}
\begin{split}
\langle r(t) , \xi \rangle &= 
-
\Bigl\langle 
d^{(1)}(t\, ; \mu,z_{1})  \bigl( J \star 
\tilde d^{(1)}(t\, ; \mu,z_{2})  \bigr) + \tilde d^{(1)}(t \, ;\mu,z_{2})  \bigl( J \star d^{(1)}(t \, ; \mu,z_{1}) \bigr) , 
\xi'
\Bigr\rangle
\\
&=
\Bigl\langle 
\tilde d^{(1)}(t\, ; \mu,z_{2}), J \star  \bigl( d^{(1)}(t\, ; \mu,z_{1})  \xi'  \bigr) 
\Bigr\rangle 
- 
\Bigl\langle \tilde d^{(1)}(t\, ; \mu,z_{2}) ,  \bigl( J \star d^{(1)}(t\, ; \mu,z_{1})  \bigr)    
\xi'
\Bigr\rangle. 
\end{split}
\end{equation*}
Invoking \eqref{eq:tangent:proc:kuramoto:d} with $\alpha=1$ (using in addition the smoothness of $J$), we then have 
\begin{equation*}
\| r(t) \|_{(1,\infty)'} \leq C e^{-\lambda t}, \qquad t \geq 1, 
\end{equation*}
where the values of $C$ and $\lambda$ are allowed to vary as long as they only depend on $\kappa$, $\alpha$, $\beta$ and $\eta$. 
By item (2) in Proposition 
\ref{prop:5:5}, 
but initiated from 
time $1/2$ (recalling that we have a bound for 
 {$\| \bar d^{(2)}(t \, ; \mu,z_{1},z_{2}) \|_{(1,\infty)'}$ at $t=1/2$})
and with $k=1$,
we deduce   
that there exists a constant $q_\infty^{(2)}(\mu,z_{1},z_{2})$ such that 
\begin{equation}
\label{eq:estimates:d2}
%\begin{cases}
%&\bigl\|\bar d^{(2)}(\mu,z_{1},z_{2})(t,\cdot)  
%\bigr\|_{-1,\infty}
%\leq C,
%\\
%&
 \bigl\|
\tilde d^{(2)}(t \, ; \mu,z_{1},z_{2}) \bigr\|_{(0,\infty)'}
\leq C
e^{-\lambda t},
%\end{cases}
\qquad t \geq 1, 
\end{equation}
with 
$\tilde d^{(2)}(t\, ; \mu,z_{1},z_{2})
= 
 \bar d^{(2)}(t\, ; \mu,z_{1},z_{2})
 - 
q_\infty^{(2)}(\mu,z_{1},z_{2}) p'(t\, ; \mu)$.
This completes the proof of 
\eqref{eq:tangent:proc:kuramoto:d}.
\vskip 4pt
 
\textit{Second Step.}
{We now turn to the proof 
of 
\eqref{eq:tangent:proc:kuramoto:m}, which is quite similar to the first step.
By Proposition
\ref{recap theorem 4.5}, we know that, for any 
$\mu \in {\mathcal Q}_{\eta}$ and $\nu \in {\mathcal P}(\bT)$, 
$(m^{(1)}(t\, ; \mu,\nu))_{t \geq 0}$ solves the same equation  
\eqref{eq:Kuramoto:eq:d1}
but with $\nu-\mu$ as initial condition. 
Therefore, 
we have the bound for 
$m^{(1)}(t \, ; \mu,\nu)$
in \eqref{eq:tangent:proc:kuramoto:m}, using item (2) in 
Proposition 
\ref{prop:5:5}, with $k=0$.
We have the same for 
$\tilde m^{(1)}(t \, ; \mu,\nu)$, with an obvious definition for the latter. 

We now treat $(m^{(2)}(t\, ; \mu,\nu_{1},\nu_{2}))_{t \geq 0}$. For $t \in [0,1]$, the bound follows from {\hergol}. 
To address the case $t>1$, we write the analogue of 
\eqref{eq:Kuramoto:partialxd1}
but for $(\partial_{x} m^{(1)}(t\, ; \mu,\nu))_{t \geq 0}$ and the analogue of $r$ in 
\eqref{eq:r:m:d:tilde}. We have $\| r(t)\|_{(1,\infty)'} \leq C \exp(-\lambda t)$.
By  (2) in Proposition \ref{prop:5:5}, with $k=1$, 
we recover 
\eqref{eq:estimates:d2} but 
for 
$\tilde m^{(2)}(t \, ; \mu,\nu_{1},\nu_{2})$ (with an obvious definition for it).
We get a bound for 
$\|  m^{(2)}(t \, ; \mu,\nu_{1},\nu_{2}) \|_{(0,\infty)'}$ for $t \geq 1$.}% This is enough to conclude. 
\end{proof}

In the rest of the subsection, ${\mathcal U}$ denotes the same functional as in 
\eqref{eq: defofflow}. We start with:
\begin{lemma}
\label{lem:4.11}
Let $\Phi : {\mathcal P}(\bT) \rightarrow {\mathbb R}$ be 
a rotation-invariant function that satisfies \emph{\hintphi{4}{3}} (see Definition $\ref{def:rotation:invariant:function})$. Then, for any  $\mu \in {\mathcal P}(\bT)$ and any  
$q \in (W^{1,\infty}(\bT))'$ with $\langle q, \one \rangle=0:$
\begin{equation}
\label{eq:2nd:order:psi:1}
\begin{split}
&\frac{\delta \Phi}{\delta m}(\mu)\bigl( \mu' \bigr) =0,
\qquad \textrm{and}
\qquad    \frac{\delta^2 \Phi}{\delta m^2} (\mu) \bigl( \mu' , q \bigr)
 +
\frac{\delta \Phi}{\delta m} (\mu)\bigl( q'\bigr) 
=0,
\end{split}
\end{equation}
where $\mu'$ and $q'$ are the derivatives, in the sense of distributions, of $\mu$ and $q$
(so that the above expressions are well-defined thanks to the regularity of $\Phi$).   
\end{lemma}

\begin{proof}
We start from the very definition of rotation-invariant function. It says that, for any 
$\psi \in \bT$, 
\begin{equation*}
\frac{\ud}{\ud \psi}_{\vert \psi=0} \Phi\bigl( \mu \circ \tau_{\psi}^{-1} \bigr) =0. 
\end{equation*}
The left-hand side writes in the form 
$- \langle [\delta\Phi/\delta m](\mu),\mu' \rangle$. We deduce the first identity in the statement. 

Assume now that $\mu$ has a (strictly) positive continuous density $\ud \mu /\ud x$ and consider an element $q \in L^\infty(\bT)$ such that 
$\langle q, \one \rangle=0$. Then, for $\varepsilon$ small enough, 
$\ud \mu/\ud x + \varepsilon q$ may be regarded as a density on $\bT$. With a slight abuse of notation, we 
write $\mu + \varepsilon q \in {\mathcal P}(\bT)$. Replacing $\mu$ by $\mu + \varepsilon q$ in the first identity in the statement and then taking the derivative with respect to $\varepsilon$ at $\varepsilon=0$, we deduce that 
the second identity in the statement holds at any pair $(\mu,q)$ satisfying the prescribed conditions. Using the
density of $L^\infty(\bT)$ in $(W^{1,\infty})'(\bT)$  
together with the fact that $\Phi$ satisfies {\hintphi{4}{3}}, the second identity also holds  at any pair $(\mu,q)$, with $\mu$ as before and 
$q \in (W^{1,\infty})'(\bT)$ such that $\langle q,\one \rangle=0$. 
Approximating (for the $1$-Wasserstein topology) any $\mu \in {\mathcal P}(\bT)$ by probability measures with 
a positive density and then invoking again the regularity properties of  
$\Phi$, we finally obtain the second identity in full generality. 
\end{proof}

We deduce the following important proposition.
\begin{proposition}
\label{corol:1}
Assume that 
$\Phi$ satisfies
\emph{\hintphi{4}{3}}.  
Then, for any $\alpha \in (0,1)$ and $\eta \in (0,1)$, there exist two positive constants $\lambda$ and $C$, with $\lambda$ only depending on $\kappa$, $\alpha$ and $\eta$, and $C$ only depending on $\kappa$, $\alpha$, $\eta$ and
the bounds in 
\emph{\hintphi{\alpha}{2}}, such that, 
 for any $\mu \in {\mathcal Q}_{\eta}$, 
 any $t > 0$ and 
 any $z_{1},z_{2} \in \bT$,  
\begin{align}
&\Bigl\vert \sld[\cU](t  , \mu)( z_1, z_2) \Big\vert
\leq C, \quad t \geq 0,
\label{eq:derivatives:proc:kuramoto:m}
\\
&\min(1,t^{1-\alpha/4}) \Bigl\vert \partial_{z_2} \partial_{z_1} \sld[\cU](t  , \mu)( z_1, z_2) \Big\vert 
\leq C e^{- \lambda t}, \quad t >0. 
\label{eq:derivatives:proc:kuramoto:d}
\end{align}
\end{proposition}

\begin{proof}
\textit{First Step.}
We start with the proof of 
\eqref{eq:derivatives:proc:kuramoto:m}.
It is a mere consequence of 
the representation formula in 
Proposition \ref{recap theorem 4.5}
and the bound
\eqref{eq:tangent:proc:kuramoto:m}
in the statement of Proposition 
\ref{prop:tangent:proc:kuramoto}. 
%In time $1$, it is a mere consequence of 
%\eqref{eq:regularity:mathcalU}. For $t \geq 1$, 
%it follows from 
%\eqref{eq:derivatives:proc:kuramoto:d} (provided that the latter holds true). 
%Indeed, since 
%$\int_{\bT} [\delta^2 {\mathcal U}/\delta m^2](t,\mu)(z_{1},z_{2}) \ud \mu(z_{2})$
%and  
%$\int_{\bT} \partial_{z_{1}} [\delta^2 {\mathcal U}/\delta m^2](t,\mu)(z_{1},z_{2}) \ud \mu(z_{2})$
%are both null, 
%the bound for the second order derivatives implies 
%a bound for 
%$[\partial_{z_{1}} \delta^2 {\mathcal U}/\delta m^2](t,\mu)(z_{1},z_{2})$, 
%which in turn implies a bound for 
%$[\delta^2 {\mathcal U}/\delta m^2](t,\mu)(z_{1},z_{2})$. 
%\textcolor{red}{Cela d\'ecoule de la borne pour $d^{(1)}$.}
%
%This is a mere consequence of 
%the formula 
%\eqref{eq: main rep} together with 
%the bound
%\eqref{eq:tangent:proc:kuramoto:m}.
\vskip 4pt

\textit{Second Step.}
We turn to the proof of 
\eqref{eq:derivatives:proc:kuramoto:d}. 
We recall the following formula from Proposition \ref{diff second order L-deriv}:
\begin{equation*}
\begin{split}
\partial_{z_2} \partial_{z_1} \sld[\cU](t , \mu)(z_1, z_2)
 &=  \sld[\Phi]\bigl(m(t \, ; \mu)\bigr)\Big( d^{(1)}(t\, ; \mu, z_1),d^{(1)}(t\, ; \mu, z_2) \Big) 
  + \ld[\Phi]\bigl(m(t \, ; \mu)\bigr)\Big( d^{(2)}(t \, ; \mu, z_1,z_2)  \Big),
\end{split}
\end{equation*}
for $\mu \in {\mathcal P}(\bT)$, $t \geq 0$ and $z_{1},z_{2} \in \bT$. 
Using the same notation as in 
Proposition \ref{prop:tangent:proc:kuramoto}, we have
\begin{equation*}
\begin{split}
&\partial_{z_2} \partial_{z_1} \sld[\cU](t, \mu)(z_1, z_2)
=\sld[\Phi]\bigl(m(t \, ; \mu)\bigr)\Big( \tilde d^{(1)}(t \, ; \mu, z_1),  d^{(1)}(t \, ; \mu, z_2) \Big)
\\
&\hspace{5pt}
+ q_{\infty}^{(1)}(\mu,z_{2}) 
\sld[\Phi]\bigl(m(t \, ; \mu)\bigr)\Big( p'(t \, ;\mu),d^{(1)}(t \, ; \mu, z_2) \Big)
+ \ld[\Phi]\bigl(m(t \, ;\mu)\bigr)\Big( d^{(2)}(t \, ; \mu, z_1,z_2)  \Big).
\end{split}
\end{equation*}
With the same notation as in \eqref{eq:bar:d2} and 
\eqref{eq:Kuramoto:partialxd1}  {and thanks to the second identity in 
\eqref{eq:2nd:order:psi:1}}, we get 
\begin{equation*}
\begin{split}
\partial_{z_2} \partial_{z_1} \sld[\cU](t , \mu)(z_1, z_2)
&=\sld[\Phi]\bigl(m(t \, ; \mu)\bigr)\Big( \tilde d^{(1)}(t\, ; \mu, z_1),  d^{(1)}(t \, ; \mu, z_2)  \Big)
\\
&\hspace{15pt}
-   q_{\infty}^{(1)}(\mu,z_{2})
\ld[\Phi]\bigl(m(t \, ;\mu)\bigr) \Big( \partial_{x} d^{(1)}(t\, ; \mu, z_2)  \Big)
+ \ld[\Phi]\bigl(m(t\, ; \mu)\bigr)\Big( d^{(2)}(t\, ; \mu, z_1,z_2)  \Big)
\\
&=
\sld[\Phi]\bigl(m(t \, ; \mu)\bigr)\Big( \tilde d^{(1)}(t\, ; \mu, z_1) ,  d^{(1)}(t\, ; \mu, z_2)
 \Big)
+
\ld[\Phi]\bigl(m(t \, ; \mu)\bigr) \Big( \bar d^{(2)}(t\, ; \mu, z_1,z_2)  \Bigr).  
\end{split}
\end{equation*}
By the first identity in 
\eqref{eq:2nd:order:psi:1}, we can remove for free $q_{\infty}^{(2)}(\mu,z_{1},z_{2})p'(t \, ;\mu)$ in the second term on the last line. 
Using again the notation from the statement of Proposition 
\ref{prop:tangent:proc:kuramoto}, we obtain 
\begin{equation*}
\begin{split}
\partial_{z_2} \partial_{z_1} \sld[\cU](t,\mu)(z_1, z_2)
&=
\sld[\Phi]\bigl(m(t \, ; \mu)\bigr)\Big( \tilde d^{(1)}(t\, ;\mu, z_1) ,  d^{(1)}(t\, ; \mu, z_2)  \Big)
+
\ld[\Phi]\bigl(m(t \, ; \mu)\bigr) \Big( \tilde d^{(2)}(t\,;\mu, z_1,z_2)  \Bigr).  
\end{split}
\end{equation*}
The end of the proof follows  {from
\eqref{eq:tangent:proc:kuramoto:d}}
and 
{\hintphi{\alpha}{2}}  (see 
in particular, \eqref{eq:Int-phi-alpha-k}).
%
%
% \, 
% with 
%the first line in 
%\eqref{eq:tangent:proc:kuramoto:d}, we then notice that the function 
%\begin{equation*}
%x \mapsto \min(1,t^{(1-\alpha)/2}) \Bigl\langle 
%\sld[\Phi]\bigl(m(t \, ; \mu)\bigr)( x,\cdot),   d^{(1)}(t\, ; \mu, z_2) 
%\Bigr\rangle
%\end{equation*}
%is in $W^{1,\infty}(\bT)$, with a norm that is uniformly bounded by a constant that only depends on $\kappa$ and $\eta$.
%By 
%the second line in
%\eqref{eq:estimates:d1}
%and
%by
%\eqref{eq:estimates:d2}, we
%complete the proof. 
\end{proof}

\subsection{Semi-group generated by the empirical distribution}
\label{subse:kuramoto:weak}
We now address the weak error, as in the statement of Theorem 
 \ref{main:thm:kuramoto}. By the same regularisation argument as in 
 {the conclusion of the proof of Proposition} \ref{thm main result}, 
we can assume that $\Phi$ satisfies \hintphi{4}{3}.  
Following 
the statement of Proposition  
\ref{corol:1}, we  must then prove that the constant $C$ that we obtain  in the main inequality of 
Theorem \ref{main:thm:kuramoto} only depends on $\Phi$ through
the 
bounds 
in {\hintphi{\gamma}{2}}, for a fixed value of $\gamma$ as in the statement of 
Theorem \ref{main:thm:kuramoto}.
%, namely 
%through the quantity
%\begin{equation}
%\label{eq:kuramoto:weak:M}
%M_{\alpha}(\Phi)=\max\Bigl(\| \Phi \|_{\infty}, \sup_{\|m\|_{-2+2\alpha,\infty} \leq 1}
%[\delta \Phi/\delta m](m),\sup_{\|m_{1}\|_{-1+\alpha,\infty},\|m_{2}\|_{-1+\alpha,\infty} \leq 1}
%[\delta^2 \Phi/\delta m^2](m_{1},m_{2})\Bigr),
%\end{equation} 
%for a real $\alpha \in [0,1)$ that is fixed throughout the whole section. 

Throughout the proof, we make use of the
notation introduced in Subsection  \ref{subse:metastable}, letting: 
\begin{equation*}
\overline{\mathcal U}^N(t,\mu) = {\mathbb E} \bigl[ \Phi\bigl( \mu_{t}^N \bigr) \, \big\vert \, \mu^N_{0} = \mu \bigr],
\quad \mu \in {\mathcal P}_{N}(\bT).
\end{equation*}
%This definition indeed makes since, for $\mu$ being the uniform distribution $(1/N) \sum_{i=1}^N \delta_{x_i}$ for some $(x_1,\cdots,x_N) \in {\mathbb T}^N$, 
%the right-hand side is symmetric in $(x_1,\cdots,x_N)$ and may hence be regarded as a function of $\mu$. 
By the Markov property of 
$(\mu^N_t)_{t \geq 0}$, 
we have
\begin{equation}
\label{eq:markov:muN}
{\mathbb E} \bigl[ \Phi\bigl( {\mu}_{t}^N \bigr) \vert {\mathcal F}_{s} \bigr]
%= \overline{\mathcal U}^N\bigl(t-s,(Y^{1,N}_{s},\cdots,Y^{N,N}_{s}) \bigr) 
= \overline{\mathcal U}^N \bigl( t-s,  \mu^N_{s} \bigr). 
\end{equation}
%with the notation in the middle making sense thanks to the symmetry  of $\overline{\mathcal U}^N$.} 

%Throughout, we use the notation ${\mathcal I}:=\{p_{\psi}, \ \psi \in \bT\}$.
\vskip 4pt

We then proceed step by step. 
The first step is dictated by the analysis of the strong error in \cite{coppini2019long}
and consists of a preliminary form of the result up until time $\exp(N^{1/2})$, for 
$\mu$ close enough to ${\mathcal I}=\{p_{\psi}, \psi \in {\mathbb T}\}$. 

\begin{lemma}
\label{lem:kura:expN1/2}
There exist  $\delta >0$, only depending on $\kappa$,  
and a constant $C \geq 0$, only depending on $\delta$, $\kappa$
and the bounds for $\Phi$ in 
\emph{\hintphi{\gamma}{2}},
such that, 
for any $N \geq 1$
and  $\mu \in {\mathcal P}_{N}(\bT)$ 
with $\textrm{\rm dist}_{\| \cdot \|_{-1,2}}(\mu,{\mathcal I}) \leq \delta$,
\begin{equation*}
\begin{split}
\forall t \in \bigl[0,\exp(N^{1/2})\bigr], \quad \Bigl\vert 
\overline{\mathcal U}^N(t,\mu) 
 -
{\mathcal U}(t,\mu) \Bigr\vert \leq \frac{C}{N}, 
\end{split}
\end{equation*}
where $\textrm{\rm dist}_{\| \cdot \|_{-1,2}}(\mu,{\mathcal I})
= \inf_{\psi \in \bT} \| \mu - p_{\psi} \|_{-1,2}$.
Here, we may choose
$\delta$ %may be taken 
such that 
$\{ \mu : 
\textrm{\rm dist}_{\| \cdot \|_{-1,2}}( \mu,{\mathcal I})
\leq \delta \}
\subset {\mathcal Q}_{1/2}$.
\end{lemma}
\begin{proof}
The proof relies on some auxiliary results obtained in \cite{coppini2019long}. 
\vskip 4pt

\textit{First Step.} 
%We start with the following (quite standard) computation: 
%\begin{equation}
%\label{eq:LLN:H-1}
%\begin{split}
%{\mathbb E} \Bigl[ \| \mu_{0} - \mu^N_{0} \|_{-1}^2 \Bigr]
%&= \sum_{n \geq 0}
% \frac{1}{(1+n^2)} {\mathbb E}
% \biggl[ \biggl\vert \int_{\bT} e^{i 2\pi n \theta} d \bigl( \mu_{0} - \mu^N_{0} \bigr) (\theta) 
% \biggr\vert^2 \biggr] \leq \frac{1}{N}
% \sum_{n \geq 0}
% \frac{1}{(1+n^2)} \leq \frac{c}{N},
%\end{split}
%\end{equation}
%for a universal constant $c \geq 0$, where we used the obvious fact that, for any bounded measurable 
%function $f : \bT \rightarrow \bR$,  
%\begin{equation*}
%\begin{split}
%&{\mathbb E}
%\biggl[ \biggl\vert 
% \int_{\bT} f(\theta) d \bigl( \mu_{0} - \mu^N_{0} \bigr) (\theta) \biggr\vert^2 \biggr]
% = \frac1N \int_{\bT} \biggl\vert f(\theta) - \int_{\bT} f(\theta') d \theta' \biggr\vert^2 d \theta 
% \leq \frac{\| f \|_{\infty}^2}N. 
% \end{split}
%\end{equation*}
%We deduce that, for any $\delta >0$, 
%\begin{equation*}
%%\label{eq:weak:weak:error:step:1:kura}
%{\mathbb P} \Bigl( \bigl\| \mu_{0} - \mu^N_{0}  \bigr\|_{-1} \geq \varrho \Bigr) 
%\leq \frac{c}{\varrho N}.
%\end{equation*} 
%In particular, under the additional condition that 
%\begin{equation}
%\label{eq:weak:weak:error:step:1:kura}
%\textrm{\rm dist}_{\| \cdot \|_{-1}}(\mu_{0},{\mathcal I}) \leq \delta,
%\end{equation}
%we get 
%\begin{equation}
%\label{eq:weak:weak:error:step:2:kura}
%{\mathbb P} \Bigl( 
%\textrm{\rm dist}_{\| \cdot \|_{-1}}(\mu_{0}^N,{\mathcal I}) 
% \geq 2\delta \Bigr) 
%\leq \frac{c}{\delta N}.
%\end{equation} 
%
%\textit{Second Step.}
We apply \cite[(98-99)]{coppini2019long}. It says that there exist a time $T>0$, a  real $\varepsilon >0$ and a constant $C$, all independent of $N$ and the initial conditions, such that, for any integer $n \geq 1$, 
\begin{equation}
\label{eq:weak:error:coppini}
\begin{split}
&{\mathbb P}
\biggl(
\sup_{0 \le t \le n T}\textrm{\rm dist}_{\| \cdot \|_{-1,2}}(\mu^N_{t},{\mathcal I})
\leq \varepsilon
\, \big\vert \, 
\bigl\|\mu^N_{0} - p_{\textrm{proj}(\mu^N_{0})} \bigr\|_{-1,2} \leq \frac{\varepsilon}2
\biggr)
  \geq \Bigl[ 1 - \exp \bigl( - C N ) \Bigr]^n,
\end{split}
\end{equation}
where $\textrm{\rm proj} : {\mathcal P}(\bT) \rightarrow \bT$ is the projection mapping defined in
\cite[Lemma 2.8]{MR3689966}
and 
\cite[Lemma  3.5]{coppini2019long} (with the small difference\footnote{There is another difference. In 
\cite{coppini2019long,MR3689966}, the dual space is not $W^{-1,2}(\bT)$, but $H^{-1}_{0}=\{ f \in W^{1,2}(\bT) : \langle f, \one \rangle =0\}'$. Of course, any element of $W^{-1,2}(\bT)$ may be  projected onto an element of $H^{-1}_{0}$ by discarding its constant Fourier mode.}  that the function $\bT$ is defined here as ${\mathbb R}/{\mathbb Z}$, whilst it is defined as ${\mathbb R}/(2 \pi {\mathbb Z})$ in  \cite{coppini2019long,MR3689966}). 
Notice that the conditioning in the left-hand side 
is implicitly required
in \cite{coppini2019long} %but it 
% since the computation in 
%\cite{coppini2019long} is conditional on the event $\{
%\| \mu_{0}^N - p_{\textrm{\rm proj}(\mu_{0}^N)} \|_{-1} \leq {\varepsilon}/2 \}$, 
(see (84)). More importantly, there is no need to assume that $\mu_0^{N}$ 
is the $N$-sample of a common distribution
and the result holds for an arbitrary initial condition $\mu_{0}^N=\mu \in {\mathcal P}_{N}(\bT)$. %provided
%it is close enough to ${\mathcal I}$, the precise choice of $\delta$
%in the statement being clarified right below. 

On ${\mathcal I}$, the projection $\textrm{\rm proj}$ reduces to the trivial mapping $\textrm{\rm proj}(p_{\psi})=\psi$, $\psi \in \bT$. 
Therefore, if some probability measure $\nu \in {\mathcal P}(\bT)$ is close  
to ${\mathcal I}$ (for $\| \cdot \|_{-1,2}$), then it is close to some $p_{\psi} \in {\mathcal I}$ and, by continuity of 
$\textrm{\rm proj}$ with respect to $W^{-1,2}(\bT)$, $\textrm{\rm proj}(\nu)$ is close to $\psi$. In the end, 
 $\nu$ is close to $p_{\textrm{\rm proj}(\nu)}$. 
 This continuity result, combined with a standard compactness argument, may be formulated as follows. For the same $\varepsilon$ as above, we can find some $\delta_{\varepsilon}>0$ such that 
 $\textrm{\rm dist}_{\| \cdot \|_{-1,2}}(\nu,{\mathcal I}) 
 \leq \delta_{\varepsilon}$ implies 
 $\| \nu - p_{\textrm{\rm proj}(\nu)} \|_{-1,2} \leq \varepsilon/2$. 
Hence, 
taking $\mu \in {\mathcal P}_{N}(\bT)$ with 
$\textrm{\rm dist}_{\| \cdot \|_{-1,2}}(\mu,{\mathcal I}) \leq \delta_{\varepsilon}$, we get 
%\begin{equation*}
%{\mathbb P} \Bigl( 
%\bigl\| \mu^N_{0} - p_{\textrm{\rm proj}(\mu^N_{0})} \bigr\|_{-1} > \varepsilon/2
% \Bigr) 
%\leq \frac{c}{\delta_{\varepsilon} N},
%\end{equation*} 
%provided we choose $\delta=\varepsilon/2$ in 
%\eqref{eq:weak:weak:error:step:1:kura}, whence 
%the value of $\delta$ in the statement. 
%
%By inserting into 
%\eqref{eq:weak:error:coppini}, we deduce that 
\begin{equation*}
{\mathbb P}
\biggl(
\sup_{0 \le t \le n T}\textrm{\rm dist}_{\| \cdot \|_{-1,2}}(\mu^N_{t},{\mathcal I})
\leq \varepsilon \, \Big\vert \,
\mu_{0}^N = \mu
\biggr) \geq \Bigl[ 1 - \exp \bigl( - C N ) \Bigr]^n.
\end{equation*}
Following 
\cite[(100)]{coppini2019long}, we deduce that, for $n= \exp(N^{1/2})$ and for a new value of $C$, 
\begin{equation*}
{\mathbb P}
\biggl(
\sup_{0 \le t \le n T}\textrm{\rm dist}_{\| \cdot \|_{-1,2}}(\mu^N_{t},{\mathcal I})
\leq \varepsilon
\, \Big\vert \,
\mu_{0}^N = \mu
\biggr) \geq 1 -\frac{C}{N}.
\end{equation*}
%which hence holds true under the assumption 
%that 
%$\textrm{\rm dist}_{\| \cdot \|_{-1}}(\mu_{0},{\mathcal I}) \leq \eta_{\varepsilon}$, see
%\eqref{eq:weak:weak:error:step:1:kura}.
%Invoking the latter reference, we know that, if
%$\textrm{\rm dist}_{\| \cdot \|_{-1}}(\mu,{\mathcal I}) \leq \delta/2$, then, for any $k \geq 1$, 
%\begin{equation*}
%{\mathbb P} 
%\Bigl(
%\textrm{\rm dist}_{\| \cdot \|_{-1}}( \mu^N_{t},{\mathcal I}) \geq \delta, \quad t \in [0,\exp(N^{1/2})] \Bigr) \leq C_{k}
%N^{-k}.
%\end{equation*}
With $\tau_{N}
:= \inf \{ t \geq 0 : 
\textrm{\rm dist}_{\| \cdot \|_{-1,2}}( \mu^N_{t},{\mathcal I})
\geq \varepsilon \},$
this may be rewritten in the form 
\begin{equation}
\label{eq:Kuramoto:exit:time}
{\mathbb P}
\Bigl(   \tau_{N} \leq \exp(N^{1/2}) 
\, \big\vert \,
\mu_{0}^N = \mu
\Bigr) \leq \frac{C}{N}.
\end{equation}
\vskip 4pt

\textit{Second Step.} Following 
\eqref{eq:errordecomp}, we 
have the following expansion, 
which holds for any $t \geq 0$,
\begin{align}
\label{eq:expand:phi:u:stopping time}
&{\mathbb E} \bigl[ \Phi\bigl( \mu^N_{t}\bigr) 
  \vert \,
\mu_{0}^N = \mu
\bigr] -
   {\mathcal U}(t,\mu)
   \\ 
&=
{\mathbb E} \bigl[ \Phi\bigl( \mu^N_{t} \bigr)   \vert \,
\mu_{0}^N = \mu \bigr] -
{\mathbb E} \bigl[ {\mathcal U}\bigl( t-  t \wedge \tau_{N} , \mu^N_{t \wedge \tau_{N}}\bigr) 
  \vert \,
\mu_{0}^N = \mu
\bigr]
+  {\mathbb E} \Bigl[ 
{\mathcal U}\bigl( t-  t \wedge \tau_{N} , \mu^N_{t \wedge \tau_{N}}\bigr)
- {\mathcal U}\bigl(t,\mu^N_{0}\bigr)  \vert \,
\mu_{0}^N = \mu
\Bigr]. \nonumber
%+
% {\mathbb E} \Bigl[ 
% {\mathcal U}\bigl(t,\mu^N_{0}\bigr) 
% -
%  {\mathcal U}(t,\mu)   \vert \,
%\mu_{0}^N = \mu
%\Bigr].
\end{align}
If we restrict ourselves to $t \in [0,\exp(N^{1/2})]$, then 
\eqref{eq:Kuramoto:exit:time}
says that the first difference in the right-hand side is bounded by $C/N$, by recalling that $\Phi$ and thus ${\mathcal U}$ are bounded 
and allowing $C$ to depend on $\| \Phi \|_{\infty}$. 

As for the second term in the right-hand side of \eqref{eq:expand:phi:u:stopping time}, we may follow
\eqref{eq:second:term:main result intro formula}, 
by expanding the 
term 
$({\mathcal U}( t-  s, \mu^N_{s}))_{s \geq 0}$ up until
time $t \wedge \tau_{N}$.
%In order to do so, it is crucial to notice that the expansion 
%is a mere consequence of It\^o's formula and thus holds regardless of the choice of the initial condition in the particle system. 
%In clear, there is no need to take $\bar \mu_{0}^N$ as the empirical distribution of an $N$-sample  
%from a common distribution. 
Using the fact that, up until time $\tau_{N}$,  
the $\| \cdot \|_{-1,2}$-distance between 
$\mu^N$ and ${\mathcal I}$
remains less than $\varepsilon$ and 
assuming w.l.o.g. that 
$\varepsilon$ is 
small enough so that 
$\{ \mu : 
\textrm{\rm dist}_{\| \cdot \|_{-1,2}}( \mu,{\mathcal I})
\leq \varepsilon  \}
\subset {\mathcal Q}_{1/2}$, we get from 
Proposition \ref{corol:1} that
this term is also bounded by 
$C/N$, with $C$ now depending on $\kappa$ and 
the bounds for $\Phi$   
in {\hintphi{\gamma}{2}}.
%As for the last term in the expansion, it is zero since $\mu_{0}^N=\mu$. 
\end{proof}

The result is  extended to initial conditions 
$\mu \in {\mathcal P}_{N}(\bT) \cap {\mathcal Q}_{\eta}$, for any $\eta >0$, on a smaller time scale.

\begin{proposition}
\label{prop:5:11}
For any $\eta \in (0,1)$, there exists a constant $C$, only depending on $\eta$, $\kappa$ and
the bounds for $\Phi$   
in \emph{\hintphi{\gamma}{2}}, such that, for any $N \geq 1$ and 
$\mu \in {\mathcal Q}_{\eta} \cap {\mathcal P}_{N}(\bT)$, 
\begin{equation}
\label{eq:propagation:chaos:-3-00}
\forall t \in [0,\exp(N^{1/4})], \quad \bigl\vert \overline{\mathcal U}^N(t,\mu) - {\mathcal U}(t,\mu) \bigr\vert \leq \frac{C}{N}. 
\end{equation}
\end{proposition}

\begin{proof}
\textit{First step.}
Lemma 
\ref{lem:kura:expN1/2} says that, for $\delta$ as therein, the result holds  up until
time $\exp(N^{1/2})$ if
$\textrm{\rm dist}_{\| \cdot \|_{-1,2}}(\mu,{\mathcal I}) \leq \delta$. 
Then, for any $\mu \in {\mathcal Q}_{\eta} \cap {\mathcal P}_{N}(\bT)$, 
Proposition \ref{prop:5:10} says that  there exists a (fixed hence independent of $N$) time $T$, only depending on $\delta$, $\eta$ and $\kappa$, such that 
\begin{equation}
\label{eq:propagation:chaos:-1:-000}
\forall t \geq T, \quad \textrm{\rm dist}_{\| \cdot \|_{-1,2}}\bigl(m(t\, ; \mu),{\mathcal I}\bigr) \leq \frac{\delta}2. 
\end{equation}
Next, we claim that (the proof is given in the third step below)
\begin{equation}
\label{eq:propagation:chaos:-1}
{\mathbb P} 
\Bigl( 
\bigl\| m(T \, ; \mu) -  \mu^N_{T}
\bigr\|_{-1,2}
> \frac{\delta}2
\, \big\vert \,
\mu_{0}^N = \mu
\Bigr) 
 \leq \frac{C}{N},
\end{equation}
where $C$ 
only depends on $\eta$, $\kappa$ and 
the bounds for $\Phi$   
in  {\hintphi{\gamma}{2}}, from which we deduce that%, for $N$ large enough, 
\begin{equation}
\label{eq:propagation:chaos:-1-00}
{\mathbb P}
\Bigl( 
\textrm{\rm dist}_{\| \cdot \|_{-1,2}}\bigl(  \mu^N_{T},{\mathcal I}\bigr) > \delta 
\, \big\vert \,
\mu_{0}^N = \mu
\Bigr) \leq \frac{C}{N}.
\end{equation}
%Moreover,
%\begin{equation*}
%\bigl\vert U(T,\mu) - \overline U^N(T,\mu) \bigr\vert \leq \frac{C}{N^{1/2}}, 
%\end{equation*}
%for any $\mu$. 
By 
\eqref{eq:markov:muN}, we have, for $t \geq T$,
\begin{equation*}
\overline{\mathcal U}^N(t,\mu) = {\mathbb E} \bigl[ \overline{\mathcal U}^N\bigl(t-T,  \mu^N_{T}\bigr)  \big\vert \,
\mu_{0}^N = \mu \bigr],
\end{equation*}
when $\mu^N_{0}=\mu$. 
Hence, 
by
\eqref{eq:propagation:chaos:-3-00}
(for 
$\textrm{\rm dist}_{\| \cdot \|_{-1,2}}(\mu,{\mathcal I}) \leq \delta$)
and
\eqref{eq:propagation:chaos:-1-00}, 
we get, 
for 
$t \in [T,\exp(N^{1/2})/2]$, 
\begin{equation*}
%\label{eq:propagation:chaos:-5}
\begin{split}
&\Bigl\vert \overline{\mathcal U}^N(t,\mu) 
- 
{\mathbb E} \Bigl[ {\mathcal U}\bigl(t-T, \mu^N_{T}\bigr) {\mathbf 1}_{\{ 
\textrm{\rm dist}_{\| \cdot \|_{-1,2}}( \mu^N_{T},{\mathcal I}) \leq \delta
\}} \big\vert \,
\mu_{0}^N = \mu \Bigr]
\Bigr\vert 
\\
&\leq C {\mathbb P}
\Bigl( 
\textrm{\rm dist}_{\| \cdot \|_{-1}}\bigl(  \mu^N_{T},{\mathcal I}\bigr) > \delta 
\, 
\big\vert \,
\mu_{0}^N = \mu
\Bigr)
 + 
{\mathbb E} \Bigl[ \Bigl\vert  {\overline{\mathcal U}^N}\bigl(t-T,   \mu^N_{T}\bigr)
-   {{\mathcal U}}\bigl(t-T,   \mu^N_{T}\bigr)
\Bigr\vert {\mathbf 1}_{\{ 
\textrm{\rm dist}_{\| \cdot \|_{-1,2}}( \mu^N_{T},{\mathcal I} ) \leq \delta
\}}
\big\vert \,
\mu_{0}^N = \mu
\Bigr]
\\
&\leq \frac{C}{N},
\end{split}
\end{equation*}
where we used the fact that $\overline {\mathcal U}^N$ is bounded by $\| \Phi \|_{\infty}$. 
On the first line,  
%of \eqref{eq:propagation:chaos:-5}, 
%we can always write that 
%\begin{equation*}
%\begin{split}
%{\mathbb E} \bigl[ {\mathcal U}\bigl(t-T, \mu^N_{T}\bigr) {\mathbf 1}_{\{ 
%\textrm{\rm dist}_{\| \cdot \|_{-1}}( \mu^N_{T},{\mathcal I}) \leq \delta
%\}}\bigr]
%&= 
%{\mathbb E} \bigl[ {\mathcal U}\bigl(t-T, \mu^N_{T}\bigr) {\mathbf 1}_{\{ 
%\textrm{\rm dist}_{\| \cdot \|_{-1}}( \mu^N_{T},{\mathcal I}) \leq \delta
%\}} \Xi(\mu^N_{T}) \bigr],
%\end{split}
%\end{equation*}
we can replace the indicator function
(inside the  expectation)  
by $\Xi(\mu_{T}^N)$ for some function $\Xi : {\mathcal P}(\bT) \rightarrow [0,1]$ that is equal to $1$ on the ball
$\{ \nu \in {\mathcal P}(\bT) : \textrm{\rm dist}_{\| \cdot\|_{-1,2}}( \nu ,{\mathcal I} ) \leq \delta\}$.
By 
\eqref{eq:propagation:chaos:-1-00} again, we deduce that, for
$\mu \in {\mathcal Q}_{\eta} \cap {\mathcal P}_{N}(\bT)$, 
\begin{equation}
\label{eq:propagation:chaos:-1-01}
\Bigl\vert \overline{\mathcal U}^N(t,\mu) 
- 
{\mathbb E} \bigl[ {\mathcal U}\bigl(t-T, \mu^N_{T}\bigr) 
\Xi(\mu^N_{T})  \vert \,
\mu_{0}^N = \mu
\bigr]
\Bigr\vert 
\leq \frac{C}N. 
\end{equation}

\textit{Second Step.}
{We complete the proof of 
\eqref{eq:propagation:chaos:-3-00}, recalling that the latter is already known to hold 
on $[0,T]$ (since there is no need for any ergodic estimates in finite time)}. In order to proceed, assume for a while that we are given a function $\Xi : {\mathcal P}({\mathbb T}) \rightarrow [0,1]$ such that %provided $\delta$ is small enough (which 
%is not a hindrance for us), 
\begin{enumerate}
\item 
$\Xi$ matches 1 on  ${\mathcal Q}_{1/2}$ (the latter containing 
$\{ \nu \in {\mathcal P}(\bT) : \textrm{\rm dist}_{\| \cdot\|_{-1,2}}( \nu ,{\mathcal I} ) \leq \delta\}$) and $0$ outside   
${\mathcal Q}_{3/4}$;
\item 
$\Xi$ is a smooth functional of $\mu \in {\mathcal P}(\bT)$. 
%$\sup_{m \in {\mathcal P}({\mathbb T})}
%\sup_{\| q \|_{-2,\infty} \leq 1}
%[\delta \Xi_{n}/\delta m](m,q)$ and 
%$\sup_{m \in {\mathcal P}({\mathbb T})}
%\sup_{\| q_{1} \|_{-1,\infty},\| q_{1} \|_{-1,\infty} \leq 1}
%[\delta^2 \Xi_{n}/\delta m^2](m,q_{1},q_{2})$
%are bounded, uniformly in $n \geq 1$, by a constant $C$.  
\end{enumerate}

Then, since ${\mathcal Q}_{1/2}$ contains 
$\{ \nu \in {\mathcal P}(\bT) : \textrm{\rm dist}_{\| \cdot \|_{-1,2}}( \nu ,{\mathcal I} ) \leq \delta\}$, we
can apply 
\eqref{eq:propagation:chaos:-1-01}. 
Since $\Xi$ is zero outside 
${\mathcal Q}_{3/4}$, 
we know from Proposition
\ref{corol:1} that ${\mathcal U}(t-T,\cdot) \Xi(\cdot)$
satisfies 
\hintphi{\gamma}{2} with explicit bounds that 
are uniform 
with respect to $t \geq T$. 
We then apply 
the finite-horizon version  of 
\eqref{eq:propagation:chaos:-3-00}
to $\Phi(\cdot) = {\mathcal U}(t-T,\cdot) \Xi(\cdot)$. 
%This only requires the coefficients to be smooth, without any further long time analysis of the tangent process. 
%The proof mostly consists in expanding the analogue of the second term in the right-hand side of 
%\eqref{eq:expand:phi:u:stopping time}, with $\tau^N$ being formally equal to $t$ here; recall that the analogue of the third term 
%in the right-hand side should be regarded as $0$ since 
%$\mu_{0}^N=\mu$ in our case. 
By 
\eqref{eq:propagation:chaos:-1-01}, we get, for any $t \geq T$
and 
$\mu \in {\mathcal Q}_{\eta} \cap {\mathcal P}_{N}(\bT)$,
\begin{equation*}
\Bigl\vert \overline{\mathcal U}^N(t,\mu) 
- 
 {\mathcal U}\bigl(t-T, m(T \, ; \mu)\bigr) 
\Xi\bigl(m(T \, ; \mu)\bigr) 
\Bigr\vert 
\leq \frac{C}N, 
\end{equation*}
for $C$ only depending
on
$\eta$, $\kappa$ and 
the bounds for $\Phi$   
in  {\hintphi{\gamma}{2}}. 
Noticing that 
${\mathcal U}(t-T, m(T \, ; \mu)) 
= {\mathcal U}(t , \mu)$ 
and that $\Xi(m(T \, ; \mu))=1$ for our choice of $T$ (see 
\eqref{eq:propagation:chaos:-1:-000}), we 
get the announced result.  

The function $\Xi$ is constructed as in  Proposition 
\ref{prop:choose:Phi:kuramoto}. We consider a smooth non-decreasing cut-off function $\varphi : [0,1] \rightarrow [0,1]$ that is equal to $1$ on 
$[0,1/2]$ and to $0$ on $[3/4,1]$. 
We then let $\Xi(\mu) = \varphi(\vert \mu^1 \vert)$. 
\vskip 4pt

\textit{Third Step.} 
We now prove
\eqref{eq:propagation:chaos:-1}.
{It is again a consequence of the finite time horizon of
\eqref{eq:propagation:chaos:-3-00}, 
but with a difference of choosing (temporarily) 
 the functional $\Phi$ in the definition of ${\mathcal U}$ as in  
 Proposition
\ref{prop:4:phi:norm:-d} with 
$(d+\alpha)/2=1$ and $\nu_{0}=\mu_{T}$ therein. 
The result then follows from Markov's inequality.}
\end{proof}

Our last step is to extend the previous result to times greater than $\exp(N^{1/4})$. The key idea is that, in long time, the empirical measure necessarily visits  the 
set ${\mathcal Q}_{\eta}$ quite often, for any $\eta \in (0,1)$. 
\begin{proposition}
\label{prop:uniform:N}
For any $\eta \in (0,1)$, there exists a constant $C$, only depending on $\eta$,  $\kappa$ and the bounds for $\Phi$   
in \emph{\hintphi{\gamma}{2}}, such that, for any $N \geq 1$ and 
$\mu \in {\mathcal Q}_{\eta} \cap {\mathcal P}_{N}(\bT)$, 
\begin{equation}
\label{eq:propagation:chaos:-3}
\forall t \geq 0, \quad \bigl\vert \overline{\mathcal U}^N(t,\mu) - {\mathcal U}(t,\mu) \bigr\vert \leq \frac{C}{N}. 
\end{equation}
\end{proposition}
\begin{proof}
By Proposition \ref{prop:5:11}, it suffices to prove \eqref{eq:propagation:chaos:-3} for $t \geq \exp(N^{1/4})$. 
For $T_{N} = \exp(N^{1/4})/2$, we call 
\begin{equation*}
\tau_{N} := \inf\biggl\{ s > 0 : \biggl\vert \int_{\mathbb T} e^{-i 2 \pi \theta}   \mu^N_{s + t -T_{N}}(\ud \theta) \biggr\vert \geq 
\eta
\biggr\}.
\end{equation*}
We prove in Lemma \ref{lem:exit:time:cos} that 
${\mathbb P} \bigl( \tau_{N}\geq N^{1/4} \bigr) \leq C/{N}$, 
for $C$ as in the statement and {for $\eta$ small enough}. 
By the strong version of the Markov property
\eqref{eq:markov:muN} (noticing that $N^{1/4} \leq \exp(N^{1/4})/4$ for $N$ large enough), 
\begin{equation*}
\overline {\mathcal U}^N(t,\mu) = {\mathbb E}\Bigl[ \overline {\mathcal U}^N \Bigl( T_{N} - T_{N} \wedge \tau_{N} , \mu^N_{t - T_{N}+T_{N}\wedge \tau_{N}} \Bigr) \big\vert \, 
\mu_{0}^N = \mu \Bigr], 
\end{equation*}
and then, {assuming without any loss of generality that $\eta$ is small enough}, 
\begin{equation*}
\begin{split}
&\Bigl\vert \overline {\mathcal U}^N(t,\mu) 
-
{\mathbb E}\Bigl[  \overline {\mathcal U}^N \Bigl( T_{N} - T_{N} \wedge \tau_{N} ,  \mu^N_{t - T_{N}+T_{N}\wedge \tau_{N}} \Bigr) {\mathbbm 1}_{\{ \tau_{N} < T_{N} / 2  \}} \big\vert \, 
\mu_{0}^N = \mu \Bigr] 
\Bigr\vert
\leq \frac{C}{N}. 
\end{split}
\end{equation*}
On the event $\{\tau_{N} < T_{N}/2\}$, 
$\mu^N_{t - T_{N}+T_{N}\wedge \tau_{N}}  \in {\mathcal Q}_{\eta}$. 
Therefore, by Proposition 
\ref{prop:5:11}, 
we 
can replace $\overline{\mathcal U}^N$ by ${\mathcal U}$
inside the expectation
and hence get 
\begin{equation}
\label{end:proof:main:statement:-0}
\begin{split}
&\Bigl\vert \overline {\mathcal U}^N(t,\mu) 
-
{\mathbb E}\Bigl[  {\mathcal U} \Bigl( T_{N} - T_{N} \wedge \tau_{N} ,\mu^N_{t - T_{N}+T_{N}\wedge \tau_{N}} \Bigr) {\mathbbm 1}_{\{ \tau_{N} < T_{N} /2  \}} 
\big\vert \, 
\mu_{0}^N = \mu
\Bigr] 
\Bigr\vert
\leq \frac{C}{N}. 
\end{split}
\end{equation}
By Proposition \ref{prop:5:10},
we know that, for $s \geq T_{N}/2$ and $\nu \in {\mathcal Q}_{\eta}$, 
\begin{equation}
\label{end:proof:main:statement}
\begin{split}
{\mathcal U}(s,\nu) =   \Phi\bigl( {\mathcal L}(X_{s} \vert X_{0} \sim \nu) \bigr) 
&= \Phi({\mathcal I}) + O \bigl( \frac1N \bigr)
=  {\mathcal U}(t,\mu) + O \bigl( \frac1N \bigr),
\end{split}
\end{equation}
with the Landau symbol ${\mathcal O}( \cdot)$ being independent of 
$\nu$, 
from which we
deduce that we can replace the expectation 
in 
\eqref{end:proof:main:statement:-0}
by ${\mathcal U}(t,\mu)$. This completes the proof. 
In the above, we used the slightly abusive notation 
${\mathcal L}(X_{s} \vert X_{0} \sim \nu)$
to denote the law of $X_{s}$ in \eqref{eq:MVSDE} when 
the law of $X_{0}$ is $\nu$. 
\end{proof}

\subsection{End of the proof of  Theorem \ref{main:thm:kuramoto}}
%We now have all the ingredients to complete the proof of 
%Theorem \ref{main:thm:kuramoto}.
For an $N$-sample $(Y^{i,N}_{0})_{i=1,\cdots,N}$
with law ${\mu_{\text{init}}} \in {\mathcal Q}_{\eta}$, for some $\eta >0$, we let 
$\mu_{0}^N = (1/N) \sum_{i=1}^N \delta_{Y_{0}^{i,N}}$. 
\vskip 4pt

\begin{proof}[Proof of Theorem \ref{main:thm:kuramoto}]
\textit{First Step.}
We start with the following (quite standard) computation: 
\begin{equation}
\label{eq:LLN:H-1}
\begin{split}
{\mathbb E} \Bigl[ \|  {\mu_{\text{init}}} - \mu^N_{0} \|_{-1,2}^2 \Bigr]
&= \sum_{n \geq 0}
 \frac{1}{(1+n^2)} {\mathbb E}
 \biggl[ \biggl\vert \int_{\bT} e^{-\i 2\pi n \theta}   \bigl( {\mu_{\text{init}}}  - \mu^N_{0} \bigr) (\ud \theta) 
 \biggr\vert^2 \biggr] 
 \leq \frac{1}{N}
 \sum_{n \geq 0}
 \frac{1}{(1+n^2)} \leq \frac{c}{N},
\end{split}
\end{equation}
for some universal constant $c \geq 0$.
%, where we used the  fact that, for any bounded measurable 
%function $f : \bT \rightarrow \bR$,  
%\begin{equation*}
%\begin{split}
%&{\mathbb E}
%\biggl[ \biggl\vert 
% \int_{\bT} f(\theta) \bigl( \textcolor{blue}{\mu_{\text{init}}} - \mu^N_{0} \bigr) (\ud \theta) \biggr\vert^2 \biggr]
% = \frac1N \int_{\bT} \biggl\vert f(\theta) - \int_{\bT} f(\theta') 
%  \textcolor{blue}{\mu_{\text{init}}
%(\ud  \theta')} \biggr\vert^2   \textcolor{blue}{\mu_{\text{init}}
%(\ud \theta)} 
% \leq \frac{\| f \|_{\infty}^2}N. 
% \end{split}
%\end{equation*}
We deduce that, for any $\varrho >0$, 
\begin{equation*}
%\label{eq:weak:weak:error:step:1:kura}
{\mathbb P} \Bigl( \bigl\|  {\mu_{\text{init}}} - \mu^N_{0}  \bigr\|_{-1,2} \geq \varrho \Bigr) 
\leq \frac{c}{\varrho^2 N}.
\end{equation*} 
Now, we choose $\varrho$ such that, for any two probability measures
$\nu_{1},\nu_{2} \in {\mathcal P}(\bT)$
with 
$\| \nu_{1} - \nu_{2} \|_{-1,2} \leq \varrho$, it holds that 
$\vert \nu_{1}^1 - \nu_{2}^1 \vert \leq \eta/2$ (we recall that 
$\nu_{1}^1$ and $\nu_{2}^1$ are the 1-Fourier modes of $\nu_{1}$
and $\nu_{2}$), from which we get that, for a constant $C$ depending on 
$\eta$, 
\begin{equation*}
%\label{eq:weak:weak:error:step:1:kura}
{\mathbb P} \Bigl(  \mu^N_{0} {\not \in} {\mathcal Q}_{\eta/2}\Bigr) 
\leq \frac{C}{N}. 
\end{equation*} 
Therefore, by Proposition 
\ref{prop:uniform:N}, 
%\begin{equation*}
%\forall t \geq 0, \quad {\mathbb P} \biggl( \bigl\vert \overline{\mathcal U}^N\bigl(t,\mu_{0}^N\bigr) - {\mathcal U}\bigl(t,\mu_{0}^N\bigr) \bigr\vert \leq \frac{C}{N} \biggr) \geq 1 - \frac{C}{N},
%\end{equation*}
there exists a constant $C$, only depending on $\kappa$, $\eta$
and
the bounds for $\Phi$   
in  {\hintphi{\alpha}{2}}, such that
%as a result of which we obtain 
\begin{equation*}
\forall t \geq 0, \quad 
\Bigl\vert 
{\mathbb E} \bigl[ \overline{\mathcal U}^N\bigl(t, \mu_{0}^N \bigr) \bigr] - 
{\mathbb E} \bigl[ {\mathcal U}\bigl(t, \mu_{0}^N \bigr) 
{\mathbf 1}_{\{ \mu_{0}^N  \in {\mathcal Q}_{\eta/2}\} }
\bigr]
\Bigr\vert \leq \frac{C}N.
\end{equation*}
%It thus remains to show that, for any $t \geq 0$, 
%\textcolor{red}{il faut v\'erifier que la mesure empirique du $N$-sample est bien dans ${\mathcal Q}_{\eta}$.}
%\begin{equation*}
%\Bigl\vert 
%{\mathcal U}\bigl(t,{\mathcal L}(X_{0}) \bigr) - {\mathbb E} \bigl[ 
%{\mathcal U}\bigl(t,\mu_{0}^N \bigr)
%\bigr] \Bigr\vert
%=
%\biggl\vert 
%{\mathcal U}\bigl(t,{\mathcal L}(X_{0}) \bigr) - {\mathbb E} \Bigl[ {\mathcal U} \Bigl(t, \frac1N \sum_{i=1}^N \delta_{X_{0}^i} \Bigr)
%\Bigr] \biggr\vert \leq \frac{C}{N},
%\end{equation*}
%%which we rewrite, with obvious notations, in the form 
%%\begin{equation*}
%%\Bigl\vert U \bigl(t, \mu) - {\mathbb E} \bigl[ U \bigl( t, \bar \mu^N \bigr) \bigr] \Bigr\vert,
%%\end{equation*}
%%When $t \leq T$, 
%%for any fixed $T>0$, 
%%this is a consequence of the proof of Theorem 2.11 in  
%%\cite{chassagneux2019weak} (see for instance (2.21) therein).
%Following 
%\eqref{eq:propagation:chaos:-1}, we have
%\begin{equation*}
%{\mathbb P} 
%\Bigl( 
%\bigl\| \mu -  \mu^N_{0}
%\bigr\|_{-1}
%> \delta
%\Bigr) 
% \leq \frac{C}{N},
%\end{equation*}
%for a given $\delta>0$ that is chosen 
%as follows.
%We can indeed choose $\delta$ such that, for any two probability measures
%$\nu_{1},\nu_{2} \in {\mathcal P}(\bT)$
%with 
%$\| \nu_{1} - \nu_{2} \|_{-1} \leq \delta$, it holds that 
%$\vert \nu_{1}^1 - \nu_{2}^1 \vert \leq \eta$ (we recall that 
%$\nu_{1}^1$ and $\nu_{2}^1$ are the mode 1 Fourier coefficients of $\nu_{1}$
%and $\nu_{2}$), from which we deduce that 
%\begin{equation*}
%{\mathbb P}
%\Bigl( \mu^N_{0} \in {\mathcal Q}_{2\eta} \Bigr) \geq 1 - \frac{C}{N}.
%\end{equation*}
We deduce that it suffices to show that 
\begin{equation}
\label{eq:conclusion:kuramoto:main:model}
\Bigl\vert 
{\mathbb E} \Bigl[ 
\Bigl( 
{\mathcal U}(t,{\mu_{\text{init}}}) -  
 {\mathcal U}\bigl(t, \mu_{0}^N \bigr)
 \Bigr) 
{\mathbf 1}_{\{ \mu_{0}^N  \in {\mathcal Q}_{\eta/2}\} }
\Bigr] \Bigr\vert \leq \frac{C}{N}.
\end{equation}

\textit{Second Step.} In order to prove 
\eqref{eq:conclusion:kuramoto:main:model}, we may argue as in the second step of the proof of 
Proposition 
\ref{prop:5:11}.
Indeed, we can consider a smooth function $\Xi : {\mathcal P}({\mathbb T}) \rightarrow [0,1]$ such that 
$\Xi$ is 1 on the set ${\mathcal Q}_{\eta/2}$ and 0 outside the set  
${\mathcal Q}_{\eta/4}$. 
Then, instead of proving 
\eqref{eq:conclusion:kuramoto:main:model}, it suffices to show
that 
\begin{equation}
\label{eq:conclusion:kuramoto:main:model:2}
\Bigl\vert 
{\mathbb E} \Bigl[ 
{\mathcal U}(t,{\mu_{\text{init}}}) \Xi({\mu_{\text{init}}}) -  
 {\mathcal U}\bigl(t, \mu_{0}^N \bigr)
 \Xi \bigl( \mu_{0}^N \bigr)
 \Bigr] \Bigr\vert \leq \frac{C}{N}.
\end{equation}
Thanks to the cut-off function $\Xi$, 
the function $(t,\mu) \mapsto 
{\mathcal U}(t,\mu)\Xi(\mu)$ satisfies the conclusion 
of Corollary 
\ref{corol:1}, even though $\mu \not \in {\mathcal Q}_{\eta}$. 
This suffices to 
apply  
\eqref{eq:second:term:main result intro formula:2}
with ${\mathcal U}(t,{\mu_{\text{init}}}) \Xi({\mu_{\text{init}}})$ instead of ${\mathcal U}(t,{\mu_{\text{init}}})$ therein.
%
%Because of the cut-off function $\Xi$, 
%the function $(t,\mu) \mapsto {\mathcal U}(t,\mu) \Xi(\mu)$
%satisfies 
%the conclusions of Corollary 
%\ref{corol:1}, even though $\mu \in {\mathcal Q}_{\eta}$. 
%
%
%we have a bound for 
%$\delta^2 [{\mathcal U}(t,\cdot) \Xi]/\delta m^2$ that is uniform in $t$ and that is global in space (whether the input belongs to ${\mathcal Q}_{\eta}$ or not) thanks to the truncation 
%function $\Xi$. 
%Then, we may 
%invoke 
%\eqref{consequence master equation}
%but with time $t$ therein being equal to $0$ and  
%with $\Phi(\mu) = {\mathcal U}(t,\mu) \Xi(\mu)$, the 
%latter $t$ being the same as in 
%\eqref{eq:conclusion:kuramoto:main:model:2}.
\end{proof}

The proof of Theorem \ref{main:thm:kuramoto} is hence completed provided that we prove
the following lemma, which we invoked in the proof of Lemma 
\ref{prop:uniform:N}:
\begin{lemma}
\label{lem:exit:time:cos}
There exist a constant $\eta \in (0,1)$ and constant $C$, both independent of $N$, such that, 
for any initial distribution $\mu \in {\mathcal P}(\bT)$
and any $t \geq 0$, the distribution of the stopping time
%\begin{equation*}
$\tau_{N} := \inf\{ s > 0 :  \vert \int_{\mathbb T} e^{- \i 2 \pi \theta}  \mu^N_{s + t }(\ud \theta)  \vert \geq 
\eta \}$
%\end{equation*}
satisfies
${\mathbb P} \bigl( \tau_{N}\geq N^{1/4} \bigr) \leq C /N$. 
\end{lemma}

\begin{proof}
{Without any loss of generality, we can assume that 
$t=0$.}
\vskip 4pt

\textit{First Step.} We go back to the shape of  particle system \eqref{eq particles} with $b$ as in 
\eqref{eq:b:kuramoto}:
\begin{equation*}
\begin{split}
\ud  Y_{t}^{j,N} &=   - \frac{2 \pi \kappa}{N} \sum_{k=1}^N  \sin\Bigl( 2 \pi \bigl(  Y_{t}^{j,N} -   Y_{t}^{k,N} \bigr) \Bigr) 
\ud t +  \ud  W_{t}^j, \quad t \geq 0.
\end{split}
\end{equation*}
Let 
$E_{t}^{j,\ell,N} = \exp(\i 2 \pi   \ell Y_{t}^{j,N})$. 
Then, 
recalling the notation $\overline{z}$ for denoting the complex conjugate of a complex number $z \in {\mathbb C}$, we obtain that 
\begin{equation*}
\begin{split}
\ud 
E_{t}^{j,\ell,N}
&= 
 - \frac{2 \pi^2 \kappa \ell}{N}
E_{t}^{j,\ell,N}
\sum_{k=1}^N   
\Bigl( 
E_{t}^{j,1,N} \overline E_{t}^{k,1,N}
- 
\overline E_{t}^{j,1,N} E_{t}^{k,1,N} \Bigr) dt
- 2 \pi^2 
\ell^2
E_{t}^{j,\ell,N}
\ud t 
 +
\i 2   \pi   \ell  E_{t}^{j,\ell,N} \ud W_{t}^j
 \\
 &= 
 - 2 \pi^2 \kappa \ell
\Bigl(E_{t}^{j,\ell+1,N}
\mu_{t}^{1,N}
- 
\overline  E_{t}^{j,\ell-1,N}
  \overline \mu_{t}^{1,N}
 \Bigr) \ud t
- 2 \pi^2 
\ell^2
E_{t}^{j,\ell,N}
\ud t 
 +
 \i 2  \pi   \ell  E_{t}^{j,\ell,N} \ud W_{t}^j,
\end{split}
\end{equation*}
where 
$\mu_{t}^{\ell,N} = N^{-1} \sum_{j=1}^N \overline E_{t}^{j,\ell,N}$
is the $\ell$-Fourier mode of $\mu_{t}^N$. 
Taking the mean over $j \in \{1,\cdots,N\}$, we get 
\begin{equation*}
\begin{split}
\ud 
\overline \mu_{t}^{\ell,N}
&= 
 -2 \pi^2 \kappa \ell
\Bigl(
\overline \mu_{t}^{\ell+1,N}
 \mu_{t}^{1,N}
- 
\mu_{t}^{\ell-1,N}
  \overline \mu_{t}^{1,N}
 \Bigr) \ud t
- 2 \pi^2 
\ell^2
\overline \mu_{t}^{\ell,N}
\ud t 
 +
 \i 2   \pi   \ell  \frac1{N} \sum_{j=1}^N  E_{t}^{j,\ell,N} \ud W_{t}^j.
\end{split}
\end{equation*}
Choosing $\ell=1$ and recalling that $\kappa >1$, we  {can find} a constant $c>1$, only depending on $\kappa$, such that 
\begin{equation}
\label{eq:mut1}
\ud \bigl[ \vert 
\mu_{t}^{1,N} \vert^2
\bigr]
= \bigl( c^{-1} - c \vert \mu_{t}^{2,N}\vert \bigr) \vert \mu_{t}^{1,N} \vert^2  \ud t 
+ \frac{c^{-1}}{{N}} \ud t + \ud K_{t}^1 + \ud M_{t}^1,
\end{equation}
where 
$(K_{t}^1)_{t \geq 0}$
is a non-decreasing absolutely continuous process and 
$(M_{t}^1)_{t \geq 0}$ is a martingale satisfying 
$[\ud/\ud t] \langle M^1 \rangle_{t} \leq c/N$. 
Similarly, 
choosing $\ell=2$, we get
\begin{equation}
\label{eq:mut2}
\ud \bigl[ \vert 
\mu_{t}^{2,N} \vert^2
\bigr]
=
\bigl(c \vert 
\mu_{t}^{1,N} \vert
\vert 
\mu_{t}^{2,N} \vert
-
c^{-1} \vert 
\mu_{t}^{2,N} \vert^2
\bigr)
 \ud t 
 - \ud K_{t}^2
+ \frac{c}{ {N}} \ud t + \ud M_{t}^2,
\end{equation}
where 
$(K_{t}^2)_{t \geq 0}$
is a non-decreasing absolutely continuous  process and 
$(M_{t}^2)_{t \geq 0}$ is a martingale satisfying 
$[\ud/\ud t]  \langle M^2 \rangle_{t} \leq c/N$. 
We now let
\begin{equation}
\label{eq:LambdatN}
\Lambda_{t}^N = \vert 
\mu_{t}^{2,N} \vert^2 - c^4 \vert 
\mu_{t}^{1,N} \vert^2, \quad t \geq 0. 
\end{equation}
Using the expansions \eqref{eq:mut1}
and 
\eqref{eq:mut2}, we obtain
\begin{equation}
\label{eq:exp:Lambda:t:N}
\begin{split}
\ud \Lambda_{t}^N  &= 
\Bigl[ \bigl(c \vert 
\mu_{t}^{1,N} \vert
\vert 
\mu_{t}^{2,N} \vert
-
c^{-1} \vert 
\mu_{t}^{2,N} \vert^2
\bigr)
- c^4 \bigl( c^{-1} - c \vert \mu_{t}^{2,N}\vert \bigr) \vert \mu_{t}^{1,N} \vert^2 \Bigr]  \ud t 
- \frac{c^{3}- c}{{N}} \ud t 
  - \ud K_{t} + \ud M_{t},
\end{split}
\end{equation}
where $(K_{t})_{t \geq 0}$ is a non-decreasing absolutely continuous  process and 
$(M_{t})_{t \geq 0}$ 
is a martingale satisfying 
$(\ud / \ud t) \langle M \rangle_{t} \leq C(c)/N$, 
in which $C(c)$ is a constant that only depends on $c$. 
We then consider the same stopping time
$\tau_{N}$ 
as in 
the statement, but with $\eta =c^{-4}/4$.
%\begin{equation*}
%\begin{split}
%&\tau_{N} = \inf \bigl\{ t \geq 0 : \vert \mu^{1,N}_{t} \vert \geq \frac14 c^{-4} \bigr\}.
%%\\
%%&\tau_{N}^2 = \inf \bigl\{ t \geq 0 : \vert \mu^{2,N}_{t} \vert \leq 2 c^2 \vert \mu^{1,N}_{t} \vert \bigr\}.
%\end{split}
%\end{equation*}
As long as $t \leq \tau_{N}$, we have (notice that the term below is nothing but the first term in the expansion of 
$\ud \Lambda_{t}^N$)
\begin{equation*}
\begin{split}
& \bigl(c \vert 
\mu_{t}^{1,N} \vert
\vert 
\mu_{t}^{2,N} \vert
-
c^{-1} \vert 
\mu_{t}^{2,N} \vert^2
\bigr)
- c^4 \bigl( c^{-1} - c \vert \mu_{t}^{2,N}\vert \bigr) \vert \mu_{t}^{1,N} \vert^2
\\
&= 
- c^3 
\vert \mu_{t}^{1,N} \vert^2
- 
c^{-1} \vert 
\mu_{t}^{2,N} \vert^2
+ 
c \vert 
\mu_{t}^{1,N} \vert
\vert 
\mu_{t}^{2,N} \vert
+
c^5  \vert \mu_{t}^{2,N}\vert   \vert \mu_{t}^{1,N} \vert^2
\\
&\leq 
- c^3 
\vert \mu_{t}^{1,N} \vert^2
- 
c^{-1} \vert 
\mu_{t}^{2,N} \vert^2
+ 
\frac{5}4 c  \vert \mu_{t}^{1,N}\vert   \vert \mu_{t}^{2,N} \vert
\\
&= 
- \frac{3}{8} c^{-1}
\Bigl( 
\vert \mu_{t}^{2,N} \vert^2
+ c^4
\vert \mu_{t}^{1,N} \vert^2
\Bigr) 
- \frac{5}8
 \Bigl(
 c^{3/2} 
\vert \mu_{t}^{1,N} \vert
- 
 c^{-1/2} \vert \mu_{t}^{2,N} \vert \Bigr)^2
 \leq  - \frac{3}{8} c^{-1} \Lambda_{t}^N. 
\end{split}
\end{equation*}
By modifying the definition of 
$(K_{t})_{t \geq 0}$ in \eqref{eq:exp:Lambda:t:N} and by assuming  that $c^3 \geq 2c$, we then get 
\begin{equation*}
\begin{split}
\ud \Lambda_{t}^N  
= - \frac{3}{8} c^{-1} \Lambda_{t}^N
\ud t - \frac{c}{N} \ud t - \ud K_{t} + \ud M_{t}. 
\end{split}
\end{equation*}
Therefore,
\begin{equation}
\begin{split}
\Lambda_{t}^N 
%&= 
%\exp \Bigl( - \frac38 c^{-1} t \Bigr) 
%\biggl[ \Lambda_{0}^N - \frac{c}{N} \int_{0}^t  
%\exp \Bigl( \frac38 c^{-1} s \Bigr) 
%\ud s 
%- \int_{0}^t  
%\exp \Bigl( \frac38 c^{-1} s \Bigr) \ud K_{s}
%+ \int_{0}^t  
%\exp \Bigl( \frac38 c^{-1} s \Bigr) \ud M_{s}
%\biggr]
%\\
&\leq 
\exp \Bigl( - \frac38 c^{-1} t \Bigr) 
\biggl[ \Lambda_{0}^N  + \int_{0}^t  
\exp \Bigl( \frac38 c^{-1} s \Bigr) \ud M_{s}
\biggr].
\end{split}
\label{eq:upper:bound:Lambda_t}
\end{equation}
We now observe that, for any integer $n \geq 0$, 
\begin{equation*}
\begin{split}
&\sup_{n \leq t \leq n+1}
\biggl[ 
\exp \Bigl( - \frac38 c^{-1} t \Bigr) 
\biggl\vert  \int_{0}^t  
\exp \Bigl( \frac38 c^{-1} s \Bigr) \ud M_{s}
\biggr\vert 
\biggr] 
 \leq  
%\exp \Bigl( - \frac38 c^{-1} n \Bigr) 
\sup_{n \leq t \leq n+1}
\biggl[
\biggl\vert  \int_{0}^t  
\exp \Bigl( \frac38 c^{-1} (s-n) \Bigr) \ud M_{s}
\biggr\vert 
\biggr].
\end{split}
\end{equation*}
Accordingly, by Burkholder-Davis-Gundy inequalities, we deduce that, for any integer $p \geq 1$, 
there exists a constant $C_{p}(c)$, depending on $p$ and $c$, such that 
\begin{equation*}
\begin{split}
&{\mathbb E} \biggl[ 
\sup_{n \leq t \leq n+1}
\biggl[ 
\exp \Bigl( - \frac38 c^{-1} t \Bigr) 
\biggl\vert  \int_{0}^t  
\exp \Bigl( \frac38 c^{-1} s \Bigr) \ud M_{s}
\biggr\vert^p 
\biggr] \biggr] 
\\
&\leq \frac{C_{p}(c)}{N^{p/2}}
\biggl( \int_{n}^{n+1} 
\exp \Bigl( \frac34 c^{-1} \bigl(s-n) \Bigr) \ud s \biggr)^{p/2}
=
\frac{C_{p}(c)}{N^{p/2}} 
\biggl( \int_{0}^{1} 
\exp \Bigl( \frac34 c^{-1} s \Bigr) \ud s \biggr)^{p/2}. 
\end{split}
\end{equation*}
In turn, by Markov's inequality, we deduce that, for any $\varepsilon >0$, 
\begin{equation*}
{\mathbb P}
\biggl( \biggl\{ 
\sup_{n \leq t \leq n+1}
\biggl[ 
\exp \Bigl( - \frac38 c^{-1} t \Bigr) 
\biggl\vert  \int_{0}^t  
\exp \Bigl( \frac38 c^{-1} s \Bigr) \ud M_{s}
\biggr\vert 
\biggr]
\geq \varepsilon
\biggr\}
\biggr) \leq \frac{C_{p}(c)}{\varepsilon^p N^{p/2}}, 
\end{equation*}
for a new value of the constant $C_{p}(c)$, and then
\begin{equation*}
{\mathbb P}
\biggl( \bigcup_{n=0}^{\lfloor N^{1/4} \rfloor} \biggl\{ 
\sup_{n \leq t \leq n+1}
\biggl[ 
\exp \Bigl( - \frac38 c^{-1} t \Bigr) 
\biggl\vert  \int_{0}^t  
\exp \Bigl( \frac38 c^{-1} s \Bigr) \ud M_{s}
\biggr\vert
\biggr]
\geq \varepsilon
\biggr\}
\biggr) \leq \frac{C_{p}(c)}{\varepsilon^p N^{p/2-1/4}}, 
\end{equation*}
where $\lfloor N^{1/4} \rfloor$ denotes the floor of $N^{1/4}$. 
Choosing $p$ large enough and using a new value of the constant $C(c)$, we end up with 
\begin{equation*}
{\mathbb P}
\biggl(\biggl\{ 
\sup_{0 \leq t \leq N^{1/4}}
\biggl[ 
\exp \Bigl( - \frac38 c^{-1} t \Bigr) 
\biggl\vert  \int_{0}^t  
\exp \Bigl( \frac38 c^{-1} s \Bigr) \ud M_{s}
\biggr\vert 
\biggr]
\geq \varepsilon
\biggr\}
\biggr) \leq \frac{C(c)}{\varepsilon^p N}. 
\end{equation*}
The value of $p$ that appears in the right-hand side is hence fixed. 
This prompts us to introduce the event
\begin{equation*}
A_{N}(\varepsilon)
:=
\biggl\{ 
\sup_{0 \leq t \leq N^{1/4}}
\biggl[ 
\exp \Bigl( - \frac38 c^{-1} t \Bigr) 
\biggl\vert  \int_{0}^t  
\exp \Bigl( \frac38 c^{-1} s \Bigr) \ud M_{s}
\biggr\vert
\biggr]
\geq \varepsilon
\biggr\}. 
\end{equation*}
On the complementary of the latter event, we have, by 
\eqref{eq:upper:bound:Lambda_t},
\begin{equation*}
\begin{split}
\forall t \in \bigl[0,N^{1/4} \wedge \tau_{N}\bigr], \quad \Lambda_{t}^N &\leq 
%\exp \Bigl( - \frac38 c^{-1} t \Bigr) 
%\biggl[ \Lambda_{0}^N  + \int_{0}^t  
%\exp \Bigl( \frac38 c^{-1} s \Bigr) \ud M_{s}
%\biggr]
%\leq 
\exp \Bigl( - \frac38 c^{-1} t \Bigr) 
 \Lambda_{0}^N + \varepsilon. 
\end{split}
\end{equation*}
Back to the definition of \eqref{eq:LambdatN}, this yields
\begin{equation*}
\forall t \in \bigl[0,N^{1/4} \wedge \tau_{N}\bigr], \quad
\vert 
\mu_{t}^{2,N} \vert^2 - c^4 \vert 
\mu_{t}^{1,N} \vert^2
\leq \exp \Bigl( - \frac38 c^{-1} t \Bigr) 
 \Lambda_{0}^N + \varepsilon,
 \end{equation*}
 that is, for all 
 $t \in [0,N^{1/4} \wedge \tau_{N}]$,
\begin{equation*}
\vert 
\mu_{t}^{2,N} \vert^2 
\leq \exp \Bigl( - \frac38 c^{-1} t \Bigr) 
 \Lambda_{0}^N + \varepsilon
 + c^4 \vert 
\mu_{t}^{1,N} \vert^2
\leq \exp \Bigl( - \frac38 c^{-1} t \Bigr)
+ \varepsilon
+  c^4 \vert 
\mu_{t}^{1,N} \vert^2,
 \end{equation*} 
 where we used the obvious inequality 
 $\Lambda_{0}^N \leq 1$. 
\vskip 4pt

\textit{Second Step.}  
 We thus introduce the following times. First, we call $t_{0}(\varepsilon)$ the smallest (deterministic) time such that 
 $\exp ( - (3/8) c^{-1} t_{0}(\varepsilon)) \leq \varepsilon$. Secondly, we let 
 \begin{equation*}
\begin{split}
&\sigma_{N}(\varepsilon) := \inf \bigl\{ t \geq {t_{0}(\varepsilon)} : \vert \mu^{1,N}_{t} \vert \geq \varepsilon c^{-4} \bigr\}.
%\\
%&\tau_{N}^2 = \inf \bigl\{ t \geq 0 : \vert \mu^{2,N}_{t} \vert \leq 2 c^2 \vert \mu^{1,N}_{t} \vert \bigr\}.
\end{split}
\end{equation*}
Then, assuming that $\varepsilon \in (0,1/4)$ and 
recalling that $\eta=c^{-4}/4$ in the definition of $\tau_{N}$, we obviously have  {$\sigma_{N}(\varepsilon) \leq \tau_{N} $ if $t_{0}(\varepsilon) \leq \tau_N$}, which implies 
$\vert 
\mu_{t}^{2,N} \vert^2 
\leq
3 \varepsilon$  {(as a consequence of the first step)},
for all $t \in [t_{0}(\varepsilon),N^{1/4} \wedge \sigma_{N}(\varepsilon) \wedge \tau_{N} ]$, 
at least if the  latter interval is not empty.
 Subsequently, plugging the latter into 
\eqref{eq:mut1}, we obtain, for $t$ in the same interval, 
 \begin{equation*}
\ud \bigl[ \vert 
\mu_{t}^{1,N} \vert^2
\bigr]
= \bigl( c^{-1} - \sqrt{3} c \varepsilon^{1/2} \bigr) \vert \mu_{t}^{1,N} \vert^2  \ud t 
+ \frac{c^{-1}}{{N}} \ud t + \ud \widetilde K_{t}^1 + \ud M_{t}^1,
\end{equation*}
for a new non-decreasing absolutely continuous process 
$(\widetilde K_{t}^1)_{t \geq 0}$. 

So far, $\varepsilon$ has been a free parameter. Now, we can choose it such that 
$c^{-1} - \sqrt{3} c \varepsilon^{1/2} =c^{-1}/2$. For this given value of $\varepsilon$ (which is now frozen in terms of $c$), 
we get 
 \begin{equation*}
\ud \bigl[ \vert 
\mu_{t}^{1,N} \vert^2
\bigr]
=  \frac{c^{-1}}2 \vert \mu_{t}^{1,N} \vert^2  \ud t 
+ \frac{c^{-1}}{{N}} \ud t + \ud \widetilde K_{t}^1 + \ud M_{t}^1,
\end{equation*}
for all $t
\in [t_{0}(\varepsilon),N^{1/4} \wedge \sigma_{N}(\varepsilon)\wedge \tau_{N} ]$. 
Next, for $t$ in the same interval, 
 \begin{equation*}
 \vert 
\mu_{t}^{1,N} \vert^2
\geq 
\exp \Bigl(  \frac{c^{-1}}2 t \Bigr) 
\biggl[ \frac{c^{-1}}{{N}} \int_{t_{0}(\varepsilon)}^t
\exp \bigl( - \frac{c^{-1}}2 s \bigr) 
\ud s +
 \int_{t_{0}(\varepsilon)}^t
\exp \bigl( - \frac{c^{-1}}2 s \bigr) 
 \ud M_{s}^1 \biggr].
\end{equation*}
In particular, for all $t \geq 0$, 
 \begin{equation*}
 \begin{split}
&{\mathbb E} \Bigl[  
\one_{\{\tau_{N} > t_{0}(\varepsilon)\}}
\bigl\vert 
\mu_{t_{0}(\varepsilon) \vee (t \wedge N^{1/4} \wedge \sigma_{N}(\varepsilon))}^{1,N} \bigr\vert^2 
\Bigr]
 \geq 
 \frac{c^{-1}}{{N}}
 {\mathbb E}
\biggl[ \one_{\{\tau_{N} > t_{0}(\varepsilon)\}}
 \int_{t_{0}(\varepsilon)}^{t_{0}(\varepsilon) \vee (t\wedge N^{1/4} \wedge \sigma_{N}(\varepsilon))}
  \exp \Bigl( \frac{c^{-1}}2 (t-s) \Bigr) 
  \ud s
  \biggr].
%  \\
%  &=  \frac{2}{{N}}
% {\mathbb E}
%\biggl[
%\one_{\{\tau_{N} > t_{0}(\varepsilon)\}}
%\biggl(
%  \exp \Bigl( \frac{c^{-1}}2 \bigl(t-t_{0}(\varepsilon) \bigr) \Bigr)
%- 
%\exp \Bigl( \frac{c^{-1}}2 \bigl(t- t_{0}(\varepsilon) \vee \bigl[t\wedge N^{1/4} \wedge \sigma_{N}(\varepsilon)\bigr] \bigr) \Bigr)\biggr)
%  \biggr].  
  \end{split}  
\end{equation*}
Choosing $t=N^{1/4}$, we obtain that
 \begin{equation*}
 \begin{split}
1
&\geq \frac{2}{{N}}
    \exp \Bigl( \frac{c^{-1}}2 \bigl(N^{1/4} -t_{0}(\varepsilon) \bigr) \Bigr)
    {\mathbb P} \Bigl( \bigl\{ \sigma_{N}(\varepsilon) \geq N^{1/4},\tau_{N} > t_{0}(\varepsilon) \bigr\} \Bigr), 
  \end{split}  
\end{equation*}
at least if $N^{1/4} \geq t_{0}(\varepsilon)$, which yields
 \begin{equation*}
    {\mathbb P} \Bigl( \bigl\{ \sigma_{N}(\varepsilon) \geq N^{1/4}, 
    \tau_{N} > t_{0}(\varepsilon) \bigr\} \Bigr)
\leq \frac{N}{2}    \exp \Bigl( - \frac{c^{-1}}2 \bigl(N^{1/4} -t_{0}(\varepsilon) \bigr) \Bigr).  
\end{equation*}
Since $\tau_{N} \geq \sigma_{N}(\varepsilon)$ if $\tau_{N} > t_{0}(\varepsilon)$, with $\eta =c^{-4}/4$ in the definition of $\tau_{N}$, this completes the proof. 
\end{proof}

\end{section}       
        
\begin{section}{Appendix}
\label{se:appendix}

\subsection{Regularization of real-valued functions defined on the space of probability measures}

We state here 
a regularisation result
for a function $\Phi$ defined 
on the space of probability measures on $\bT^d$.
The 
 proof was introduced for the first time in the arXiv version v1 of this work, see 
 \cite{DelarueTse-arXiv}. The regularisation procedure has been then reexplained in deep in 
 \cite{cecchin2022weak}. For this reason, we have just decided to give the main statement 
but to omit the proof.  

\begin{theorem}
\label{thm:regul}
Let $\Phi : {\mathcal P}(\bT^d) \rightarrow \R$ 
satisfy 
\emph{\hintphi{\alpha}{2}} for $\alpha \in [0,1)$. 
Then there exists a sequence of smooth functions $(\Phi_{n})_{n \geq 1}$, 
satisfying 
\emph{\hintphi{4}{3}}
for each $n \geq 1$, 
and converging to 
$\Phi$, uniformly on $\bT^d$, such that 
the bounds satisfied by 
$(\Phi_{n})_{n \geq 1}$ in 
\emph{\hintphi{\alpha}{2}}
are uniform in $n \geq 1$. 
\end{theorem}

\subsection{Marginal regularity of transition densities of large SDEs}
\begin{lemma}
\label{lem:annex:regularity:highd}
Consider a collection of bounded drifts $(b^i : [0,\infty) \times ({\mathbb T}^d)^n \rightarrow {\mathbb R}^d)_{i=1,\cdots,n}$, for two integers
$d$ and $n$, together with the solution ${\boldsymbol X}_t=(X_t^1,\cdots,X_t^n)_{t \geq 0}$ to the particle system
\begin{equation*}
\ud X_t^i = b^i(t,{\boldsymbol X}_t) \ud t + \ud W_t^i, \quad t \geq 0, 
\end{equation*}
for some deterministic initial condition ${\boldsymbol X}_0=(X_0^1,\cdots,X_0^n) \in ({\mathbb T}^d)^n$. For
any $t>0$, denote
 by $[{\mathbb T}^d]^n \ni {\boldsymbol x}=(x_1,\cdots,x_n) \mapsto p_t({\boldsymbol x})=p_t(x_1,\cdots,x_n)$ 
 the density of ${\boldsymbol X}_t$.

Consider a bounded measurable function $\varphi : {\mathbb T}^d \times {\mathcal P}({\mathbb T}^d) \ni (x,m) \mapsto \varphi(x,m) \in {\mathbb R}$
that 
satisfies the following two properties: 
$(i)$ 
for any $x \in {\mathbb T}^d$, 
the function 
$m \in {\mathcal P}({\mathbb T}^d) \mapsto \varphi(x,m)$ is differentiable w.r.t $m$;
$(ii)$ the function ${\mathbb T}^d \times {\mathcal P}({\mathbb T}^d) \times {\mathbb T}^d \ni (x,m,y) 
\mapsto 
[ \delta \varphi/\delta m](x,m,y)$ is bounded (so that 
$\varphi$ is Lipschitz continuous in $m$ w.r.t. $\textrm{\rm dist}_{\rm TV}$, uniformly in $x$).  

Then, for any ${\rho} \in (0,1)$, there exists a constant $c_{{\rho}}$ only depending on the parameters 
$d$, {$\rho$} and  the quantity 
$\max_{i=1,\cdots,n} 
\sup_{t \geq 0}
\sup_{m \in {\mathcal P}({\mathbb T}^d)} \| b^i(t,\cdot,m)\|_\infty$ such that, 
for all $y \in {\mathbb T}^d$
\begin{equation*}
\begin{split}
& \biggl\vert \int_{[{\mathbb T}^d]^n} 
\varphi  {\Bigl(}x_1+y,\bar \mu^{n {-(1)}}_{{\boldsymbol x}} {\Bigr)}
 \frac{p_t(x_1+y,\cdots,x_n)-p_t(x_1,\cdots,x_n)}{\vert y \vert^{ {\rho}}}
\ud x_1\cdots \ud x_n
\biggr\vert 
\\
&\leq 
c_\rho  
\Bigl( \frac1{1 \wedge t^{{\rho}/2}} + t \Bigr) 
\Bigl( 
\sup_{m \in {\mathcal P}({\mathbb T}^d)} \|\varphi(\cdot,m) \|_{ {0},\infty} 
+ 
\sup_{m \in {\mathcal P}({\mathbb T}^d)} \bigl\|\frac{\delta \varphi}{\delta m}(\cdot,m,\cdot) \|_{{0},\infty} 
\Bigr), 
\end{split}
\end{equation*}
where, in the above left-hand side, 
$\bar \mu^{n{-(1)}}_{,{\boldsymbol x}}$ is equal to $(n-1)^{-1} \sum_{i=2}^n \delta_{x_i}$.
\end{lemma}

\begin{proof}
For $z \in {\mathbb T}^d$ and $i \in \{1,\cdots,n\}$, 
we write $z e_i$ for the element of $[{\mathbb T}^d]^n$ whose coordinate $i$ is equal to 
$z$ and whose coordinates $j$, for $j \not = i$, are zero. 
Moreover, for a fixed $y \in {\mathbb T}^d$, with $y \not = 0$ and for a {function} $\varphi$ as in the statement, we let 
\begin{equation*}
\begin{split}
&\widetilde \varphi({\boldsymbol x}) := \varphi(x_1,\bar \mu^{n {-(1)}}_{{\boldsymbol x}}), \quad 
\Phi({\boldsymbol x})  {:=} \frac{\widetilde\varphi({{\boldsymbol x}}+ye_1) - \widetilde\varphi({{\boldsymbol x}})}{\vert y \vert^{ {\rho}}}, \quad {\boldsymbol x}=(x_1,\cdots,x_n) \in \bigl[ {\mathbb T}^d \bigr]^n,
\end{split}
\end{equation*}
For 
$i \in \{2,\cdots,n\}$ and $z \in {\mathbb T}^d$, we have 
\begin{equation}
\label{eq:annex:Phi:Psi}
\begin{split}
\Phi({\boldsymbol x}+ ze_i) - 
\Phi({\boldsymbol x}) 
&= \frac{ \widetilde \varphi({\boldsymbol x} +y e_1 + ze_i) -
\widetilde \varphi({\boldsymbol x} +y e_1)}{\vert y\vert^{ {\rho}}}
- 
 \frac{ \widetilde \varphi({\boldsymbol x}   + ze_i) -
\widetilde \varphi({\boldsymbol x}  )}{\vert y\vert^{ {\rho}}}
= 
\frac{\Psi({\boldsymbol x}   + ye_1,z) -
\Psi({\boldsymbol x}  ,z)}{\vert y\vert^{ {\rho}}},
\end{split}
\end{equation}
where
\begin{equation*}
\begin{split}
&\Psi({\boldsymbol x} ,z)
 =  \widetilde \varphi({\boldsymbol x}   + ze_i) -
\widetilde \varphi({\boldsymbol x}  ) 
 = \varphi\Bigl(x_1,\bar \mu^{n {-(1)}}_{{\boldsymbol x}}
+ 
\frac1{n-1} \bigl[ \delta_{x_i+z} 
-
\delta_{x_i} \bigr]
\Bigr)
-
\varphi(x_1,\bar \mu^{n {-(1)}}_{{\boldsymbol x}})
\\
&=
\int_0^1 
\frac{\ud \lambda}{n-1} 
\Bigl[ 
\frac{\delta \varphi}{\delta m}\Bigl(x_1,\bar \mu^{n {-(1)}}_{{\boldsymbol x}}
+ 
\frac{\lambda}{n-1} \bigl[ \delta_{x_i+z} 
-
 \delta_{x_i} \bigr]
\Bigr)(x_i+z)
-
\frac{\delta \varphi}{\delta m}\Bigl(x_1,\bar \mu^{n {-(1)}}_{{\boldsymbol x}}
+ 
\frac{\lambda}{n-1} \bigl[ \delta_{x_i+z} 
-
 \delta_{x_i} \bigr]
\Bigr)(x_i) \Bigr]. 
\end{split}
\end{equation*}
Next, we consider the PDE (in dimension $d \times n$)
\begin{equation*}
\partial_s u(s,{\boldsymbol x}) + \tfrac12 \Delta_{{\boldsymbol x}} u(s,{\boldsymbol x})  = 0, \quad \bigl(s,{\boldsymbol x}\bigr) \in [0,t] \times 
\bigl({\mathbb T}^d\bigr)^n \, ; 
\quad u(t,{\boldsymbol x}) = \Phi({\boldsymbol x}). 
\end{equation*}
Obviously, 
\begin{equation*}
u(s,{\boldsymbol x}) = \Phi * g^{(d \times n)}_{t-s}({\boldsymbol x}) 
= \frac1{\vert y \vert^{ {\rho}}} \bigl( \widetilde \varphi(\cdot + ye_1 ) 
 -  \widetilde \varphi(\cdot) \bigr) * g_{t-s}^{(d \times n)}({\boldsymbol x})= 
\frac1{\vert y \vert^{ {\rho}}} \bigl( \widetilde \varphi * g^{(d \times n)}_{t-s}\bigl({\boldsymbol x}+  {ye_1} \bigr) -  \widetilde \varphi * g_{t-s}^{(d \times n)}({\boldsymbol x}) \bigr), 
\end{equation*}
where $g_r^{(d \times n)}$ is the usual Gaussian kernel at time $r>0$ in dimension $d \times n$. And then, 
%it is standard to prove that 
\begin{equation}
\label{eq:annex:gradient:x1}
\begin{split}
&\vert  u(s,{\boldsymbol x}) \vert \leq \frac{c_{ {\rho}}}{1 \wedge (t-s)^{ {\rho}/2}} 
%\vert y \vert^\epsilon
\sup_{m \in {\mathcal P}({\mathbb T}^d)}
\bigl\| \varphi(\cdot,m) \bigr\|_{0,\infty}, \quad \vert \nabla_{x_1} u(s,{\boldsymbol x}) \vert \leq \frac{c_{{\rho}}}{1 \wedge (t-s)^{(1+ {\rho})/2}} 
%\vert y \vert^\epsilon
\sup_{m \in {\mathcal P}({\mathbb T}^d)}
\bigl\| \varphi(\cdot,m) \bigr\|_{0,\infty},
\end{split}
\end{equation}
for a constant $c_{ {\rho}} \geq 0$ (only depending on $d$ and $ {\rho}$  {--and in particular, not depending on $n$--}). 
In order to provide a similar bound but for 
the gradients w.r.t. $(x_i)_{ {2 \leq i \leq n}}$, we  use 
\eqref{eq:annex:Phi:Psi}
which allows us to write
\begin{equation*}
\begin{split}
&u(s,{\boldsymbol x}+ze_i) 
-
u(s,{\boldsymbol x}) 
=  
%\frac1{n-1} 
\frac1{\vert y \vert^{ {\rho}}} \Psi(\cdot,z) *\Bigl(  g^{(d \times n)}_{t-s}\bigl({\boldsymbol x}+ ye_1\bigr) -   g_{t-s}^{(d \times n)}({\boldsymbol x}) \Bigr), 
\end{split}
\end{equation*}
which yields in turn: 
\begin{equation*}
\Bigl\vert \nabla_{x_i} \Bigl[ u(s,{\boldsymbol x}+ze_i)
- u(s,{\boldsymbol x}) 
\Bigr] \Bigr\vert
  \leq \frac{c_\gamma}{(n-1)[ 1 \wedge (t-s)^{(1+{\rho})/2}]} 
% \vert y \vert^\epsilon
 \sup_{m \in {\mathcal P}({\mathbb T}^d)}
 \bigl\|\frac{\delta \varphi}{\delta m}(\cdot,m)(\cdot)\bigr\|_{ {0},\infty}, 
\end{equation*}
for $i \in \{ 2,\cdots,n\}$. 
Integrating in $z \in {\mathbb T}^d$,
we get 
\begin{equation}
\label{eq:annex:gradient:xi}
\Bigl\vert \nabla_{x_i} 
 u(s,{\boldsymbol x}) 
 \Bigr\vert
  \leq \frac{c_\rho}{(n-1) (t-s)^{(1+ {\rho})/2}} 
 %\vert y \vert^\epsilon
  \sup_{m \in {\mathcal P}({\mathbb T}^d)}
 \bigl\|\frac{\delta \varphi}{\delta m}(\cdot,m)(\cdot)\bigr\|_{ {0},\infty},  
\end{equation}
for $i \in \{ 2,\cdots,n\}$.
Now, we expand ${\mathbb E}[u(s,{\boldsymbol X}_s)]$ (w.r.t. $s$). We get 
\begin{equation*}
\begin{split}
{\mathbb E}\bigl[\Phi({\boldsymbol X}_t)\bigr] = u(0,{\boldsymbol X}_0) + \sum_{i=1}^n \int_0^t b^i(s,{\boldsymbol X}_s) \cdot \nabla_{x_i} u(s,{\boldsymbol X}_s) \ud s,
\end{split}
\end{equation*}
and 
then,  {using \eqref{eq:annex:gradient:x1} together with the boundedness of $b$}, we get 
\begin{equation*}
\Bigl\vert {\mathbb E}\bigl[\Phi({\boldsymbol X}_t)\bigr] 
\Bigr\vert \leq c_\rho %\vert y \vert^\epsilon 
\Bigl( \sup_{m \in {\mathcal P}({\mathbb T}^d)} \Bigl[ \| \varphi(\cdot,m) \|_{{0},\infty}
+
\bigl\| \frac{\delta \varphi}{\delta m}(\cdot,m)(\cdot) \bigr\|_{{0},\infty} \Bigr]
\Bigr)
  \Bigl( \frac{1}{1 \wedge t^{ {\rho}/2}}  +  t  \Bigr).
\end{equation*}
 {Using a change of variable to rewrite 
the above left-hand side in terms of the left-hand side appearing in the statement, 
we get the conclusion}. 
\end{proof}

 \section*{Acknowledgment}
 François Delarue acknowledges the financial support of the French ANR projet ANR-19-P3IA-0002 ``3IA C\^ote d'Azur - Nice - Interdisciplinary Institute for Artificial Intelligence'' (2019-22) and 
is now supported by the European Research Council (ERC) under the European Union’s Horizon 2020 research and innovation programme (ELISA project, Grant agreement No. 101054746).
{The research of Alvin Tse benefited from the support of the ``Chaire Risques Financiers'', Fondation du Risque.}

\end{section}

%%%%%%%%%%%%%%%%%%%%%%%%%%%%%%%%%%%%%%%%%%%%%%%%%%%%%%%%%%%%%%%%%%%%%%%%%%%%%%%%%%%%%%%%%%%%%%%%%%%%%%%%%%%%%%%%%%%
%\bibliographystyle{imsart-number}

\small
%\setstretch{1.0}

\bibliographystyle{plain} 

%%  \bibliography{<your bibdatabase>}%wmaainf}%alpha}%plain}
\bibliography{regularity} 

%%%%%%%%%%%%%%%%%%%%%%%%%%%%%%%%%%%%%%%%%%%%%%%%%%%%%%%%%%%%%%%%%%%%%%%%%%
%%%% \END BIBLIOGRAPHY
%%%%%%%%%%%%%%%%%%%%%%%%%%%%%%%%%%%%%%%%%%%%%%%%%%%%%%%%%%%%%%%%%%%%%%%%%%
\end{document}